\documentclass[twoside, 14pt]{article}
\usepackage{amsmath,amssymb,amsthm,mathrsfs}
\usepackage[margin=2.0cm]{geometry} 
\usepackage{lipsum}
\usepackage{titlesec,hyperref}
\usepackage{fancyhdr}
\usepackage[numbers,sort&compress]{natbib}
\usepackage{color}
\usepackage[titletoc]{appendix}

\fancypagestyle{plain}
{
	\fancyhf{}
	
}

\titleformat{\subsection}{\it}{\thesubsection.\enspace}{1.5pt}{}
\titleformat{\subsubsection}{\it}{\thesubsubsection.\enspace}{1.5pt}{}

\newtheorem{theo}{Theorem}[section]
\newtheorem{lemm}[theo]{Lemma}

\newtheorem{prop}[theo]{Proposition}
\newtheorem{rema}{Remark}[section]
\numberwithin{equation}{section}

\allowdisplaybreaks

\def\p{\partial}
\def\re{\rho_e}
\def\ue{u_e}

\def\i{\int_{\mathbb{T}}}
\def\j{\int_{0}^{+\infty}}

\def\f{\frac}
\def\dl{\delta}

\def\beq{\begin{equation}}
	\def\bal{\begin{aligned}}
		\def\dal{\end{aligned}}
	\def\deq{\end{equation}}
\def\beqq{\begin{equation*}}
	\def\deqq{\end{equation*}}

\def\vr{\varrho}
\def\vrf{\varrho_\infty}
\def\uf{u_\infty}
\def\xw{\langle y\rangle}
\def\gm{\gamma}
\def\sg{\sigma_1}
\def\p{\partial}
\def\al{\alpha}
\def\var{\varepsilon}
\def\x{\mathcal{X}}
\def\y{\mathcal{Y}}
\def \e{\mathcal{E}_1(t)}
\def \ea {\mathcal{E}_2(t)}
\def\se{\mathcal{X}(x)}

\def\sy{\mathcal{Y}(x)}
\def\br{\bar{\rho}}
\def\bde{\delta_1}
\def\sde{\delta_2}
\def\ssg{\sigma_2}
\def\bu{\bar{u}}
\def\tu{\tilde{u}}
\def\tv{\tilde{v}}
\def\tq{\tilde{q}}
\def\tr{\tilde{\rho}}
\def\hu{\hat{u}}
\def\hr{\hat{\rho}}
\def\hq{\hat{q}}

\begin{document}
	\title{Local existence and uniqueness of
		solution to the two-dimensional inhomogeneous Prandtl equations
		by energy method \hspace{-4mm}}
	\author{Jincheng Gao$^\dag$\quad Lianyun Peng$^\ddag$ \quad Zheng-an Yao$^\sharp$ \\[10pt]
		\small {School of Mathematics, Sun Yat-Sen University,}\\
		\small {510275, Guangzhou, P. R. China}\\[5pt]
	}

	\footnotetext{Email: \it $^\dag$gaojch5@mail.sysu.edu.cn,
		\it $^\ddag$pengly8@mail2.sysu.edu.cn,
		\it $^\sharp$mcsyao@mail.sysu.edu.cn}
	\date{}
	
	\maketitle

	\begin{abstract}
		{
			In this paper, we consider the local  existence and uniqueness result for the inhomogeneous Prandtl equations in dimension two by energy method.
			First of all, for the homogeneous case, the local-in-time
well-posedness theory of unsteady
			Prandtl equations was obtained by
			[Alexandre, Wang, Xu, Yang, J. Am. Math. Soc., 28 (3), 745-784 (2015)]
and [Masmoudi, Wong, Comm. Pure Appl. Math., 68 (10), 1683-1741 (2015)]
			independently by energy method without any transformation.
			However, for the inhomogeneous case, the appearance of density will create
			some new difficulties for us to overcome the loss of tangential derivative
			of horizontal velocity.
			Thus, our first result is to overcome the loss of tangential derivative
			such that one can establish the local-in-time well-posedness result for the
			inhomogeneous Prandtl equations by energy method.
			Secondly, for the homogeneous case, the local-in-$x$
 well-posedness in higher regular space
			for the steady Prandtl equations was obtained by [Guo, Iyer, Comm. Math. Phys., 382 (3), 1403-447 (2021)] by energy method since
			they firstly found the good quantity(called `quotient').
            With the help of this quotient, our second result is to
            establish the local-in-$x$ well-posedness in higher regular Sobolev space
            for the steady inhomogeneous Prandtl equations.
		}
		
		\vspace*{5pt}
	\end{abstract}

	\tableofcontents

	\section{Introduction}
	
	In this paper, we are concerned with the initial-boundary value problem for the
	following inhomogeneous Prandtl equations in a periodic domain $\Omega \overset{def}{=} \{(x, y)|(x, y) \in \mathbb{T}\times \mathbb{R}^+\}$:
	\beq\label{Prandtl}
	\left\{\begin{aligned}
		&\p_t \rho+u \p_x \rho+ v \p_y \rho=0,\\
		&\rho \p_t u+\rho u \p_x u+\rho v \p_y u-\p_y^2 u+\p_x p=0,\\
		&\p_x u+\p_y v=0,\\
		&(u, v)|_{y=0}=0,
		\underset{y\rightarrow +\infty}{\lim}(\rho, u)(t,x, y)
		=(\re(t,x),\ue(t,x)),\\
		&(\rho, u)|_{t=0}=(\rho_0, u_{0})(x,y).
	\end{aligned}\right.
	\deq
	where $\rho$ denotes the density, $(u, v)$ denotes the fluid horizontal
	and vertical velocity field, $p$ denotes the pressure.
	The known function $(\re(t,x), \ue(t,x), p(t,x))$
	satisfies the Bernoulli law:
	\beq\label{Bernoulli}
	\left\{\begin{aligned}
		&\p_t \re +\ue \p_x \re=0,\\
		&\re \p_t \ue+\re \ue \p_x \ue+\p_x p=0.
	\end{aligned}\right.
	\deq
	The Prandtl equations \eqref{Prandtl}-\eqref{Bernoulli} can be derived from the inhomogeneous Navier-Stokes equations, which can be found in Appendix \ref{appendix-derivation} in detail.
	
	When the flow is steady, the system \eqref{Prandtl} is reduced to the following equations:
	\beq\label{s-Prandtl}
	\left\{\begin{aligned}
		&u \p_x \rho+ v \p_y \rho=0,\\
		&\rho u \p_x u+\rho v \p_y u-\p_y^2 u+\p_x p=0,\\
		&\p_x u+\p_y v=0,\\
		&(u, v)|_{y=0}=0,
		\underset{y\rightarrow +\infty}{\lim}(\rho, u)(x, y)
		=(\re(x),\ue(x)),\\
		&(\rho, u)|_{x=0}=(\rho_0, u_{0})(y).
	\end{aligned}\right.
	\deq	
	In this paper, we deal with the case
	$(\re(t, x), \ue(t, x))=(\varrho_\infty, u_\infty)$, where $\varrho_\infty$ and $u_\infty$ are constants. It means $\p_x p=0$.
	We now review some related work to the system \eqref{Prandtl} and \eqref{s-Prandtl} in the sequence.
	
	\textbf{(I) Some results for unsteady Prandtl equations}
	
	The vanishing viscosity limit of the incompressible Navier-Stokes equations that,
	in a domain with Dirichlet boundary condition, is an important problem in
	both physics and mathematics.
	As the viscosity coefficient tends to zero, the solution undergoes a sharp transition
	from a solution of the Euler system to the zero non-slip boundary condition on boundary of the
	Navier-Stokes system. This sharp transition will lead to the formation of the boundary layer.
	Indeed, Prandtl \cite{Prandtl} derived the Prandtl equations for boundary layer
	from the incompressible Navier-Stokes equations with non-slip boundary condition. If the density in \eqref{Prandtl} is a constant, this system is the classical Prandtl equations.
	Now, let us introduce the related matter of well-posedness results for the Prandtl equation.
	If the tangential velocity field in the normal direction to the boundary satisfies
	the monotonicity condition, Oleinik \cite{{Oleinik1999},{Oleinik4}} applied the Crocco transform
	to establish the global-in-time regular solutions on $[0, L]\times \mathbb{R}^+$
	for small $L$, and local-in-time solutions on $[0, L]\times \mathbb{R}^+$ for arbitrary
	large by finite $L$.
	Under the monotonicity condition and a favorable pressure gradient of the Euler flow, the
	global-in-time weak solutions were obtained in \cite{Xin-Zhang-2004} for arbitrarily $L$ by Xin and Zhang. Very recently, Xin, Zhang and Zhao \cite{Xin2022} proved that such a weak solution is unique and in fact is a classical solution in some sense.
	It should be noted that all of the above results are achieved by using Crocco transformations.
	By taking care of the cancellation in the convection term to overcome the
	loss of derivative in the tangential direction of velocity, the researchers
	in \cite{Xu-Yang-Xu} and \cite{Masmoudi} independently used the simply
	energy method to establish well-posedness theory for the
	two-dimensional Prandtl equations in the framework of Sobolev space.
	For more results in this direction, the interested readers can refer to
	the well-posedness results in the analytic or Gevrey setting without monotonicity
	\cite{{Sammartino-Caflisch1},{Sammartino-Caflisch2},{Kukavica-Vicol-2013},{Li-Yang2020},{Li-Yang2017},{Li-Wu-2016},
		{Lombardo-Cannone-Sammartino-2003},{Ignatova-Vicol-2016},{David-Masmoudi-2015},{Paicu2021}},
	ill-posedness results in the Sobolev setting without monotonicity
	\cite{{David-Dormy-2010},{David-Nguyen-2012}, {Liu-Yang-JMPA}}, generic invalidity of boundary
	layer expansions in the Sobolev spaces \cite{{Grenier-Guo-Nguyen-2015},{Grenier-Guo-Nguyen-2016-01},
		{Grenier-Guo-Nguyen-2016-02},{Guo-Nguyen-2011},
		{Grenier-Nguyen-2017},{Grenier-Nguyen-2018},{Grenier-Nguyen-2019}}
	and references therein.
	It should be pointed out that most works focus on two-dimensional case, and there are only a few results in the three-dimensional case such as \cite{Li-Masmoudi-2022, Liu-Wang-Yang2016,Liu-Wang-Yang2017,Luo-Xin2018}.
	
	\textbf{(II) Some results for steady Prandtl equations}
	
	If the density in \eqref{s-Prandtl} is  constant, this system is the classical steady Prandtl equations. It should be pointed out that the boundary conditions
$\eqref{s-Prandtl}_4$ are in themselves not enough to determine uniquely a solution
of Prandtl equations. In order to obtain the well-set problem, one should suppose
the initial data condition at the initial position $x=0$.
The initial problem of steady Prandtl equations is analogous to the initial
value problem in the theory of parabolic partial differential equations.
Thus, the coordinate $x$ plays the role of time and the velocity component
$u$ plays the role of a temperature.
The local well-posedness result
	was obtained by Oleinik \cite{Oleinik3}, who used the Von-Mises transformation
	and maximum principle. Due to the degenerate property  of horizontal velocity
	near the boundary, it is hard to obtain the higher regularity for the steady
	Prandtl equation. This problem was settled by Guo and Iyer \cite{Guo-Iyer-2021}
	since they found the good unknown quantity (called ``quotient") to establish the closed energy estimate.
	By the way, the estimates for the ``quotient" play an important role in establishing the validity of the Prandtl layer expansion, see \cite{Guo-Iyer-2023-1, Iyer-Masmoudi-2020}.
	And \cite{Wang-Zhang-2021} established the global $C^{\infty}$ regularity by using further maximum principle techniques despite the degeneracy near the boundary.
	For the incompressible steady Navier-Stokes equations, Guo and Nguyen \cite{Guo-Nguyen2} justified the boundary layer expansion for the flow with a non-slip boundary
	condition on a moving plate.
	This result has been extended to the case of a rotating disk and to the case
	of nonshear Euler flows(\cite{{Iyer-ARMA},{Iyer-SIAM}}).
	Recently, Guo and Iyer \cite{Guo-Iyer} studied the boundary layer expansion for the small viscous flows with the classical no slip boundary conditions or on the static plate.
	This work was extended to the case of global theory in the $x-$variable
	for a large class of boundary layer with sharp decay rates.
	For more results about the boundary layer expansion, the reader should consult
	\cite{{Li-Ding-2020},{Gao-Zhang-2020},{Gerard-Varet-2019-ARMA}}.
	In terms of the asymptotic behavior of the solution, Serrin \cite{Serrin-1967}
	used the maximum principle techniques to single out that similarity solution as those which asymptotically develop downstream, whatever may be the state of motion at the initial position at $x=0$.
	In the case of localized data near the Blasius solution, Iyer \cite{Iyer-2020-ARMA}
	applied the energy method instead of maximum principle method for the good known quantity to establish specific convergence rate.
	Recently Wang and Zhang \cite{Wang-Zhang-2023} used the maximum principle techniques to prove the explicit decay estimate for general initial data with exponential decay.
	There are many results on the steady Prandtl equations, for the dynamic stability in \cite{Guo-Wang-Zhang-2023}, for the separation boundary layer in \cite{Dalibard-2019,Shen-Wang-Zhang-2021}.

	To the best of authors' knowledge,
there are no results concerning the well-posedness of the inhomogeneous
Prandtl equations \eqref{Prandtl} and \eqref{s-Prandtl}.
	\textit{Thus, our purpose in this paper is to establish the well-posedness result for the inhomogeneous Prandtl equations \eqref{Prandtl} and \eqref{s-Prandtl} respectively.} Motivated by \cite{Masmoudi}, our first target is to use the energy method to investigate the local-in-time well-posedness  for the inhomogeneous unsteady Prandtl equations \eqref{Prandtl}.
Motivated by the work on the homogeneous case \cite{Guo-Iyer-2021}, our second target is to establish the local-in-$x$ well-posedness result in higher regular Sobolev space for the inhomogeneous steady Prandtl equations \eqref{s-Prandtl}.
	Before we state our results, we give some definitions at first.
	
	\textbf{Definitions.}
	Now, let us define the weighted $H^s$ norm  $\|\cdot\|_{H^{s}_{\gamma}}$ ,  $\|\cdot\|_{H^{s}_{y,\gamma}}$ and
	$\|\cdot\|_{\bar{H}^{s}_{y,\gamma}}$ as follows,
	\begin{align*}
		&\|f\|_{H^{s}_\gamma}^2 \overset{def}{=}
		\sum_{|\alpha|\le s}\|\xw^{\gamma+\alpha_2}\p^\alpha f
		\|_{L^2(\mathbb{T}\times \mathbb{R}^+)}^2;\\
		&\|f\|_{H^{s}_{y,\gamma}}^2\overset{def}{=}
		\sum_{|\alpha|\le s}\|\xw^{\gamma}\p^\alpha f
		\|_{L^2(\mathbb{R}^+)}^2;\\
		&\|f\|_{\bar{H}^{s}_{y,\gamma}}^2\overset{def}{=}
		\sum_{\alpha_2 \le s}\|\xw^{\gamma}\p_y^{\alpha_2} f \|_{L^2(\mathbb{R}^+)}^2,
	\end{align*}	
	where $\alpha=(\alpha_1, \alpha_2)$, $|\alpha| =\alpha_1+\alpha_2$, $\p^\alpha \overset{def}{=} \p_x^{\alpha_1} \p_y^{\al_2}$ and the weight is defined by $\langle y \rangle \overset{def}{=} 1 + y$.
	Notice that $H_{0}^m \neq H^m$, where $H^m=H^m(\mathbb{T}\times \mathbb{R}^+)$ is the classical Sobolev space.
	For any $\beta=(\beta_1,\beta_2)$, set $\beta \le \alpha(\text{or} ~\beta < \alpha)$ means $|\beta| \le |\alpha|( \text{or} ~|\beta| < |\alpha|)$ and $\beta_1 \le \alpha_1, \beta_2 \le \alpha_2$.
	We shall use the notation
	$C_{\alpha}^{\beta}\overset{def}{=}C_{\alpha_1}^{\beta_1} C_{\alpha_2}^{\beta_2}$ to denote the quadratic coefficient, $[A, B] \overset{def}{=} AB - B A$ to denote the commutator between $A$ and $B$. $C_{D,E}$ means the constant $C$ depends on the parameters $D$ and $E$. For the sake of simplicity, let us set
$
	\int f(x,y) dxdy \overset{def}{=} \j\i  f(x,y)dxdy.
$
	
	First of all, for the inhomogeneous unsteady Prandtl equations \eqref{Prandtl}, let us denote
	$w(t,x,y) \overset{def}{=} \p_y u(t,x,y)$, $ \vr(t,x,y)\overset{def}{=}\rho(t,x,y)-\varrho_\infty$
	and $w|_{t=0}=w_0
	\overset{def}{=}\p_y u_{0}(x,y)$.
	Define the energy norm
	\beqq
	\begin{aligned}
		\mathcal{E}(t) \overset{def}{=}
		&\sum_{|\alpha|\le 6, \alpha_1 \le 5}
		\|\p^\alpha w \xw^{\gm+\alpha_2}\|_{L^2}^2
		+\sum_{|\alpha|\le 6, \alpha_1 \le 5}
		\|\p^\alpha \vr \xw^{\sg+\alpha_2}\|_{L^2}^2\\
		&+\|\vr_g \xw^{\gm}\|_{L^2}^2+\|w_g\xw^{\gm}\|_{L^2}^2
		+\sum_{|\alpha|\le 2}\! \xw^{2\sg+2\alpha_2}|\p^\alpha w|^2,
	\end{aligned}
	\deqq
	and the dissipation norm
	\beqq
	\begin{aligned}
		\mathcal{D}(t) \overset{def}{=}& \sum_{|\alpha|\le 6, \alpha_1 \le 5}
		\|\p_y \p^\alpha w \xw^{\gm+\alpha_2}\|_{L^2}^2
		+\|\p_y w_g \xw^{\gm} \|_{L^2}^2,
	\end{aligned}
	\deqq
	where
	\beqq
	(g_w, g_\vr)\overset{def}{=}(\frac{\p_y w}{w}, \frac{\p_y \vr}{w}),\quad
	(\vr_g, w_g)\overset{def}{=}
	(\p_x^6 \vr-g_\vr \p_x^6 u, \p_x^6 w-g_w \p_x^6 u).
	\deqq	
	Therefore, our first result can be stated as follows:
	\begin{theo}[Well-posedness result of the unsteady Prandtl equation]\label{main-result-un}
		Let the constants satisfy $\gm > \frac32$, $\gm+\frac12<\sg \leq 2\gm-1$, $\bde>0$ and $\kappa_2 \ge 4 \kappa_1>0$.
		For positive constants $(\varrho_\infty, u_\infty)$, suppose
		the outer Euler flow $(\re(t, x), \ue(t, x))=(\varrho_\infty, u_\infty)$.
		Assume that the initial data satisfy
		$\mathcal{E}(0)<+\infty$
		and
		$$
		\quad w_0(x,y) \xw^{\sg} \ge 2 \bde,
		\quad 2\kappa_1 \le \rho_0(x,y) \le \frac{\kappa_2}{2},
		$$
		for all $(x,y) \in  \mathbb{T}\times \mathbb{R}^+$.
		Then, there exist a time
		$0<T_0<<1$
		(depending on initial data $\mathcal{E}(0)$ and parameters $\bde$,  $\kappa_1$, $\kappa_2$, $\gm$, $\sg$)
		and a unique classical solution $(\rho, u, v)$ to the inhomogeneous
		unsteady Prandtl equations \eqref{Prandtl} satisfying
		\beq\label{estimate-w,rho}
		w(t,x,y) \xw^{\sg}\ge \bde, \quad  \quad \kappa_1 \leq \rho(t,x,y) \le \kappa_2,
		\deq
		for all $(t,x,y) \in [0,T_0]\times \mathbb{T}\times \mathbb{R}^+$ and
		\beq\label{estimate-X}
		\begin{aligned}
			\underset{t\in [0, T_0]}{\sup} \mathcal{E}(t)
			+ \int_0^{T_0}\mathcal{D}(t)d t
			\le 2 \mathcal{E}(0).
		\end{aligned}
		\deq
	\end{theo}

	\begin{rema}
		Due to the relation $w \xw^\gamma \in L^2$ and $w \xw^{\sg}\ge \bde$,
		we need to require $\sg$ and $\gamma$ satisfy the relation $\sg>\gamma+\frac12$.
		In order to control the nonlinear term in density equation(see estimate \eqref{2304}),
		we require $\sg \le 2\gamma-1$. Thus, the weight index $\sg$ should satisfy
		$\gamma+\frac12<\sg\le 2\gamma-1$.
	\end{rema}
	
	\begin{rema}
		The outer constant Euler flow will give rise to the important boundary condition
		$\p_y w|_{y=0}=0$. This boundary condition will play an important role
		as we deal with the nonlinear term $I_{4}$(see \eqref{2210}).
		Therefore, the local-in-time well-posedness in Theorem \ref{main-result-un}
		holds only for the case of outer constant Euler flow.
	\end{rema}

	\begin{rema}
		For the sake of simplicity, the local-in-time existence and uniqueness results
		in Theorem \ref{main-result-un} are established only under the $H^6$-framework.
		However, it should be pointed out that one can establish
		similar result in more higher regular Sobolev space.
	\end{rema}

	Next, for the inhomogeneous steady Prandtl equations \eqref{s-Prandtl}, we denote $\br(x,y)\overset{def}{=}\rho(x,y)-\varrho_{\infty}$ and $\bu(x,y)\overset{def}{=}u(x,y)-u_{\infty}$.
	Let us define the energy norm
	\beqq
	\begin{aligned}
		\mathcal{X}(x)\overset{def}{=}
		\sum_{|\alpha|\le m}\|\p^\alpha \br \|_{L_y^2}^2
		+\sum_{|\alpha|\le m}
		\|\sqrt{\rho} u \p^{\al}\p_y (\f{v}{u}) \xw^{\ssg}\|_{L_y^2}^2
		+\|\bu \|_{L_y^2}^2+\|\p_y \bu \|_{L_y^2}^2,
	\end{aligned}
	\deqq
	and the dissipation norm
	\beqq
	\mathcal{Y}(x)\overset{def}{=}
	\sum_{|\alpha|\le m}\|\sqrt{u}\p_y^2 \p^\alpha(\f{v}{u}) \xw^{\ssg}\|_{L_y^2}^2.
	\deqq
	
	Our second result can be stated as follows:
	\begin{theo}
		[Well-posedness result of the steady Prandtl equation]\label{main-result-steady}
		Let the constants satisfy $m \ge 3$, $\ssg \ge 2$ and  $\kappa_3>0$.
		For positive constants $(\varrho_\infty, u_\infty)$, suppose
		the outer Euler flow $(\re(x), \ue(x))=(\varrho_\infty, u_\infty)$.
		Assume $(\rho_0(y)-\varrho_\infty, u_0(y)-u_\infty)\in C^\infty(\mathbb{R}^+)$
		decays fast and satisfies the following conditions:
		\beq\label{stherom-intial-u}
		u_0(y)>0 ~~\text{for~all}~ y>0, \quad u_0(0)=0,   \quad \lim_{y \rightarrow+\infty} u_0(y)= u_\infty, \quad u_0'(0)>0, \quad u_0''(0)=0,
		\deq
		and for all $y \in \mathbb{R}^+$,
		\beq\label{density-assumption}
		\rho_0(y) \geq \kappa_3.
		\deq
		Suppose also the generic compatibility conditions hold at the corner $(0, 0)$
		up to order $m$. Then there exists a time $0<L_a<<1$
		(depending on initial data $\rho_0$, $u_0$ and parameters $m$, $\ssg$, $\kappa_3$)
		and a unique classical solution $(\rho, u, v)$ to the inhomogeneous
		steady Prandtl equations \eqref{s-Prandtl} satisfying
		\beq \label{estimate-s-rho}
		u(x,y)>0, \quad \rho(x,y) \ge \kappa_3,
		\deq
		for all $(x,y) \in [0,L_a]\times \mathbb{R}^+$ and
		\beq\label{estimate-uX}
		\begin{aligned}
			\underset{x \in [0, L_a]}{\sup} \se+ \int_0^{L_a}{\sy}d x
			\le 2C_{\rho_0, u_0}.
		\end{aligned}
		\deq
	\end{theo}
	
	\begin{rema}
		Although the initial data of vertical velocity is not given,
		one can give the initial control of $\mathcal{X}(0)$ in terms of $\rho_0$ and $u_0$.
		To achieve the target, we need the generic compatibility conditions hold at the
		corner $(0, 0)$ up to order $m$, see Appendix \ref{appendix-com} in detail.
		Under the conditions of Theorem \ref{main-result-steady},
		then we can apply the equation \eqref{s-Prandtl} to obtain the following estimates:
		\beqq
		\|u-u_{\infty}\|_{L^\infty_x L^2_y}
		+\|\p_y u \|_{L^\infty_x L^2_y}
		+\|\p_x^{\alpha_1} \p_y^{\alpha_2}  \p_y^2 u \xw^{\ssg}\|_{L^\infty_x L^2_y}
		+\|\p_x^{\alpha_1} \p_y^{\alpha_2} (\rho-\varrho_{\infty})\|_{L^\infty_x L^2_y}
		\le C_{m,\ssg,\kappa_3, \rho_0, u_0},
		\deqq
		for all $0\le \alpha_1+\alpha_2 \le m$.
	\end{rema}

	The rest of the paper is organized as follows.
	In Section \ref{section-approach}, we explain the difficulties and our approach to establish the well-posedness theory for the inhomogeneous Prandtl equations \eqref{Prandtl} and \eqref{s-Prandtl} respectively.
	In Section  \ref{prior-estimate} and \ref{well-posedness}, we prove the local-in-time well-posedness of \eqref{Prandtl} in weighted Sobolev spaces by energy methods. Precisely, in Section \ref{prior-estimate}, we  derive a priori estimate and
	obtain the closed energy estimates of the regularized inhomogeneous Prandtl system \eqref{Prandtl-01}.
	Finally, the well-posedness result for the original Prandtl equations \eqref{Prandtl} will be
	investigated in Section \ref{well-posedness}.
	In Sections \ref{sec:priori-steady} and \ref{sec:well-posedness-steady}, the 
local-in-$x$ well-posedness of the inhomogeneous steady Prandtl equations \eqref{s-Prandtl} in higher regular Sobolev space is established.
	
	\section{Difficulties and outline of our approach}\label{section-approach}
	In this section, we will explain the main difficulties of proving Theorems \ref{main-result-un} and  \ref{main-result-steady} as well as our strategies for overcoming them.
	
	\subsection{Difficulties and outline of our approach for the inhomogeneous unsteady  Prandtl system}
	First of all, similar to the classical Prandtl equations, the main difficulty of Prandtl system \eqref{Prandtl} comes from the loss of tangential derivative
	of horizontal velocity, therefore the standard energy estimate can not apply.
	Without using the classical Crocco transformation as well as other change of variables, Masmoudi and Wong \cite{Masmoudi} introduce a new weighted $H^s$-norm to avoid the regularity loss created by the vertical velocity $v$ through a cancellation property.
	Since the outer Euler flow $(\rho_e, u_e)$ is independent of variables $x$ and $t$,
	motivated by their idea, we consider the following approximated system
	\beq\label{Prandtl-01}
	\left\{\begin{aligned}
		&\p_t \vr^\var+u^\var \p_x \vr^\var+ v^\var \p_y \vr^\var-\var\p_x^2 \vr^\var=0,\\
		&\p_t u^\var+u^\var \p_x u^\var+v^\var \p_y u^\var
		-\var \p_x^2 u^\var
		-\frac{1}{\vr^\var+\vrf}\p_y^2 u^\var=0,\\
		&\p_x u^\var+\p_y v^\var=0,\\
		&(u^\var, v^\var)|_{y=0}=0,
		\underset{y\rightarrow +\infty}{\lim}(\vr^\var, u^\var)(t,x,y)=(0, \uf),\\
		&(\vr^\var, u^\var)|_{t=0}
		\overset{def}{=}(\rho_0-\vrf, u_{0})(x,y),
	\end{aligned}\right.
	\deq
	where $\vr^\var \overset{def}{=}\rho^\var-\vrf$. Let us define $w^\var \overset{def}{=}\p_y u^\var$,
	and taking $\partial_y$-operator to the equation $\eqref{Prandtl-01}_2$, we have
	\beq\label{Prandtl-02}
	\left\{\begin{aligned}
		&\p_t w^\var+u^\var \p_x w^\var+v^\var \p_y w^\var
		-\var \p_x^2 w^\var
		-\p_y(\frac{1}{\vr^\var+\vrf}\p_y w^\var)=0,\\
		&\p_y w^\var |_{y=0}=0,\\
		&w^\var|_{t=0}
		\overset{def}{=}\p_y u_{0}(x,y).
	\end{aligned}\right.
	\deq
	It should be pointed out that the nonlinear terms $v^\var \p_y \vr^\var$
	and $v^\var \p_y w^\var$ will create loss of tangential derivative.
	Similar to \cite{Masmoudi}, taking $\partial_x^6$-operator to the density equation
	and velocity equation respectively, we have
	\beq\label{0201}
	(\p_t+u^\var \p_x +v^\var \p_y)\p_x^6 \vr^\var
	+\p_x^6 v^\var \p_y \vr^\var
	=\var\p_x^2 \p_x^6 \vr^\var+...,
	\deq
	and
	\beq\label{0202}
	(\p_t+u^\var \p_x +v^\var \p_y)\p_x^6 u^\var
	+\p_x^6 v^\var w^\var=\var \p_x^8 u^\var+\p_x^6 \{\frac{1}{\vr^\var+\vrf}\p_y w^\var\}+...,
	\deq
	where the symbol ... represents the lower-order terms that
	we want the reader to ignore at this moment.
	The main obstacle in both \eqref{0201} and \eqref{0202} is the term $\p_x^6 v^\var=-\p_x^6 \p_y^{-1}u^\var$, which has
	7 $x$-derivatives so that the standard energy estimates cannot apply.
	However,  we can eliminate the problematic term by subtracting them
	in an appropriate way by using the combination of equations \eqref{0201} and \eqref{0202}.
	Subtracting $g_\vr^\var(\overset{def}{=}\frac{\p_y \vr^\var}{w^\var})\times$ \eqref{0202} from \eqref{0201}, we obtain
	\beqq
	\p_t \vr_g^\var+u^\var \p_x \vr_g^\var+v^\var \p_y \vr_g^\var-\var \p_x^2 \vr_g^\var
	=2\var\p_x^7 u^\var \p_x g_\vr^\var+...,
	\deqq
	where the function $\vr_g^\var\overset{def}{=}\p_x^6 \vr^\var-g_\vr^\var \p_x^6 u^\var$.
	This good quantity $\vr_g^\var$ will avoid the loss of tangential derivative.
	Similarly, in order to deal with the term $\p_x^6 v^\var w^\var$, we establish
	the following equation
	\beqq
	\p_t w_g^\var +u^\var \p_x w_g^\var +v^\var \p_y w_g^\var
	=\var \p_x^2 w_g^\var +\p_y\{\frac{1}{\vr^\var+\vrf}\p_y w_g^\var \}
	+...,
	\deqq
	where $w_g^\var \overset{def}{=} \p_x^6 w^\var-g_w^\var \p_x^6 u^\var$
	and $g_w^\var \overset{def}{=} \frac{\p_y w^\var}{w^\var}$.
	Instead of estimating the weighted norm of $(\p_x^6 \vr^\var, \p_x^6 w^\var)$,
	we estimate that of $(\vr_g^\var, w_g^\var)$ with the same weight $\gamma$ to close the energy estimate, which can avoid the loss of $x$-variable thanks to some cancellation mechanism.
	
	Secondly, since the domain $\mathbb{T} \times \mathbb{R}^+$ is unbounded,
	we use the Hardy-type inequalities (see \eqref{Hardy1} and \eqref{Hardy2})
	to control the quantities $u^\var$ and  $v^\var =-\int^{y}_0 \p_x u^\var dy'$
	by the vorticity $w^\var$.
	Since it will generate some space weight, we add the suitable weights respectively to vorticity and density, including the weight $\xw$ for each $y$-variable.
	In order to close the energy estimate, it remains to derive the weighted $L^{\infty}$ controls on the lower-order derivatives of $w^\var$ and $\vr^\var$
	with the same weight $\sg$.
	On the one hand, we view the equation of $\p^\alpha w^\var $ as a ``linear" parabolic equations and control the lower-order derivatives of $w^\var$
	by using the classical maximum principle.
	On the other hand, we use the Sobolev embedding inequality and good estimate
	of transport equation to control the lower-order derivatives
	of $\vr^\var$ with weight $\sg$.
	Then, we require the lower order derivative of density in energy norm
	is along with the weight $\sg$.
	This weight, as we deal with the term $I_8$(see estimate \eqref{2304}),
	needs us to require the index condition $\sg\le 2\gamma-1$.
	The weights should be chosen carefully so that
	we can establish the closed energy estimate.

	Finally, the lack of high-order boundary conditions prevents us from using the integration by part in the $y$-variable. However, when the highest order of derivative is even, one can use the trace estimate \eqref{trace} and the equation $\eqref{Prandtl-02}_1$ to reduce
	the order of derivatives (see the term $I_{6}$).
	However, due to the appearance of density, this procedure is somewhat complicated.
	Thus, we only consider the $H^6$-framework in our paper just for the sake of simplicity.
	Therefore, the local-in-time well-posedness for the original system \eqref{Prandtl} can be obtained directly as the parameter $\var$ tends to $0^+$.
	
	\subsection{Difficulties and outline of our approach for
		the  inhomogeneous steady Prandtl system}
	
	For the case of homogeneous flow,  the main difficulty of steady Prandtl
	equation arises from the loss of tangential derivative of horizontal velocity.
	Thus, with the help of Von-Mises transformation, Oleinik and Samokhin \cite{Oleinik1999}
	established the local-in-$x$ well-posedness in lower regular space.
	Due to the degenerate property of horizontal velocity on the boundary,
	it fails to establish the higher regularity of solution.
	In this respect, Guo and Iyer \cite{Guo-Iyer-2021} have settled this open problem since they found
	the good quotient $\p_y(\f{v}{\bar{u}})$. Here $\bar{u}$ is the classical Prandtl solution obtained by Oleinik and Samokhin \cite{Oleinik1999}.
	
	Now let us focus on the local-in-$x$ well-posedness of inhomogeneous
	steady Prandtl system \eqref{s-Prandtl}.
	Motivated by the idea by Guo and Iyer \cite{Guo-Iyer-2021}, we rewrite the original system
	as follows
	\beq\label{e:rewrite}
	\left\{
	\bal
	&\p_x \rho+ q \p_y \rho=0,\\
	&\p_x( \rho u^2 \p_y q)-\p_y^3 (u q) =0,\\
	&\p_x u+\p_y(u q)=0,
	\dal\right.
	\deq
	where the quotient $q \overset{def}{=} \frac{v}{u}$.
	From the Eqs. $\eqref{e:rewrite}_1$ and $\eqref{e:rewrite}_2$, one can establish some estimate
	for the density $\rho$ and quotient $\p_y(\frac{v}{u})$.
	However, in order to close the energy estimate, we need to provide estimate for
	the horizontal velocity $u$. Indeed, due to assumption of positive initial
	vorticity on the boundary, the horizontal velocity $u$ is equivalent to $y$
	near the boundary. Thus, it only needs to give the estimate for the far field
	component of horizontal velocity $u$.
	As mentioned by Serrin \cite{Serrin-1967}, the classical steady Prandtl equation
	is acting like an evolution equation, with $x$ being a time-like variable, $y$ being space-like.
	\textit{Thus, the divergence-free condition is essentially an evolution equation of
		horizontal velocity $u$.
	Then, one can establish the estimate for horizontal velocity $u$ to
	establish closed a priori energy estimate.}
	Finally, we point out that we need to control the term $q$ by $\p_y q$ by using Hardy inequality. Since it will generate one space weight, our method here is to add the suitable weight $\xw^{\ssg}$ to the quantity $\p_y q$.
	Therefore, we can establish the local-in-$x$ well-posedness of inhomogeneous
	steady Prandtl equation in higher regular space.
	
	\section{A priori estimate of the approximated unsteady
		Prandtl equations}\label{prior-estimate}
	
	In this section, we will establish a priori estimate for the
	regularized inhomogeneous Prandtl system \eqref{Prandtl-01}.
	First of all, we define the energy norms
	\beqq
	\begin{aligned}
		\e \overset{def}{=}
		&\sum_{|\alpha|\le 6, \alpha_1 \le 5}\|\p^\alpha w^\var \xw^{\gm+\alpha_2}\|_{L^2}^2
		+\sum_{|\alpha|\le 6, \alpha_1 \le 5}\|\p^\alpha \vr^\var \xw^{\sg+\alpha_2}\|_{L^2}^2\\
		&+\|\vr_g^\var \xw^{\gm}\|_{L^2}^2+\|w_g^\var \xw^{\gm}\|_{L^2}^2
		+\sum_{|\alpha|\le 2}|\xw^{\sg+\alpha_2}\p^\alpha w^\var|^2,
	\end{aligned}
	\deqq
	and
	\beqq
	\begin{aligned}
		\ea \overset{def}{=}
		&\sum_{|\alpha|\le 6, \alpha_1 \le 6}\|\p^\alpha w^\var \xw^{\gm+\alpha_2}\|_{L^2}^2
		+\sum_{|\alpha|\le 6, \alpha_1 \le 5}\|\p^\alpha \vr^\var \xw^{\sg+\alpha_2}\|_{L^2}^2\\
		&+\|\p_x^6 \vr^\var \xw^{\gm}\|_{L^2}^2
		+\sum_{|\alpha|\le 2}|\xw^{\sg+\alpha_2}\p^\alpha w^\var|^2.
	\end{aligned}
	\deqq
	Then we define the dissipation norm
	\beqq
	\begin{aligned}
		\mathcal{D}_1 (t)\overset{def}{=}&\sum_{|\alpha|\le 6, \alpha_1 \le 5}
		\var \|\p_x \p^\alpha w^\var \xw^{\gm+\alpha_2}\|_{L^2}^2 +\sum_{|\alpha|\le 6, \alpha_1 \le 5}
		\|\frac{1}{\sqrt{\vr^\var+\vrf}}\p_y \p^\alpha w^\var \xw^{\gm+\alpha_2}\|_{L^2}^2\\
		&+\sum_{|\alpha|\le 6, \alpha_1 \le 5}
		\var \|\p_x \p^\alpha \vr^\var \xw^{\sg+\alpha_2}\|_{L^2}^2
		+\var\|\p_x w_g^{\var} \xw^{\gm} \|_{L^2}^2
		+\|\frac{1}{\sqrt{\vr^\var+\vrf}}\p_y w_g^{\var} \xw^{\gm} \|_{L^2}^2
		+\var\|\p_x \vr_g^{\var} \xw^{\gamma}\|_{L^2}^2.
	\end{aligned}
	\deqq
	
	Now, let us state the following Proposition.
	\begin{prop}\label{Main-Pro}
		Assume $\gm > \frac32, \gm+\frac12<\sg\leq2\gm-1$, $\dl>0$, $0 < \kappa \le \bar{\kappa}$
		and $\mathcal{E}_1(0)<+\infty$.
		Suppose
		\beq\label{Lowerb}
		w^{\var}\xw^{\sg}\ge \dl,
		\quad \kappa \leq \vr^\var+\vrf \le \bar{\kappa},
		\deq
		holds for all $(t, x, y)\in [0, T]\times \mathbb{T}\times \mathbb{R}^+$.
		Then, it holds for all $T_0 \le T$
		\beq\label{uiform-estimate}
		\begin{aligned}
			E_1(T_0)+\int_0^{T_0} \mathcal{D}_1 (\tau) d\tau
			\le &\; \mathcal{E}_1(0)
			+C_{\bar{\kappa},\kappa,\dl,\gm, \sg} (1+E_1(T_0)^{20})T_0 \\
			&\; + (\|I(0)\|_{L^\infty(\mathbb{T}\times \mathbb{R}^+)}
			+C_{\kappa} T_0 (1+E_1(T_0)^4))
			e^{C_{\sg,\kappa} (1+E_1(T_0)^2)T_0},
		\end{aligned}
		\deq
		where $E_1(T) \overset{def}{=} \sup_{0\le t \le T} \e$.
	\end{prop}

	\subsection{Weighted estimates with normal derivatives}\label{reformulation}
	
	Before we obtain the weighted estimates with normal derivatives for the vorticity and density, we need to investigate the estimates as follows, which will be used frequently to deal with the nonlinear term.
	\begin{lemm}\label{lemma21}
		Under the condition of \eqref{Lowerb},
		the solution $(u^\var, v^\var, w^\var)$ of regularized system
		\eqref{Prandtl-01} will satisfy
		\beq\label{2101}
		\ea \le C_{\gamma, \sg, \dl} (1+\e^3),
		\deq
		and
		\beq\label{2102}
		\|\p_y\p_x^6 w^\var \xw^\gm\|_{L^2}^2
		\le \|\p_y w_g^\var \xw^\gm\|_{L^2}^2+C_{\gamma, \sg, \dl}(1+\e^4).
		\deq
	\end{lemm}
	\begin{proof}
		First of all, it is easy to check that
		\beq\label{2103}
		w_g^\var=\p_x^6 w^\var-\frac{\p_y w^\var}{w^\var} \p_x^6 u^\var
		=w^\var \p_y(\frac{\p_x^6 u^\var}{w^\var}).
		\deq
		Using relation \eqref{2103}, the Hardy inequality \eqref{Hardy2}
		and lower bound condition \eqref{Lowerb}, it holds
		\beq\label{2105}
		\begin{aligned}
			\|\p_x^6 w^\var \xw^\gm\|_{L^2}
			&\le \|w_g^\var \xw^\gm\|_{L^2}
			+\|\p_y w^\var \xw^{\sg+1}\|_{L^\infty}
			\| \frac{\p_x^6 u^\var}{w^\var} \xw^{\gm-\sg-1}\|_{L^2}\\
			&\le \|w_g^\var \xw^\gm\|_{L^2}
			+C_{\gamma,\sg} \|\p_y w^\var \xw^{\sg+1}\|_{L^\infty} \|\p_y\left(\frac{\p_x^6 u^\var}{w^\var}\right) \xw^{\gm-\sg}\|_{L^2}
			\\
			&\le \|w_g^\var \xw^\gm\|_{L^2}
			+C_{\gamma,\sg} \|\p_y w^\var \xw^{\sg+1}\|_{L^\infty}
			\|\frac{w_g^\var}{w^\var} \xw^{\gm-\sg}\|_{L^2}
			\\
			&\le C_{\gamma, \sg, \dl}\|w_g^\var \xw^\gm\|_{L^2}
			(1+\|\p_y w^\var \xw^{\sg+1}\|_{L^\infty}).
		\end{aligned}
		\deq
		Due to the relation
		$\vr_g^\var=\p_x^6 \vr^\var-g_\vr^\var \p_x^6 u^\var
		=\p_x^6 \vr^\var-\frac{\p_y \vr^\var}{w^\var} g_\vr^\var \p_x^6 u^\var$,
		using the Hardy inequality \eqref{Hardy2}
		we can obtain
		\beq\label{2106}
		\begin{aligned}
			\|\p_x^6 \vr^\var \xw^\gm\|_{L^2}
			&\le \|\vr_g^\var \xw^\gm\|_{L^2}
			+\|g_\vr^\var \p_x^6 u^\var \xw^\gm\|_{L^2}\\
			&\le \|\vr_g^\var \xw^\gm\|_{L^2}
			+\frac{1}{\dl}\|\p_y \vr^\var \xw^{\sg+1}\|_{L^\infty}
			\| \p_x^6 u^\var \xw^{\gm-1}\|_{L^2}\\
			&\le \|\vr_g^\var \xw^\gm\|_{L^2}
			+C_{\gamma, \dl} \|\p_y \vr^\var \xw^{\sg+1}\|_{H^2}
			\| \p_x^6 w^\var \xw^{\gm}\|_{L^2}\\
			&\le \|\vr_g^\var \xw^\gm\|_{L^2} +C_{\gamma, \sg, \dl} \|\p_y \vr^\var \xw^{\sg+1}\|_{H^2}
			\|w_g^\var \xw^\gm\|_{L^2}
			(1+\|\p_y w^\var \xw^{\sg+1}\|_{L^\infty}),
		\end{aligned}
		\deq
		where we have used \eqref{2105} in the last inequality.
		The combination of estimates \eqref{2105} and \eqref{2106} yields directly
		\beqq
		\ea \le C_{\gamma, \sg, \dl} (1+\e^3).
		\deqq
		Finally, due to the relation \eqref{2103} and
		lower bound condition \eqref{Lowerb}, we have
		\beqq
		\begin{aligned}
			\|\p_y\p_x^6 w^\var \xw^\gm\|_{L^2}
			\le &\| \p_y (w_g^\var+\frac{\p_y w^\var}{w^\var}
			\p_x^6 u^\var)\xw^\gm\|_{L^2}\\
			\le &\;\|\p_y w_g^\var \xw^\gm\|_{L^2}
			+\|\frac{\p_y w^\var}{w^\var} \p_x^6 w^\var \xw^\gm\|_{L^2}
			+\|\frac{\p_y^2 w^\var w^\var-(\p_y w^\var)^2}
			{(w^\var)^2} \p_x^6 u^\var \xw^\gm\|_{L^2}\\
			\le & \;\|\p_y w_g^\var \xw^\gm\|_{L^2}
			+C_\dl \|\p_y w^\var \xw^{\sg+1}\|_{L^\infty}
			\|\p_x^6 w^\var \xw^\gm\|_{L^2}\\
			&\;+C_{\dl,\gamma}(\|\p_y^2 w^\var \xw^{\sg+2}\|_{L^\infty}
			+\|\p_y w^\var \xw^{\sg+1}\|_{L^\infty}^2)
			\|\p_x^6 w^\var \xw^{\gm}\|_{L^2},
		\end{aligned}
		\deqq
		which, together with  estimate \eqref{2105}, yields directly
		\beqq
		\begin{aligned}
			\|\p_y\p_x^6 w^\var \xw^\gm\|_{L^2}
			\le \|\p_y w_g^\var \xw^\gm\|_{L^2}+C_{\gamma, \sg, \dl}(1+\e^2).
		\end{aligned}
		\deqq
		Therefore, we complete the proof of this lemma.
	\end{proof}
	
	Now we first establish the weighted estimates with normal derivatives for the vorticity.
	\begin{lemm}\label{lemma22}
		Under the condition of \eqref{Lowerb},
		the solution $(u^\var, v^\var, w^\var)$ of regularized system
		\eqref{Prandtl-01} will satisfy
		\beq\label{2201}
		\begin{aligned}
			&\frac{d}{dt}\!\! \sum_{|\alpha|\le 6, \alpha_1 \le 5}
			\|\p^\alpha w^\var \xw^{\gm+\alpha_2}\|_{L^2}^2
			+\!\!\!\!\sum_{|\alpha|\le 6, \alpha_1 \le 5}
			\var \|\p_x \p^\alpha w^\var \xw^{\gm+\alpha_2}\|_{L^2}^2
			+\!\!\!\!\sum_{|\alpha|\le 6, \alpha_1 \le 5}
			\|\frac{1}{\sqrt{\vr^{\var}+\vrf}}\p_y \p^\alpha w^\var \xw^{\gm+\alpha_2}\|_{L^2}^2\\
			\le &~ \nu_1 \|\p_y \p^\alpha w^\var \xw^{\gm+\al_2}\|_{L^2}^2 +\nu_2 \|\p_y^2 w^\var\|_{H^5}^2 +C_{\nu_1,  \nu_2, \kappa,\gamma, \sg, \dl}(1+\e^{20}),
		\end{aligned}
		\deq
		where $\nu_1$ and $\nu_2$ are the small constants that will be chosen later.
	\end{lemm}
	\begin{proof}
		For any $|\alpha|\le 6$ and $\alpha_1\le 5$,
		applying the $\p^\al$ differential operator to the
		equation $\eqref{Prandtl-02}_1$, we have
		\beq\label{2202}
		\begin{aligned}
			&\p_t(\p^\al w^\var)+u^\var \p_x (\p^\al w^\var)+v^\var \p_y (\p^\al w^\var)
			-\var \p_x^2(\p^\al w^\var)
			-\p_y\left\{\frac{1}{\vr^{\var}+\vrf}\p_y \p^\al w^\var \right\}\\
			=&-[\p^\al, u^\var]\p_x w^\var-[\p^\al, v^\var]\p_y w^\var
			+\p_y\left\{\p^\alpha (\frac{1}{\vr^{\var}+\vrf})\p_y w^\var\right\}+ \p_y\left\{
			\sum_{0<\beta<\alpha}C^\alpha_\beta
			\p^\beta (\frac{1}{\vr^{\var}+\vrf})\p^{\alpha-\beta}\p_y w^\var\right\}.
		\end{aligned}
		\deq
		Multiplying \eqref{2202} by $\p^\alpha w^\var \xw^{2\gm+2\al_2}$
		and integrating over $\mathbb{T}\times \mathbb{R}^+$, we have
		\beq\label{2203}
		\frac{d}{dt}\frac{1}{2}\int |\p^\alpha w^\var |^2 \xw^{2\gm+2\al_2}dxdy
		+\var\int |\p_x \p^\alpha w^\var |^2 \xw^{2\gm+2\al_2}dxdy
		=\sum_{1\le i \le 6}I_{i},
		\deq
		where the terms $I_i(1\le i \le 6)$ are defined as follows:
		\beqq
		\begin{aligned}
			&I_{1}=(\gm+\al_2)\int u^\var_2 |\p^\al w^\var|^2 \xw^{2\gm+2\al_2-1} dxdy;\\
			&I_{2}=-\int [\p^\al, u^\var]\p_x w^\var \cdot \p^\alpha w^\var \xw^{2\gm+2\al_2} dxdy;\\
			&I_{3}=-\int [\p^\al, v^\var]\p_y w^\var \cdot \p^\alpha w^\var \xw^{2\gm+2\al_2} dxdy;\\
			&I_{4}=\int \p_y\left\{\p^\alpha (\frac{1}{\vr^{\var}+\vrf})\p_y w^\var\right\}
			\cdot \p^\alpha w^\var \xw^{2\gm+2\al_2} dxdy;\\
			&I_{5}=\sum_{0<\beta<\alpha}C^\alpha_\beta
			\int  \p_y\left\{\p^\beta (\frac{1}{\vr^{\var}+\vrf})\p^{\alpha-\beta}\p_y w^\var\right\}
			\cdot \p^\alpha w^\var \xw^{2\gm+2\al_2} dxdy; \\
			&I_{6}=\int \p_y\left\{\frac{1}{\vr^{\var}+\vrf}\p_y \p^\al w^\var \right\}
			\cdot \p^\alpha w^\var \xw^{2\gm+2\al_2} dxdy.
		\end{aligned}
		\deqq
		In the sequence, we will give the estimate from terms $I_1$ to $I_6$.
		\textbf{Deal with term $I_{1}$.}
		It is easy to check that
		\beq\label{2204}
		|I_{1}|\le C_\gm\|v^\var \xw^{-1}\|_{L^\infty}\|\p^\al w^\var \xw^{\gm+\al_2}\|_{L^2}^2
		\le C_\gm \|\p_x w^\var\|_{H^1_1}\|\p^\al w^\var \xw^{\gm+\al_2}\|_{L^2}^2
		\le C_\gm (1+\e^2),
		\deq
		where we have used the following estimate
		\beq\label{u2-est}
		\begin{aligned}
			\|v^\var \xw^{-1}\|_{L^\infty}
			&\le C   \left(\|v^\var \xw^{-1}\|_{L^2}+\|\p_x v^\var \xw^{-1}\|_{L^2}
			+\|\p_y^2 \left(v^\var \xw^{-1}\right)\|_{L^2}  \right)\\
			&\le C (\|\p_y v^\var\|_{L^2}+\|\p_{xy} v^\var\|_{L^2}+\|\p_y^2 v^\var\|_{L^2})\\
			&\le C \|\p_x w^\var\|_{H^1_1}.
		\end{aligned}
		\deq
		Here we have used Hardy inequality \eqref{Hardy1} and the Sobolev inequality \eqref{inf}.
		
		\textbf{Deal with term $I_{2}$.}
		For $i=1, 2$, using Hardy inequality \eqref{Hardy1} and the Morse-type inequality \eqref{morse}, we have
		\beq\label{2205}
		\begin{aligned}
			|I_{2}|
			&=|\sum_{0<\beta \le \alpha}
			C_\alpha^\beta \int \p^{\beta-e_i} \p^{e_i}u^\var \p^{\al-\beta}\p_x w^\var
			\cdot \p^\alpha w^\var \xw^{2\gm+2\al_2} dxdy|\\
			&\le C\|\p^{e_i}u^\var \xw^{i-1} \|_{H^5_{0}}\|\p_x w^\var\|_{H^5_\gm}
			\|\p^\alpha w^\var \xw^{\gm+\al_2}\|_{L^2}\\
			&\le C\|\p^{e_i}w^\var \xw^{i}\|_{H^5_{0}}\|\p_x w^\var\|_{H^5_\gm}
			\|\p^\alpha w^\var \xw^{\gm+\al_2}\|_{L^2}\\
			&\le C_{\gamma, \sg, \dl}(1+\e^4),
		\end{aligned}
		\deq
		where we have used the estimate \eqref{2101} in the last inequality.
		
		\textbf{Deal with term $I_{3}$.}
		If $\beta_2>0$, the divergence-free condition $\eqref{Prandtl-01}_3$ yields directly

		\beq\label{2206}
		\begin{aligned}
			|I_{3}|
			&=|\sum_{0<\beta\le \alpha}C_\alpha^\beta
			\int \p_x^{\beta_1} \p_y^{\beta_2-1}\p_x u^\var \p^{\al-\beta}\p_y w^\var
			\cdot \p^\alpha w^\var \xw^{2\gm+2\al_2} dxdy|\\
			&\le C\|\p_x u^\var\|_{H^5_0}\|\p_y w^\var \|_{H^5_{\gm+1}}\|\p^\alpha w^\var \xw^{\gm+\al_2}\|_{L^2}\\
			&\le C\|\p_x w^\var\|_{H^5_1}\|\p_y w^\var \|_{H^5_{\gm+1}}\|\p^\alpha w^\var \xw^{\gm+\al_2}\|_{L^2}.
		\end{aligned}
		\deq
		If $\beta_2=0$, we have for all $|\alpha|\le 5$,
		\beq\label{2207}
		\begin{aligned}
			|I_{3}|
			&=|\sum_{0<\beta\le \alpha}C_\alpha^\beta
			\int \p_x^{\beta_1-1} \p_x v^\var \p^{\al-\beta}
			\p_y w^\var \cdot \p^\alpha w^\var \xw^{2\gm+2\al_2} dxdy|\\
			&\le C\|\p_x v^\var\|_{H^4_{-1}}\|\p_y w^\var\|_{H^4_{\gm+1}}
			\|\p^\alpha w^\var \xw^{\gm+\al_2}\|_{L^2}\\
			&\le C\|\p_x^2 u^\var\|_{H^4_{0}}\|\p_y w^\var\|_{H^4_{\gm+1}}
			\|\p^\alpha w^\var \xw^{\gm+\al_2}\|_{L^2}\\
			&\le C\|\p_x^2 w^\var\|_{H^4_{1}}\|\p_y w^\var\|_{H^4_{\gm+1}}
			\|\p^\alpha w^\var \xw^{\gm+\al_2}\|_{L^2}.
		\end{aligned}
		\deq
		If $\beta_2=0$ and $|\alpha|=6$, then it implies $\alpha_2\ge 1$.
		Thus, we have
		\beq\label{2208}
		\begin{aligned}
			|I_{3}|
			&=|\sum_{0<\beta\le \alpha}C_\alpha^\beta
			\int \p_x^{\beta_1-1} \p_x v^\var \p^{\al-\beta-e_2}
			\p_y^2 w^\var \cdot \p^\alpha w^\var \xw^{2\gm+2\al_2} dxdy|\\
			&\le C\|\p_x v^\var\|_{H^4_{-1}}\|\p_y^2 w^\var\|_{H^4_{\gm+2}}
			\|\p^\alpha w^\var \xw^{\gm+\al_2}\|_{L^2}\\
			&\le C\|\p_x^2 w^\var\|_{H^4_{1}}\|\p_y^2 w^\var\|_{H^4_{\gm+2}}
			\|\p^\alpha w^\var \xw^{\gm+\al_2}\|_{L^2}.
		\end{aligned}
		\deq
		Thus, the combination of estimates \eqref{2206}-\eqref{2208}
		and \eqref{2101} yields directly
		\beq\label{2209}
		|I_{3}|\le C_{\gamma, \sg, \dl}(1+\e^3).
		\deq
		
		\textbf{Deal with term $I_{4}$.}
		Integrating by part and applying the boundary condition $\p_y w^\var|_{y=0}=0$, we have
		\beq\label{2210}
		\begin{aligned}
			|I_{4}|
			=\Big{|}&\;\int \p^\alpha (\frac{1}{\vr^\var+\vrf})\p_y w^\var
			\cdot \p_y \p^\alpha w^\var \xw^{2\gm+2\al_2} dx\\
			&\;+(2\gm+2\al_2)\int \p^\alpha (\frac{1}{\vr^\var+\vrf})\p_y w^\var
			\cdot \p^\alpha w^\var \xw^{2\gm+2\al_2-1} dx\Big{|}\\
			\le &\;C_\gm\|\p_y w^\var\|_{L^\infty}
			\|\p^\alpha (\frac{1}{\vr^\var+\vrf})\xw^{\gm+\al_2}\|_{L^2}
			\|\p_y \p^\alpha w^\var \xw^{\gm+\al_2}\|_{L^2}\\
			&+C_\gm\|\p_y w^\var\|_{L^\infty}
			\|\p^\alpha (\frac{1}{\vr^\var+\vrf})\xw^{\gm+\al_2}\|_{L^2}
			\| \p^\alpha w^\var \xw^{\gm+\al_2-1}\|_{L^2}\\
			\le &\;C_{\gm,\kappa}\|\p_y w^\var\|_{H^2}(1+\|\vr^\var\|_{H^6_\gm}^6)
			\|\p_y \p^\alpha w^\var \xw^{\gm+\al_2}\|_{L^2}\\
			&\;+C_{\gm,\kappa}\|\p_y w^\var\|_{H^2}(1+\|\vr^\var\|_{H^6_\gm}^6)
			\| \p^\alpha w^\var \xw^{\gm+\al_2-1}\|_{L^2}\\
			\le &\;\nu_1 \|\p_y \p^\alpha w^\var \xw^{\gm+\al_2}\|_{L^2}^2
			+C_{\nu_1, \kappa, \gamma, \sg, \dl}(1+\e^{19}),
		\end{aligned}
		\deq
		where we have used the estimate \eqref{2101} in the last inequality.
		
		\textbf{Deal with term $I_{5}$.}
		Integrating by part and applying the boundary condition $\p_y w^\var|_{y=0}=0$, we have
		\beqq
		\begin{aligned}
			I_{5}
			=&\sum_{0<\beta<\alpha}C^\alpha_\beta
			\int  \p_y\left\{\p^\beta (\frac{1}{\vr^\var+\vrf})\p^{\alpha-\beta}\p_y w^\var\right\}
			\cdot \p^\alpha w^\var \xw^{2\gm+2\al_2} dxdy\\
			=&\sum_{0<\beta<\alpha}C^\alpha_\beta \left( \int_{\mathbb{T}}
			\left(\p^\beta(\frac{1}{\vr^\var+\vrf})   \p^{\alpha-\beta} \p_y w^\var \cdot \p^{\alpha} w^\var \right)_{y=0} d x
			-
			\int  \p^\beta (\frac{1}{\vr^\var+\vrf})\p^{\alpha-\beta}\p_y w^\var \cdot \p_y \p^{\alpha} w^\var \xw^{2\gm+2\al_2} dxdy \right)\\
			&-(2\gm+2\al_2) \sum_{0<\beta<\alpha}C^\alpha_\beta
			\int  \p^\beta (\frac{1}{\vr^\var+\vrf})\p^{\alpha-\beta}\p_y w^\var
			\cdot \p^\alpha w^\var \xw^{2\gm+2\al_2-1} dx dy\\
			\overset{def}{=}&I_{5,1}+I_{5,2}+I_{5,3}. \\
		\end{aligned}
		\deqq
		Similar to the term $I_4$, for any $i=1,2$, we have
		\beqq
		\begin{aligned}
			|I_{5,2}|
			&\le C\|\xw^{i-1} \p^{e_i}(\frac{1}{\vr^\var+\vrf})\|_{H^5_0}
			\|\p_y w^\var\|_{H^5_\gm}
			\|\p_y \p^\alpha w^\var \xw^{\gm+\al_2}\|_{L^2}\\
			&\le C_\kappa(1+\|\vr^\var\|_{H^6_0}^6)\|\p_y w^\var\|_{H^5_\gm}
			\|\p_y \p^\alpha w^\var \xw^{\gm+\al_2}\|_{L^2}\\
			&\le \nu_1 \|\p_y \p^\alpha w^\var \xw^{\gm+\al_2}\|_{L^2}^2
			+ C_{\nu_1,\kappa,\gamma, \sg, \dl} (1+\e^{19});\\
			|I_{5,3}|
			&\le C_{\gamma}\|\xw^{i-1} \p^{e_i}(\frac{1}{\vr^\var+\vrf})\|_{H^5_0}
			\|\p_y w^\var\|_{H^5_\gm}
			\| \p^\alpha w^\var \xw^{\gm+\al_2}\|_{L^2}\\
			&\le C_{\gamma,\kappa}(1+\|\vr^\var\|_{H^6_0}^6)\|\p_y w^\var\|_{H^5_\gm}
			\| \p^\alpha w^\var \xw^{\gm+\al_2}\|_{L^2}\\
			&\le  C_{\gamma, \kappa,  \sg, \dl} (1+\e^{10}).
		\end{aligned}
		\deqq
		Using the trace inequality \eqref{trace},
		\beqq
		\begin{aligned}
			|I_{5,1}|
			\le&\; \sum_{0<\beta<\alpha} C (\| \p^\beta \p_y (\frac{1}{\vr^\var+\vrf})   \p^{\alpha-\beta} \p_y w^\var \|_{L^2} \|\p^\al w^\var\|_{L^2}
			+\| \p^\beta(\frac{1}{\vr^\var+\vrf})   \p^{\alpha-\beta} \p_y^2 w^\var \|_{L^2} \|\p^\al w^\var\|_{L^2}\\
			&\;+\| \p^\beta(\frac{1}{\vr^\var+\vrf})   \p^{\alpha-\beta} \p_y w^\var \|_{L^2} \|\p_y \p^\al w^\var\|_{L^2}
			)\\
			\le&\; C_{\kappa}(1+\|\vr^\var \|_{H^6}^6) \| \p_y w^\var\|_{H^5}
			\| \p^\al w^\var\|_{L^2}
			+C_{\kappa} (1+\| \vr^\var \|_{H^5}^5) \| \p_y^2 w^\var\|_{H^5}
			\| \p^\al w^\var\|_{L^2}\\
			&\; +C_{\kappa} (1+\| \vr^\var \|_{H^5}^5) \| \p_y w^\var\|_{H^5}
			\| \p^\al \p_y w^\var\|_{L^2}\\
			\le &\; \nu_2 \| \p_y^2 w^\var\|_{H^5}^2
			+ C_{\nu_2,\kappa, \gamma, \sg, \dl}(1+\e^{10}).
		\end{aligned}
		\deqq
		The combination of estimates from terms $I_{5,1}$ to $I_{5,3}$ gives directly
		\beq\label{2211}
		|I_{5}|
		\le \nu_1 \|\p_y \p^\alpha w^\var \xw^{\gm+\al_2}\|_{L^2}^2
		+\nu_2  \| \p_y^2 w^\var\|_{H^5}^2
		+ C_{\nu_1, \nu_2, \kappa, \gamma, \sg, \dl} (1+\e^{19}).
		\deq
		
		\textbf{Deal with term $I_{6}$.}
		Indeed, integrating by part, we have
		\beqq
		\begin{aligned}
			I_6
			=&\underset{I_{6,1}}{\underbrace{\int_{\mathbb{T}}
					\left(\frac{1}{\vr^\var+\vrf}\p_y \p^\al w^\var \cdot
					\p^\alpha w^\var\right)_{y=0} d x}}
			-\int \frac{1}{\vr^\var+\vrf}|\p_y \p^\al w^\var|^2 \xw^{2\gm+2\al_2} dxdy\\
			&-\underset{I_{6,2}}{\underbrace{\int \frac{2\gm+2\al_2}{\vr^\var+\vrf}\p_y \p^\al w^\var
					\cdot \p^\alpha w^\var \xw^{2\gm+2\al_2-1}dx dy}}. \\
		\end{aligned}
		\deqq
		Using the H\"{o}lder inequality, we have
		\beqq
		|I_{6,2}|
		\le C_{\gm}\|\frac{1}{\sqrt{\vr^\var+\vrf}}\p_y \p^\al w^\var \xw^{\gm+\al_2}\|_{L^2}
		\|\p^\al w^\var \xw^{\gm+\al_2-1}\|_{L^2}.
		\deqq
		Let us deal with the boundary term $I_{6,1}$.
		\textbf{Deal with the case: $|\al|\le 5$.}
		The application of boundary embedding inequality \eqref{trace} yields
		\beqq
		\begin{aligned}
			|I_{6,1}|
			&\le \|\p_y(\frac{1}{\vr^\var+\vrf}\p_y \p^\al w^\var)\|_{L^2}
			\|\p^\alpha w^\var\|_{L^2}
			+\|\frac{1}{\vr^\var+\vrf}\p_y \p^\al w^\var\|_{L^2}
			\|\p_y \p^\alpha  w^\var\|_{L^2}\\
			&\le (\|\p_y(\frac{1}{\vr^\var+\vrf})\p_y \p^\al w^\var\|_{L^2}
			+\|\frac{1}{\vr^\var+\vrf} \p_y^2 \p^\al w^\var\|_{L^2})
			\|\p^\alpha w^\var\|_{L^2}
			+\|\frac{1}{\vr^\var+\vrf}\p_y \p^\al w^\var\|_{L^2}
			\|\p_y \p^\alpha  w^\var\|_{L^2}\\
			&\le \mu\|\frac{1}{\sqrt{\vr^\var+\vrf}} \p_y^2 \p^\al w^\var\|_{L^2}^2
			+C_{\mu, \kappa}(1+\e^2).
		\end{aligned}
		\deqq
		The combination of $I_{6,1}$ and $I_{6,2}$ yields directly
		\beq\label{2212}
		\begin{aligned}
			I_6 \le (\mu-1)\|\frac{1}{\sqrt{\vr^\var+\vrf}}\p_y \p^\alpha w^\var \xw^{\gm+\alpha_2}\|_{L^2}^2 + C_{\mu,\kappa,\gm} (1+\e^2).
		\end{aligned}
		\deq
		\textbf{Deal with the case: $|\al|=6$}.
		\textbf{Case 1: $\alpha_2$ is even.}
		For $\alpha_2=2k(k=1,2,3)$, we can apply the vorticity
		equation $\eqref{Prandtl-02}_1$ to obtain
		\beq\label{bl-redu}
		\left\{
		\begin{aligned}
			\p_y^3 w^\var|_{y=0}
			=&((\vr^\var+\vrf)w^\var\p_x w^\var)_{y=0}+\left(\frac{2 \p_y \vr^\var}{\vr^\var+\vrf}\p_y^2 w^\var\right)_{y=0},\\
			\p_y^5 w^\var|_{y=0}=&N_1,\\
			\frac{1}{\vr^\var+\vrf}\p_y^7 w^\var|_{y=0}
			=&(\vr^\var+\vrf)N_2-\var \p_x^2 N_1
			+\left(\p_y^5(u^\var \p_x w^\var +v^\var \p_y w^\var)\right)_{y=0}\\
			&-\left(\p_y^k(\frac{1}{\vr^\var+\vrf})\p_y^{7-k}w^\var\right)_{y=0}.
		\end{aligned}
		\right.
		\deq
		where the functions $N_i(i=1,2)$ are defined in \eqref{a206} and \eqref{a225}.
		The boundary reduction relation \eqref{bl-redu}
		can be proven in Appendix \ref{appendix-b}.
		Thus, for $k=1,2,3$, it is easy to check that
		\beqq
		\begin{aligned}
			|I_{6,1}|
			=&|\int_{\mathbb{T}}
			\left(\frac{1}{\vr^\var+\vrf} \p_x^{6-2k} \p_y^{2k+1} w^\var \cdot
			\p_x^{6-2k} \p_y^{2k} w^\var\right)_{y=0} d x|
			\le  \nu_2 \|\p_y^2 w^\var\|_{H^5}^2
			+C_{\nu_2,\kappa, \gamma, \sg, \dl}(1+\e^{20}).
		\end{aligned}
		\deqq
		\textbf{Case 2: $\alpha_2$ is odd.}
		Then, $\alpha_1\le 5$ and $\alpha_1+\alpha_2=6$ implies $\alpha_2 \ge 1$.
		Obviously, if $\alpha_2=1$, then the condition $\p_y w^\var|_{y=0}=0$ yields directly
		\beqq
		I_{6,1}
		=\int_{\mathbb{T}}
		\left(\frac{1}{\vr^\var+\vrf} \p_x^5 \p_y^2 w^\var \cdot  \p_x^5 \p_y w^\var \right)_{y=0} d x=0.
		\deqq
		Thus, for $\alpha_2=2k+1(k=1,2)$, we integrate by part
		with respect to $x$ to obtain
		\beqq
		\begin{aligned}
			I_{6,1}
			=&\int_{\mathbb{T}}
			\left(\frac{1}{\vr^\var+\vrf}  \p_x^{5-2k}\p_y^{2k+2} w^\var \cdot
			\p_x^{5-2k} \p_y^{2k+1} w^\var \right)_{y=0} d x\\
			=&-\int_{\mathbb{T}}\left(
			\p_x(\frac{1}{\vr^\var+\vrf})  \p_x^{4-2k} \p_y^{2k+2} w^\var \cdot
			\p_x^{5-2k} \p_y^{2k+1} w^\var \right)_{y=0} dx \\
			&-\int_{\mathbb{T}} \left(
			\frac{1}{\vr^\var+\vrf}  \p_x^{4-2k} \p_y^{2k+2}w^\var \cdot
			\p_x^{6-2k} \p_y^{2k+1} w^\var \right)_{y=0} d x\\
			\overset{def}{=}&I_{6,1,1}+I_{6,1,2}.
		\end{aligned}
		\deqq
		Using the trace inequality \eqref{trace}, we have
		\beqq
		\begin{aligned}
			|I_{6,1,1}|
			\le &\;\|\p_x(\frac{1}{\vr^\var+\vrf})\p_x^{4-2k} \p_y^{2k+2} w^\var\|_{L^2}
			 \|\p_x^{5-2k} \p_y^{2k+2}  w^\var\|_{L^2}\\
			&\;+\|\p_y\left\{\p_x(\frac{1}{\vr^\var+\vrf})\p_x^{4-2k} \p_y^{2k+2} w^\var\right\}\|_{L^2}
			 \|\p_x^{5-2k} \p_y^{2k+1}  w^\var\|_{L^2}\\
			\le &\; C_\kappa\|\p_x \vr^\var\|_{H^2_0}
			(\|\p_x^{4-2k} \p_y^{2k+2} w^\var\|_{L^2}
			\|\p_x^{5-2k} \p_y^{2k+2}  w^\var\|_{L^2}
			+\|\p_x^{4-2k} \p_y^{2k+3} w^\var\|_{L^2}
			\|\p_x^{5-2k} \p_y^{2k+1}  w^\var\|_{L^2})\\
			&\;+C_\kappa(1+\|\vr^\var\|_{H^4_0}^2)
			\|\p_x^{4-2k} \p_y^{2k+2} w^\var\|_{L^2}
			\|\p_x^{5-2k} \p_y^{2k+1}  w^\var\|_{L^2}\\
			\le&\;\nu_2 \|\p_x^{4-2k} \p_y^{2k+3} w^\var\|_{L^2}^2
			+C_{\nu_2, \kappa}(1+\e^2).
		\end{aligned}
		\deqq
		For $k=1,2$, the quantity $\p_y^{2k+1} \p_x^{6-2k} w^\var |_{y=0}$
		in term $I_{6,1,1}$ has an odd number of $y$ derivatives,
		and hence, we can apply the boundary reduction equality
		$\eqref{bl-redu}_1$ and $\eqref{bl-redu}_2$ to further reduce
		the order of derivative in term $I_{6,1,1}$. Then, similar to the
		case 1, we can apply the trace inequality \eqref{trace} to obtain
		\beqq
		\begin{aligned}
			|I_{6,1,2}|
			\le  \nu_2 \|\p_y^2 w^\var\|_{H^5}^2
			+C_{\nu_2,\kappa, \gamma, \sg, \dl}(1+\e^{20}).
		\end{aligned}
		\deqq
		The combination of estimates from $I_{6,1,1}$ and $I_{6,1,2}$ yields directly
		\beqq
		\begin{aligned}
			|I_{6,1}|
			\le  \nu_2 \|\p_y^2 w^\var\|_{H^5}^2
			+C_{\nu_2,\kappa, \gamma, \sg, \dl}(1+\e^{20}),
		\end{aligned}
		\deqq
		which, together with the estimate of $I_{6,2}$, yields directly
		\beq\label{2215}
		\begin{aligned}
			I_6 \le \nu_2 \|\p_y^2 w^\var\|_{H^5}^2
			+(\mu-1)\|\frac{1}{\sqrt{\vr^\var+\vrf}}\p_y \p^\alpha w^\var \xw^{\gm+\alpha_2}\|_{L^2}^2
			+ C_{\nu_2,\mu,\kappa,\gamma, \sg, \dl} (1+\e^{20}).
		\end{aligned}
		\deq
		Substituting the estimates to \eqref{2204}, \eqref{2205}, \eqref{2209}-\eqref{2212} and \eqref{2215} into \eqref{2203} and choosing $\mu$ small enough, we can obtain the estimate \eqref{2201}.
		Therefore, we complete the proof of this lemma.
	\end{proof}
	
	Next, we establish the weighted estimates with normal derivatives for the density.
	\begin{lemm}\label{lemma23}
		Under the condition of \eqref{Lowerb},
		the solution $(u^\var, v^\var, w^\var)$ of regularized system
		\eqref{Prandtl-01} will satisfy
		\beqq
		\frac{d}{dt}\sum_{|\alpha|\le 6, \alpha_1 \le 5}\frac12\|\p^\al \vr^\var \xw^{\sg+\al_2}\|_{L^2}^2
		+\sum_{|\alpha|\le 6, \alpha_1 \le 5}
		\var \|\p_x \p^\alpha \vr^\var \xw^{\sg+\alpha_2}\|_{L^2}^2
		\le C_{\sg, \gm, \dl}(1+\e^4).
		\deqq
	\end{lemm}
	\begin{proof}
		For any $|\alpha|\le 6$ and $\alpha_1\le 5$,
		applying the $\p^\al$ differential operator to the
		equation $\eqref{Prandtl-01}_1$, we have
		\beqq
		\p_t(\p^\al \vr^\var)+u^\var \p_x (\p^\al \vr^\var)
		+v^\var \p_y(\p^\al \vr^\var)-\var \p_x^2(\p^\al \vr^\var)
		=-[\p^\al, u^\var]\p_x \vr^\var-[\p^\al, v^\var]\p_y \vr^\var.
		\deqq
		Applying this equation by $\p^\al \vr^\var \xw^{2\sg+2\al_2}$ and integrating
		over $\mathbb{T}\times \mathbb{R}^+$, we have
		\beq\label{2302}
		\begin{aligned}
			&\frac{d}{dt}\frac12\|\p^\al \vr^\var \xw^{\sg+\al_2}\|_{L^2}^2
			+\var\|\p_x \p^\al \vr^\var \xw^{\sg+\al_2}\|_{L^2}^2
			-\underset{I_{7}}{\underbrace{
					(\sg+\alpha_2)\int |\p^\al \vr^\var|^2 v^\var \xw^{2\sg+2\alpha_2-1}dxdy}}\\
			=&-\underset{I_{8}}{\underbrace{
					\int [\p^\al, u^\var]\p_x \vr^\var \cdot \p^\al \vr^\var \xw^{2\sg+2\al_2} dxdy}}
			-\underset{I_{9}}{\underbrace{
					\int [\p^\al, v^\var]\p_y \vr^\var \cdot \p^\al \vr^\var \xw^{2\sg+2\al_2} dxdy}}.
		\end{aligned}
		\deq
		Using the estimate \eqref{u2-est}, one has
		\beq\label{2303}
		\begin{aligned}
			|I_7|
			\le C_{\sg}\|v^\var \xw^{-1}\|_{L^\infty}
			\|\p^\al \vr^\var \xw^{\sg+\alpha_2}\|_{L^2}^2
			\le C_{\sg}\|\p_x w^\var\|_{H^1_1}
			\|\p^\al \vr^\var \xw^{\sg+\alpha_2}\|_{L^2}^2
			\le C_{\sg} (1+\e^2).
		\end{aligned}
		\deq
		For $i=1,2$, one can apply the Morse-type inequality \eqref{morse}, Hardy inequality \eqref{Hardy1} and $\sg\le2\gm-1$ to obtain
		\beq\label{2304}
		\begin{aligned}
			|I_8|
			&=|\sum_{0<\beta \le \alpha}C^\alpha_\beta
			\int \p^{\beta-e_i} \p^{e_i}u^\var \p^{\alpha-\beta} \p_x \vr^\var
			\cdot \p^\al \vr^\var \xw^{2\sg+2\al_2} dxdy|\\
			&\le C\|\p^{e_i}u^\var\|_{H^5_{\sg-\gm+i-1}}
			\|\p_x \vr^\var\|_{H^5_\gm}\|\p^\al \vr^\var \xw^{\sg+\al_2}\|_{L^2}\\
			&\le C_{\sg,\gamma,\dl}(1+\e^4),
		\end{aligned}
		\deq
		where we have used the estimate \eqref{2101}.
		Finally, let us deal with the term $I_{9}$.
		If $\beta_2>0$, then we apply the divergence free condition
		$\eqref{Prandtl-01}_3$ and Morse-type inequality \eqref{morse} to obtain
		\beqq
		\begin{aligned}
			|I_{9}|
			&=|\sum_{0<\beta\le \alpha}C_\alpha^\beta
			\int \p_x^{\beta_1} \p_y^{\beta_2-1}\p_x u^\var \p^{\al-\beta}\p_y \vr^\var
			\cdot \p^\alpha \vr^\var \xw^{2\sg+2\al_2} dxdy|\\
			&\le C\|\p_x u^\var\|_{H^5_0}\|\p_y \vr^\var\|_{H^5_{\sg+1}}
			\|\p^\alpha \vr^\var \xw^{\sg+\al_2}\|_{L^2}\\
			&\le C_{\sg,\gamma,\dl}(1+\e^3),
		\end{aligned}
		\deqq
		where we have used Hardy inequality \eqref{Hardy1}
		and estimate \eqref{2101}.
		If $\beta_2=0$, we have for all $|\alpha|\le 5$,
		\beqq
		\begin{aligned}
			|I_{9}|
			&=|\sum_{0<\beta\le \alpha}C_\alpha^\beta
			\int \p_x^{\beta_1-e_1} \p_x v^\var \p^{\al-\beta}
			\p_y \vr^\var \cdot \p^\alpha \vr^\var \xw^{2\sg+2\al_2} dxdy|\\
			&\le C\|\p_x v^\var\|_{H^4_{-1}}\|\p_y \vr^\var\|_{H^4_{\sg+1}}
			\|\p^\alpha \vr^\var \xw^{\sg+\al_2}\|_{L^2}\\
			&\le C\|\p_x^2 u^\var\|_{H^4_{0}}\|\p_y \vr^\var\|_{H^4_{\sg+1}}
			\|\p^\alpha \vr^\var \xw^{\sg+\al_2}\|_{L^2}\\
			&\le C_{\sg,\gamma,\dl}(1+\e^3),
		\end{aligned}
		\deqq
		where we have used the Hardy inequalities
		\eqref{Hardy1}, \eqref{Hardy2} and estimate \eqref{2101}.
		If $\beta_2=0$ and $|\alpha|=6$, then it implies $\alpha_2\ge 1$, and we have
		\beqq
		\begin{aligned}
			|I_{9}|
			&=|\sum_{0<\beta\le \alpha}C_\alpha^\beta
			\int \p_x^{\beta_1-e_1} \p_x v^\var \p^{\al-\beta-e_2}
			\p_y^2 \vr^\var \cdot \p^\alpha \vr^\var \xw^{2\sg+2\al_2} dxdy|\\
			&\le C\|\p_x v^\var\|_{H^4_{-1}}\|\p_y^2 \vr^\var\|_{H^4_{\sg+2}}
			\|\p^\alpha \vr^\var \xw^{\sg+\al_2}\|_{L^2}\\
			&\le C\|\p_x^2 u^\var\|_{H^4_{0}}\|\p_y^2 \vr^\var\|_{H^4_{\sg+2}}
			\|\p^\alpha \vr^\var \xw^{\sg+\al_2}\|_{L^2}\\
			&\le C_{\sg,\gamma,\dl}(1+\e^3).
		\end{aligned}
		\deqq
		Thus, the term $I_{9}$ can be estimated as follows
		\beq\label{2305}
		|I_{9}| \le  C_{\sg,\gamma,\dl}(1+\e^3).
		\deq
		Substituting the estimates \eqref{2303}, \eqref{2304}
		and \eqref{2305} into \eqref{2302}, we have
		\beqq
		\frac{d}{dt}\frac12\|\p^\al \vr^\var \xw^{\sg+\al_2}\|_{L^2}^2
		+\var\|\p_x \p^\al \vr^\var \xw^{\sg+\al_2}\|_{L^2}^2
		\le C_{\sg,\gamma,\dl}(1+\e^4).
		\deqq
		Therefore, we complete the proof of this lemma.
	\end{proof}

	\subsection{ Weighted estimates only in tangential Variables}\label{section-only tangential}
	First of all, we establish the estimate of only tangential derivative of vorticity.
	Applying $\p_x^6$-operator to the vorticity
	equation $\eqref{Prandtl-02}_1$, we have
	\beq\label{201}
	\p_t(\p_x^6 w^\var)+u^\var \p_x(\p_x^6 w^\var)+v^\var \p_y(\p_x^6 w^\var)
	+\p_x^6 v^\var \p_y w^\var-\var \p_x^8 w^\var
	-\p_x^6 \p_y \{\frac{1}{\vr^\var+\vrf}\p_y w^\var\}=Q_1,
	\deq
	where the function $Q_1$ is defined as follows
	\beqq\label{202}
	Q_1\overset{def}{=}-\sum_{0< k \le 6}C_6^k \p_x^k u^\var \p_x^{7-k} w^\var
	-\sum_{0< k <6}C_6^k \p_x^k v^\var \p_x^{6-k} \p_y w^\var.
	\deqq
	Applying $\p_x^6$-operator to the velocity
	equation $\eqref{Prandtl-01}_2$, we have
	\beq\label{203}
	\p_t(\p_x^6 u^\var)+u^\var \p_x (\p_x^6 u^\var)+v^\var \p_y (\p_x^6 u^\var)
	+\p_x^6 v^\var w^\var-\var \p_x^8 u^\var-\p_x^6 \{\frac{1}{\vr^\var+\vrf}\p_y w^\var\}=Q_2,
	\deq
	where the function $Q_2$ is defined as follows
	\beqq\label{204}
	Q_2\overset{def}{=}-\sum_{0< k \le 6}C_6^k \p_x^k u^\var \p_x^{7-k} u^\var
	-\sum_{0< k <6}C_6^k \p_x^k v^\var \p_x^{6-k} w^\var.
	\deqq
	The most difficult term in both  \eqref{201} and \eqref{203}
	is $\p_x^6 v^\var = -\p_y^{-1} \p_x^7  u^\var$,
	which causes the loss of $x$-derivative of horizontal velocity,
	so that the standard energy estimate cannot apply.
	To eliminate the problematic term $\p_x^6 v^\var$, we subtract the equations \eqref{201} and \eqref{203} in an appropriate way under the condition of \eqref{Lowerb} which makes sure that $w^\var >0$.
	Recall that $g_w^\var =\frac{\p_y w^\var}{w^\var}$, then multiplying  \eqref{203} by $g_w^\var$, it holds
	\beq\label{205}
	\p_t(\p_x^6 u^\var)g_w^\var+u^\var \p_x (\p_x^6 u^\var)g_w^\var
	+v^\var \p_y (\p_x^6 u^\var)g_w^\var
	+\p_x^6 v^\var \p_y w^\var-\var \p_x^8 u^\var g_w^\var
	-\p_x^6 \{\frac{1}{\vr^\var+\vrf}\p_y w^\var\}g_w^\var=Q_2\cdot g_w^\var.
	\deq
	Then, recall that $w_g^\var =\p_x^6 w^\var-g_w^\var \p_x^6 u^\var
	=\p_x^6 w^\var-\frac{\p_y w^\var}{w^\var} \p_x^6 u^\var$, it is easy to check that
	\beq\label{206}
	\begin{aligned}
		\p_t w_g^\var +u^\var \p_x w_g^\var +v^\var \p_y w_g^\var
		-\var \p_x^2 w_g^\var -\p_y\{\frac{1}{\vr^\var+\vrf}\p_y w_g^\var \}
		=Q_1-Q_2\cdot g_w^\var-Q_3-Q_4 \cdot \p_x^6 u^\var,
	\end{aligned}
	\deq
	here $Q_i(i=3,4)$ are defined as follows.
	\beqq\label{207}
	\begin{aligned}
		Q_3\overset{def}{=}&-\p_y\left\{\sum_{0<k \le 6}C_6^k \p_x^k(\frac{1}{\vr^\var+\vrf})
		\p_x^{6-k}\p_y w^\var\right\}
		-\p_y\left\{\frac{1}{\vr^\var+\vrf}\p_x^6 u^\var \p_y g_w^\var\right\}\\
		&-\p_x^6 w^\var \p_y\{\frac{1}{\vr^\var+\vrf}g_w^\var\}
		+\sum_{0< k \le 6}C_6^k \p_x^k(\frac{1}{\vr^\var+\vrf})\p_x^{6-k}\p_y w^\var g_w^\var
		-2\var\p_x^7 u^\var \p_x g_w^\var;\\
		Q_4\overset{def}{=}&\frac{1}{w^\var}\p_y \left\{-u^\var \p_x w^\var-v^\var \p_y w^\var
		+\p_y(\frac{1}{\vr^\var+\vrf}\p_y w^\var)\right\}
		-\frac{\p_y w^\var}{(w^\var)^2}
		\left(-u^\var \p_x w^\var-v^\var \p_y w^\var
		+\p_y(\frac{1}{\vr^\var+\vrf}\p_y w^\var)\right)\\
		&+u^\var\left(\frac{\p_{xy}w^\var}{w^\var}
		-\frac{\p_y w^\var \p_x w^\var}{(w^\var)^2}\right)
		+v^\var\left(\frac{\p_{y}^2 w^\var}{w^\var}
		-\frac{(\p_y w^\var)^2}{(w^\var)^2}\right)
		-\var \p_x^2 (\frac{\p_y w^\var}{w^\var})+\var \frac{\p_y \p_x^2 w^\var}{w^\var}-\var \frac{\p_y w^\var \p_x^2 w^\var}{(w^\var)^2}.
	\end{aligned}
	\deqq
	The above equation \eqref{206} is proven in
	Appendix \ref{good-unkonw} in detail.
	Now, we are going to establish the weighted estimate for
	this good quantity $w_g^\var$.
	\begin{lemm}\label{lemma24}
		Under the condition of \eqref{Lowerb},
		the solution $(u^\var, v^\var, w^\var)$ of regularized system
		\eqref{Prandtl-01} will satisfy
		\beqq
		\begin{aligned}
			&\frac{d}{dt}\|w_g^\var \xw^{\gm}\|_{L^2}^2
			+\var\|\p_x w_g^\var \xw^{\gm} \|_{L^2}^2
			+\|\frac{1}{\sqrt{\vr^\var+\vrf}}\p_y w_g^\var \xw^{\gm} \|_{L^2}^2\\
			\le&\; \nu_3 \|\p_y w_g^\var \xw^{\gm}\|_{L^2}^2
			+C_{\nu_3, \kappa, \dl, \gamma,\sg}(1+\e^{19}),
		\end{aligned}
		\deqq
		where $\nu_3$ is the small constant that will be chosen later.
	\end{lemm}
	\begin{proof}
		Multiplying the equation \eqref{206} by $w_g^\var \xw^{2\gm}$ and
		integrating over $\mathbb{T}\times \mathbb{R}^+$, we have
		\beq\label{2402}
		\begin{aligned}
			&\;\frac{d}{dt}\frac12 \int (w_g^\var)^2 \xw^{2\gm}dxdy
			+\var\int |\p_x w_g^\var|^2 \xw^{2\gm} dxdy
			+\int \frac{1}{\vr^\var+\vrf}|\p_y w_g^\var|^2 \xw^{2\gm}dxdy\\
			=&\;\underset{J_{1}}{\underbrace{\gm \int v^\var (w_g^\var)^2 \xw^{2\gm-1}dxdy}}
			-\underset{J_{2}}{\underbrace{\int \frac{2\gm}{\vr^\var+\vrf}\p_y w_g^\var \cdot w_g^\var \xw^{2\gm-1}dxdy}}
			+\underset{J_{3}}{\underbrace{\int Q_1\cdot w_g^\var \xw^{2\gm} dxdy}}\\
			&\;-\underset{J_{4}}{\underbrace{\int Q_2\cdot g_w^\var \cdot w_g^\var \xw^{2\gm} dxdy}}
			-\underset{J_{5}}{\underbrace{\int Q_3 \cdot w_g^\var \xw^{2\gm} dxdy}}
			-\underset{J_{6}}{\underbrace{\int Q_4 \cdot \p_x^6 u^\var \cdot w_g^\var \xw^{2\gm} dxdy}}.
		\end{aligned}
		\deq
		Using the estimate \eqref{u2-est}, we have
		\beq\label{2403}
		|J_{1}|
		\le C_\gm\|v^\var \xw^{-1}\|_{L^\infty}\| w_g^\var \xw^{\gm}\|_{L^2}^2
		\le C_\gm\|\p_x w^\var\|_{H^1_1}\| w_g^\var \xw^{\gm}\|_{L^2}^2,
		\deq
		and
		\beq\label{2404}
		|J_2|
		\le C_{\gm,\kappa}\|\p_y w_g^\var \xw^{\gm}\|_{L^2}
		\|w_g^\var \xw^{\gm-1}\|_{L^2}
		\le \nu_3 \|\p_y w_g^\var \xw^{\gm}\|_{L^2}^2
		+C_{\nu_3, \gm,\kappa} \|w_g^\var \xw^{\gm-1}\|_{L^2}^2.
		\deq
		
		\textbf{Deal with $J_3$ and $J_4$ terms.}
		Using Hardy inequality \eqref{Hardy1}, the Morse-type inequality \eqref{morse}
		 and estimate \eqref{2101}, we have
		\beq\label{2405}
		\begin{aligned}
			|J_3|
			\le &\;|\sum_{1\le k \le 6}C_6^k
			\int \p_x^{k-1} \p_x u^\var \p_x^{6-k} \p_x w^\var \cdot w_g^\var \xw^{2\gm} dxdy|\\
			&\;+|\sum_{1\le  k \le 5}C_6^k \int \p_x^{k-1} \p_x v^\var
			\p_x^{5-k} \p_{xy} w^\var \cdot w_g^\var \xw^{2\gm} dxdy|\\
			\le&\;C\|\p_x u^\var\|_{H^5_0}\|\p_x w^\var\|_{H^5_\gm} \|w_g^\var \xw^{\gm}\|_{L^2}
			+C\|\p_x v^\var \|_{H^4_{-1}}
			\|\p_{xy} w^\var\|_{H^4_{\gm+1}}\|w_g^\var \xw^{\gm}\|_{L^2}\\
			\le&\;C\|\p_x w^\var\|_{H^5_1}\|\p_x w^\var\|_{H^5_\gm} \|w_g^\var \xw^{\gm}\|_{L^2}
			+C\|\p_x^2 w^\var\|_{H^4_{1}}\|\p_{xy} w^\var\|_{H^4_{\gm+1}}
			\|w_g^\var \xw^{\gm}\|_{L^2}\\
			\le &\;C_{\gamma, \dl, \sg}(1+\e^4),
		\end{aligned}
		\deq
		and using the lower bound condition \eqref{Lowerb}, we arrive at
		\beq\label{2406}
		\begin{aligned}
			|J_4|
			\le &\;C_\dl\|\p_x u^\var\|_{H^5_0}\|\p_x u^\var\|_{H^5_{\gm-1}}
			\|\p_y w^\var \xw^{\sg+1}\|_{L^\infty}\|w_g^\var \xw^\gm\|_{L^2}\\
			&\;+C_\dl\|\p_x v^\var\|_{H^4_{-1}}\|\p_x w^\var\|_{H^4_\gm}
			\|\p_y w^\var \xw^{\sg+1}\|_{L^\infty}\|w_g^\var \xw^\gm\|_{L^2}\\
			\le &\;C_\dl\|\p_x w^\var\|_{H^5_1}\|\p_x w^\var\|_{H^5_{\gm}}
			\|\p_y w^\var \xw^{\sg+1}\|_{L^\infty}\|w_g^\var \xw^\gm\|_{L^2}\\
			&\;+C_\dl\|\p_x^2 w^\var\|_{H^4_{1}}\|\p_x w^\var\|_{H^4_\gm}
			\|\p_y w^\var \xw^{\sg+1}\|_{L^\infty}\|w_g^\var \xw^\gm\|_{L^2}\\
			\le &\; C_{\gamma, \dl, \sg}(1+\e^4).
		\end{aligned}
		\deq
		
		\textbf{Deal with $J_5$ term.}
		It is easy to check that
		\beqq
		\begin{aligned}
			J_5
			=&\!\!\sum_{0<k \le 6}C_6^k\!\!
			\int\p_y\left\{ \p_x^k(\frac{1}{\vr^\var+\vrf})\p_x^{6-k}\p_y w^\var\right\}
			\cdot w_g^\var \xw^{2\gm}dxdy
			+\!\!\int \p_y\left\{\frac{1}{\vr^\var+\vrf}\p_x^6 u^\var \p_y g_w^\var\right\} \cdot w_g^\var \xw^{2\gm}dxdy\\
			&+\int \p_x^6 w^\var \p_y\{\frac{1}{\vr^\var+\vrf}g_w^\var\} \cdot w_g^\var \xw^{2\gm}dxdy
			-\sum_{0< k \le 6}C_6^k \int \p_x^k(\frac{1}{\vr^\var+\vrf})\p_x^{6-k}
			\p_y w^\var g_w^\var \cdot w_g^\var \xw^{2\gm}dxdy\\
			&-2\var\int \p_x^7 u^\var \p_x g_w^\var \cdot w_g^\var \xw^{2\gm} dxdy\\
			\overset{def}{=}&J_{5,1}+J_{5,2}+J_{5,3}-J_{5,4}-J_{5,5}.
		\end{aligned}
		\deqq
		Integrating by part and using inequalities \eqref{morse},
		\eqref{f02} and \eqref{2101}, we have
		\beqq
		\begin{aligned}
			|J_{5,1}|
			&\le C_{\gm}\|\p_x(\frac{1}{\vr^\var+\vrf})\|_{H^5_0}\|\p_y w^\var\|_{H^5_\gm}
			(\|\p_y w_g^\var \xw^{\gm}\|_{L^2}+\|w_g^\var \xw^{\gm-1}\|_{L^2})\\
			&\le C_{\gm,\kappa}(1+\|\vr^\var\|_{H^6_0}^6)\|\p_y w^\var\|_{H^5_\gm}
			(\|\p_y w_g^\var \xw^{\gm}\|_{L^2}+\|w_g^\var \xw^{\gm-1}\|_{L^2})\\
			&\le \nu_3 \|\p_y w_g^\var \xw^{\gm}\|_{L^2}^2
			+C_{\nu_3,\kappa,\gamma, \dl, \sg}(1+\e^{19}).
		\end{aligned}
		\deqq
		Integrating by part and using lower bound condition \eqref{Lowerb}, it holds
		\beqq
		\begin{aligned}
			|J_{5,2}|
			\le &\; C_{\kappa}\|\p_x^6 u^\var \xw^{\gm-2}\|_{L^2}\|\p_y g_w^\var \xw^2\|_{L^\infty}
			(\|\p_y w_g^\var \xw^{\gm}\|_{L^2}+\|w_g^\var \xw^{\gm-1}\|_{L^2})\\
			\le &\;C_{\kappa,\dl}\|\p_x^6 w^\var \xw^{\gm-1}\|_{L^2}
			(\|\p_y^2 w^\var \xw^{\sg+2}\|_{L^\infty}
			+\|\p_y w^\var \xw^{\sg+1}\|_{L^\infty}^2)\\
			&\;\times(\|\p_y w_g^\var \xw^{\gm}\|_{L^2}+\|w_g^\var \xw^{\gm-1}\|_{L^2})\\
			\le &\; \nu_3 \|\p_y w_g^\var \xw^{\gm}\|_{L^2}^2
			+C_{\nu_3,\kappa,\gamma, \dl, \sg}(1+\e^5).
		\end{aligned}
		\deqq
		Similarly, it is easy to obtain
		\beqq
		\begin{aligned}
			|J_{5,3}|
			=&\;|\int \p_x^6 w^\var \left(-\frac{\p_y \vr^\var g_w^\var}{(\vr^\var+\vrf)^2}
			+\frac{\p_y g_w^\var}{\vr^\var+\vrf}\right)\cdot w_g^\var \xw^{2\gm}dx|\\
			\le&\;C_{\dl,\kappa}
			(\|\p_y \vr^\var \xw\|_{L^\infty}\|\p_y w^\var \xw^{\sg+1}\|_{L^\infty}
			+\|\p_y^2 w^\var \xw^{\sg+2}\|_{L^\infty}
			+\|\p_y w^\var \xw^{\sg+1}\|_{L^\infty}^2)\\
			&\;\times \|\p_x^6 w^\var \xw^{\gm-2}\|_{L^2}\|w_g^\var \xw^{\gm}\|_{L^2}\\
			\le &\; C_{\kappa,\gamma, \dl, \sg}(1+\e^3),
		\end{aligned}
		\deqq
		and
		\beqq
		\begin{aligned}
			|J_{5,4}|
			\le \; & C_{\dl}\|\p_x(\frac{1}{\vr^\var+\vrf})\|_{H^5_0}\|\p_y w^\var\|_{H^5_{\gm-1}}
			\|\p_y w^\var \xw^{\sg+1}\|_{L^\infty}\|w_g^\var \xw^{\gm}\|_{L^2}\\
			\le & \; C_{\dl,\kappa}(1+\|\vr^\var\|_{H^6_0}^6)\|\p_y w^\var\|_{H^5_{\gm-1}}
			\|\p_y w^\var \xw^{\sg+1}\|_{L^\infty}\|w_g^\var \xw^{\gm}\|_{L^2}\\
			\le & \; C_{\kappa, \gamma, \dl, \sg}(1+\e^{11});\\
			|J_{5,5}|
			= &\;2\var|\int \p_x^7 u^\var (\frac{\p_{xy} w^\var}{w^\var}-\frac{\p_{x} w^\var \p_{y} w^\var}{(w^\var)^2}) \cdot w_g^\var \xw^{2\gm} dx|\\
			\le &\; \var C_{\dl} \|w_g^\var \xw^{\gm}\|_{L^2}
			(\| \p_{xy} w^\var \xw^{\sg+1}\|_{L^\infty} + \|\p_x w^\var \xw^{\sg} \|_{L^\infty} \|\p_y w^\var \xw^{\sg+1} \|_{L^\infty})\\
			&\;\times \|\xw^{\gm -\sg -1} \frac{\p_x^7 u^\var }{w^\var}  \|_{L^2}\\
			\le &\; \var C_{\dl,\gm,\sg}\|w_g^\var \xw^{\gm}\|_{L^2}
			(\| \p_{xy} w^\var \xw^{\sg+1}\|_{L^\infty} + \|\p_x w^\var \xw^{\sg} \|_{L^\infty} \|\p_y w^\var \xw^{\sg+1} \|_{L^\infty})\\
			&\; \times  (\|\p_x w
			_g \xw^{\gm} \|_{L^2} + \|\p_x g_w^\var \xw \|_{L^\infty}\|\p_x^6 w^\var \xw^\gm\|_{L^2})\\
			\le &\; \frac{\var}{2} \|\p_x w
			_g \xw^{\gm} \|_{L^2}^2 + C_{\dl,\gm,\sg} (1+\e^4),
		\end{aligned}
		\deqq
		where we have used the relation
		\beqq
		\p_x w_g^\var=\p_x^7 w^\var-g_w^\var \p_x^7 u^\var
		-\p_x g_w^\var \p_x^6 u^\var
		=w^\var \p_y\{\frac{\p_x^7 u^\var }{w^\var} \}
		-\p_x g_w^\var \p_x^6 u^\var,
		\deqq
		to establish the following estimate
		\beqq
		\begin{aligned}
			\|\xw^{\gm -\sg -1} \frac{\p_x^7 u^\var }{w^\var}  \|_{L^2}
			&\le  C_{\gm,\sg} \|\xw^{\gm -\sg }
			\p_y\{\frac{\p_x^7 u^\var }{w^\var} \} \|_{L^2}\\
			&\le   C_{\dl,\gm,\sg} \|\xw^{\gm  }
			w^\var \p_y\{\frac{\p_x^7 u^\var }{w^\var} \} \|_{L^2}\\
			&\le  C_{\dl,\gm,\sg} (\|  \p_x w_g^\var \xw^{\gm} \|_{L^2} + \|\p_x g_w^\var \xw \|_{L^\infty}\|\p_x^6 u^\var_1 \xw^{\gm-1}\|_{L^2})\\
			&\le  C_{\dl,\gm,\sg} (\|  \p_x w_g^\var \xw^{\gm} \|_{L^2} + \|\p_x g_w^\var \xw \|_{L^\infty}\|\p_x^6 w^\var \xw^\gm\|_{L^2}).
		\end{aligned}
		\deqq
		Thus, the combination of estimates of terms from
		$J_{5,1}$ to $J_{5,5}$ yields directly
		\beq\label{2407}
		|J_{5}|\le \nu_3 \|\p_y w_g^\var \xw^{\gm}\|_{L^2}^2+\frac{\var}{2} \|\p_x w
		_g \xw^{\gm} \|_{L^2}^2
		+C_{\nu_3, \kappa, \gamma, \dl, \sg}(1+\e^{19}).
		\deq
		
		\textbf{Deal with $J_6$ term.}
		First of all, let us decompose the quantity $Q_4$ as follows:
		\beqq
		Q_4 \overset{def}{=}Q_{4,1}+Q_{4,2}+Q_{4,3}+Q_{4,4},
		\deqq
		where $Q_{4,i}(i=1,2,3,4)$ is defined as follows:
		\beqq
		\begin{aligned}
			Q_{4,1}&\overset{def}{=}\frac{1}{w^\var}\p_y \left\{-u^\var \p_x w^\var-v^\var \p_y w^\var
			\right\};\\
			Q_{4,2}&\overset{def}{=}-\frac{\p_y w^\var}{(w^\var)^2}\left(-u^\var \p_x w^\var
			-v^\var \p_y w^\var
			+\p_y(\frac{1}{\vr^\var+\vrf}\p_y w^\var)\right);\\
			Q_{4,3}&\overset{def}{=}u^\var\left(\frac{\p_{xy}w^\var}{w^\var}
			-\frac{\p_y w^\var \p_x w^\var}{(w^\var)^2}\right)
			+v^\var\left(\frac{\p_{y}^2 w^\var}{w^\var}
			-\frac{(\p_y w^\var)^2}{(w^\var)^2}\right);\\
			Q_{4,4}&\overset{def}{=}-\var \p_x^2\left\{ \frac{ \p_y w^\var}{w^\var} \right\}+\f{1}{w^\var}\p_y^2(\frac{1}{\vr^\var+\vrf}\p_y w^\var).
		\end{aligned}
		\deqq
		Therefore,
		\beqq
		J_{6}= \sum_{1\leq i \leq 4} \int Q_{4,i} \cdot \p_x^6 u^\var \cdot w_g^\var \xw^{2\gm}dxdy
		\overset{def}{=} \sum_{1\leq i \leq 4} J_{6,i}.
		\deqq
		It is easy to check that for $i=1,2,3$,
		\beqq
		\begin{aligned}
			|J_{6,i}|
			&\le \|Q_{4,i} \xw\|_{L^\infty}\|\p_x^6 u^\var \xw^{\gm-1}\|_{L^2}
			\|w_g^\var \xw^{\gm}\|_{L^2}\\
			&\le C_{\gm}\|Q_{4,i} \xw\|_{L^\infty}\|\p_x^6 w^\var \xw^{\gm}\|_{L^2}
			\|w_g^\var \xw^{\gm}\|_{L^2}.
		\end{aligned}
		\deqq
		Let us first deal with estimate $\|Q_{4,i} \xw\|_{L^\infty}(i=1,2,3)$.
		Using the inequality \eqref{inf}, it is easy to check that
		\beqq
		\begin{aligned}
			&\;\|Q_{4,1}\xw\|_{L^\infty}\\
			\le &\; C_\dl(\|\p_x w^\var \xw\|_{L^\infty}
			+\|\p_{xy} w^\var \xw^{\sg+1}\|_{L^\infty}\|u^\var\|_{L^\infty}
			+\|\p_y w^\var \xw^{\sg+1}\|_{L^\infty}\|\p_x u^\var\|_{L^\infty})\\
			&\;+C_\dl(\|\p_y^2 w^\var \xw^{\sg+2}\|_{L^\infty}
			\|v^\var \xw^{-1}\|_{L^\infty}
			+\|\p_y^2(\frac{1}{\vr^\var+\vrf})\|_{L^\infty}
			\|\p_y w^\var \xw^{\sg+1} \|_{L^\infty})\\
			&\;+C_{\dl,\kappa}\|\p_y(\frac{1}{\vr^\var+\vrf})\|_{L^\infty}
			\|\p_y^2 w^\var \xw^{\sg+1}\|_{L^\infty}\\
			\le &\; C_\dl\|\p_x w^\var \xw\|_{H^2}
			+C_\dl\|\p_{xy} w^\var \xw^{\sg+1}\|_{L^\infty}(1+\|w^\var\|_{H^1_1})
			+C_\dl\|\p_y w^\var \xw^{\sg+1}\|_{L^\infty}\|\p_x w^\var\|_{H^2_1}\\
			&\;+C_\dl\|\p_y^2 w^\var \xw^{\sg+2}\|_{L^\infty}\|w^\var\|_{H^1_1}
			+C_{ \dl, \kappa}(1+\|\vr^\var\|_{H^4_0}^4)\|\p_y w^\var \xw^{\sg+1} \|_{L^\infty}\\
			&\;+C_{\dl, \kappa}(1+\|\vr^\var\|_{H^3_0}^3)\|\p_y^2 w^\var \xw^{\sg+1}\|_{L^\infty}\\
			\le &\; C_{\dl,\kappa}(1+\e^3).
		\end{aligned}
		\deqq
		Similarly, it is easy to check that
		\beqq
		\begin{aligned}
			&\;\|Q_{4,2}\xw\|_{L^\infty}\\
			\le &\;C_\dl \|\p_y w^\var \xw^{\sg+1}\|_{L^\infty}
			(\|u^\var\|_{L^\infty}\|\p_x w^\var \xw^{\sg}\|_{L^\infty}
			+\|v^\var \xw^{-1}\|_{L^\infty}
			\|\p_y w^\var \xw^{\sg+1}\|_{L^\infty})\\
			&\;+C_{\dl,\kappa} \|\p_y w^\var \xw^{\sg+1}\|_{L^\infty}
			(\|\p_y \vr^\var\|_{L^\infty}\|\p_y w^\var \xw^{\sg}\|_{L^\infty}
			+\|\p_y^2 w^\var \xw^{\sg}\|_{L^\infty})\\
			\le &\; C_{\dl,\kappa}(1+\e^2),
		\end{aligned}
		\deqq
		and
		\beqq
		\begin{aligned}
			&\;\|Q_{4,3}\xw\|_{L^\infty}\\
			\le &\;C_\dl\|u^\var\|_{L^\infty}(\|\p_{xy}w^\var \xw^{\sg+1}\|_{L^\infty}
			+\|\p_y w^\var \xw^{\sg+1}\|_{L^\infty}
			\|\p_x w^\var \xw^{\sg}\|_{L^\infty})\\
			&\;+C_\dl\|v^\var \xw^{-1}\|_{L^\infty}
			(\|\p_y^2 w^\var \xw^{\sg+2}\|_{L^\infty}
			+\|\p_y w^\var \xw^{\sg+1}\|_{L^\infty}^2)\\
			\le &\; C_{\dl} (1+\e^2).
		\end{aligned}
		\deqq
		Thus, the combination of estimates of terms from
		$Q_{4,1}$ to $Q_{4,3}$ yields directly
		\beq\label{2408}
		\sum_{1\leq i \leq 3}|J_{6,i}| \le C_{\kappa,\dl}(1+\e^5).
		\deq
		Finally, let us deal with $J_{6,4}$. Integrating by part and using the estimate of $J_{5,5}$, we have
		\beqq
		\begin{aligned}
			&\;|-\var \int   \p_x^2\left\{ \frac{ \p_y w^\var}{w^\var} \right\} \cdot \p_x^6 u^\var \cdot w_g^\var \xw^{2\gm}dxdy|\\
			=&\;|\var \int \p_x \left\{\frac{ \p_y w^\var}{w^\var} \right\} \cdot \p_x^7 u^\var \cdot w_g^\var \xw^{2\gm} dxdy
			+ \var \int \p_x  \left\{\frac{ \p_y w^\var}{w^\var} \right\} \cdot \p_x^6 u^\var \cdot \p_x w_g^\var \xw^{2\gm} dxdy|\\
			\le &\; \frac{\var}{4} \|\p_x w_g^\var \xw^{\gm}\|_{L^2}^2 +
			C_{\dl,\gm,\sg} (1+\e^4)\\
			&\; +\var \|\p_x \left\{ \frac{ \p_y w^\var}{w^\var} \right\}  \xw\|_{L^\infty}\|\p_x^6 u^\var \xw^{\gm-1}\|_{L^2}
			\|\p_x w_g^\var \xw^{\gm}\|_{L^2}\\
			\le &\; \frac{\var}{4} \|\p_x w_g^\var \xw^{\gm}\|_{L^2}^2 + C_{\dl,\gm,\sg} (1+\e^4) + C_{\dl}\|\p_{xy} w^\var \xw^{\sg+1}\|_{L^\infty}\|\p_x^6 w^\var \xw^{\gm}\|_{L^2}
			\|\p_x w_g^\var \xw^{\gm}\|_{L^2}\\
			&\;+		
			C_{\dl} \var\|\p_{y} w^\var \xw^{\sg+1}\|_{L^\infty}
			\|\p_{x} w^\var \xw^{\sg}\|_{L^\infty}
			\|\p_x^6 w^\var \xw^{\gm}\|_{L^2}
			\|\p_x w_g^\var \xw^{\gm}\|_{L^2}\\
			\le &\; \frac{\var}{2} \|\p_x w_g^\var \xw^{\gm}\|_{L^2}^2 +
			C_{\dl,\gm,\sg} (1+\e^4),
		\end{aligned}
		\deqq
		and
		\beqq
		\begin{aligned}
			&\;| \int  \f{1}{w^\var} \p_y^2\left\{ \frac{ \p_y w^\var}{\rho^\var+\varrho_\infty} \right\} \cdot \p_x^6 u^\var \cdot w_g^\var \xw^{2\gm}dxdy|\\
			= &\; | \int   \p_y \left\{ \frac{ \p_y w^\var}{\rho^\var+\varrho_\infty} \right\} \cdot  \p_y \left\{\f{\p_x^6 u^\var}{w^\var} \cdot w_g^\var \xw^{2\gm} \right\}dxdy|\\
			\le  &\; \|  \p_y \left\{ \frac{ \p_y w^\var}{\rho^\var+\varrho_\infty} \right\} \xw^{\sg+1}\|_{L^{\infty}}
			(\| \p_y \left\{\f{\p_x^6 u^\var}{w^\var}\right\} \xw^{\gamma- \sg-1}\|_{L^2}    \| w_g^\var \xw^{\gm}\|_{L^2}
			+ \|\p_x^6 u^\var \xw^{\gamma-1}\|_{L^2} \| \p_y w_g^\var \xw^{\gm}\|_{L^2}\\
			&\; +  \|\p_x^6 u^\var \xw^{\gamma-1}\|_{L^2} \| w_g^\var \xw^{\gm}\|_{L^2})\\
			\le &\; \nu_3 \|\p_y w_g^\var \xw^{\gm}\|_{L^2}^2 +
			+C_{\nu_3, \dl,\kappa,\gm,\sg}(1+\e^{3}).
		\end{aligned}
		\deqq
		Combining the above estimates, we can obtain that
		\beqq
		|J_{6,4}| \le \nu_3 \|\p_y w_g^\var \xw^{\gm}\|_{L^2}^2 + \frac{\var}{2} \|\p_x w_g^\var \xw^{\gm}\|_{L^2}^2+ C_{\nu_3, \kappa, \gamma, \dl, \sg}(1+\e^4),
		\deqq
		which, together with the estimate \eqref{2408}, yields directly
		\beq\label{2409}
		|J_{6}| \le \nu_3 \|\p_y w_g^\var \xw^{\gm}\|_{L^2}^2 + \frac{\var}{2} \|\p_x w_g^\var \xw^{\gm}\|_{L^2}^2+ C_{\nu_3, \kappa, \gamma, \dl, \sg}(1+\e^5).
		\deq
		Substituting the estimates \eqref{2403}-\eqref{2407}
		and \eqref{2409} into \eqref{2402}, we can obtain
		\beqq
		\begin{aligned}
			&\;\frac{d}{dt}\|w_g^\var \xw^\gm\|_{L^2}^2
			+\var\|\p_x w_g^\var \xw^{\gm}\|_{L^2}^2
			+\|\frac{1}{\sqrt{\vr^\var+\vrf}}\p_y w_g^\var \xw^\gm\|_{L^2}^2\\
			\le &\; \nu_3 \|\p_y w_g^\var \xw^{\gm}\|_{L^2}^2 +\frac{\var}{2} \|\p_x w_g^\var \xw^{\gm}\|_{L^2}^2
			+C_{\nu_3, \kappa, \gamma, \dl, \sg}(1+\e^{19}).
		\end{aligned}
		\deqq
		Therefore, we complete the proof of this lemma.
	\end{proof}
	
	Finally, let us establish the estimate of tangential derivative of density.
	Applying $\p_x^6$-operator to the density equation $\eqref{Prandtl-01}_1$,
	we have
	\beq\label{208}
	\p_t (\p_x^6 \vr^\var)+u^\var \p_x(\p_x^6 \vr^\var)
	+v^\var \p_y (\p_x^6 \vr^\var)+\p_x^6 v^\var \p_y \vr^\var
	-\var \p_x^8 \vr^\var=Q_5,
	\deq
	where $Q_5$ is define as follows
	\beqq\label{209}
	Q_5\overset{def}{=}-\sum_{0< k \le 6}C_6^k \p_x^k u^\var \p_x^{7-k} \vr^\var
	-\sum_{0< k <6}C_6^k \p_x^k v^\var \p_x^{6-k} \p_y \vr^\var.
	\deqq
	Recall that $g_\vr^\var =\frac{\p_y \vr^\var}{w^\var}$
	and multiplying \eqref{203} by $g_\vr^\var$, we have
	\beq\label{2010}
	\p_t(\p_x^6 u^\var)g_\vr^\var+u^\var \p_x (\p_x^6 u^\var)g_\vr^\var
	+v^\var \p_y (\p_x^6 u^\var)g_\vr^\var+\p_x^6 v^\var \p_y \vr^\var
	-\var \p_x^8 u^\var g_\vr^\var
	-\p_x^6 \left\{\frac{1}{\vr^\var+\vrf}\p_yw^\var\right\}g_\vr^\var=Q_2\cdot g_\vr^\var.
	\deq
	Recall that $\vr_g^\var =\p_x^6 \vr^\var-g_\vr^\var \p_x^6 u^\var$,
	then it is easy to check that
	\beq\label{2011}
	\p_t \vr_g^\var+u^\var \p_x \vr_g^\var+v^\var \p_y \vr_g^\var-\var \p_x^2 \vr_g^\var
	=Q_5-Q_6 \cdot g_\vr^\var-Q_7 \cdot \p_x^6 u^\var+2\var\p_x^7 u^\var \p_x g_\vr^\var,
	\deq
	here $Q_i(i=6,7)$ are defined as follows
	\beqq\label{2012}
	\begin{aligned}
		Q_6\overset{def}{=}&\;Q_2+\p_x^6(\frac{1}{\vr^\var+\vrf}\p_y w^\var),\\
		Q_7\overset{def}{=}&\;-\frac{\p_y\{u^\var \p_x \vr^\var+v^\var \p_y \vr^\var-\var \p_x^2 \vr^\var\}}{w^\var}
		-\frac{\p_y \vr^\var}{(w^\var)^2}
		\left(-u^\var \p_x w^\var-v^\var \p_y w^\var
		+\p_y\{\frac{1}{\vr^\var+\vrf}\p_y w^\var\}+\var \p_x^2 w^\var \right)\\
		&\;+u^\var\left(\frac{\p_{xy}\vr^\var}{w^\var}
		-\frac{\p_y \vr^\var \p_x w^\var}{(w^\var)^2}\right)
		+v^\var\left(\frac{\p_y^2 \vr^\var}{w^\var}
		-\frac{\p_y \vr^\var \p_y w^\var}{(w^\var)^2}\right)
		-\var \p_x^2 \left\{\frac{\p_y \vr^\var}{w^\var}\right\}.
	\end{aligned}
	\deqq
	The above equation \eqref{2011} is proven in Appendix \ref{good-unkonw} in detail.
	Now, let us establish the estimate for this good quantity $\vr_g^\var$ as follows.	
	\begin{lemm}\label{lemma25}
		Under the condition of \eqref{Lowerb},
		the solution $(u^\var, v^\var, w^\var)$ of regularized system
		\eqref{Prandtl-01} will satisfy
		\beqq
		\frac{d}{dt}\|\vr_g^\var \xw^{\gm}\|_{L^2}^2
		+\var\|\p_x \vr_g^\var \xw^{\gamma}\|_{L^2}^2
		\le \nu_3  \|\p_y w_g^\var \xw^{\gm}\|_{L^2}^2+C_{\nu_3,\dl,\kappa}(1+\e^{11}),
		\deqq
		where $\nu_3$ is the small constant that will be chosen later.
	\end{lemm}
	\begin{proof}
		Multiplying the equation \eqref{2011} by $\vr_g^\var \xw^{2\gm}$ and
		integrating over $\mathbb{T}\times \mathbb{R}^+$, we have
		\beq\label{2502}
		\begin{aligned}
			&\;\frac{d}{dt}\frac{1}{2}\int (\vr_g^\var)^2 \xw^{2\gm}dxdy
			+\var\int |\p_x \vr_g^\var|^2 \xw^{2 \gamma}dxdy
			-\underset{J_{7}}{\underbrace{\gm\int |\vr_g^\var|^2 v^\var \xw^{2\gm-1}dx dy}}\\
			=&\;\underset{J_{8}}{\underbrace{\int Q_5 \cdot \vr_g^\var \xw^{2\gm}dxdy}}
			-\underset{J_{9}}{\underbrace{\int (Q_6 \cdot g_\vr^\var)\cdot \vr_g^\var \xw^{2\gm}dxdy}}
			-\underset{J_{10}}{\underbrace{\int (Q_7 \cdot \p_x^6 u^\var)
					\cdot \vr_g^\var \xw^{2\gm}dxdy}}.
		\end{aligned}
		\deq
		
		\textbf{Deal with $J_{7}$ and $J_8$ terms.} Obviously, it is easy to check that
		\beq\label{2503}
		|J_{7}|\le C_\gm \|v^\var \xw^{-1}\|_{L^\infty}
		\|\vr_g^\var \xw^{\gm}\|_{L^2}^2
		\le C_\gm \|\p_x w^\var\|_{H^1_1}\|\vr_g^\var \xw^{\gm}\|_{L^2}^2,
		\deq
		and
		\beq\label{2504}
		\begin{aligned}
			|J_8|
			\le &\;|\sum_{1\le  k \le 6}C_6^k
			\int \p_x^{k-1} \p_x u^\var \p_x^{6-k} \p_x \vr^\var
			\cdot \vr_g^\var \xw^{2\gm}dxdy|\\
			&\;+|\sum_{1\le k \le 5}C_6^k
			\int \p_x^{k-1} \p_x v^\var \p_x^{5-k} \p_{xy} \vr^\var
			\cdot \vr_g^\var \xw^{2\gm}dxdy|\\
			\le &\;C\|\p_x u^\var\|_{H^5_0}\|\p_x \vr^\var\|_{H^5_\gm}
			\|\vr_g^\var \xw^{\gm}\|_{L^2}
			+C\|\p_x v^\var\|_{H^4_{-1}}\|\p_{xy} \vr^\var\|_{H^4_{\gm+1}}
			\|\vr_g^\var \xw^{\gm}\|_{L^2}\\
			\le &\; C_{\gamma, \dl, \sg}(1+\e^4).
		\end{aligned}
		\deq
		
		\textbf{Deal with $J_{9}$ term.}
		Using Hardy inequality \eqref{Hardy1}, the Morse-type inequality \eqref{morse}
		and estimate \eqref{2101},
		we have
		\beq\label{2505}
		\begin{aligned}
			&\;\int (Q_2 \cdot g_\vr^\var)\cdot \vr_g^\var \xw^{2\gm}dx dy\\
			\le &C_{\dl} \|\p_x u^\var\|_{H^5_0}\|\p_x u^\var\|_{H^5_{\gm-1}}
			\|\p_y \vr^\var \xw^{\sg+1}\|_{L^\infty}\|\vr_g^\var \xw^{\gm}\|_{L^2}\\
			&\;+C_{\dl}\|\p_x v^\var\|_{H^4_{-1}}\|\p_x w^\var\|_{H^4_\gm}
			\|\p_y \vr^\var \xw^{\sg+1}\|_{L^\infty}
			\|\vr_g^\var \xw^{\gm}\|_{L^2}\\
			\le &\;C_{\dl}\|\p_x w^\var\|_{H^5_1}\|\p_x w^\var\|_{H^5_{\gm}}
			\|\p_y \vr^\var \xw^{\sg+1}\|_{H^2}
			\|\vr_g^\var \xw^{\gm}\|_{L^2}\\
			&\;+C_{\dl}\|\p_x^2 w^\var\|_{H^4_{1}}\|\p_x w^\var\|_{H^4_\gm}
			\|\p_y \vr^\var \xw^{\sg+1}\|_{H^2}
			\|\vr_g^\var \xw^{\gm}\|_{L^2}\\
			\le &\; C_{\gamma, \dl, \sg}(1+\e^4).
		\end{aligned}
		\deq
		Using the inequality \eqref{f02}, one arrives at
		\beq\label{2506}
		\begin{aligned}
			&\int \left(\p_x^6\{\frac{1}{\vr^\var+\vrf}\p_y w^\var\}
			\cdot g_\vr^\var\right)\cdot \vr_g^\var \xw^{2\gm}dx dy\\
			= &\int \frac{1}{\vr^\var+\vrf}\p_y \p_x^6 w^\var
			\cdot g_\vr^\var \cdot \vr_g^\var \xw^{2\gm} dx dy\\
			&\;+\sum_{1\le k\le 6}C_6^k \int \p_x^{k-1}\p_x(\frac{1}{\vr^\var+\vrf})
			\p_x^{6-k}\p_y w^\var \cdot g_\vr^\var \cdot \vr_g^\var \xw^{2\gm}dx dy\\
			\le &\;C_{\dl,\kappa} \|\p_y \vr^\var \xw^{\sg+1}\|_{L^\infty}
			\|\p_y \p_x^6 w^\var \xw^{\gm}\|_{L^2} \|\vr_g^\var \xw^{\gm-1}\|_{L^2}\\
			&\;+C_\dl \|\p_y \vr^\var \xw^{\sg+1}\|_{L^\infty}
			\|\p_x(\frac{1}{\vr^\var+\vrf})\|_{H^5_0}\|\p_y w^\var\|_{H^5_\gm}
			\|\vr_g^\var \xw^{\gm-1}\|_{L^2}\\
			\le &\;\nu_3  \|\p_y \p_x^6 w^\var \xw^{\gm}\|_{L^2}^2
			+C_{\nu_3,\dl,\kappa}
			\|\p_y \vr^\var \xw^{\sg+1}\|_{H^2}^2
			\|\vr_g^\var \xw^{\gm-1}\|_{L^2}^2\\
			&\;+C_{\dl,\kappa}\|\p_y \vr^\var \xw^{\sg+1}\|_{H^2}(1+\|\vr^\var\|_{H^6_0}^6)
			\|\p_y w^\var\|_{H^5_\gm}\|\vr_g^\var \xw^{\gm-1}\|_{L^2}\\
			\le &\; \nu_3  \|\p_y \p_x^6 w^\var \xw^{\gm}\|_{L^2}^2
			+C_{\nu_3, \kappa, \gamma, \dl, \sg}(1+\e^{11}).
		\end{aligned}
		\deq
		Therefore, the combination of estimates \eqref{2505}
		and \eqref{2506} yields directly
		\beq\label{2507}
		|J_{9}|
		\le \nu_3  \|\p_y \p_x^6 w^\var \xw^{\gm}\|_{L^2}^2
		+C_{\nu_3, \kappa, \gamma, \dl, \sg}(1+\e^{11}).
		\deq
		
		\textbf{Deal with $J_{10}$ term}. Obviously, it holds
		\beq\label{2508}
		|J_{10}|\le \|Q_7 \xw\|_{L^\infty}\|\p_x^6 u^\var \xw^{\gm-1}
		\|_{L^2}\|\vr_g^\var \xw^\gm\|_{L^2}.
		\deq
		Using the lower bound condition \eqref{Lowerb}, it holds
		\beq\label{2509}
		\begin{aligned}
			&\;\|\frac{\p_y\{u^\var \p_x \vr^\var+v^\var \p_y \vr^\var-\var \p_x^2 \vr^\var\}\xw}{w^\var}\|_{L^\infty}\\
			= &\;\|\frac{(\p_y u^\var \p_x \vr^\var+u^\var\p_{xy}\vr^\var
				+\p_y v^\var \p_y \vr^\var+v^\var \p_y^2 \vr^\var-\var \p_x^2 \p_y \vr^\var) \xw}{w^\var}\|_{L^\infty}\\
			\le&\;C_\dl\|\p_x \vr^\var \xw\|_{L^\infty}
			+C_\dl(1+\|u^\var-u_\infty\|_{L^\infty})
			\|\p_{xy}\vr^\var \xw^{\sg+1}\|_{L^\infty}\\
			&\;+C_\dl\|\p_x u^\var\|_{L^\infty}
			\|\p_y \vr^\var \xw^{\sg+1}\|_{L^\infty}
			+C_\dl\|v^\var \xw^{-1}\|_{L^\infty}
			\|\p_y^2 \vr^\var \xw^{\sg+2}\|_{L^\infty}+C_\dl \|\p_x^2 \p_y \vr^\var \xw^{\sg+1}\|_{L^\infty}\\
			\le&\;C_\dl\|\p_x \vr^\var \xw\|_{H^2}   +C_\dl(1+\|w^\var\|_{H^2_1})\|\p_{xy}\vr^\var \xw^{\sg+1}\|_{H^2_0}\\
			&\;+C_\dl\|\p_x w^\var\|_{H^2_1}\|\p_y \vr^\var \xw^{\sg+1}\|_{H^2_0}
			+C_\dl\|\p_xw^\var\|_{H^1_1}\|\p_y^2 \vr^\var\xw^{\sg+2}\|_{H^2_0}+C_\dl \|\p_x^2 \p_y \vr^\var \xw^{\sg}\|_{L^\infty} \\
			\le&\; C_\dl(1+\e).
		\end{aligned}
		\deq
		Similarly, it is easy to check that
		\beq\label{2510}
		\begin{aligned}
			&\;\|\frac{\p_y \vr^\var}{(w^\var)^2}
			\left(-u^\var \p_x w^\var-v^\var \p_y w^\var
			+\p_y(\frac{1}{\vr^\var+\vrf}\p_y w^\var)+\var \p_x^2 w^\var\right)\xw\|_{L^\infty}\\
			\le &\;C_\dl\|\p_y \vr^\var \xw^{\sg+1}\|_{L^\infty}
			(1+\|u^\var-u_\infty\|_{L^\infty})
			\|\p_x w^\var \xw^{\sg}\|_{L^\infty}\\
			&\;+C_\dl\|\p_y \vr^\var \xw^{\sg+1}\|_{L^\infty}
			\|v^\var \xw^{-1}\|_{L^\infty}
			\|\p_y w^\var \xw^{\sg+1}\|_{L^\infty}\\
			&\;+C_\dl \|\p_y \vr^\var \xw^{\sg+1}\|_{L^\infty}
			\left(\|\p_y \vr^\var\|_{L^\infty}
			\|\p_y w^\var \xw^{\sg}\|_{L^\infty}
			+\|\p_y^2 w^\var \xw^{\sg}\|_{L^\infty}\right)\\
			&\;+C_{\dl,\kappa} \|\p_y \vr^\var \xw^{\sg+1}\|_{L^\infty} \|\p_x^2 w^\var \xw^{\sg}\|_{L^\infty}\\
			\le &\;C_{\dl,\kappa}(1+\e^2),
		\end{aligned}
		\deq
		and
		\beq\label{2511}
		\begin{aligned}
			&\;\|u^\var \left(\frac{\p_{xy}\vr^\var}{w^\var}
			-\frac{\p_y \vr^\var \p_x w^\var}{(w^\var)^2}\right)\xw\|_{L^\infty}
			+\|v^\var\left(\frac{\p_y^2 \vr^\var}{w^\var}
			-\frac{\p_y \vr^\var \p_y w^\var}{(w^\var)^2}\right)\xw\|_{L^\infty}
			+\var \|\p_x^2 \left\{\frac{\p_y \vr^\var}{w^\var}\right\}\|_{L^\infty}\\
			\le &\;C_\dl(1+\|u^\var-u_\infty\|_{L^\infty})
			(\|\p_{xy}\vr^\var \xw^{\sg+1}\|_{L^\infty}
			+\|\p_y \vr^\var \xw^{\sg+1}\|_{L^\infty}
			\|\p_x w^\var \xw^{\sg}\|_{L^\infty})\\
			&\;+C_\dl\|v^\var \xw^{-1}\|_{L^\infty}
			(\|\p_y^2 \vr^\var \xw^{\sg+2}\|_{L^\infty}
			+\|\p_y \vr^\var \xw^{\sg+1}\|_{L^\infty}
			\|\p_y w^\var \xw^{\sg+1}\|_{L^\infty})\\
			&\;+C_\dl \|\p_x^2 \p_y \vr^\var \xw^{\sg+1}\|_{L^\infty}
			+C_\dl \|\p_x \p_y \vr^\var \xw^{\sg+1}\|_{L^\infty} \|\p_x w^\var \xw^{\sg}\|_{L^\infty}\\
			&\;+C_\dl \|\p_y \vr^\var \xw^{\sg+1}\|_{L^\infty} (\|\p_x^2  w^\var \xw^{\sg}\|_{L^\infty}+\|\p_x w^\var \xw^{\sg}\|_{L^\infty})\\
			\le &\;C_\dl(1+\|w^\var\|_{H^2_1})
			(\|\p_{xy}\vr^\var \xw^{\sg+1}\|_{H^2}
			+\|\p_y \vr^\var \xw^{\sg+1}\|_{H^2}
			\|\p_x w^\var \xw^{\sg}\|_{L^\infty})\\
			&\;+C_\dl\|w^\var\|_{H^2_1}
			(\|\p_y^2 \vr^\var \xw^{\sg+2}\|_{H^2}
			+\|\p_y \vr^\var \xw^{\sg+1}\|_{H^2}
			\|\p_y w^\var \xw^{\sg+1}\|_{L^\infty})\\
			&\;+C_\dl \|\p_x^2 \p_y \vr^\var \xw^{\sg+1}\|_{L^\infty}
			+C_\dl \|\p_x \p_y \vr^\var \xw^{\sg+1}\|_{L^\infty} \|\p_x w^\var \xw^{\sg}\|_{L^\infty}\\
			&\;+C_\dl \|\p_y \vr^\var \xw^{\sg+1}\|_{L^\infty} (\|\p_x^2  w^\var \xw^{\sg}\|_{L^\infty}+\|\p_x w^\var \xw^{\sg}\|_{L^\infty})\\
			\le &\;C_\dl(1+\e^2).
		\end{aligned}
		\deq
		The combination of estimates \eqref{2509}-\eqref{2511} yields directly
		\beqq
		\|Q_7 \xw\|_{L^\infty} \le C_{\dl,\kappa}(1+\e^2),
		\deqq
		which, together with estimate \eqref{2508}, yields directly
		\beq\label{2512}
		|J_{10}| \le C_{\dl,\kappa}(1+\e^4) .
		\deq
		Substituting the estimates \eqref{2503}, \eqref{2504},
		\eqref{2507} and \eqref{2512} into \eqref{2502}, we have
		\beqq
		\frac{d}{dt}\|\vr_g^\var \xw^{\gm}\|_{L^2}^2
		+\var\|\p_x \vr_g^\var \xw^{\gamma}\|_{L^2}^2
		\le \nu_3  \|\p_y \p_x^6 w^\var \xw^{\gm}\|_{L^2}^2
		+C_{\nu_3, \kappa, \gamma, \dl, \sg}(1+\e^{11}),
		\deqq
		which, together with \eqref{2102}, yields directly
		\beqq
		\frac{d}{dt}\|\vr_g^\var \xw^{\gm}\|_{L^2}^2
		+\var\|\p_x \vr_g^\var \xw^{\gamma}\|_{L^2}^2
		\le \nu_3  \|\p_y w_g^\var \xw^{\gm}\|_{L^2}^2
		+C_{\nu_3,\kappa, \gamma, \dl, \sg}(1+\e^{11}).
		\deqq
		Therefore, we complete the proof of this lemma.
	\end{proof}

	\subsection{Weighted $L^\infty$-estimates for lower order terms}
	
	In order to close the energy estimate, we need to establish the weighted $L^\infty$-estimates for lower-order derivatives of $w^\var$ and density.
	First of all, we will establish the uniform (in $\var$) weighted $L^\infty$-estimates for $\p^\alpha w^\var (|\alpha| \leq 2)$. Through viewing the equation of $\p^\alpha w^\var $ as a ``linear" parabolic equations, the estimate can be obtained by
	using the classical maximum principle stated in Lemmas
	\ref{max} and  \ref{min}(see Appendix \ref{appendix-a}).
	For all $|\alpha|\le 2$, let us define
	$$
	B_\alpha\overset{def}{=}\p^\al w^\var \xw^{\sg+\alpha_2},
	$$	
	and
	$$I\overset{def}{=}\sum_{0\le \beta\le \alpha}|B_{\beta}^2|.$$
	Then, we can establish the following estimate.
	\begin{lemm}\label{lemma26}
		Under the condition of \eqref{Lowerb},
		the solution $(u^\var, v^\var, w^\var)$ of regularized system
		\eqref{Prandtl-01} will satisfy
		\beqq
		\|\sum_{|\alpha|\le 2}|\p^\alpha w^\var|^2 \xw^{2\sg+2\alpha_2}\|_{L^\infty}
		\le \left\{\|I(0)\|_{L^\infty(\mathbb{T}\times \mathbb{R}^+)}+C_{\kappa} t(1+\e^4)\right\}
		e^{C_{\sg,\kappa}(1+\e^2)t}.
		\deqq
	\end{lemm}
	
	\begin{proof}
		By a direct computation, the quantity $B_\alpha$ satisfies
		\beq\label{2602}
		(\p_t+u^\var \p_x+v^\var \p_y-\var\p_x^2-\frac{1}{\vr^\var+\vrf}\p_y^2)B_\alpha
		=Q_\alpha \p_y B_\alpha +R_\alpha B_\alpha+S_\alpha,
		\deq
		where $Q_\alpha, R_\alpha$ and $S_\alpha$ are defined as follows:
		\beqq
		\begin{aligned}
			Q_\alpha\overset{def}{=}&\p_y(\frac{1}{\vr^\var+\vrf})
			-\frac{2(\sg+\alpha_2)}{\xw(\vr^\var+\vrf)};\\
			R_\alpha\overset{def}{=}&(\sg+\alpha_2) \xw^{-1} v^\var
			+\frac{(\sg+\alpha_2)(\sg+\alpha_2+1)}{\xw^2 (\vr^\var+\vrf)}
			+\frac{(\sg+\alpha_2)\p_y \vr^\var}{\xw (\vr^\var+\vrf)^2}.\\
		\end{aligned}
		\deqq
		If $|\alpha|=0$, then $S_\alpha\overset{def}{=}0$. If $|\alpha|>0$, then
		\beqq
		\begin{aligned}
			S_\alpha\overset{def}{=}&
			\sum_{0<\beta\le \alpha}C_\alpha^\beta\p^{\beta}
			(\frac{1}{\vr^\var+\vrf})\xw^{\beta_2-1} \p_y B_{\alpha-\beta+e_2}
			+\sum_{0<\beta\le \alpha}C_\alpha^\beta\p^{\beta+e_2}
			(\frac{1}{\vr^\var+\vrf})\xw^{\beta_2-1}B_{\alpha-\beta+e_2}\\
			&-(\sg+\alpha_2-\beta_2+1)\sum_{0<\beta\le \alpha}
			C_\alpha^\beta\p^{\beta}(\frac{1}{\vr^\var+\vrf})\xw^{\beta_2-2}
			B_{\alpha-\beta+e_2}\\
			&-\sum_{0<\beta\le \alpha}C_\alpha^\beta \p^\beta u^\var \xw^{\beta_2}
			B_{\alpha-\beta+e_1}
			-\sum_{0<\beta\le \alpha}C_\alpha^\beta
			\p^\beta v^\var \xw^{\beta_2-1}B_{\alpha-\beta+e_2}.
		\end{aligned}
		\deqq
		Obviously, it is easy to check that
		\beq\label{2603}
		\begin{aligned}
			|Q_\alpha|\le &\;C_{\sg,\kappa}(1+\|\p_y \vr^\var\|_{H^2_0}),\quad
			|R_\alpha|\le C_{\sg,\kappa}(1+\|w^\var\|_{H^2_1}+\|\p_y \vr^\var\|_{H^2_0}),\\
			|S_\alpha| \le &\; C_{\sg, \kappa}(1+\|\vr^\var\|_{H^5_1}^3+\|w^\var\|_{H^4_1})
			\sum_{0<\beta\le \alpha}(|B_{\alpha-\beta+e_2}|+|B_{\alpha-\beta+e_1}|)\\
			&\;+C_{\sg, \kappa}(1+\|\vr^\var\|_{H^4_1}^2)
			\sum_{0<\beta\le \alpha}|\p_y B_{\alpha-\beta+e_2}|.
		\end{aligned}
		\deq		
		According to the equation \eqref{2602}, quantity $I$ will satisfy
		\beqq
		\begin{aligned}
			&\;(\p_t+u^\var \p_x+v^\var \p_y-\var \p_x^2-\frac{1}{\vr^\var+\vrf}\p_y^2)I\\
			=&\;-2\sum_{|\alpha|\le 2}(\var |\p_x B_\alpha|^2
			+\frac{1}{\vr^\var+\vrf}|\p_y B_\alpha|^2)
			+2\sum_{|\alpha|\le 2}  Q_\alpha B_\alpha \p_y B_\alpha +R_\alpha |B_\alpha|^2
			+S_\alpha B_\alpha,\\
			\le &\;C_{\sg,\kappa}(1+\|\vr^\var\|_{H^5_1}^4+\|w^\var\|_{H^4_1}^4)I,
		\end{aligned}
		\deqq
		where we have used the estimate \eqref{2603} in the last inequality.
		Applying the classical maximum principle inequality \eqref{max-est}
		in Lemma \ref{max} for parabolic equation to the quantity $I$ under condition
		of $\vr^\var + \vrf$ in \eqref{Lowerb},
		we have
		\beq\label{2605}
		\begin{aligned}
			\|I(t)\|_{L^\infty(\mathbb{T}\times \mathbb{R}^+)}
			\le &\max \left\{e^{C_{\sg,\kappa}(1+\|\vr^\var\|_{H^5_1}^4+\|w^\var\|_{H^4_1}^4)t}
			\|I(0)\|_{L^\infty(\mathbb{T}\times \mathbb{R}^+)},\right.\\
			&\quad \quad \quad \left.\underset{\tau\in [0, t]}{\max}\{e^{C_{\sg,\kappa}
				(1+\|\vr^\var\|_{H^5_1}^4+\|w^\var\|_{H^4_1}^4)(t-\tau)}
			\|I(\tau)|_{y=0}\|_{L^\infty(\mathbb{T})}\}\right\}.
		\end{aligned}
		\deq
		Finally, let us estimate the term $\|I(t)|_{y=0}\|_{L^\infty(\mathbb{T})}$.
		Due to the fact that
		\beqq
		\|I(t)|_{y=0}\|_{L^\infty(\mathbb{T})}
		\le \|I(0)|_{y=0}\|_{L^\infty(\mathbb{T})}
		+t\|\p_t I(t)|_{y=0}\|_{L^\infty(\mathbb{T})},
		\deqq
		and
		\beqq
		\p_t (\p^\alpha w^\var)
		=\p^\alpha\left\{\var\p_x^2 w^\var+\p_y(\frac{1}{\vr^\var+\vrf} \p_y w^\var)\right\}
		-\p^\alpha(u^\var \p_x w^\var+v^\var \p_y w^\var).
		\deqq
		For all $|\alpha|\le 2$, it is easy to check that
		\beqq
		(\p^\alpha\{\var\p_x^2 w^\var+\p_y(\frac{1}{\vr^\var+\vrf} \p_y w^\var)\})_{y=0}
		\le C_\kappa(1+\|\vr^\var\|_{H^4_0(\mathbb{R}^+)}^4 )\| \p_y w^\var\|_{H^4_1(\mathbb{R}^+)},
		\deqq
		and
		\beqq
		(\p^\alpha\{u^\var \p_x w^\var+v^\var \p_y w^\var\})_{y=0}
		\le C(1+\|w^\var\|_{H^4_1(\mathbb{R}^+)}^2).
		\deqq
		Thus, we have
		\beqq
		\|\p_t I(t)|_{y=0}\|_{L^\infty(\mathbb{T})}
		\le C_{\kappa}(1+\|\vr^\var\|_{H^5_0}^8+\|(w^\var, \p_y w^\var)\|_{H^5_1}^2),
		\deqq
		and hence, it holds
		\beq\label{2606}
		\|I(t)|_{y=0}\|_{L^\infty(\mathbb{T})}
		\le \|I(0)|_{y=0}\|_{L^\infty(\mathbb{T})}
		+C_{\kappa} t(1+\|\vr^\var\|_{H^5_0}^8+\|(w^\var, \p_y w^\var)\|_{H^5_1}^2).
		\deq
		Thus, substituting the estimate \eqref{2606} into \eqref{2605},
		then we have
		\beqq
		\|I(t)\|_{L^\infty(\mathbb{T}\times \mathbb{R}^+)}
		\le \left\{\|I(0)\|_{L^\infty(\mathbb{T}\times \mathbb{R}^+)}
		+C_{\kappa} t(1+\|\vr^\var\|_{H^5_0}^8+\|(w^\var, \p_y w^\var)\|_{H^5_1}^4)\right\}
		e^{C_{\sg,\kappa}(1+\|\vr^\var\|_{H^5_1}^4+\|w^\var\|_{H^4_1}^4)t}.
		\deqq
		Therefore, we complete the proof of this lemma.
	\end{proof}
	
	Next, we will establish estimates for the lower bound of vorticity
	and the bound of density respectively.
	\begin{lemm}\label{lemma27}
		Under the condition of \eqref{Lowerb},
		the solution $(u^\var, v^\var, w^\var)$ of regularized system
		\eqref{Prandtl-01} will satisfy
		\beq\label{2701}
		\underset{\mathbb{T}\times\mathbb{R}^+}{\min}\xw^{\sg}w (t)
		\ge\left\{1-T\ \underset{t\in[0, T]}{\max}(1+\e)
		e^{t \ \underset{t \in [0, T]}{\max}(1+\e)}\right\}
		\left\{\underset{\mathbb{T}\times\mathbb{R}^+}{\min} \xw^{\sg} w_0
		-C_\kappa T(1+\underset{t\in[0, T]}{\max}\e^2)\right\},
		\deq
		and
		\beq\label{2702}
		\rho_0
		-CT(1+\underset{ t \in [0,T]}{\max}\e)
		\le \vr^\var+\vrf
		\le \rho_0
		+CT(1+\underset{ t \in [0,T]}{\max}\e).
		\deq
	\end{lemm}
	\begin{proof}
		Let us define $B_{(0, 0)}\overset{def}{=}w^\var \xw^{\sg}$,
		this quantity will satisfy
		\beqq
		\p_t B_{(0, 0)}+u^\var \p_x B_{(0, 0)}
		+\left(v^\var+\frac{2\sg \xw^{-1}}{\vr^\var+\vrf}
		+\frac{\p_y \vr^\var}{(\vr^\var+\vrf)^2}\right)\p_y B_{(0, 0)}
		-\frac{1}{\vr^\var+\vrf}\p_y^2 B_{(0, 0)}
		=Q_8 B_{(0, 0)},
		\deqq
		where $Q_8$ is defined as follows
		\beqq
		Q_8\overset{def}{=}
		\sg \xw^{-1} v^\var+\frac{\sg(\sg+1)\xw^{-2}}{\vr^\var+\vrf}
		+\frac{\sg\xw^{-1}\p_y \vr^\var}{(\vr^\var+\vrf)^2}.
		\deqq
		Then, using the classical minimum principle \eqref{min-est}
		in Lemma \ref{min}, we obtain
		\beq\label{2704}
		\begin{aligned}
			\underset{\mathbb{T}\times\mathbb{R}^+}{\min}\xw^{\sg}w^\var (t)
			\ge&(1-t\|Q_8\|_{L^\infty([0, T]\times\mathbb{T}\times\mathbb{R}^+)}
			e^{t\|Q_8\|_{L^\infty([0, T]\times\mathbb{T}\times\mathbb{R}^+)}})
			\min \left\{\underset{\mathbb{T}\times\mathbb{R}^+}{\min} \xw^{\sg} w_0,
			\underset{[0, T]\times \mathbb{T}}{\min}w^\var|_{y=0}\right\}.
		\end{aligned}
		\deq
		Obviously, it is easy to check that
		\beq\label{2705}
		\|Q_8\|_{L^\infty([0, T]\times\mathbb{T}\times\mathbb{R}^+)}
		\le \underset{[0, T]}{\max}\left\{1+\|w^\var\|_{H^2_1}
		+\|\p_y \vr^\var\|_{H^2_0}\right\}.
		\deq
		On the other hand, we have
		\beq\label{2706}
		\underset{[0, T]\times \mathbb{T}}{\min}w^\var|_{y=0}
		\ge \underset{\mathbb{T}}{\min}w_0|_{y=0}
		-T\|\p_t w^\var|_{y=0}\|_{L^\infty([0, T]\times \mathbb{T})}.
		\deq
		Using the vorticity equation $\eqref{Prandtl-02}_1$, we can obtain
		\beq\label{2709}
		\|\p_t w^\var|_{y=0}\|_{L^\infty([0, T]\times \mathbb{T})}
		=\|\p_y(\frac{\p_y w^\var}{\vr^\var+\vrf})
		|_{y=0}\|_{L^\infty([0, T]\times \mathbb{T})}.
		\deq
		Due to the fact that
		\beqq
		|(\p_y(\frac{1}{\vr^\var+\vrf})\p_y w^\var)_{y=0}|
		\le  \|\p_y^2(\frac{1}{\vr^\var+\vrf})\|_{L^2(\mathbb{R}^+)}
		\|\p_y w^\var\|_{L^2(\mathbb{R}^+)}
		+\|\p_y (\frac{1}{\vr^\var+\vrf})\|_{L^2(\mathbb{R}^+)}
		\|\p_y^2 w^\var\|_{L^2(\mathbb{R}^+)},
		\deqq
		and
		\beqq
		|(\frac{1}{\vr^\var+\vrf}\p_y^2 w^\var)_{y=0}|
		\le C_\kappa(\|\p_y \vr^\var\|_{L^2(\mathbb{R}^+)}\|\p_y w^\var\|_{L^2(\mathbb{R}^+)}
		+\|\p_y^3 w^\var\|_{L^2(\mathbb{R}^+)}),
		\deqq
		it holds
		\beq\label{2707}
		\|\p_y(\frac{\p_y w^\var}{\vr^\var+\vrf})|_{y=0}
		\|_{L^\infty([0, T]\times \mathbb{T})}
		\le C_{\kappa}( 1+\underset{[0, T]}{\max}(\|\vr^\var \|_{H^3_1}^3+\|w^\var\|_{H^4_1}^3)).
		\deq
		Substituting the estimates \eqref{2709} and \eqref{2707}
		into \eqref{2706}, we have
		\beqq
		\underset{[0, T]\times \mathbb{T}}{\min}w^\var|_{y=0}
		\ge \underset{\mathbb{T}}{\min}w_0|_{y=0}
		-C_{\kappa}T(1+\underset{[0, T]}{\max}(\|\vr^\var \|_{H^3_1}^3+\|w^\var\|_{H^4_1}^3)),
		\deqq
		which, together with estimates \eqref{2704} and \eqref{2705}, yields directly
		\beqq
		\begin{aligned}
			\underset{\mathbb{T}\times\mathbb{R}^+}{\min}\xw^{\sg}w^\var (t)
			\ge&\left\{1-T\ \underset{[0, T]}{\max}(1+\|w^\var\|_{H^2_1}+\|\p_y \vr^\var\|_{H^2_0})
			e^{t \ \underset{[0, T]}{\max}(1+\|w^\var\|_{H^2_1}
				+\|\p_y \vr^\var\|_{H^2_0})}\right\}\\
			&\times \left\{\underset{\mathbb{T}\times\mathbb{R}^+}{\min} \xw^{\sg} w_0
			-C_{\kappa}T(1+\underset{[0, T]}{\max}(\|\vr^\var \|_{H^3_1}^3+\|w^\var\|_{H^4_1}^3))\right\}.
		\end{aligned}
		\deqq
		Finally, we hope to establish the  bound for density.
		Obviously, it is easy to check that
		\beq\label{2708}
		(\vr^\var+\vrf)|_{t=0}-t\|\p_t \vr^\var\|_{L^\infty} \le \vr^\var+\vrf \le (\vr^\var+\vrf)|_{t=0}+t\|\p_t \vr^\var\|_{L^\infty} .
		\deq
		Due to the density equation $\eqref{Prandtl-01}_1$, we can obtain
		\beq\label{2710}
		\begin{aligned}
			\|\p_t \vr^\var\|_{L^\infty}
			&\le \var\|\p_x^2 \vr^\var\|_{L^\infty}+\|u^\var \p_x \vr^\var\|_{L^\infty}
			+\|v^\var \p_y \vr^\var\|_{L^\infty}\\
			&\le C\|\p_x^2 \vr^\var\|_{H^2_0}+C(1+\|w^\var\|_{H^1_1})\|\vr^\var\|_{H^3_0}\\
			&\le C(1+\|(\vr^\var, w^\var)\|_{H^4_1}^2).
		\end{aligned}
		\deq
		The combination of \eqref{2708} and \eqref{2710} yields directly
		\beqq
		\rho_0
		-CT(1+\underset{t\in[0, T]}{\max}\|(\vr^\var, w^\var)(t)\|_{H^4_1}^2)
		\le \vr^\var+\vrf \le \rho_0
		+CT(1+\underset{t\in[0, T]}{\max}\|(\vr^\var, w^\var)(t)\|_{H^4_1}^2).
		\deqq
		Therefore, we complete the proof of this lemma.
	\end{proof}

	\subsection{Proof of Proposition \ref{Main-Pro}}
	In this subsection, we will give the proof of Proposition \ref{Main-Pro}.
	Indeed, due to the the upper bound of $\vr^\var+\vrf$
	and smallness of $\nu_i(i=1,2,3)$, the combination of estimates obtained in Lemmas \ref{lemma22}-\ref{lemma26} yields directly
	\beqq
	\begin{aligned}
		\e
		+ \int_0^t \mathcal{D}_1 (\tau) d\tau
		\le  \mathcal{E}_1(0)
		+C_{\bar{\kappa},\kappa, \gm, \dl,\sg} \int_0^t (1+\e^{20}) d\tau
		+\left\{\|I(0)\|_{L^\infty(\mathbb{T}\times \mathbb{R}^+)}
		+C_{\kappa} t(1+\e^4)\right\}
		e^{C_{\sg,\kappa}(1+\e^2)t}.
	\end{aligned}
	\deqq
	Therefore, we complete the proof of Proposition \ref{Main-Pro}.
	
	\section{Existence and uniqueness of the original unsteady  Prandtl equations }\label{well-posedness}
	In this section, we will establish the local-in-time existence and uniqueness
	stated in the Theorem \ref{main-result-un} for the original system
	\eqref{Prandtl}.
	Indeed, with the help of uniform estimate \eqref{uiform-estimate}
	in Proposition \ref{Main-Pro} and estimates \eqref{2701}
	and \eqref{2702} in Lemma \ref{lemma27}, we can obtain
	the uniform life-span existence time and the uniform estimate.
	Then, one can obtain the solution  of original Prandtl equation \eqref{Prandtl}
	from the approximated system \eqref{Prandtl-01} as $\var$ tends to $0^+$.
	
	\subsection{Local-in-time existence of the original system}
	\textbf{Step 1: Uniform estimate and life-span time.}
	For the parameters $R_0$, $\dl$, $\kappa$ and $\bar{\kappa}$, which will be defined later, we define
	\beq\label{criterion}
	\begin{aligned}
		T_*^\var \overset{\text{def}}{=}
		&\sup\left\{T\in (0, 1]|  \e \le R_0, w^\var \xw^{\sg} \ge \dl,
		\kappa \le \vr^\var + \vrf \le \bar{\kappa},
		(t, x, y) \in [0, T] \times \mathbb{T} \times \mathbb{R}^+ \right\}.
	\end{aligned}
	\deq
	Then, from the Proposition \ref{Main-Pro}, we conclude
	for all $T \le T_*^\var$  that
	\beq\label{estimate-E-1-1}
	\begin{aligned}
		&\;E_1(T)+\int_0^{T} \mathcal{D}_1 (\tau) d\tau \\
		\le &\; \mathcal{E}_1(0)
		+C_{\bar{\kappa},\kappa,\dl,\gm, \sg} (1+E_1(T)^{20})T
		+ (\|I(0)\|_{L^\infty(\mathbb{T}\times \mathbb{R}^+)}
		+C_{\kappa} T (1+E_1(T)^4))
		e^{C_{\sg,\kappa} (1+E_1(T)^2)T_0}.
	\end{aligned}
	\deq
	Choose  $(\kappa, \bar\kappa, \dl)\overset{def}{=}
	(2\kappa_1, \f{\kappa_2}{2}, \frac{\bde}{2})$,
	then \eqref{estimate-E-1-1} yields directly
	\beq\label{estimate-E-1}
	\begin{aligned}
		&\; E_1(T)+\int_0^{T} \mathcal{D}_1 (\tau) d\tau \\
		\le & \; \mathcal{E}_1(0)
		+C_{\kappa_1,\kappa_2,\bde,\gm, \sg} (1+E_1(T)^{20})T
		+ (\|I(0)\|_{L^\infty(\mathbb{T}\times \mathbb{R}^+)}
		+C_{\kappa_1} T (1+E_1(T)^4))
		e^{C_{{\sg},\kappa_1} (1+E_1(T)^2)T_0}.
	\end{aligned}
	\deq
	Let us choose
	$T_1 \overset{def}{=} \min \left\{1,
	\frac{\mathcal{E}_1(0) }{2 C_{\kappa_1,\kappa_2, \bde,\gm, \sg}
		(1+(4 \mathcal{E}_1(0))^{20})},
	\frac{ \|I(0)\|_{L^\infty(\mathbb{T}\times \mathbb{R}^+)} }{ C_{\kappa_1}(1+(4 \mathcal{E}_1(0))^{4})},
	\frac{1}{C_{\sg,\kappa_1}(1+(4 \mathcal{E}_1(0))^{2})}
	\ln(\frac{\mathcal{E}_1(0)}{4 \|I(0)\|_{L^\infty(\mathbb{T}\times \mathbb{R}^+)}}) \right\}$ and $R_0 \overset{def}{=} 4 \mathcal{E}_1(0)$,
	then the inequality \eqref{estimate-E-1} yields directly
	\beq \label{E-2}
	E_1(T)+\int_0^{T} \mathcal{D}_1 (\tau) d\tau \le 2 \mathcal{E}_1(0) =\f {R_0}{2},
	\deq
	for all $T\le \min \left\{T_1, T_*^\var \right\}$.
	Choose  $T_2 \overset{def}{=} \min \left\{T_1,
	\frac{1}{6(1+4 \mathcal{E}_1(0))},
	\frac{\bde }{2C_{\kappa_1} (1+(4 \mathcal{E}_1(0))^{2})},
	\frac{ \ln 2 }{ 1+4 \mathcal{E}_1(0)}\right\}$,
	according to estimate \eqref{2701}, we have
	\beq \label{estimate-w}
	\begin{aligned}
		\underset{\mathbb{T}\times\mathbb{R}^+}{\min}\xw^{\sg}w (t)
		&\ge\left\{1-T(1+E_1(T))
		e^{T (1+E_1(T))} \right\}
		\left\{\underset{\mathbb{T}\times\mathbb{R}^+}{\min} \xw^{\sg} w_0
		-C_{\kappa_1} T(1+E_1(T)^2)\right\}
		\ge \bde=2\dl,
	\end{aligned}
	\deq
	for all $T\le \min \left\{T_2, T_*^\var \right\}$.
	Choose
	$T_3 \overset{def}{=} \min \left\{T_2,
	\frac{ \kappa_1}{ C (1+4 \mathcal{E}_1(0))},
	\frac{ \kappa_2}{ 2C (1+4 \mathcal{E}_1(0))} \right\}$,
	then we have
	\beq \label{estimate-rho}
	\vr^\var+\vrf  \ge \rho_0
	-CT(1+E_1(T)) \ge  \kappa_1 =\f{\kappa}{2},
	\deq
	and
	\beq \label{estimate-rho-1}
	\vr^\var+\vrf \le \rho_0 +CT(1+E_1(T)) \le \kappa_2
	=2\bar{\kappa},
	\deq
	for all $T\le \min \left\{T_3, T_*^\var \right\}$.	This yields $T_3 \le T_*^\var$.
	Indeed otherwise, our criterion about the continuation of the solution
	would contradict the definition of $T_*^\var$ in \eqref{criterion}.
	Let us define $T_{0}\overset{def}{=} T_3$, then we find the
	uniform existence time $T_{0}$(independent of $\var$)
	such that the estimates \eqref{E-2}, \eqref{estimate-w},\eqref{estimate-rho}
	and \eqref{estimate-rho-1} hold.
	Therefore, estimate \eqref{E-2} and lower semi-continuity of norm yield
	the estimate \eqref{estimate-X} in Theorem \ref{main-result-un}.	\\
	\textbf{Step 2: Local-in-time existence.}
	Using the estimates \eqref{E-2}, we have
	\beqq
	\sup_{0\leq t \leq T_0}  ( \|w^\var\|_{H^{6}_{loc}} +\|\vr^\var\|_{H^{6}_{loc}} +\|u^\var-\uf \|_{H^{6}_{loc}} )
	\le \sqrt{2}\mathcal{E}_1(0)^{\f12}.
	\deqq
	One can also find that
	$\p_t (w^\var, \vr^\var)$ and $\p_t u^\var$ are uniformly (in $\var$) in $L^2(0, T_0; H^{4}_{loc})$ and $L^2(0, T_0; H^{4}_{loc})$ respectively.
	Then it follows from a strong compactness argument(see Lemma 4 in \cite{Simon1990}) that
	$(w^\var, \vr^\var, u^\var-\uf)$ is compact in $C(0, T_0; H^s_{loc})(s<6
	)$.
	In particular, there exists a sequence $\var_n \rightarrow 0^+$
	and $w, \vr, u-\uf \in L^\infty([0,T_0]; H^{6}_{loc})\cap C(0, T_0; H^s_{loc})$ such that
	\begin{equation} \label{e:converge w,u-uf,vr}
		\left\{\begin{aligned}
			w^{\var_n} &\rightarrow w &&\text{ in } C([0,T_0]; H^{s}_{loc});\\
			\vr^{\var_n} &\rightarrow \vr &&\text{ in } C([0,T_0]; H^{s}_{loc});\\
			u^{\var_n} -\uf &\rightarrow u-\uf &&\text{ in } C([0,T_0]; H^{s}_{loc}).
		\end{aligned}\right.
	\end{equation}
	Then, we can define $\rho\overset{def}{=}\vr+\vrf$.
	Using the local uniform convergence of $\p_x u^{\var_n}$, we also have the
	pointwise convergence of $v^{\var_n}$: as $\var_n \to 0^+$,
	\begin{equation}  \label{e:converge v}
		v^{\var_n} = - \int^{y}_0 \p_x u^{\var_n}\, dy \rightarrow - \int^{y}_0 \p_x u \, dy \overset{def}{=} v.
	\end{equation}
	Therefore, combining \eqref{e:converge w,u-uf,vr}-\eqref{e:converge v}, one
	may justify the pointwise convergences of all terms in the regularized Prandtl equations
	\eqref{Prandtl-01}. Thus, passing to the limit $\var_n \rightarrow 0^+$ in
	\eqref{Prandtl-01},
	we know that the limit $(u,v,\rho)$ solves the Prandtl equations \eqref{Prandtl}
	in the classical sense. Furthermore, it is easy to check that
	\beqq
	w(t,x,y) \xw^{\sg} \ge \bde, \quad  \quad \kappa_1 \leq \rho(t,x,y) \le \kappa_2,
	\deqq
	for all $(t,x,y) \in [0,T_0]\times \mathbb{T}\times \mathbb{R}^+$, which implies the estimate \eqref{estimate-w,rho} in Theorem \ref{main-result-un}.
	Therefore, we complete the proof of the local-in-time existence of solution
	and estimates \eqref{estimate-w,rho} and \eqref{estimate-X}
	in Theorem \ref{main-result-un}.

	\subsection{Uniqueness of solution of the original system}
	In this subsection, we will give the proof for the uniqueness of
	inhomogeneous unsteady Prandtl equation \eqref{Prandtl}.
	Let $(u_1, v_1, \rho_1)$ and $(u_2, v_2, \rho_2)$ be two solutions of the original system \eqref{Prandtl} satisfying \eqref{estimate-w,rho} and \eqref{estimate-X} on the existence time $T_0$, constructed in the previous subsection, with the same initial data $(u_0, \rho_0)$.
	Let us denote $(\tilde{u}, \tilde{v}, \tilde{\rho}) \overset{def}{=} (u_1, v_1, \rho_1) - (u_2, v_2, \rho_2)$, $\tilde{w}\overset{def}{=} \p_y u_1- \p_y u_2=w_1-w_2$,
	$a_2 \overset{def}{=} \frac{\p_y w_2}{w_2}$, $b_2\overset{def}{=}\frac{\p_y \rho_2}{w_2}$. Then one may check that $\tilde{g} \overset{def}{=} \tilde{w}
	-a_2 \tilde{u} = w_2 \p_y \left(\frac{\tilde{u}}{w_2} \right)$, $\tilde{f} \overset{def}{=}  \tilde{\rho} - b_2 \tilde{u}$ and they satisfy
	\beq \label{e:tilde}
	\begin{aligned}
		(\p_t+u_1 \p_x + v_1 \p_y - \frac{\p_y^2}{\rho_1}) \tilde{g}
		= & -\frac{2}{\rho_1} \tilde{w} \p_y a_2
		+ \frac{\tilde{\rho}}{\rho_1 \rho_2} \p_y w_2 a_2
		+ (\frac{1}{\rho_1}-\frac{1}{\rho_2})\p_y^2 w_2 + \p_y w_1 \p_y(\frac{1}{\rho_1})- \p_y w_2 \p_y(\frac{1}{\rho_2})\\
		&- \tilde{u} \left(\tilde{u} \p_x a_2 + \tilde{v} \p_y a_2
		+ \f{2}{\rho_1} a_2 \p_y a_2
		-\frac{2 \p_y^2 w_2}{w_2} \cdot \frac{\p_y \rho_2}{\rho_2}
		-a_2 \frac{\p_y^2 \rho_2}{\rho_2^2}
		+ 2 a_2 \frac{(\p_y \rho_2)^2}{\rho_2^3} \right)\\
		&- \tilde{u} \left(
		a_2^2 \frac{\p_y \rho_2}{\rho_2^2}
		+\frac{\rho_2 -\rho_1}{\rho_1 \rho_2}(\frac{ \p_y^3 w_2}{w_2}-a_2 \frac{ \p_y^2 w_2}{w_2})
		\right),\\
		(\p_t+u_1 \p_x + v_1 \p_y) \tilde{f}
		= & -(\frac{\p_y \tilde{w}}{\rho_1} - \frac{\tilde{\rho}}{\rho_1 \rho_2}  \p_y w_2) \frac{\p_y \rho_2}{w_2}
		- \tilde{u} \left( \tilde{u} \p_x b_2 + \tilde{v} \p_y b_2
		-  \frac{\p_y \rho_2 }{w_2^2} \cdot \frac{\p_y^2 w_2}{\rho_2}
		\right).
	\end{aligned}
	\deq
	To derive the $L^2$ estimates on $\tilde{g}$ and $\tilde{f}$, let us first define the cutoff function
	$\chi_R (y) \overset{def}{=} \chi (\frac{y}{R} )$ for any $R \ge 1$, where $\chi \in C^\infty_c
	([0, +\infty))$ satisfies the following properties:
	\beqq
	0\leq \chi \leq 1, \quad \chi|_{[0,1]} \equiv 1, \quad \text{supp} \chi \subset [0,2], \quad -2 \leq \chi' \leq 0.
	\deqq
	Then $\chi_R$ has
	the following pointwise properties: as $R \to + \infty$,
	\beqq
	\chi_R \to \mathbf{1}_{\mathbb{R}^+}, \quad |\chi'_R| \le \frac{2}{R} \to 0^+ \quad \text{ and }\quad
	|\chi''_R| \le O(\frac{1}{R^2}) \to 0^+.
	\deqq
	For any $t \in (0,T_0]$, multiplying equation $\eqref{e:tilde}_1$ by $2 \chi_R
	\tilde{g}$, and then integrating over $[0,t] \times \mathbb{T}\times \mathbb{R}^+$, we obtain, via
	integration by parts,
	\beq \label{e:chi tilde g}
	\begin{aligned}
		&  \int \chi_R \tilde{g}^2 (t) \, dx dy - \int \chi_R \tilde{g}^2 |_{t=0}\,
		dx dy \notag\\
		= &  -2 \int^t_0 \int \frac{\chi_R}{\rho_1} |\p_y \tilde{g} |^2 dx dy d\tau - 2 \int^t_0 \int_{\mathbb{T}}
		\frac{\tilde{g} \p_y \tilde{g}}{\rho_1} |_{y=0} dx d\tau
		+ \int^t_0 \int_{\mathbb{T}} \frac{\p_y \rho_1}{\rho_1^2} \tilde{g}^2|_{y=0} dx d\tau
		- 4 \int^t_0 \int \chi_R \tilde{g}
		\frac{\p_y a_2 \tilde{w}}{\rho_1}dx dy d\tau \\
		& + 2 \int^t_0 \int \chi_R \tilde{g} \left( \frac{\tilde{\rho}}{\rho_1 \rho_2} \p_y w_2 a_2
		+ (\frac{1}{\rho_1}-\frac{1}{\rho_2})\p_y^2 w_2 + \p_y w_1 \p_y(\frac{1}{\rho_1})- \p_y w_2 \p_y(\frac{1}{\rho_2}) \right) dx dy d\tau\\
		&  -2 \int^t_0 \int \chi_R \tilde{g} \tilde{u}
		\Big( \tilde{u} \p_x a_2 + \tilde{v} \p_y a_2
		+ \f{2}{\rho_1} a_2 \p_y a_2
		-\frac{2 \p_y^2 w_2}{w_2} \cdot \frac{\p_y \rho_2}{\rho_2}
		-a_2 \frac{\p_y^2 \rho_2}{\rho_2^2}
		+ 2 a_2 \frac{(\p_y \rho_2)^2}{\rho_2^3}+
		a_2^2 \frac{\p_y \rho_2}{\rho_2^2} \\
		&+\frac{\rho_2 -\rho_1}{\rho_1 \rho_2}(\frac{ \p_y^3 w_2}{w_2}-a_2 \frac{ \p_y^2 w_2}{w_2})
		\Big) 	dx dy d\tau
		+\int^t_0 \int \chi'_R v_1 \tilde{g}^2 dx dy d\tau
		+\int^t_0 \int (\p_y(\frac{\chi'_R}{\rho_1})-\chi'_R\frac{ \p_y \rho_1 }{\rho_1^2}) \tilde{g}^2 dx dy d\tau\\
		&+\int^t_0 \int \chi_R \p_y\{\frac{\p_y \rho_1 }{\rho_1^2}\} \tilde{g}^2 dx dy d\tau  \\
		\overset{def}{=}& -2 \int^t_0 \int \frac{\chi_R}{\rho_1} |\p_y \tilde{g} |^2 dx dy d\tau + \sum_{i=1}^{8 }R_i.  \\
	\end{aligned}
	\deq
	Since  $\p_y \tilde{g}|_{y=0} =0 $, we can get that $R_1=0$ and we only need to deal with the boundary integral $R_2$.
	\beqq
	\begin{aligned}
		|R_2| &\le C \left( \int^t_0  \int^1_0 \int_{\mathbb{T}} \left|\frac{\p_y \rho_1}{\rho_1^2} \tilde{g}^2 \right| dx dy d\tau
		+\int^t_0  \int^1_0 \int_{\mathbb{T}} \left|\p_y\left(\frac{\p_y \rho_1}{\rho_1^2} \tilde{g}^2\right)\right| dx dy d\tau \right) \\
		& \le  \frac{1}{4} \int^t_0 \int \frac{\chi_R}{\rho_1} |\p_y\tilde{g}|^2 dx dy d\tau
		+ C_{\kappa_1} \left( \|\frac{\p_y \rho_1}{\rho_1^2}\|_{L^\infty} + \|\p_y(\frac{\p_y \rho_1}{\rho_1^2})\|_{L^\infty}
		+ \| \frac{\p_y \rho_1}{\rho_1^{\frac{3}{2}}} \|^2_{L^\infty} \right)
		\int^t_0 \int \tilde{g}^2 dx dy d\tau\\
		&\le \frac{1}{4} \int^t_0 \int \frac{\chi_R}{\rho_1} |\p_y\tilde{g}|^2 dx dy d\tau
		+ C_{ \kappa_1}(1+\e)\int^t_0 \|\tilde{g}\|_{L^2}^2 d\tau,\\
	\end{aligned}
	\deqq
	where we have used the simple trace estimate:
	\beqq
	\int_{\mathbb{T}} |f| \; dx \Big|_{y=0} \le C \left(\int^1_0 \int_{\mathbb{T}} |f| dx dy
	+ \int^1_0 \int_{\mathbb{T}} |\p_y f|  dx dy\right)	.	
	\deqq
	It's easy to check that
	\beqq
	\begin{aligned}
		&\; \|\xw
		\Big(\tilde{u} \p_x a_2 + \tilde{v} \p_y a_2
		+ 2 a_2 \p_y a_2
		-\frac{2 \p_y^2 w_2}{w_2} \cdot \frac{\p_y \rho_2}{\rho_2}
		-a_2 \frac{\p_y^2 \rho_2}{\rho_2^2}
		+ 2 a_2 \frac{(\p_y \rho_2)^2}{\rho_2^3}\\
		&\; + a_2^2 \frac{\p_y \rho_2}{\rho_2^2}
		+\frac{\rho_2 -\rho_1}{\rho_1 \rho_2}(\frac{ \p_y^3 w_2}{w_2}-a_2 \frac{ \p_y^2 w_2}{w_2})
		\Big) \|_{L^\infty}\\
		\le &\; \|\p_x a_2 \xw  \|_{L^\infty} \|\tilde{u}\|_{L^\infty}
		+\|\tilde{v}\xw^{-1}\|_{L^\infty} \|\p_y^2 a_2 \xw^2  \|_{L^\infty}
		+\|a_2\|_{L^\infty} \|\p_y a_2 \xw\|_{L^\infty} \\
		&\;+ C_{\bde,\kappa_1} \Big( \|\p_y^2 w_2 \xw^{\sg+2}  \|_{L^\infty} \|\p_y \rho_2 \xw  \|_{L^\infty}
		+\| a_2 \xw\|_{L^\infty}(\|\p_y^2 \rho_2 \|_{L^\infty}+ \|\p_y \rho_2 \|_{L^\infty}^2)\\
		&\;+\| a_2 \xw\|_{L^\infty}^2 \|\p_y \rho_2 \|_{L^\infty}
		+(\| \rho_1 \|_{L^\infty}+\| \rho_2 \|_{L^\infty})
		(\|\p_y^3 w_2 \xw^{\sg+1} \|_{L^\infty}+\|a_2 \xw \|_{L^\infty}\|\p_y^2 w_2 \|_{L^\infty}) \Big)\\
		\le &\; C_{\bde,\kappa_1}(1+\e^2).
	\end{aligned}
	\deqq
	Therefore, it's easy to check that
	\beqq
	|R_5| \le C_{\bde,\kappa_1}(1+\e^2) \int^t_0 \|\tilde{g}\|_{L^2}
	\ \left\| \tilde{u} \xw^{-1}\right\|_{L^2}
	d\tau,
	\deqq
	and
	\beqq
	\begin{aligned}
		|R_3|&\le  C_{\bde,\kappa_1} (1+\e) \int^t_0 \|\tilde{g}\|_{L^2}  \|\tilde{w}\|_{L^2} d\tau;\\
		|R_8| &\le  C_{\kappa_1} (1+\e) \int^t_0 \|\tilde{g}\|_{L^2}^2 d\tau.
	\end{aligned}
	\deqq
	Next, we deal with the term $R_4$ and decompose this term as follows:
	\beqq
	R_4 \overset{def}{=} R_{4,1}+R_{4,2},
	\deqq
	where
	\beqq
	\begin{aligned}
		R_{4,1}&\overset{def}{=} 2 \int^t_0 \int \chi_R \tilde{g} \left( \frac{\tilde{\rho}}{\rho_1 \rho_2} \p_y w_2 a_2 \right) dx dy d\tau; \\
		R_{4,2}&\overset{def}{=} 2 \int^t_0 \int \chi_R \tilde{g} \left(  \p_y w_1 \p_y(\frac{1}{\rho_1})- \p_y w_2 \p_y(\frac{1}{\rho_2}) + (\frac{1}{\rho_1}-\frac{1}{\rho_2})\p_y^2 w_2 \right) dx dy d\tau.
	\end{aligned}
	\deqq
	It's easy to check that
	\beqq
	|R_{4,1}| \le C_{\bde,\kappa_1} \int^t_0 \|\tilde{g}\|_{L^2}  \|\tilde{\rho}\|_{L^2} d\tau.
	\deqq
	Integrating by parts and the relation $\p_y \tilde{w}= \p_y \tilde{g} + \p_y (a_2 \tilde{u})=\p_y \tilde{g} + \p_y a_2 \tilde{u}+ a_2 \tilde{w}$, we can obtain that
	\beqq
	\begin{aligned}
		|R_{4,2}|
		= &\; 2\left| \int^t_0 \int \chi_R \tilde{g} \left(
		\p_y\tilde{w} \p_y(\frac{1}{\rho_1})+ \p_y w_2 \p_y(\frac{1}{\rho_1}-\frac{1}{\rho_2}  ) +(\frac{1}{\rho_1}-\frac{1}{\rho_2})\p_y^2 w_2\right) dx dy d\tau \right|\\
		=&\;2 \left|\int^t_0 \int \left(
		\chi_R \tilde{g} \p_y\tilde{w} \p_y(\frac{1}{\rho_1})
		- \p_y w_2 (\frac{1}{\rho_1}-\frac{1}{\rho_2}) \chi_R \p_y \tilde{g}
		- \p_y w_2 (\frac{1}{\rho_1}-\frac{1}{\rho_2}) \chi'_R  \tilde{g} \right)dx dy d\tau  \right|\\
		=&\;2\left| \int^t_0 \int \left(
		\chi_R \tilde{g}  \p_y(\frac{1}{\rho_1})
		(\p_y \tilde{g} + \p_y a_2 \tilde{u}+ a_2 \tilde{w})
		- \p_y w_2 \frac{\tilde{\rho}}{\rho_2\rho_1} \chi_R  \p_y \tilde{g}
		- \p_y w_2 \frac{\tilde{\rho}}{\rho_2\rho_1} \chi'_R  \tilde{g} \right) dx dy d\tau  \right|\\
		\le&\;C_{\kappa_1} \|\p_y \rho_1 \|_{L^\infty} \|\tilde{g} \|_{L^2}
		( \| \sqrt{\chi_R} \frac{\p_y \tilde{g}}{\sqrt{\rho_1}} \|_{L^2}
		+ \|\p_y a_2 \xw \|_{L^\infty} \| \tilde{u} \xw^{-1} \|_{L^2}
		+ \|a_2 \|_{L^\infty} \| \tilde{w} \|_{L^2})\\
		&\;+ C_{\kappa_1} \|\p_y w_2 \|_{L^\infty} \|\tilde{\rho} \|_{L^2} (\|\sqrt{\chi_R} \frac{\p_y \tilde{g}}{\sqrt{\rho_1}} \|_{L^2} +\|\tilde{g} \|_{L^2} )\\
		\le &\; \frac{1}{4} \int^t_0 \int \frac{\chi_R}{\rho_1} |\p_y\tilde{g}|^2 dx dy d\tau
		+ C_{ \bde, \kappa_1}(1+\e)\int^t_0 \|(\tilde{g}, \tilde{\rho},\tilde{w}, \tilde{u} \xw^{-1}) \|_{L^2}^2 d\tau.
	\end{aligned}	
	\deqq
	Combining the estimate of $R_{4,1}$ and $R_{4,2}$, it holds
	\beqq
	|R_4| \le  \frac{1}{4} \int^t_0 \int \frac{\chi_R}{\rho_1} |\p_y\tilde{g}|^2 dx dy d\tau
	+ C_{ \bde, \kappa_1}(1+\e)\int^t_0 \|(\tilde{g}, \tilde{\rho},\tilde{w}, \tilde{u} \xw^{-1})\|_{L^2}^2 d\tau.
	\deqq
	Then combining the estimates from  $R_{1}$ to $R_{8}$, one has
	\beq\label{all-tilde-g}
	\begin{aligned}
		& \int \chi_R \tilde{g}^2 (t) \, dx dy - \int \chi_R \tilde{g}^2 |_{t=0}\,
		dx dy \\
		\le &- \frac{5}{4} \int^t_0 \int \frac{\chi_R}{\rho_1} |\p_y \tilde{g} |^2 dx dy d\tau
		+ C_{ \bde, \kappa_1}(1+\e^2)\int^t_0 \|(\tilde{g}, \tilde{\rho},\tilde{w}, \tilde{u}\xw^{-1}) \|_{L^2}^2 d\tau
		+R_6 +R_7.
	\end{aligned}
	\deq

	Next, we deal with the estimate of $\tilde{f}$. Multiplying equation $\eqref{e:tilde}_2$ by $2 \chi_R
	\tilde{f}$, and then integrating over $[0,t] \times \mathbb{T}\times \mathbb{R}^+$, we obtain,
	\beqq
	\begin{aligned}
		&  \int \chi_R \tilde{f}^2 (t) \, dx dy - \int \chi_R \tilde{f}^2 |_{t=0}\,
		dx dy \notag\\
		= &  -2 \int^t_0 \int (\frac{\p_y  \tilde{w}}{\rho_1}-\frac{\tilde{\rho}}{\rho_1 \rho_2}\p_y w_2 )  \chi_R \tilde{f} \, dx dy d\tau
		- 2 \int^t_0 \int
		(\tilde{u} \p_x b_2 + \tilde{v} \p_y b_2 - \frac{ \p_y \rho_2}{\rho_2} \cdot \frac{\p_y^2 w_2}{w_2} )\xw  \cdot	\tilde{f}  \tilde{u} \xw^{-1} dx dy d\tau\\
		&+\int^t_0 \int \chi'_R v_1 \tilde{f}^2 \, dx dy d\tau \\
		\overset{def}{=} & R_9+R_{10}+R_{11}.\\
	\end{aligned}
	\deqq
	Similarly, we can get the following estimate
	\beq \label{all-tilde-f}
	\begin{aligned}
		&\;  \int \chi_R \tilde{f}^2 (t) \, dx dy- \int \chi_R \tilde{f}^2 |_{t=0}\,
		dx dy \\
		\le &\; \frac{1}{4} \int^t_0 \int \frac{\chi_R}{\rho_1} |\p_y \tilde{g} |^2 dx dy d\tau
		+ C_{ \bde, \kappa_1}(1+\e)\int^t_0 \|(\tilde{f}, \tilde{\rho},\tilde{w}, \tilde{u}\xw^{-1} )\|_{L^2}^2 d\tau
		+R_{11}.
	\end{aligned}
	\deq
	We emphasize that $\tilde{w}, \tilde{u}\xw^{-1}$ can be controlled by $\tilde{g}$, and $\tilde{\rho}$ can be controlled by $\tilde{g}$ and $\tilde{f}$, that is
	\beqq
	\|( \tilde{w},\tilde{u}\xw^{-1} )\|_{L^2} \le C_{\bde} \| \tilde{g} \|_{L^2},
	\deqq
	and
	\beqq
	\| \tilde{\rho}\|_{L^2} \le C_{\bde} \| (\tilde{g}, \tilde{f}) \|_{L^2}.
	\deqq
	Combining the estimates of \eqref{all-tilde-g} with \eqref{all-tilde-f},
	we can obtain that
	\beq \label{all-tilde-R}
	\begin{aligned}
		&  \int \chi_R (\tilde{g}^2+\tilde{f}^2) (t,x,y) \, dx dy - \int \chi_R (\tilde{g}^2 + \tilde{f}^2 )|_{t=0}\,
		dx dy +  \int^t_0 \int \frac{\chi_R}{\rho_1} |\p_y \tilde{g}|^2 dx dy d\tau \\
		\le &  C_{ \bde, \kappa_1}(1+\e^2)\int^t_0 \|(\tilde{f}, \tilde{g})\|_{L^2}^2 d\tau + R_6+R_7+R_{11}.
	\end{aligned}
	\deq
	Finally both integrands of $R_6$, $R_7$ can be controlled by a multiple of $\tilde{g}^2$ and $R_{11}$ can be controlled by a multiple of $\tilde{f}^2$, which belongs to
	$L^1 ([0,T]; \mathbb{T} \times \mathbb{R}^+)$, so applying Lebesgue's dominated convergence theorem, we
	have
	\begin{equation} \label{e:limit Ri}
		\lim_{R \to + \infty} R_i = 0,  \qquad \text{ for } i=6,7,11.
	\end{equation}
	Using monotone convergence theorem and \eqref{e:limit Ri}, we can pass to the limit $R \to + \infty$ in
	\eqref{all-tilde-R} to obtain:
	\beq \label{all-tilde}
	\begin{aligned}
		&   \|(\tilde{f}(t,x,y), \tilde{g}(t,x,y))\|_{L^2}^2
		+  \int^t_0   \|\frac{1}{\sqrt{\rho_1}} \p_y \tilde{g}\|_{L^2}^2 d\tau\\
		\le & \;\|(\tilde{f}(0,x,y), \tilde{g}(0,x,y))\|_{L^2}^2 + C_{ \bde, \kappa_1}(1+\e^2)\int^t_0 \|(\tilde{f}, \tilde{g})\|_{L^2}^2 d\tau.
	\end{aligned}
	\deq
	Applying Gronwall's lemma to \eqref{all-tilde}, we obtain
	\beqq
	\|(\tilde{f}(t,x,y), \tilde{g}(t,x,y))\|_{L^2}^2
	\le \|(\tilde{f}(0,x,y), \tilde{g}(0,x,y))\|_{L^2}^2 e^{C_{ \bde, \kappa_1}(1+\e^2)t},
	\deqq
	which implies $\tilde{f}(t)= \tilde{g}(t) \equiv 0$
	since $\tilde{f}(0)= \tilde{g}(0) =0$. Since $\tilde{g}=w_2 \p_2 \left( \frac{u_1-u_2}{w_2} \right)  \equiv 0$,
	we have
	\begin{equation}  \label{e:u1-u2}
		u_1 - u_2 = q w_2,
	\end{equation}
	for some function $q \overset{def}{=} q(t,x)$. Using the assumption $w_2 > 0$
	and Dirichlet boundary condition $(u_1, v_1)|_{y=0} \equiv 0$,
	we know via \eqref{e:u1-u2} that $q \equiv 0$, and hence, $u_1(t,x,y) \equiv u_2(t,x,y)$.
	Since $v_1, v_2$ can be uniquely determined by $u_1$ and $u_2$ respectively, we also have $v_1(t,x,y) \equiv
	v_2(t,x,y)$.
	From $\tilde{f}(t,x,y)\equiv 0$ and $\tilde{u}(t,x,y)\equiv 0$, we can obtain $\tilde{\rho}(t,x,y)\equiv 0$, that is  $\rho_1(t,x,y) \equiv \rho_2(t,x,y)$. Therefore we complete the uniqueness of the Prandtl system \eqref{Prandtl}.

	\section{A priori estimate for the steady Prandtl equations}\label{sec:priori-steady}
	In this section, we will establish a priori estimate for the inhomogeneous steady Prandtl equations \eqref{s-Prandtl}, which
	will play an important role in the local-in-$x$ existence result
	in Section \ref{sec:well-posedness-steady}.
    Under the assumption of \eqref{inital-pyu}, one can give
    the equivalent relation of $u$ near the boundary.
    Indeed, due to the conditions $u|_{y=0}=0, \p_y u(x,0) \ge \lambda>0,
		u>0$ and $\underset{y\rightarrow +\infty}{\lim}u=u_\infty$,
		then there exist two small positive constants $\delta_0$ and $\xi_0$, for any $0 < \sde \le \delta_0$, one can obtain that
		\beq\label{equ-u}
		u(x, y) \ge \frac{ \lambda}{2} y, \quad
		\forall (x, y)\in [0, L]\times [0, \sde],
		\deq
		and
		\beq\label{lower-u}
		u(x, y) \ge \xi_0, \quad \forall (x, y)
		\in [0, L]\times [\f{\sde}{2}, +\infty).
		\deq
		It should be pointed out that $\xi_0$ may depend on $\sde$.
	Now, we state a priori estimate for the system \eqref{s-Prandtl} as follows.
	\begin{prop}\label{prop: prior estimate}
		Assume $m \ge 3, \ssg \ge 2$ and
		 the initial data satisfies $\mathcal{X}(0)<+\infty$.
		For any positive constants $\lambda$, $\eta$ and $L$, suppose
		\beq \label{inital-pyu}
		\p_y u(x, 0) \ge \lambda  , \quad \forall x \in [0, L],
		\deq
		and
		\beq \label{inital-u,rho}
		u(x, y)>0, \quad \rho(x, y) \geq \eta,
		\quad \forall (x, y)\in [0, L] \times (0, +\infty),
		\deq
		hold.
		Then, the smooth solution $(\rho, u ,v)$ of system \eqref{s-Prandtl}
		will satisfy
		\beq\label{steady-priori}
		\f{d}{dx} \se+ \sy+C\left(|\p^{\al}\p_y (\f{v}{u})|^2  \p_y u\right)_{y=0}   \le C_{\lambda,\xi_0,\eta,\ssg} (1+ P(\se))+C_{\lambda,\xi_0,\eta,\ssg} \sde^{\f12} (1+ P(\se)) \sy,
		\deq
		where $P(\se)$ is the polynomial of
		$\se$  satisfying $P(0)=0$.
	\end{prop}
	
	\subsection{A priori estimate}
	In this subsection, we will establish some estimates for us to give the proof of Proposition \ref{prop: prior estimate}.
	First of all, we establish  estimate for the good unknown quantity $\p_y (\f{v}{u})$ and tangential velocity $u$, which will be used frequently.
	\begin{lemm}\label{Lemma:help}
		Under the condition of \eqref{inital-pyu} and \eqref{inital-u,rho}, for any smooth  solution $(\rho,u,v)$ of system
		\eqref{s-Prandtl}, it holds
		\beq \label{e:v/u}
		\begin{aligned}
			&\|\p^\alpha\p_y (\f{v}{u}) \xw^{\ssg}\|_{L_y^2}^2
			\le C_{\lambda, \xi_0, \eta}\se, & &|\al| \le m-1;\\
			&\|\p^\alpha\p_y (\f{v}{u}) \xw^{\ssg}\|_{L_y^2}^2
			\le  C_{\xi_0,\eta} \se+\sde C_{\lambda} \sy, & &|\al| = m ;\\
		\end{aligned}	
		\deq
		and
		\beq \label{e:u}
		\begin{aligned}	
			& \| u \|_{L_y^{\infty}}^2 \le C (1+\se);\\
			& \| \p_y^3 u \xw^{\ssg} \|_{L_y^{\infty}}^2 \le C_{\lambda, \xi_0, \eta}(1+\se^6);\\
			&\| \f{\p_x u}{u} \|_{L_y^{\infty}}^2  \le C_{\lambda,\xi_0,\eta}( 1+\se^7);\\
			&\|\p_y u\|_{H_{y,0}^{m}}^2 \le  C_{\lambda, \xi_0,\eta} (1+P(\se));\\	
			&\|\p_x u\|_{H_{y,0}^{m}}^2 \le   C_{\lambda, \xi_0,\eta} (1+P(\se))+ C_{\lambda, \xi_0,\eta} \sde (1+P(\se)) \sy;\\
			&\|\p_y^2 u\|_{H_{y,\ssg}^{m}}^2 \le C_{\lambda, \xi_0,\eta} (1+P(\se)).\\
		\end{aligned}	
		\deq
	\end{lemm}	
	
	\begin{proof}
		\textbf{Step 1: proof of estimate \eqref{e:v/u}.}
		Let us define $\chi(y) \in C^\infty_c	([0, +\infty))$ which satisfies the following properties:
		\beqq
		0\leq \chi \leq 1, \quad \chi|_{[0,\f{\sde}{2}]} \equiv 1, \quad \text{supp} \chi \subset [0,\sde], \quad -2 \leq \chi' \leq 0.
		\deqq
		For all $|\al| \leq m $, we decompose
		the term $\j |\p^\alpha\p_y (\f{v}{u})|^2 \xw^{2\ssg} dy$
		as follows
		\beqq
		\begin{aligned}
			\j |\p^\alpha\p_y (\f{v}{u})|^2 \xw^{2\ssg} dy
			=&\j |\p^\alpha\p_y (\f{v}{u})|^2 \xw^{2\ssg}(\chi+1-\chi)^2 dy\\
			\le &\j |\p^\alpha\p_y (\f{v}{u})|^2 \xw^{2\ssg}(1-\chi)^2 dy
			+ \j |\p^\alpha\p_y (\f{v}{u})|^2 \xw^{2\ssg}\chi^2 dy
			\\
			\overset{def}{=} & K_{1,1}+K_{1,2}.
		\end{aligned}
		\deqq
		Due to the lower bound of $u$ away from the boundary  \eqref{lower-u}, we have
		\beqq
		|K_{1,1}| \le  C_{\xi_0} \j  {u^2} |\p^\alpha\p_y (\f{v}{u})|^2 \xw^{2\ssg}dy.
		\deqq
		Integrating by part, it holds
		\beqq
		\begin{aligned}
			|K_{1,2}|&\le  \left| \j |\p^\alpha\p_y (\f{v}{u})|^2 \chi^2 dy \right|
			=  \left| \left(y |\p^\alpha\p_y (\f{v}{u} )|^2 \chi^2\right)^{+\infty}_0 -\j y \p_y \left\{|\p^\alpha\p_y (\f{v}{u} )|^2 \chi^2\right\}dy \right|\\
			&\le \left| \j 2 y \p^\alpha\p_y (\f{v}{u}) \p^\alpha\p_y^2 (\f{v}{u}) \chi^2
			+ 2 y |\p^\alpha\p_y (\f{v}{u} )|^2 \chi \chi'
			dy \right|\\
			&\le C \|\p^\alpha\p_y (\f{v}{u}) \chi\|_{L_y^2}
			( \| y \p^\alpha\p_y^2 (\f{v}{u}) \chi\|_{L_y^2}
			+ \| y \p^\alpha\p_y (\f{v}{u}) \chi'\|_{L_y^2}
			)\\
			&\le   \f12 \|\p^\alpha\p_y (\f{v}{u}) \chi\|_{L_y^2} ^2
			+ C \| y \p^\alpha\p_y^2 (\f{v}{u}) \chi\|_{L_y^2} ^2
			+ C \| y \p^\alpha\p_y (\f{v}{u}) \chi'\|_{L_y^2}^2.
		\end{aligned}
		\deqq
		Combining the estimate of $K_{1,1}$ and $K_{1,2}$,  we can get
		\beqq
		\begin{aligned}
			&\j |\p^\alpha\p_y (\f{v}{u})|^2 \xw^{2\ssg} dy\\
			\le &\;\f12 \j |\p^\alpha\p_y (\f{v}{u})|^2 \chi^2 dy
			+ C \j y^2 |\p^\alpha\p_y^2 (\f{v}{u})|^2 \chi^2 dy \\
			&\; + C \j y^2| \p^\alpha\p_y (\f{v}{u})|^2 \chi'^2 dy
			+C_{\xi_0} \j  {u^2} |\p^\alpha\p_y (\f{v}{u})|^2 \xw^{2\ssg} dy.
		\end{aligned}
		\deqq
		Using the lower bound of $u$  near the boundary \eqref{equ-u}, we have for $|\al| \le m-1$,
		\beqq
		\begin{aligned}
			\j |\p^\alpha\p_y (\f{v}{u})|^2 \xw^{2\ssg} dy
			\le   C_\lambda \j u^2 |\p^\alpha\p_y^2 (\f{v}{u})|^2 \chi^2 dy
			+ C_{\xi_0} \j u^2 |\p^\alpha\p_y (\f{v}{u})|^2 \xw^{2\ssg} dy
			\le C_{\lambda, \xi_0,\eta}\se,
		\end{aligned}
		\deqq
		and for $|\al| = m$,
		\beqq
		\begin{aligned}
			\j |\p^\alpha\p_y (\f{v}{u})|^2 \xw^{2\ssg} dy
			&\le
			\sde C_\lambda \j u |\p^\alpha\p_y^2 (\f{v}{u})|^2 \chi^2 dy
			+C_{\xi_0} \j u^2 |\p^\alpha\p_y (\f{v}{u})|^2 \xw^{2\ssg} dy\\
			&\le  \sde C_\lambda \sy+C_{\xi_0,\eta} \se,
		\end{aligned}
		\deqq
		which implies the estimate $\eqref{e:v/u}_1$ and $\eqref{e:v/u}_2$.
		
		
		\textbf{Step 2: proof of estimates $\eqref{e:u}_1$-$\eqref{e:u}_3$. } Due to the condition $\underset{y\rightarrow +\infty}{\lim}\bar{u}=0$,
		it holds
		\beqq
		\| u\|_{L_y^{\infty}}^2 \le  C(1+ \|  \bar{u} \|_{L_y^{\infty}}^2 )\le C (1+ \|  (\bar{u}, \p_y \bar{u}) \|_{L_y^2}^2) \le C (1+ \se),
		\deqq
		which implies the estimate $\eqref{e:u}_1$.
		 Using the equation $\eqref{s-Prandtl}_2$, we have
		\beqq
		\begin{aligned}
			\| \p_y^2 u \xw^2  \|_{L_y^2}^2 &= \| \rho u^2 \p_{y}(\f{v}{u})  \xw^2  \|_{L_y^2}^2
			\le C  (1+ \|\rho\|_{L_y^{\infty}}^2)\| u\|_{L_y^{\infty}}^2  \| \sqrt{\rho} u \p_{y}(\f{v}{u}) \xw^2 \|_{L_y^2}^2 \le C (1+\se^3)
			,
		\end{aligned}
		\deqq
	   which yields that
		\beqq
		\begin{aligned}
			\| \p_y^3 u \xw^{\ssg}  \|_{L_y^{\infty}}^2 =& \| \p_{y}\{ \rho u^2 \p_{y}(\f{v}{u}) \} \xw^{\ssg}  \|_{L_y^{\infty}}^2 \\
			\le & \| \p_{y}\rho \|_{L_y^{\infty}}^2 \| u \|_{L_y^{\infty}}^4
			\| \p_{y}(\f{v}{u}) \xw^{\ssg}   \|_{L_y^{\infty}}^2
			+ \| \rho \|_{L_y^{\infty}}^2 \| u \|_{L_y^{\infty}}^4
			\| \p_{y}^2(\f{v}{u}) \xw^{\ssg}   \|_{L_y^{\infty}}^2
			\\   	
			& + \| \rho \|_{L_y^{\infty}}^2 \| u \|_{L_y^{\infty}}^2
			\| \p_{y} u   \|_{L_y^{\infty}}^2
			\| \p_{y}(\f{v}{u}) \xw^{\ssg}   \|_{L_y^{\infty}}^2\\
			\le &  C (\| \p_{y}\rho \|_{H_y^1}^2
			\| u \|_{L_y^{\infty}}^4
			\| \p_{y}(\f{v}{u}) \xw^{\ssg}   \|_{H_y^1}^2
			+ \| \rho \|_{H_y^1}^2 \| u \|_{L_y^{\infty}}^4
			\| \p_{y}^2(\f{v}{u}) \xw^{\ssg}   \|_{H_y^1}^2
			\\   	
			& + \| \rho \|_{H_y^1}^2 \| u \|_{L_y^{\infty}}^2
			\| \p_{y}^2 u \xw   \|_{L_y^2}^2
			\| \p_{y}(\f{v}{u}) \xw^{\ssg}   \|_{H_y^1}^2 )	\\
			\le & C_{\lambda, \xi_0, \eta}(1+\se^6).
		\end{aligned}
		\deqq
		Therefore we establish the estimate $\eqref{e:u}_2$.
		
		Finally, we estimate the term $\f{\p_x u}{u}$.
		\beq\label{p_x u}
		\begin{aligned}
			\|\f{\p_x u}{u}\|_{L_y^{\infty}}^2&=\| \f{\p_x u}{u} (\chi+ 1-\chi)\|_{L_y^{\infty}}^2
			\le \| \f{\p_x u}{u} \chi\|_{L_y^{\infty}}^2 + \|\f{\p_x u}{u} (1-\chi)\|_{L_y^{\infty}}^2\\
			&\le C_{\lambda,\xi_0} (\| \f{\p_x u}{y}\|_{L_y^{\infty}}^2 + \|\p_x u\|_{L_y^{\infty}}^2).
		\end{aligned}
		\deq
		We first deal with $\|\p_x u\|_{L_y^{\infty}}$.
		\beq \label{p_x u1}
		\begin{aligned}
			\|\p_x u\|_{L_y^{\infty}}& =  \|\p_y v\|_{L_y^{\infty}}
			=\|\p_y (\f{v}{u} u)\|_{L_y^{\infty}}\\
			&\le \|u\|_{L_y^{\infty}} \|\p_y (\f{v}{u}) \|_{L_y^{\infty}} +
			\|\p_y u\|_{L_y^{\infty}} \|\f{v}{u} \|_{L_y^{\infty}}\\
			&\le  \| \p_y^2 u \xw^2  \|_{L^2_y} \| (\p_y (\f{v}{u}),\p_y^2 (\f{v}{u})) \xw  \|_{L^2_y}.\\
		\end{aligned}
		\deq
		Then we deal with $\| \f{\p_x u}{y}\|_{L_y^{\infty}}$.
		\beq \label{p_x u2}
		\begin{aligned}
			\| \f{\p_x u}{y}\|_{L_y^{\infty}} &=\| \f{1}{y} \int_0^y \p_{y'} \p_x u dy' \|_{L_y^{\infty}} \le \|\p_y^2 v \|_{L_y^{\infty}}\\
			&\le C (\|\p_y^2 (\f{v}{u}) \|_{L_y^{\infty}}  \|u \|_{L_y^{\infty}}  + \|\p_y (\f{v}{u}) \|_{L_y^{\infty}}  \|\p_y u \|_{L_y^{\infty}} +\|\f{v}{u} \|_{L_y^{\infty}}    \|\p_y^2 u \|_{L_y^{\infty}})\\
			&\le C (\|\p_y^2 (\f{v}{u}) \|_{H_y^1}  \|u \|_{L_y^{\infty}}
			+ \|\p_y (\f{v}{u}) \|_{H_y^1}  \| \p_y^2 u \xw  \|_{L_y^2} +\|\p_y(\f{v}{u}) \xw\|_{L_y^{2}}    \|\p_y^3 u \xw \|_{L_y^{2}}).  \\
		\end{aligned}
		\deq
		Combining the estimates \eqref{p_x u1} and \eqref{p_x u2} into \eqref{p_x u} and using the estimates $\eqref{e:v/u}_1$, $\eqref{e:u}_1$, $\eqref{e:u}_2$, we can obtain that
		\beqq
		\begin{aligned}
			\|\f{\p_x u}{u}\|_{L_y^{\infty}}^2
			&\le C_{\lambda,\xi_0,\eta}( 1+\se^7),
		\end{aligned}
		\deqq
		which implies the estimate $\eqref{e:u}_3$.
		
		\textbf{Step 3: proof of estimates $\eqref{e:u}_4$-$\eqref{e:u}_6$.}
		First we estimate the term $\|\p_y^{\al_2} \p_y u \|_{L_y^2}^2, |\al_2| \le m+1$ by induction.
		If $\al_2=0$, then we have
		$$
		\|\p_y u \|_{L_y^2}^2 \le \se.
		$$
		Assume for $0 \le k \le l \le m$, it holds
		\beq\label{assum-u}
		\|\p_y^{k} \p_y u \|_{L_y^2}^2 \le C_{\lambda, \xi_0, \eta} (1+P(\se)).
		\deq
		Then
		\beq \label{e:K1,3-4}
		\begin{aligned}
			\p_y^{l+1} \p_y u  &= \p_y^l \p_y^2 u =\p_y^l \{-\rho u^2 \p_y(\f{v}{u})\}\\
			&= -\rho u^2 \p_y^l \p_y (\f{v}{u}) -\sum_{0< j \le l} C_l^j \p_y^{j}(\rho u^2) \cdot  \p_y^{l-j} \p_y(\f{v}{u}) \\
			&\overset{def}{=} K_{1,3}+K_{1,4}.
		\end{aligned}
		\deq
		Using H\"{o}lder inequality and the estimate $\eqref{e:u}_1$, it is easy to check that
		\beq\label{e:K1,3}
		\|K_{1,3}\|_{L^2_y}^2 \le C (1+\|\br\|_{H_y^1}^2) \|u\|_{L_y^{\infty}}^2 \|\sqrt{\rho} u \p_y^l \p_y (\f{v}{u})\|_{L^2_y}^2 \le C(1+\se^3).
		\deq
		Using the estimate $\eqref{e:v/u}_1$, it holds
		\beqq
		\bal
		\|K_{1,4}\|_{L^2_y}^2
		&\;\le \|\sum_{0< j \le l} C_l^j \p_y^{j-1} \p_y(\rho u^2) \cdot \p_y^{l-j} \p_y (\f{v}{u}) \|_{L^2_y}^2
		\le C \|\p_y (\rho u^2)\|_{\bar{H}_{y,0}^{l-1}}^2 \|\p_y (\f{v}{u})\|_{\bar{H}_{y,0}^{l-1}}^2\\
		&\; \le C_{\lambda, \xi_0, \eta} \se \|\p_y (\rho u^2)\|_{\bar{H}_{y,0}^{l-1}}^2.
		\dal
		\deqq
		According to the assumption \eqref{assum-u}, we have
		\beqq
		\begin{aligned}
			&\|\bu\|_{\bar{H}_{y,0}^{l}}^2 \le C (\|\bu \|_{L_y^2}^2 +\| \p_y{\bu}\|_{\bar{H}_{y,0}^{l-1}}^2) \le C_{\lambda, \xi_0, \eta} (1+P(\se));\\	
			&\|\bu^2\|_{\bar{H}_{y,0}^{l}}^2 \le  C \|\bu\|_{\bar{H}_{y,0}^{l}}^4 \le C_{\lambda, \xi_0, \eta} (1+P(\se));\\
			&	\|\br \bu\|_{\bar{H}_{y,0}^{l}}^2 \le C \|\br\|_{\bar{H}_{y,0}^{l}}^2 \|\bu\|_{\bar{H}_{y,0}^{l}}^2 \le C_{\lambda, \xi_0, \eta} (1+P(\se));\\
			&	\|\br \bu^2\|_{\bar{H}_{y,0}^{l}}^2 \le  C \|\br\|_{\bar{H}_{y,0}^{l}}^2 \|\bu^2\|_{\bar{H}_{y,0}^{l}}^2 \le C_{\lambda, \xi_0, \eta} (1+P(\se)),\\
		\end{aligned}
		\deqq
		which yields directly
		\beq\label{e:K1,4}
		\begin{aligned}
			\|K_{1,4}\|_{L^2_y}^2 &\le  C_{\lambda, \xi_0,\eta} \se \|\p_y (\rho u^2)\|_{\bar{H}_{y,0}^{l-1}}^2\\
			&\le  C_{\lambda, \xi_0,\eta} \se \|\p_y\{(\br+\varrho_{\infty})(\bu+u_{\infty})^2\}\|_{\bar{H}_{y,0}^{l-1}}^2\\
			&\le  C_{\lambda, \xi_0,\eta} \se (\|\br \bu^2\|_{\bar{H}_{y,0}^{l}}^2 +
			\|\br \bu\|_{\bar{H}_{y,0}^{l}}^2+
			\|\br \|_{\bar{H}_{y,0}^{l}}^2+
			\|\bu\|_{\bar{H}_{y,0}^{l}}^2+\|\bu^2\|_{\bar{H}_{y,0}^{l}}^2
			)\\
			&\le
			C_{\lambda, \xi_0,\eta} (1+P(\se)).
		\end{aligned}
		\deq
		Combining the estimates \eqref{e:K1,3} and \eqref{e:K1,4} into \eqref{e:K1,3-4},
		we have
		\beqq
		\begin{aligned}
			\|\p_y^{l+1} \p_y u \|_{L^2_y}^2 \le C_{\lambda, \xi_0, \eta} (1+P(\se)),
		\end{aligned}
		\deqq
		which yields that
		\beq\label{e:all-py-u}
		\|\p_y^{\al_2} \p_y u \|_{L_y^2}^2 \le C_{\lambda, \xi_0,\eta} (1+P(\se)) ,\quad |\al_2| \le m+1.
		\deq
		
		Next, we estimate the term $\|\p^{\al} \p_y v \|_{L_y^2}^2, |\al| \le m$ by induction.
		
		When $|\al|=0$, using the relation $\p_y v =\p_y(\f{v}{u} u) =u \p_y(\f{v}{u}) + \f{v}{u} \p_y u $ and  the boundary condition $\f{v}{u} |_{y=0}=0$, we have
		\beqq
		\begin{aligned}
			\|\p_y v \|_{L_y^2}^2 &\le \|u \p_y(\f{v}{u}) \|_{L_y^2}^2 + \|\f{v}{u} \|_{L_y^{\infty}}^2 \| \p_y u \|_{L_y^2}^2\\
			&\le C (\|u \p_y(\f{v}{u}) \|_{L_y^2}^2 + \|\p_y (\f{v}{u}) \xw\|_{L_y^{2}}^2 \| \p_y \bu \|_{L_y^2}^2)\\
			& \le C_{\lambda, \xi_0,\eta} (1+\se^2).
		\end{aligned}
		\deqq
		Assume that for $k \le m-1$,
		\beqq
		\|\p^{\al} \p_y v \|_{L_y^2}^2 \le
		C_{\lambda, \xi_0,\eta} (1+P(\se)) ,\quad |\al| \le k,
		\deqq
		which yields that
		\beq \label{e:m-order-u}
		\| \bu  \|_{H_{y,0}^{m}}^2
		\le C (\| \bu \|_{L_y^2}^2 +
		\| \p_y^{\al_2} u \|_{L_y^2}^2 + \| \p_y v\|_{H_{y,0}^{m-1}}^2) \le
		C_{\lambda, \xi_0,\eta} (1+P(\se)), \quad 1 \le |\al_2| \le m,
		\deq
		and for $i=1,2$,
		\beq \label{e:high-order-u}
		\begin{aligned}
			&\; \| \p^{\al}\p_y^2 u \xw^{\ssg} \|_{L_y^2}^2 = \| \p^{\al}\{\rho u^2 \p_y (\f{v}{u})\}\xw^{\ssg} \|_{L_y^2}^2\\
			\le &\; \|\rho u^2 \p^{\al} \p_y (\f{v}{u})\xw^{\ssg} \|_{L_y^2}^2
			+ \| \sum_{0<\beta\le \alpha} C_{\al}^{\beta} \p^{\beta}(\rho u^2) \p^{\al-\beta}\p_y (\f{v}{u})\xw^{\ssg}\|_{L_y^2}^2\\
			\le &\; (1+\|\br\|_{H_y^1}^2) \|u\|_{L_y^{\infty}}^2 \|\sqrt{\rho} u \p^{\al} \p_y (\f{v}{u}) \xw^{\ssg} \|_{L^2_y}^2
			+ \| \sum_{0<\beta\le \alpha} C_{\al}^{\beta} \p^{\beta-e_i} \p^{e_i}(\rho u^2) \p^{\al-\beta}\p_y (\f{v}{u})) \xw^{\ssg} \|_{L_y^2}^2\\
			\le &\; 1+ \se^3 + C  (1+\|\br\|_{H_{y,0}^{m-1}}^4 + \|\bu\|_{H_{y,0}^{m-1}}^8) \|\p_y(\f{v}{u})\|_{H_{y,\ssg}^{m-2}}^2\\
			\le &\;  C_{\lambda, \xi_0,\eta} (1+P(\se)).
		\end{aligned}
		\deq	
		We first deal with the term $\|\p^{\al+e_2} \p_y v \|_{L_y^2}^2$.
		\beqq
		\begin{aligned}
			&\;\|\p^{\al+e_2} \p_y v \|_{L_y^2}^2 =  \|\p^{\al} \p_y^2 (u \f{v}{u} ) \|_{L_y^2}^2\\
			\le  &\; \|\p^{\al}  \{\f{v}{u} \p_y^2 u \} \|_{L_y^2}^2
			+ \|\p^{\al}  \{\p_y u \p_y(\f{v}{u})\} \|_{L_y^2}^2
			+\|\p^{\al}  \{u \p_y^2(\f{v}{u})\} \|_{L_y^2}^2\\
			\overset{def}{=} &\; K_{1,5}+K_{1,6}+K_{1,7}.
		\end{aligned}
		\deqq
		Using the estimate $\eqref{e:v/u}_1$ and \eqref{e:high-order-u}, we can obtain that
		\beqq
		\begin{aligned}
			K_{1,5} =
			&\;\| \sum_{0 \le \beta\le \alpha} C_{\al}^{\beta} \p^{\beta} ( \f{v}{u}) \p^{\al-\beta}\p_y^2 u   \|_{L_y^2}^2\\
			\le &\; C
			\sum_{0 \le \beta \le \alpha}
			\|\p^{\beta} (\f{v}{u}) \|_{L^{\infty}_y}^2
			\|\p^{\al-\beta} \p_y^2 u \|_{L_y^2}^2\\
			\le &\; C
			\sum_{0 \le \beta \le \alpha}
			\|\p^{\beta} \p_y (\f{v}{u}) \xw \|_{L^{2}_y}^2
			\|\p^{\al-\beta} \p_y^2 u \|_{L_y^2}^2\\
			\le &\;  C_{\lambda, \xi_0,\eta} (1+P(\se)).
		\end{aligned}
		\deqq
		Using the estimate $\eqref{e:v/u}_1$ and \eqref{e:m-order-u}, we can obtain that
		\beqq
		\begin{aligned}
			K_{1,6} =
			&\;\| \sum_{0 \le \beta\le \alpha} C_{\al}^{\beta} \p^{\beta} \p_y( \f{v}{u}) \p^{\al-\beta}\p_y u   \|_{L_y^2}^2\\
			\le &\;  C
			\sum_{0 \le \beta \le \alpha}
			\|\p_{y} (\f{v}{u}) \|_{H_{y,0}^{m-1}}^2
			\|\p_y u \|_{H_{y,0}^{m-1}}^2\\
			\le &\;  C_{\lambda, \xi_0,\eta} (1+P(\se)).
		\end{aligned}
		\deqq
		Using the Hardy inequality, Sobolev embedding, the estimate $\eqref{e:v/u}_1$ and \eqref{e:m-order-u}, we can obtain that for $i=1,2$,
		\beqq
		\begin{aligned}
			K_{1,7} \le
			& \;\|u  \p^{\al}\p_y^2 ( \f{v}{u})   \|_{L_y^2}^2 +
			\| \sum_{0 < \beta\le \alpha} C_{\al}^{\beta} \p^{\beta-e_i} \p^{e_i} u  \p^{\al-\beta}\p_y^2 ( \f{v}{u} )  \|_{L_y^2}^2\\
			\le &\; C
			(\|u  \p^{\al}\p_y^2 ( \f{v}{u})   \|_{L_y^2}^2
			+\|\p^{e_i} u \|_{H_{y,0}^{m-2}}^2  \|\p_{y}^2 (\f{v}{u}) \|_{H_{y,0}^{m-2}}^2)\\
			\le &\; C_{\lambda, \xi_0,\eta} (1+P(\se)).
		\end{aligned}
		\deqq
		Thus, the combination of estimates of terms from $K_{1,5}$ to $K_{1,7}$ yields directly
		\beq \label{e:(al+e2)-v}
		\begin{aligned}
			&\|\p^{\al+e_2} \p_y v \|_{L_y^2}^2  \le C_{\lambda, \xi_0,\eta} (1+P(\se)).
		\end{aligned}
		\deq
		Next, we deal with the term $\|\p^{\al+e_1} \p_y v \|_{L_y^2}^2 $. Applying the divergence-free condition as well as the estimates $\eqref{e:v/u}$ and \eqref{e:(al+e2)-v}, we have
		\beqq
		\begin{aligned}
			&\;\|\p^{\al+e_1} \p_y v \|_{L_y^2}^2 \\
			\le  &\; \|\p^{\al} \{u\cdot \p_{xy}(\f{v}{u})\} \|_{L_y^2}^2 + \|\p^{\al} \{\p_y v \cdot \p_{y}(\f{v}{u})\}\|_{L_y^2}^2\\
			&\;+\|\p^{\al} \{\p_y^2 v \cdot \f{v}{u}\} \|_{L_y^2}^2
			+\|\p^{\al} \{\p_y u \cdot \p_{x}(\f{v}{u})\} \|_{L_y^2}^2\\
			\le &\; \| u\p^{\al} \p_{xy}(\f{v}{u}) \|_{L_y^2}^2 + \| \p^{e_i} u\|_{H_{y,0}^{m-1}}^2  \|\p_{y} (\f{v}{u}) \|_{H_{y,0}^{m-1}}^2
			+ \| \p_{y} v \|_{H_{y,0}^{m-1}}^2  \|\p_{y} (\f{v}{u}) \|_{H_{y,0}^{m-1}}^2\\
			&\;+ \| \p_{y}^2 v \|_{H_{y,0}^{m-1}}^2  \|\f{v}{u} \|_{H_{y,0}^{m-1}}^2
			+ \| \p_{y} u \|_{H_{y,0}^{m-1}}^2 \|\p_x(\f{v}{u}) \|_{H_{y,0}^{m-1}}^2 \\
			\le &\; C_{\lambda, \xi_0,\eta} (1+P(\se))+ C_{\lambda, \xi_0,\eta} \sde (1+P(\se)) \sy.\\
		\end{aligned}
		\deqq
		Therefore we prove that
		\beq \label{e:(al+py)-v}
		\|\p^{\al} \p_y v \|_{L_y^2}^2 \le
		\left\{\begin{aligned}
			&C_{\lambda, \xi_0,\eta} (1+P(\se)) ,\quad |\al| \le m, \al_1 \le m-1; \\
			& C_{\lambda, \xi_0,\eta} (1+P(\se))+ C_{\lambda, \xi_0,\eta} \sde (1+P(\se)) \sy, \quad |\al| = m, \al_1 = m.
		\end{aligned}  \right.
		\deq
		With \eqref{e:all-py-u} and \eqref{e:(al+py)-v} at hand, we are ready to estimate $\eqref{e:u}_4$-$\eqref{e:u}_6$.
		
		If $\al_1=0$, applying the estimate \eqref{e:all-py-u}, we have
		\beqq
		\|\p^{\al} \p_y u\|_{L_y^2}^2= \|\p_y^{\al_2} \p_y u\|_{L_y^2}^2
		\le C_{\lambda, \xi_0,\eta} (1+P(\se)),
		\deqq
		and if $\al_1 \ge 1$ and $\al_2 \le m-1$, using the divergence-free condition and the estimate \eqref{e:(al+py)-v}, it holds
		\beqq
		\| \p^{\al} \p_y u \|_{L_y^2}^2
		=\| \p_{x}^{\al_1-1} \p_{y}^{\al_2+1} \p_y v \|_{L_y^2}^2 \le C_{\lambda, \xi_0,\eta} (1+P(\se)),
		\deqq
		which yields the estimate $\eqref{e:u}_4$.
		Similar to the above estimate,
		\beqq
		\| \p^{\al} \p_x u \|_{L_y^2}^2
		=\| \p^{\al} \p_y v \|_{L_y^2}^2 \le C_{\lambda, \xi_0,\eta} (1+P(\se))+ C_{\lambda, \xi_0,\eta} \sde (1+P(\se)) \sy,
		\deqq
		which yields the estimate $\eqref{e:u}_5$.
		Similar to the estimate \eqref{e:high-order-u}, for $|\al| \le m$,
		\beqq
		\begin{aligned}
			\| \p^{\al}\p_y^2 u \xw^{\ssg} \|_{L_y^2}^2
			\le &\;  (1+\|\br\|_{H_y^1}^2) \|u\|_{L_y^{\infty}}^2 \|\sqrt{\rho} u \p^{\al} \p_y (\f{v}{u}) \xw^{\ssg}\|_{L^2_y}^2+
			C(1+\|\br\|_{H_{y,0}^{m}}^4+ \|\bu\|_{H_{y,0}^{m}}^8) \|\p_y(\f{v}{u})\|_{H_{y,\ssg}^{m-1}}^2\\
			\le &\; C_{\lambda, \xi_0,\eta} (1+P(\se)) ,
		\end{aligned}
		\deqq	
		which yields the estimate $\eqref{e:u}_6$.
		Therefore, we complete the proof of this lemma.
	\end{proof}
	
	Now, we establish the estimate for the derivatives of density.
	\begin{lemm}
		Under the condition of \eqref{inital-pyu} and \eqref{inital-u,rho}, for any smooth  solution $(\rho,u,v)$ of system
		\eqref{s-Prandtl}, it holds
		\beqq
		\begin{aligned}
			&\frac{d}{dx}\sum_{|\alpha|\le m}
			\|\p^\alpha \br \|_{L_y^2}^2
			\le C_{\lambda,\xi_0,\eta}(1 + \se^2 )+ \sde  \sy.
		\end{aligned}
		\deqq
	\end{lemm}
	\begin{proof}
		For any $|\alpha|\le m$,
		applying the $\p^\al$ differential operator to the
		equation $\eqref{s-Prandtl}_1$, multiplying by $\p^{\al} \br$ and integrating over $\mathbb{R}^+$, we have
		\beqq
		\f{d}{dt} \f12 \j |\p^{\al} \br|^2  dy+ \underbrace{\j \p^{\al} (\f{v}{u} \p_y \br) \cdot \p^{\al} \br dy}_{K_2}=0.
		\deqq
		Denote $K_2$ as follows.
		\beqq
		K_2=\j \f{v}{u} \p^{\al} \p_y \br \cdot \p^{\al} \br dy
		+
		\sum_{0<\beta \le \alpha} \j \p^{\beta}(\f{v}{u}) \p^{\al-\beta} \p_y \br \cdot \p^{\al} \br dy\overset{def}{=}K_{2,1}+K_{2,2}.
		\deqq
		Integrating by part and using  the estimate $\eqref{e:v/u}_1$, we have
		\beqq
		\begin{aligned}
			|K_{2,1}|&=|\f12 \j \f{v}{u} \p_y |\p^{\al} \br|^2 dy| \\
			& =\left|\f12 \left(\f{v}{u}  |\p^{\al} \br|^2 \right)^{+\infty}_0 -\f12\j \p_y(\f{v}{u})  |\p^{\al} \br|^2  dy\right|\\
			&\le C  \|\p_y(\f{v}{u})\|_{L_y^{\infty}} \|\p^{\al} \br \|_{L^2_y}^2 \\
			&\le  C_{\lambda,\xi_0,\eta} (1+\se^2).
		\end{aligned}	
		\deqq
		Using the estimate $\eqref{e:v/u}_2$ and the H\"{o}lder inequality, the term $K_{2,2} $ can be estimated as follows
		\beqq
		|K_{2,2}| \le C \|\p^{\beta}(\f{v}{u}) \|_{L^{\infty}_y}  \|\p^{\al-\beta} \p_y \br \|_{L^2_y} \|\p^{\al} \br \|_{L^2_y}
		\le  (C_{\xi_0,\eta} \se^{\f12}+\sde^{\f12} C_{\lambda} \sy^{\f12}) \se.
		\deqq
		Combining the estimates of $K_{2,1}$ and $K_{2,2}$, we complete the proof of this lemma.
	\end{proof}
	\begin{lemm}
		Under the condition of \eqref{inital-pyu} and \eqref{inital-u,rho}, for any smooth  solution $(\rho,u,v)$ of system
		\eqref{s-Prandtl}, it holds
		\beqq
		\begin{aligned}
			&\; \frac{d}{dx}\sum_{|\alpha|\le m}
			\|\sqrt{\rho} u \p^{\al}\p_y (\f{v}{u}) \xw^{\ssg}\|_{L_y^2}^2 +
			\|\sqrt{u}\p_y^2 \p^\alpha(\f{v}{u}) \xw^{\ssg}\|_{L_y^2}^2
			+C\left(|\p^{\al}\p_y (\f{v}{u})|^2  \p_y u\right)_{y=0}\\
			\le &\;  \mu \sy + C_{\mu,\lambda,\xi_0,\eta,\ssg} (1+ P(\se))+C_{\lambda,\xi_0,\eta,\ssg} \sde^{\f12} (1+ P(\se)) \sy,
		\end{aligned}
		\deqq
		where $\mu$ is the small constant that will be chosen later.
	\end{lemm}
	\begin{proof}
		Applying the $\p_x$ differential operator to the
		equation $\eqref{s-Prandtl}_2$, using the divergence-free condition, we have
		\beq \label{e: pyy-u}
		\p_x \{\rho u^2  \p_y (\f{v}{u})\} -\p_y^3 v=0.
		\deq
		For any $|\al| \le m$, applying the $\p^{\al}$ differential operator to the
		equation \eqref{e: pyy-u} by $\p^{\al} \p_y (\f{v}{u}) \xw^{2 \ssg}$ and integrating over $\mathbb{R}^+$, we have 	
		\beqq
		\begin{aligned}
			&\underbrace{\j \p_x \p^{\al}\{\rho u^2  \p_y (\f{v}{u})\} \cdot \p^{\al} \p_y (\f{v}{u}) \xw^{2 \ssg } dy}_{K_{3,1}}
			-\underbrace{\j \p_y^3 \p^{\al} v \cdot \p^{\al} \p_y (\f{v}{u}) \xw^{2 \ssg } dy}_{K_{3,2}}=0.
		\end{aligned}
		\deqq
		\textbf{Deal with the term $K_{3,1}$}.
		\beqq
		\begin{aligned}
			K_{3,1}
			=& \j \p_x (\rho  u^2) |\p^{\al} \p_y (\f{v}{u})|^2 \xw^{2 \ssg}  dy
			+ \f12 \j \rho  u^2  \p_x\{|\p^{\al} \p_y (\f{v}{u})|^2\} \xw^{2 \ssg}  dy\\
			&+\j \sum_{0 < \beta\le \alpha} C_{\al}^{\beta} \p_x \left\{\p^{\beta}( \rho  u^2)  \p^{\al-\beta}\p_y (\f{v}{u})
			\right\}\cdot \p^{\al} \p_y (\f{v}{u}) \xw^{2 \ssg }  dy\\
			=& \f12 \f{d}{dx} \j \rho  u^2  |\p^{\al} \p_y (\f{v}{u})|^2 \xw^{2 \ssg }  dy +
			\f12\j \p_x (\rho  u^2) |\p^{\al} \p_y (\f{v}{u})|^2 \xw^{2 \ssg }  dy\\
			&+\j \sum_{0 < \beta\le \alpha} C_{\al}^{\beta} \p_x \left\{\p^{\beta}( \rho  u^2)  \p^{\al-\beta}\p_y (\f{v}{u})
			\right\}\cdot \p^{\al} \p_y (\f{v}{u}) \xw^{2 \ssg }  dy\\
			\overset{def}{=}& \f12 \f{d}{dx} \j \rho  u^2  |\p^{\al} \p_y (\f{v}{u})|^2 \xw^{2 \ssg}  dy +K_{3,1,1}+K_{3,1,2}.\\
		\end{aligned}
		\deqq
		Using  the estimate $\eqref{e:u}_3$, we can obtain that
		\beq \label{e:K311}
		\begin{aligned}
			|K_{3,1,1}| &\le \left|\f12\j (\p_x \rho  u^2 + 2\rho u \p_x u ) |\p^{\al} \p_y (\f{v}{u})|^2 \xw^{2 \ssg}  dy \right|\\
			&\le C \|\p_x \rho\|_{H_{y}^1} \| u \p^{\al}\p_y (\f{v}{u}) \xw^{\ssg}\|_{L_y^2}^2 + C \|\f{\p_x u}{u}\|_{L^{\infty}_y} \|\sqrt{\rho} u \p^{\al}\p_y (\f{v}{u}) \xw^{\ssg}\|_{L_y^2}^2\\
			&\le C_{\lambda,\xi_0,\eta} (1+\se^5).
		\end{aligned}
		\deq
		It is easy to check that
		\beqq
		\begin{aligned}
			K_{3,1,2}
			= &\sum_{1 \le \beta\le \alpha} C_{\al}^{\beta} \j  \p^{\beta}( \rho  u^2)  \p^{\al-\beta}\p_{xy} (\f{v}{u})
			\cdot \p^{\al} \p_y (\f{v}{u}) \xw^{2 \ssg }  dy\\
			&+ \sum_{1 \le \beta\le \alpha-1} C_{\al}^{\beta} \j  \p^{\beta} \p_x ( \rho  u^2)  \p^{\al-\beta}\p_{y} (\f{v}{u})
			\cdot \p^{\al} \p_y (\f{v}{u}) \xw^{2 \ssg}  dy\\
			&+2  \sum_{0 \le \beta\le \alpha} C_{\al}^{\beta}  \j  \p^{\beta}   \rho \p^{\al-\beta}( u \p_x u)
			\p_{y} (\f{v}{u})
			\cdot \p^{\al} \p_y (\f{v}{u}) \xw^{2 \ssg }  dy\\
			&+\j \left(\sum_{0 \le \beta\le \alpha-1} C_{\al}^{\beta}   \p^{\beta} \p_x  \rho \p^{\al-\beta}(u^2) + u^2   \p^{\al} \p_x  \rho \right)
			\p_{y} (\f{v}{u})
			\cdot \p^{\al} \p_y (\f{v}{u}) \xw^{2 \ssg }  dy\\
			\overset{def}{=} &\sum_{1\leq i \leq 5} K_{3,1,2,i}.
		\end{aligned}
		\deqq
		Using the density equation $\eqref{s-Prandtl}_1$ and integrating by part, the last term can be estimated as follows
		\beqq
		\begin{aligned}
			K_{3,1,2,5} = & - \j u^2  \left(\f{v}{u} \p^{\al} \p_y \rho +\sum_{1 \le \beta \le \alpha} C_{\al}^{\beta} \p^{\beta}(\f{v}{u} )   \p^{\al-\beta} \p_y \rho\right)
			\p_{y} (\f{v}{u})
			\cdot \p^{\al} \p_y (\f{v}{u}) \xw^{2 \ssg}  dy\\
			= &  \j   \p^{\al} \rho \cdot \p_y\left\{ u^2 \f{v}{u}
			\p_{y} (\f{v}{u})
			\cdot \p^{\al} \p_y (\f{v}{u}) \xw^{2 \ssg } \right\} dy + \left( \p^{\al} \rho \cdot u^2 \f{v}{u}
			\p_{y} (\f{v}{u})
			\cdot \p^{\al} \p_y (\f{v}{u}) \xw^{2 \ssg } \right)^{\infty}_0\\
			&- \sum_{1 \le \beta \le \alpha} C_{\al}^{\beta} \j u^2  \p^{\beta}(\f{v}{u} )   \p^{\al-\beta} \p_y \rho
			\p_{y} (\f{v}{u})
			\cdot \p^{\al} \p_y (\f{v}{u}) \xw^{2 \ssg }  dy\\
			= &  \j  \p^{\al} \rho \cdot \p_y\left\{ u^2 \f{v}{u}
			\p_{y} (\f{v}{u})
			\cdot \p^{\al} \p_y (\f{v}{u}) \xw^{2 \ssg } \right\} dy \\
			&- \sum_{1 \le \beta \le \alpha} C_{\al}^{\beta} \j u^2  \p^{\beta}(\f{v}{u} )   \p^{\al-\beta} \p_y \rho
			\p_{y} (\f{v}{u})
			\cdot \p^{\al} \p_y (\f{v}{u}) \xw^{2 \ssg }  dy.\\
		\end{aligned}
		\deqq
		Using the Hardy inequality, Sobolev embedding and the estimates in Lemma \ref{Lemma:help}, for $i=1,2$,
		\beq\label{e:K31251}
		\begin{aligned}
			&\;  \left|\sum_{1 \le \beta \le \alpha} C_{\al}^{\beta}\j u^2   \p^{\beta}(\f{v}{u} )   \p^{\al-\beta} \p_y \rho
			\p_{y} (\f{v}{u})
			\cdot \p^{\al} \p_y (\f{v}{u}) \xw^{2 \ssg}  dy\right|\\
			\le &\; C  \| u\|_{L^{\infty}_y} \|\p_y (\f{v}{u}) \xw^{\ssg}\|_{L^{\infty}_y} \| \p^{e_i} (\f{v}{u})\|_{H_{y,0}^{m-1}} \| \p_{y} \br\|_{H_{y,0}^{m-1}}  \| u \p^{\al}\p_y (\f{v}{u}) \xw^{\ssg}\|_{L_y^2}\\
			\le &\; C \| u\|_{L^{\infty}_y} \| \p_y (\f{v}{u}) \xw^{\ssg}\|_{H_{y}^1} \| \p_{y} (\f{v}{u}) \xw \|_{H_{y,0}^{m}} \| \br\|_{H_{y,0}^{m}}  \| u \p^{\al}\p_y (\f{v}{u}) \xw^{\ssg}\|_{L_y^2}\\
			\le &\;\sde \sy
			+ C_{\lambda, \xi_0,\eta} (1+\se^4).
		\end{aligned}
		\deq
		Similarly, we can obtain
		\beqq
		\begin{aligned}
			&\; \left|\j   \p^{\al} \rho \cdot \p_y\left\{ u^2 \f{v}{u}
			\p_{y} (\f{v}{u})
			\cdot \p^{\al} \p_y (\f{v}{u}) \xw^{2 \ssg } \right\} dy\right| 	\\
			\le &\;  \| \p^{\al} \br\|_{L_{y}^2} \| u \p^{\al}\p_y (\f{v}{u}) \xw^{\ssg}\|_{L_y^2}
			\Big(\| \p_y (\f{v}{u}) \xw^{\ssg}\|_{L^{\infty}_y}^2 \|  u\|_{L^{\infty}_y} + \| \p_y (\f{v}{u}) \xw^{\ssg}\|_{L^{\infty}_y}
			\|\f{v}{u}\|_{L^{\infty}_y} \| \p_y u\|_{L^{\infty}_y}
			\\
			&\; + \|\f{v}{u}\|_{L^{\infty}_y} \| u\|_{L^{\infty}_y}
			(\| \p_y^2 (\f{v}{u}) \xw^{\ssg}\|_{L^{\infty}_y}  +\| \p_y (\f{v}{u}) \xw^{\ssg-1}\|_{L^{\infty}_y})
			\Big) \\
			&\; +\| \p^{\al} \br\|_{L_{y}^2} \| \sqrt{u} \p^{\al}\p_y^2 (\f{v}{u}) \xw^{\ssg}\|_{L_y^2}
			\| \p_y (\f{v}{u}) \xw^{\ssg}\|_{L^{\infty}_y} \|\f{v}{u}\|_{L^{\infty}_y} \|  u\|_{L^{\infty}_y}^{\f32}\\
			\le &\;  \mu \sy+ C_{\mu, \lambda, \xi_0,\eta}(1+ \se^4),
		\end{aligned}
		\deqq
		which, together with \eqref{e:K31251}, yields
		\beqq
		|K_{3,1,2,5}|\le \sde \sy + \mu \sy +  C_{\mu, \lambda, \xi_0,\eta} (1+ \se^4).
		\deqq
		Deal with the term $K_{3,1,2,1}$. For $i=1,2$,
		\beqq
		\begin{aligned}
			|K_{3,1,2,1}| &= |\sum_{1 \le \beta\le \alpha} C_{\al}^{\beta} \j  \p^{\beta-e_i} \p^{e_i}\left\{ (\br+\varrho_{\infty} ) (\bu+u_{\infty})^2 \right\}  \p^{\al-\beta}\p_{xy} (\f{v}{u})
			\cdot \p^{\al} \p_y (\f{v}{u}) \xw^{2 \ssg }  dy|\\
			&\le  \|\p^{e_i} \left( (\br+\varrho_{\infty} ) (\bu+u_{\infty})^2 \right)  \|_{H_{y,0}^{m-1}}
			\|\p_{xy} (\f{v}{u}) \|_{H_{y,\ssg}^{m-1}}
			\|  \p^{\al}\p_y (\f{v}{u}) \xw^{\ssg}\|_{L_y^2} \\
			&\le  C (1+ \| \br \|_{H_{y,0}^{m}}^2 +  \| \bu \|_{H_{y,0}^{m}}^4)
			\|\p_{y} (\f{v}{u}) \|_{H_{y,\ssg}^{m}}
			\|  \p^{\al}\p_y (\f{v}{u}) \xw^{\ssg}\|_{L_y^2} \\
			&\le C_{\lambda, \xi_0,\eta} (1+P(\se))+ C_{\lambda, \xi_0,\eta} \sde (1+P(\se)) \sy.
		\end{aligned}
		\deqq
		Similarly, we arrive at
		\beqq
		\begin{aligned}
			|K_{3,1,2,2}| &\le
			\|\p_{x} \{(\br+\varrho_\infty)(\bu+u_\infty)^2\}\|_{H_{y,0}^{m-1}}
			\|\p^{e_i}\p_y (\f{v}{u}) \|_{H_{y,\ssg}^{m-1}}
			\|\p^{\al}\p_y (\f{v}{u}) \xw^{\ssg}\|_{L_y^2}\\
			&	\le   (1+ \| \br \|_{H_{y,0}^{m}}^2 +  \| \bu \|_{H_{y,0}^{m}}^4)
			\|\p_{y} (\f{v}{u}) \|_{H_{y,\ssg}^{m}}^2\\
			&	\le C_{\lambda, \xi_0,\eta} (1+P(\se))+ C_{\lambda, \xi_0,\eta} \sde (1+P(\se)) \sy;\\
			|K_{3,1,2,3}| &\le   2  \| \p^{\beta} (\br+\varrho_\infty) \p^{\al-\beta} \{(\bu+u_\infty) \p_x \bu\} \|_{L_y^2}
			\|\p_{y} (\f{v}{u}) \xw^{\ssg}\|_{L_{y}^{\infty}} \|\p^{\al}\p_y (\f{v}{u}) \xw^{\ssg}\|_{L_y^2}
			\\
			&\le C (1+ \| \br \|_{H_{y,0}^{m}})(1+\| (\bu,\p_x \bu) \|_{H_{y,0}^{m}}^2 )	\|\p_{y} (\f{v}{u}) \xw^{\ssg}\|_{H_y^1} \|\p^{\al}\p_y (\f{v}{u}) \xw^{\ssg}\|_{L_y^2}\\
			&\le C_{\lambda, \xi_0,\eta} (1+P(\se))+ C_{\lambda, \xi_0,\eta} \sde (1+P(\se)) \sy;\\    	
		\end{aligned}
		\deqq
		and
		\beqq
		\begin{aligned}
			|K_{3,1,2,4}| &\le  \| \p_x \br \|_{H_{y,0}^{m-1}}
			\| \p^{e_i} \{( \bu+u_{\infty})^2\} \|_{H_{y,0}^{m-1}}
			\|\p_{y} (\f{v}{u}) \xw^{\ssg}\|_{H_y^1} \|\p^{\al}\p_y (\f{v}{u}) \xw^{\ssg}\|_{L_y^2}\\\
			&\le  \|  \br \|_{H_{y,0}^{m}}
			( 1+ \| \bu \|_{H_{y,0}^{m}}^2) \|\p_{y} (\f{v}{u}) \xw^{\ssg}\|_{H_y^1} \|\p^{\al}\p_y (\f{v}{u}) \xw^{\ssg}\|_{L_y^2}\\
			&\le C_{\lambda, \xi_0,\eta}(1+P(\se)) + \sde \sy.
		\end{aligned}
		\deqq
		Combining the estimates from terms $K_{3,1,2,i}(i=1,\cdots,5)$, we have
		\beqq
		|K_{3,1,2}|\le \mu \sy + C_{\mu,\lambda, \xi_0,\eta} (1+P(\se))+ C_{\lambda, \xi_0,\eta} \sde (1+P(\se)) \sy,\\
		\deqq
		which, together with \eqref{e:K311}, yields that
		\beq\label{e:K31}
		K_{3,1}\ge \f12 \f{d}{dx} \j \rho  u^2  |\p^{\al} \p_y (\f{v}{u})|^2 \xw^{2 \ssg }  dy -\mu \sy - C_{\mu,\lambda, \xi_0,\eta} (1+P(\se))- C_{\lambda, \xi_0,\eta} \sde (1+P(\se)) \sy.\\
		\deq
		
		\textbf{Deal with the term $K_{3,2}$}.
		\beqq
		\begin{aligned}
			-K_{3,2} = &- \j \p_y^3 \p^{\al} (u \cdot \f{v}{u}) \cdot \p^{\al} \p_y (\f{v}{u}) \xw^{2 \ssg } dy \\
			=&	-\j \p_y^3 \left\{ u \p^{\al} (\f{v}{u}) + \sum_{0 \le \beta\le \alpha-1} C_{\al}^{\beta}\p^{\beta}(\f{v}{u}) \p^{\al-\beta} u \right\} \cdot \p^{\al} \p_y (\f{v}{u}) \xw^{2 \ssg } dy \\
			=&	-\j \p_y \left\{u \p_y^2 \p^{\al} (\f{v}{u}) + 2 \p_y u \p_y \p^{\al} (\f{v}{u}) + \p_y^2 u  \p^{\al} (\f{v}{u})
			\right\}
			\cdot \p^{\al} \p_y (\f{v}{u}) \xw^{2 \ssg } dy\\
			&-\j \p_y^3 \left\{ \sum_{0 \le \beta\le \alpha-1} C_{\al}^{\beta}\p^{\beta}(\f{v}{u}) \p^{\al-\beta} u \right\} \cdot \p^{\al} \p_y (\f{v}{u}) \xw^{2 \ssg } dy\\
			\overset{def}{=} &\sum_{1\leq i \leq 4} K_{3,2,i}.
		\end{aligned}
		\deqq
		Integrating by part, we have
		\beqq
		\begin{aligned}
			K_{3,2,1}= &\j u \p_y^2 \p^{\al} (\f{v}{u})\cdot\p_y  \{\p^{\al} \p_y (\f{v}{u}) \xw^{2 \ssg}\} dy - \left(u \p_y^2 \p^{\al}(\f{v}{u}) \cdot\p^{\al} \p_y (\f{v}{u}) \xw^{2 \ssg } \right)_{0}^{\infty}\\
			= &\j u |\p_y^2 \p^{\al} (\f{v}{u})|^2 \xw^{2 \ssg } dy
			+ 2 \ssg  \j u \p_y^2 \p^{\al} (\f{v}{u})\cdot \p^{\al} \p_y (\f{v}{u}) \xw^{2 \ssg -1} dy.
		\end{aligned}
		\deqq
		The last term can be estimated as follows:
		\beqq
		\begin{aligned}
			&\; 2 \ssg  |\j u \p_y^2 \p^{\al} (\f{v}{u})\cdot \p^{\al} \p_y (\f{v}{u}) \xw^{2 \ssg -1} dy |\\
			\le &\; C_{\ssg} (1+\|u\|_{L^{\infty}_y})
			\|\sqrt{u} \p^{\al}\p_y^2 (\f{v}{u}) \xw^{\ssg}\|_{L_y^2}
			\|\p^{\al}\p_y (\f{v}{u}) \xw^{\ssg}\|_{L_y^2}\\
			\le &\;  \mu \sy + C_{\mu,\xi_0,\eta,\ssg} (1+ \se^2)+C_{\gamma,\ssg} \sde^{\f12}(1+\se^{\f12} ) \sy,
		\end{aligned}
		\deqq
		which yields directly that
		\beqq
		\begin{aligned}
			K_{3,2,1} \ge &\j u |\p_y^2 \p^{\al} (\f{v}{u})|^2 \xw^{2 \ssg } dy
			-\mu \sy - C_{\mu,\xi_0,\eta,\ssg} (1+ \se^2)-C_{\gamma,\ssg} \sde^{\f12}(1+ \se^{\f12}) \sy.
		\end{aligned}
		\deqq
		Integrating by part, we can obtain that
		\beqq
		\begin{aligned}
			K_{3,2,2}+K_{3,2,3}=&-\j \left( 3 \p_y^2 u \p_y \p^{\al} (\f{v}{u}) + \p_y^3 u  \p^{\al} (\f{v}{u})
			\right)
			\cdot \p^{\al} \p_y (\f{v}{u}) \xw^{2 \ssg }
			dy\\
			&-\j 2 \p_y u \p_y^2 \p^{\al} (\f{v}{u})
			\cdot \p^{\al} \p_y (\f{v}{u}) \xw^{2 \ssg }
			dy\\
			=&-\j \left( 3 \p_y^2 u \p_y \p^{\al} (\f{v}{u}) + \p_y^3 u  \p^{\al} (\f{v}{u})
			\right)
			\cdot \p^{\al} \p_y (\f{v}{u}) \xw^{2 \ssg }
			dy\\
			&+\j |\p_y \p^{\al} (\f{v}{u})|^2
			\cdot  \p_y (\p_y u \xw^{2 \ssg +2 \al_2})dy
			-\left( |\p_y \p^{\al} (\f{v}{u})|^2 \p_y u \xw^{2 \ssg}
			\right)_{0}^{\infty}\\
			= &-\j \left( 2 \p_y^2 u \p_y \p^{\al} (\f{v}{u}) + \p_y^3 u  \p^{\al} (\f{v}{u})
			\right)
			\cdot \p^{\al} \p_y (\f{v}{u}) \xw^{2 \ssg }
			dy\\
			&+2 \ssg \j |\p_y \p^{\al} (\f{v}{u})|^2
			\cdot  \p_y u \xw^{2 \ssg -1}dy
			+\left( |\p_y \p^{\al} (\f{v}{u})|^2 \p_y u
			\right)_{y=0}.\\
		\end{aligned}
		\deqq
		Using the estimates in Lemma \ref{Lemma:help} and the Hardy inequality, we have
		\beqq
		\begin{aligned}
			&\; \Bigg|-\j \left( 2 \p_y^2 u \p_y \p^{\al} (\f{v}{u}) + \p_y^3 u  \p^{\al} (\f{v}{u})
			\right)
			\cdot \p^{\al} \p_y (\f{v}{u}) \xw^{2 \ssg }
			dy
			+2 \ssg  \j |\p_y \p^{\al} (\f{v}{u})|^2
			\cdot  \p_y u \xw^{2 \ssg -1}dy \Bigg| \\
			\le  &\; C_{\ssg}
			(\|\p_y^2 u\|_{L^{\infty}_y}+\|\p_y u\|_{L^{\infty}_y})\|\p^{\al}\p_y (\f{v}{u}) \xw^{\ssg}\|_{L_y^2}^2
			+ \|\p_y^3 u \xw \|_{L^{\infty}_y}
			\|\p^{\al} (\f{v}{u}) \xw^{\ssg-1} \|_{L_y^2}
			\|\p^{\al}\p_y (\f{v}{u}) \xw^{\ssg}\|_{L_y^2}\\
			\le &\;  C_{\lambda,\xi_0,\eta,\ssg} (1+ \se^4)+C_{\gamma,\lambda, \xi_0,\eta,\ssg} \sde(1+ \se^3) \sy.
		\end{aligned}
		\deqq	
		Therefore,
		\beqq
		K_{3,2,2}+K_{3,2,3} \ge
		\left( |\p_y \p^{\al} (\f{v}{u})|^2 \p_y u
		\right)_{y=0}
		- C_{\lambda,\xi_0,\eta,\ssg} (1+ \se^4)-C_{\gamma,\lambda, \xi_0,\eta,\ssg} \sde(1+ \se^3) \sy.
		\deqq
		Now deal with the term $K_{3,2,4}$.
		\beq\label{e=K324}
		\begin{aligned}
			K_{3,2,4}=&-\sum_{0 \le \beta\le \alpha-1} C_{\al}^{\beta} \j \left(\p_y^3\p^{\beta}(\f{v}{u}) \p^{\al-\beta} u + 3 \p_y^2\p^{\beta}(\f{v}{u}) \p_y \p^{\al-\beta} u+3 \p_y\p^{\beta}(\f{v}{u}) \p_y^2 \p^{\al-\beta} u+\p^{\beta}(\f{v}{u}) \p_y^3\p^{\al-\beta} u
			\right)\\
			& \times \p^{\al} \p_y (\f{v}{u}) \xw^{2 \ssg } dy \\
			\overset{def}{=} & \sum_{0 \le i \le 4}  K_{3,2,4,i}.
		\end{aligned}
		\deq
		Using the estimates in Lemma \ref{Lemma:help} and the Hardy inequality, we have
		\beq\label{e:K3242-3}
		\begin{aligned}
			|K_{3,2,4,2}|&\;\le C \|\p_y^2(\f{v}{u}) \|_{H^{m-1}_{y,\sg}}
			\|\p_y \p^{e_i} u \|_{H^{m-1}_{y,0}} \|\p^{\al}\p_y (\f{v}{u}) \xw^{\ssg}\|_{L_y^2} \\
			&\;\le
			C_{\lambda, \xi_0,\eta} (1+P(\se))+ C_{\lambda, \xi_0,\eta} \sde (1+P(\se)) \sy;
			\\
			|K_{3,2,4,3}|& \;\le C \|\p_y(\f{v}{u}) \|_{H^{m}_{y,\sg}}
			\|\p_y^2 u \|_{H^{m}_{y,0}} \|\p^{\al}\p_y (\f{v}{u}) \xw^{\ssg}\|_{L_y^2} \\
			&\;\le C_{\lambda, \xi_0,\eta} (1+P(\se))+ C_{\lambda, \xi_0,\eta} \sde (1+P(\se)) \sy.
		\end{aligned}
		\deq
		Then we deal with the term $K_{3,2,4,1}$.
		If $|\al| \le m-1$, for $i=1,2$,
		\beqq
		|K_{3,2,4,1}|\le C \|\p_y^3(\f{v}{u}) \|_{H^{m-2}_{y,\ssg}}
		\|\p^{e_i} u \|_{H^{m-2}_{y,0}}
		\|\p^{\al}\p_y (\f{v}{u}) \xw^{\ssg}\|_{L_y^2} \le C_{\lambda, \xi_0,\eta} (1+P(\se))+ \sde \sy.
		\deqq
		If $|\al| = m$ and  $|\beta| \le m-2$, for $|\gamma|=|(\gamma_1, \gamma_2)|=2$,
		\beqq
		\begin{aligned}
			&\; \sum_{|\al| = m, |\beta| \le m-2} C_{\al}^{\beta} |\int \p_y^3\p^{\beta}(\f{v}{u}) \p^{\al-\beta} u \cdot \p^{\al} \p_y (\f{v}{u}) \xw^{2 \ssg} dy|\\
			\le &\;C\|\p_y^3 (\f{v}{u}) \|_{H_{y,\ssg}^{m-2}}
			\|\p^{\gamma} u \|_{H_{y,0}^{m-2}}
			\|\p^{\al}\p_y (\f{v}{u}) \xw^{\ssg}\|_{L_y^2} \\
			\le &\;
			C_{\lambda, \xi_0,\eta} (1+P(\se))+ C_{\lambda, \xi_0,\eta} \sde (1+P(\se)) \sy.
		\end{aligned}
		\deqq
		If $|\al| = m$ and  $\beta=\al-e_2$, integrating by part, we have
		\beqq
		\begin{aligned}
			&\;  \j\p_y \{\p^{\al}\p_y (\f{v}{u})\} \p_y u \cdot \p^{\al} \p_y (\f{v}{u}) \xw^{2 \ssg } dy\\
			=&\; \f12  \j |\p^{\al}\p_y (\f{v}{u})|^2 \p_y( \p_y u \xw^{2 \ssg }) dy-\f12 \left(|\p^{\al}\p_y (\f{v}{u})|^2  \p_y u \xw^{2 \ssg} \right)_{0}^{\infty}  \\
			\ge &\;  \f12 \left(|\p^{\al}\p_y (\f{v}{u})|^2  \p_y u \right)_{y=0} - C(\|\p_y^{2} u \|_{L^{\infty}_y} +\|\p_y u \xw^{-1} \|_{L^{\infty}_y} )\|\p^{\al}\p_y (\f{v}{u}) \xw^{\ssg}\|_{L_y^2}^2\\
			\ge &\;  \f12 \left(|\p^{\al}\p_y (\f{v}{u})|^2  \p_y u\right)_{y=0}-
			C_{\lambda, \xi_0,\eta} (1+\se^5)- C_{\lambda, \xi_0,\eta} \sde (1+\se^4) \sy.
		\end{aligned}
		\deqq
		If $|\al| = m$ and  $\beta=\al-e_1$, we have
		\beqq
		\begin{aligned}
			&| \j\p^{\beta + e_2}\p_y^2 (\f{v}{u}) \p_x u \cdot \p^{\al} \p_y (\f{v}{u}) \xw^{2 \ssg } dy|\\
			\le &\;  (1+\|u \|_{L^{\infty}_y}) \|\f{\p_x u}{u} \|_{L^{\infty}_y}
			\|\sqrt{u} \p^{\beta + e_2}\p_y^2  (\f{v}{u}) \xw^{\ssg}\|_{L_y^2}
			\|\p^{\al}\p_y (\f{v}{u}) \xw^{\ssg}\|_{L_y^2}\\
			\le &\;  \mu \sy + C_{\mu,\lambda,\xi_0,\eta} (1+\se^{9})+C_{\lambda,\xi_0,\eta} \sde^{\f12}(1+ \se^4) \sy.
		\end{aligned}
		\deqq
		Thus, combining the above estimates of $K_{3,2,4,1}$, we have
		\beq\label{e:K3241}
		\begin{aligned}
			K_{3,2,4,1}
			\ge  C\left(|\p^{\al}\p_y (\f{v}{u})|^2  \p_y u\right)_{y=0} -\mu \sy - C_{\mu,\lambda,\xi_0,\eta} (1+ P(\se))-C_{\lambda,\xi_0,\eta} \sde^{\f12}(1+ P(\se)) \sy.
		\end{aligned}
		\deq
		Finally, let us deal with the term $K_{3,2,4,4}$.
		If $|\beta| = 0$, due to the relation $\p_y^2 u =-\rho
		u^2 \p_y(\f{v}{u})$  and $\rho=\br+\varrho_{\infty}, u=\bu+u_{\infty}$, integrating by part, for $i=1,2$, we have
		\beqq
		\begin{aligned}
			&\;  \j \f{v}{u} \p_y\p^{\al} \{\rho u^2 \p_y(\f{v}{u})\}
			\cdot \p^{\al} \p_y (\f{v}{u}) \xw^{2 \ssg} dy\\
			=&\;    \j\f{v}{u} \p^{\alpha}\p_y (\rho u^2) \p_y (\f{v}{u})
			\cdot \p^{\al} \p_y (\f{v}{u}) \xw^{2 \ssg} dy\\
			&\;
			+
			\sum_{0 \le \beta < \alpha} C_{\al}^{\beta} \j\f{v}{u} \p^{\beta}\p_y (\rho u^2) \p^{\al-\beta}\p_y (\f{v}{u})
			\cdot \p^{\al} \p_y (\f{v}{u}) \xw^{2 \ssg} dy\\
			&\; +\sum_{1 \le \beta\le \alpha} C_{\al}^{\beta} \j \f{v}{u} \p^{\beta}(\rho u^2) \p^{\al-\beta}\p_y^2 (\f{v}{u})
			\cdot \p^{\al} \p_y (\f{v}{u}) \xw^{2 \ssg } dy \\
			&\; +\j \f{v}{u} \rho u^2 \p^{\al}\p_y^2 (\f{v}{u})
			\cdot \p^{\al} \p_y (\f{v}{u}) \xw^{2 \ssg } dy\\
			\overset{def}{=} &\;\sum_{i=1}^4 K_{3,2,4,4,i}.
		\end{aligned}
		\deqq
		It is easy to check that
		\beqq
		\bal
		|\sum_{i=2}^4 K_{3,2,4,4,i}| \le &\; C
		\Big(
		\|\p_y (\rho u^2)\|_{H_{y,0}^{m-1}} \|\p^{e_i} \p_y (\f{v}{u})\|_{H_{y,\ssg}^{m-1}}
		+\|\p^{e_i} (\rho u^2)\|_{H_{y,0}^{m-1}} \|\p_y^2 (\f{v}{u})\|_{H_{y,\ssg}^{m-1}}\\
		&\;  +
		(1+\|u \|_{L^{\infty}_y}^2)
		\|\rho \sqrt{u} \p^{\al}\p_y^2  (\f{v}{u}) \xw^{\ssg}\|_{L_y^2}
		\Big) \|\f{v}{u} \|_{L^{\infty}_y}\|\p^{\al}\p_y (\f{v}{u}) \xw^{\ssg}\|_{L_y^2}\\
		\le  &\;  \mu \sy + C_{\mu,\lambda,\xi_0,\eta} (1+ P(\se))+C_{\lambda,\xi_0,\eta} \sde^{\f12} (1+ P(\se)) \sy.
		\dal
		\deqq
		As for the term  $K_{3,2,4,4,1}$,
		\beqq
		\bal
		| K_{3,2,4,4,1}|
		= &\; |\j u^2 \f{v}{u} \p^{\alpha}\p_y \rho  \p_y (\f{v}{u})
		\cdot \p^{\al} \p_y (\f{v}{u}) \xw^{2 \ssg} dy\\ &\;
		+\sum_{0 \le \beta\le \alpha-1} C_{\al}^{\beta} \j \f{v}{u} \p^{\beta}\p_y \rho \p^{\alpha-\beta}(u^2) \p_y (\f{v}{u})
		\cdot \p^{\al} \p_y (\f{v}{u}) \xw^{2 \ssg} dy\\ &\;
		+\j\f{v}{u} \p^{\alpha} \{\rho \p_y(u^2)\} \p_y (\f{v}{u})
		\cdot \p^{\al} \p_y (\f{v}{u}) \xw^{2 \ssg} dy|\\
		\le &\; \mu \sy + C_{\mu,\lambda,\xi_0,\eta} (1+ P(\se))+C_{\lambda,\xi_0,\eta} \sde^{\f12} (1+ P(\se)) \sy,
		\dal
		\deqq
		where the estimate of the first term can refer to the estimate of $K_{3,1,2,5}$. Therefore, when $|\beta|=0$,
		\beqq
		\bal
		| \j \f{v}{u} \p_y^3\p^{\al} u
		\cdot \p^{\al} \p_y (\f{v}{u}) \xw^{2 \ssg} dy | \le \mu \sy + C_{\mu,\lambda,\xi_0,\eta} (1+ P(\se))+C_{\lambda,\xi_0,\eta} \sde^{\f12} (1+ P(\se)) \sy.
		\dal
		\deqq
		If $|\beta| \ge 1$, we have
		\beqq
		\begin{aligned}
			&\; | \j \p^{\beta}(\f{v}{u}) \p_y^3\p^{\al-\beta} u
			\cdot \p^{\al} \p_y (\f{v}{u}) \xw^{2 \ssg} dy|\\
			\le &\;  \|\p_y (\f{v}{u}) \|_{H_{y,\ssg}^{m}}
			\|\p_y^2 u \|_{H_{y,1}^{m}}
			\|\p^{\al}\p_y (\f{v}{u}) \xw^{\ssg}\|_{L_y^2} \\
			\le  &\;  C_{\lambda,\xi_0,\eta} (1+ P(\se))+C_{\lambda,\xi_0,\eta} \sde (1+ P(\se)) \sy.
		\end{aligned}
		\deqq
		Thus, combining the above estimates of $K_{3,2,4,4}$, we have
		\beq \label{e:K3244}
		\begin{aligned}
			|K_{3,2,4,4}|
			\le  \mu \sy + C_{\mu,\lambda,\xi_0,\eta} (1+ P(\se))+C_{\lambda,\xi_0,\eta} \sde^{\f12} (1+ P(\se)) \sy.
		\end{aligned}
		\deq
		Combining the estimates of  \eqref{e:K3242-3} \eqref{e:K3241} and \eqref{e:K3244} into \eqref{e=K324},  we can obtain
		\beqq
		\begin{aligned}
			K_{3,2,4}
			\ge  C\left(|\p^{\al}\p_y (\f{v}{u})|^2  \p_y u\right)_{y=0}- \mu \sy - C_{\mu,\lambda,\xi_0,\eta} (1+ P(\se))-C_{\lambda,\xi_0,\eta} \sde^{\f12} (1+ P(\se)) \sy,
		\end{aligned}
		\deqq
		which, together with the estimates of $K_{3,2,i}(i=1,2,3)$, yields
		\beq\label{e:K32}
		\begin{aligned}
			-K_{3,2} \ge &\;\j u |\p_y^2 \p^{\al} (\f{v}{u})|^2 \xw^{2 \ssg } dy
			+ C\left(|\p^{\al}\p_y (\f{v}{u})|^2  \p_y u\right)_{y=0}\\
			&\;-\mu \sy - C_{\mu,\lambda,\xi_0,\eta,\ssg} (1+ P(\se))-C_{\lambda,\xi_0,\eta,\ssg} \sde^{\f12} (1+ P(\se)) \sy.
		\end{aligned}
		\deq
		Combining the estimates of \eqref{e:K31} and \eqref{e:K32}, we complete the proof of this lemma.
	\end{proof}
	Finally, we will establish the estimate for the lower order derivatives of velocity.
	\begin{lemm}\label{lemma:u-pyu}
		Under the condition of \eqref{inital-pyu} and \eqref{inital-u,rho}, for any smooth  solution $(\rho,u,v)$ of system
		\eqref{s-Prandtl}, it holds
		\beqq
		\begin{aligned}
			&\frac{d}{dx}
			\| (\bar{u}, \p_y \bar{u})\|_{L_y^2}^2
			\le C_{\lambda,\xi_0,\eta} (1+ \se^2) .
		\end{aligned}
		\deqq
	\end{lemm}
	\begin{proof}
		Multiplying $\eqref{s-Prandtl}_3$ by $\bar{u} $ and integrating over $\mathbb{R}^+$, we have
		\beqq
		\begin{aligned}
			&\j \p_x \bar{u} \cdot \bar{u} dy
			+\underbrace{\j \p_y v \cdot \bar{u} dy}_{K_{4,1}}=0.
		\end{aligned}
		\deqq
		Using H\"{o}lder inequality and the estimate \eqref{e:v/u}, we have
		\beqq
		\begin{aligned}
			|K_{4,1}| &= |\j \p_y (u \f{v}{u}) \cdot \bar{u} dy |\\
			&= |\f12 \j \f{v}{u} \p_y \bar{u} \cdot \bar{u} dy + \j  (u_{\infty}+\bar{u}) \p_y(\f{v}{u}) \cdot \bar{u} dy |\\
			&\le C (\| \bar{u}\|_{L_y^2}   \|\p_y  \bar{u}\|_{L_y^2} \| \p_y (\f{v}{u}) \|_{L_y^{\infty}} + \| \bar{u}\|_{L_y^2}^2 \| \p_y (\f{v}{u}) \|_{L_y^{\infty}} +\| \bar{u}\|_{L_y^2} \| \p_y (\f{v}{u}) \|_{L_y^{2}}),
		\end{aligned}
		\deqq
		which yields that
		\beq\label{e:K41}
		\begin{aligned}
			\frac{d}{dx}
			\| \bar{u}\|_{L_y^2}^2
			\le  C_{\lambda,\xi_0,\eta} (1+ \se^2).
		\end{aligned}
		\deq
		Applying the $\p_y$ differential operator to $\eqref{s-Prandtl}_3$, 	multiplying by $\p_y \bar{u} $ and integrating over $\mathbb{R}^+$, we have
		\beqq
		\begin{aligned}
			& \j \p_y \p_x \bar{u} \cdot \p_y \bar{u} dy + \underbrace{\j \p_y^2 v \cdot \p_y \bar{u} dy}_{K_{4,2}}=0.
		\end{aligned}
		\deqq
		Integrating by part, we have
		\beqq
		\begin{aligned}
			|K_{4,2}|&=|\j \f{v}{u} \p_y^2 \bar{u} \cdot \p_y \bar{u} + 2\p_y (\f{v}{u}) |\p_y \bar{u}|^2 + \bar{u}  \p_y^2(\f{v}{u})  \cdot \p_y \bar{u} dy|\\	
			&=|\f32\j \p_y (\f{v}{u}) |\p_y \bar{u}|^2 + \bar{u}  \p_y^2(\f{v}{u})  \cdot \p_y \bar{u} dy|\\
			& \le C (\| \p_y \bar{u}\|_{L_y^2}^2 \| \p_y (\f{v}{u}) \|_{L_y^{\infty}} +\| \bu \|_{L_y^{\infty}}\| \p_y \bar{u}\|_{L_y^2} \| \p_y^2 (\f{v}{u}) \|_{L_y^{2}}),
		\end{aligned}
		\deqq
		which yields that
		\beq\label{e:K42}
		\begin{aligned}
			\frac{d}{dx}
			\| \p_y \bar{u}\|_{L_y^2}^2
			\le C_{\lambda,\xi_0,\eta} (1+ \se^2).
		\end{aligned}
		\deq
		Combining the estimates of \eqref{e:K41} and \eqref{e:K42}, we complete the proof of this lemma.
	\end{proof}
	
	\subsection{Proof of Proposition \ref{prop: prior estimate}}
	Combining the estimates from Lemma \ref{Lemma:help} to Lemma \ref{lemma:u-pyu} and choosing $\mu$ small enough,  we have
	\beqq
	\f{d}{dx} \se + \sy + C \left(|\p^{\al}\p_y (\f{v}{u})|^2  \p_y u\right)_{y=0}
	\le C_{\lambda,\xi_0,\eta,\ssg} (1+ P(\se))+C_{\lambda,\xi_0,\eta,\ssg} \sde^{\f12} (1+ P(\se)) \sy.
	\deqq
	Therefore we finish the proof of Proposition \ref{prop: prior estimate}.

	\section{Existence and uniqueness of the steady Prandtl equations}\label{sec:well-posedness-steady}
	In this section, we will establish the local existence and uniqueness of solution stated in Theorem \ref{main-result-steady} for the steady Prandtl system \eqref{s-Prandtl}.
	Define $(\rho^0, u^0)\overset{def}{=}(\rho_0, u_0)$,
	and for any $\theta>0$, let us consider the following approximated system
	for all $k \ge 1$:
	\beq\label{e:approximate}
	\left\{
	\bal
	&\p_x \rho^k + q^{k-1} \p_y \rho^k =0,\\
	&\p_x\{ \rho^k (u^{k-1}+\theta)^2 \p_y q^k\} -\p_y^3 \{u^{k-1}q^k\} =0,\\
	&\p_x u^k + \p_y (u^{k-1} q^k) =0,\\
	&(q^k,u^k)|_{y=0}=(0,0),\\
	&\lim_{y \rightarrow+\infty}(\rho^k,u^k, \p_y q^k )
	=(\varrho_{\infty}, u_{\infty}, 0),\\
	&(\rho^k, u^k, q^k)|_{x=0}=(\rho_0, u_0, q_0).
	\dal\right.
	\deq
	Define $q_0(y) \overset{def}{=}-\int_{0}^{y} \f{\p_y^2 u_0}{\rho_0 u_0^2} dy'$.
    For any fixed $\theta>0$, it is easy to apply the Galerkin method
    to construct the local  solution $(\rho^k, q^k, u^k)$
    on $[0, L_{\theta, k}]$(cf.\cite{Guo-Iyer-2021}).
    This life-span $L_{\theta, k}$ may tend to zero
    as $\theta \rightarrow 0^+$ and $k \rightarrow +\infty$.
    Thus, we will establish some uniform estimate
    and uniform life-span $[0, L_a]$ for the solution $(\rho^k, q^k, u^k)$.
    Therefore, in section \ref{sec:uniform}, for the fixed constant $\theta$, we will
	establish some uniform estimates with respect to the parameter $k$ for the approximated system \eqref{e:approximate}. Then the local existence and uniqueness of original steady Prandtl system are to be investigated in sections \ref{sec:existence} and \ref{sec:unique} respectively.
	
	\subsection{Uniform estimate of approximated equations}\label{sec:uniform}
	In this subsection, for the fixed constant $\theta$, we will derive the uniform estimate with respect to $k$
	for the approximated system  \eqref{e:approximate}.
	From the equations $\eqref{e:approximate}_2$ and $\eqref{e:approximate}_3$, we have
	\beqq
	\p_x\{ \rho^k (u^{k-1}+\theta)^2 \p_y q^k + \p_y^2 u^k\} =0,
	\deqq
	which yields directly
	\beqq
	\rho^k (u^{k-1}+\theta)^2 \p_y q^k + \p_y^2 u^k = \rho_0 (2 u_0 \theta +\theta^2) \p_y q_0 \overset{def}{=} r_0.
	\deqq
	For all $k \ge 1$, define the energy norm and dissipation norm for the solution ($\rho^{k}, q^{k}, u^k$) as following:
	\beqq
	\bal
	\x^k(x) \overset{def}{=} & \sum_{|\alpha|\le m}
    \|\p^\alpha(\rho^k-\varrho_\infty) \|_{L_y^2}^2
	+\sum_{|\alpha|\le m}
     \|\sqrt{\rho^k} (u^{k-1}+\theta) \p^{\al}\p_y q^k \xw^{\ssg}\|_{L_y^2}^2
	+\|(u^k-u_\infty) \|_{L_y^2}^2+\|\p_y u^k\|_{L_y^2}^2,
	\dal
	\deqq
and
	\beqq
	\bal
	\y^k(x) \overset{def}{=}
	&\sum_{|\alpha|\le m}\|\sqrt{u^{k-1}} \p^\alpha \p_y^2 q^k \xw^{\ssg}\|_{L_y^2}^2.
	\dal
	\deqq
	Define
	$\lambda_0\overset{def}{=}\f{ u_0'(0)}{4}. $
Due to the assumption condition of initial data $u_0(y)$, there exist
two small positive constants $\delta_0$ and $\xi_0$, for any $0 < \sde \le \delta_0$, one can obtain that
\beq\label{u-i}
u_0(y)\ge  2\lambda_0 y, \quad \forall y \in [0, \sde],
\deq
and
\beq\label{u-ii}
u_0(y) \ge 2\xi_0, \quad \forall y \in [\frac{\sde}{2}, +\infty),
\deq
{where $\xi_0$ is a positive constant only depending on $\sde$.}
Now, we will state the following uniform estimate.
	\begin{prop}[Uniform estimate]
		\label{prop-suniform}
		Let $k\ge 1$ be an integer
		and $(\rho^k, q^k, u^k)$ be sufficiently smooth solution of the approximated system \eqref{e:approximate}. Suppose $(\rho_0, u_0)$ satisfy the conditions \eqref{stherom-intial-u}-\eqref{density-assumption}.
For any $L > 0$, assume
		\beq \label{prop-puyk}
		u^{k-1}(x, y)>0, \quad
        \p_y u^{k-1} (x, 0)\ge \lambda_0,
        \quad \forall (x,y) \in [0, L] \times (0,+\infty),
		\deq
hold for all $k \ge 1$.
		Then, there exists a time $L_a$ independent of $k$ and $\theta$ such that the following a priori estimate holds true for all $x \in [0, L_a]$:
		\beq \label{uniform-s}
		\sup_{0 \le x \le L_a}  \x^k(x) + \int_{0}^{L_a} \y^k (x) dx \le 2 C_{\rho_0, u_0},
		\deq
		and
		\beq\label{prop-re-urhok}
		\bal
		\p_y u^k (x,0) & \ge 2{\lambda_0}, \quad \rho^{k} (x, y)  \geq \kappa_3;\\
		u^k(x,y) &\ge \left\{
		\bal
		&\lambda_0  y, &&\quad  \forall y \in [0,\sde];\\
		&\xi_0 , &&\quad \forall y \in  [\f{\sde}{2}, +\infty),
		\dal \right.\\	
		\dal
		\deq
		for all $k \ge 1$.
	\end{prop}
	\begin{proof}
		Let $K\geq 1$ be a fixed large integer, and let us introduce the auxiliary functions $X^K(x)$ , $Y^K(x)$ and $Z^K(x)$:
		\beqq
		\bal
		X^K(x) \overset{def}{=} &\max_{1\le j \le K} \x^j(x);\\
		Y^K(x) \overset{def}{=} &\max_{1\le j \le K} \y^j(x);\\
        Z^K(x) \overset{def}{=} &\max_{1\le j \le K} \p_y u^j(x,0).
		\dal
		\deqq	
		For the parameter $R$, which will be defined later, we have
		\beq\label{l-definition}
	     L_{\theta, K}^* \overset{def}{=}
         \sup \{L\in (0,1]| X^{K}(x) \leq R,
         Z^K(x) \ge \lambda_0 , x \in (0, L]  \}.
		\deq	
        For all $1\le k \le K$, let us define the characteristic line
        \begin{equation*}
        \left\{
        \begin{aligned}
        &\frac{d}{dx}X^{k-1}(x, \alpha)=q^{k-1}(x, \alpha),\\
        &X^{k-1}(x, \alpha)|_{x=0}=\alpha.
        \end{aligned}
        \right.
        \end{equation*}
        Thus, on the characteristic line the density equation
        $\eqref{e:approximate}_1$ has the following form
        \beqq
        \frac{d}{dx}\rho^k(x, X^{k-1}(x, \alpha))=0,
        \deqq
        which, together with condition \eqref{density-assumption}, yields directly
        \beqq
        \rho^k(x, X^{k-1}(x, \alpha))=
        \rho^k(x, \alpha)|_{x=0}=\rho_0(\alpha)\ge \kappa_3.
        \deqq
        This implies directly the uniform estimate
        \beq
        \rho^{k} (x, y)  \geq \kappa_3.
        \deq
		Similar to the analysis of a priori estimate \eqref{steady-priori}
        of the steady Prandtl equations \eqref{s-Prandtl}, we can obtain for all
        $x \le L_{\theta, K}^*$,
		\beqq
		X^K(x)+ \int_{0}^{x} Y^K(\xi) d\xi \le C_{\rho_0, u_0} + C_{6} \sde^{\f12}\int_{0}^{x}(1+P(X^K)) Y^K d\xi + C_{7} \int_{0}^{x}(1+P(X^K))d\xi,
		\deqq
		where  the constants  $C_{6}$ and $C_{7}$  are only dependent of $\rho_0, u_0, \lambda_0, \xi_0, \kappa_3$ and  $\ssg$.
Similar to the control of the compatibility condition hold at the corner $(0,0)$ (see Appendix \ref{appendix-com} in detail), we can check that there exists a constant only depending on the initial data $\rho_0$, $u_0$ such that $X^k(0) \le C_{\rho_0, u_0}$.	
Choosing $R \overset{def}{=} 4C_{\rho_0, u_0}$,
$L_1 \overset{def}{=} \min\{ \f{C_{\rho_0, u_0}}{C_{7} (1+P(4C_{\rho_0, u_0})}, 1\}$ and  $\sde\overset{def}{=} \min\{  \frac{1}{ 4 C_{6}^2 (1+P(4C_{\rho_0, u_0}))^2}, \delta_0\}$,
then we have for all $x\le L_1$
\beq\label{uniform-E_2}
X^K(x)+\frac12 \int_0^{x} Y^K d\xi \le 2 C_{\rho_0, u_0} = \f{R}{2}.
\deq
Using the equation $\eqref{e:approximate}_3$ and estimate \eqref{uniform-E_2}, we have
		\beqq
		\bal
		\p_y u^k(x,0) =u_0'(0)
     -\int_{0}^{x} (\p_y^2 \{u^{k-1} q^k\})_{y=0} d\xi,
		\dal
		\deqq
and
\beqq
		\bal
		& \; \int_{0}^{x}| (\p_y^2 \{u^{k-1} q^k\})_{y=0}| d\xi \\
		\le &\; x \| \p_y u^{k-1} |_{y=0} \|_{L_x^{\infty}}  \| \p_y q^{k} |_{y=0} \|_{L_x^{\infty}}  \\
		\le &\;  C x(\| \p_y u^{k-1} \|_{L_x^{\infty}L_y^{\infty}} + \| \p_y u^{k-1} \|_{L_x^{\infty}L_y^{2}}^{\f12} \| \p_y^2 u^{k-1} \|_{L_x^{\infty}L_y^{2}}^{\f12})
		(\| \p_y q^{k} \|_{L_x^{\infty}L_y^{\infty}} + \| \p_y q^{k} \|_{L_x^{\infty}L_y^{2}}^{\f12} \| \p_y^2 q^{k} \|_{L_x^{\infty}L_y^{2}}^{\f12})
		\\
		\le &\; C_{8} C_{\rho_0, u_0} x,
		\dal
		\deqq
		where the constant $C_8$ and the following mentioned constants $C_{9}$, $C_{10}$ are only dependent of $\lambda_0, \xi_0, \kappa_3$.
 Thus, choose $L_2 \overset{def}{=}\min\{
              \frac{u_0'(0)}{2 C_{8} C_{\rho_0, u_0}}, L_1\}
              =\min\{\frac{u_0'(0)}{2 C_{8} C_{\rho_0, u_0}},
              \f{C_{\rho_0, u_0}}{C_{7} (1+P(4C_{\rho_0, u_0}))}, 1\}$,
               we have
\beq \label{lower-pyuK}
		\bal
		\p_y u^k(x,0)
       \ge u_0'(0)-\f{ u_0'(0)}{2}\ge \f{ u_0'(0)}{2}=2 \lambda_0.
		\dal
		\deq
    Due to the equation $\eqref{e:approximate}_3$, one can obtain
		\beqq
		u^k(x,y) = u_0(y)-\int_{0}^{x} \p_y \{u^{k-1} q^k\} d \xi,
		\deqq
        which, together with the boundary conditions
		of $q^{k}$ and $u^{k-1}$, yields directly
		\beq\label{u-exp}
        u^k(x,y) = u_0(y)-\int_{0}^{x} \p_y \{u^{k-1} q^k\} d\xi
		=u_0(y)- \int_{0}^{x} \int_{0}^{y} \p_{\eta}^2 \{u^{k-1} q^k\} d \xi d \eta.
		\deq
Thus, the combination of \eqref{u-exp} and \eqref{u-i} yields
		\beqq
		\bal
		u^k(x,y)  &\ge u_0(y) -\int_{0}^{x} \int_{0}^{y} |\p_y^2 \{u^{k-1} q^k\}| dx \ge \frac{u_0'(0)}{2} y - xy \|\p_y^2 \{u^{k-1} q^k\}\|_{L^\infty}\\
		&\ge \frac{u_0'(0)}{2} y- C_{9} C_{\rho_0, u_0} xy,
		\dal
		\deqq
for all $y\in [0,\sde]$.
		Choose  $L_{3}\overset{def}{=}
\min\{\f{u_0'(0)}{4 C_{9}  C_{\rho_0, u_0}}, L_{2}\}
=\min\{\f{u_0'(0)}{4 C_{9}  C_{\rho_0, u_0}},
\frac{u_0'(0)}{2 C_{8} C_{\rho_0, u_0}},
              \f{C_{\rho_0, u_0}}{C_{7} (1+P(4C_{\rho_0, u_0}))}, 1\}$,
then we have
		\beq\label{u-decompose-1}
		u^k(x,y)  \ge \frac{u_0'(0)}{2} y - \frac{u_0'(0)}{4} y
=\frac{u_0'(0)}{4} y
= {\lambda_0} y, \quad  \forall (x,y) \in [0,L_3] \times [0,\sde].
		\deq	
		Similarly, choose $L_4 \overset{def}{=}\min\{L_{3}, \f{\xi_0}{ C_{10}  C_{\rho_0, u_0}}\} $,
then the combination of \eqref{u-ii} and \eqref{u-exp} yields
		\beq\label{u-decompose}
		u^k(x,y) \ge 2\xi_0 - x \|\p_y^2 \{u^{k-1} q^k\}\|_{L_x^\infty L_y^2}
		\ge   2\xi_0 -C_{10}  C_{\rho_0, u_0} x
= {\xi_0}, \quad \forall (x,y) \in [0,L_4] \times [\f{\sde}{2}, +\infty).
		\deq
Then, the uniform estimates \eqref{uniform-E_2} and \eqref{lower-pyuK}
yield $L_4 \le L_{\theta, K}^*$.
Indeed  otherwise, our criterion about the continuation of the solution
would contradict the definition of $L_{\theta, K}^*$ in \eqref{l-definition}.
Let us define $L_a\overset{def}{=}L_4$, then we find the uniform
existence time $L_a$ (independent of $K$ and $\theta$) such that the estimates \eqref{uniform-E_2}, \eqref{lower-pyuK}, \eqref{u-decompose-1}
 and \eqref{u-decompose} hold, which implies the estimates \eqref{uniform-s} and \eqref{prop-re-urhok}. Thus we finish the proof of Proposition \ref{prop-suniform}.
	\end{proof}

	\subsection{Local existence of the original system }\label{sec:existence}
	In this subsection, we will show the solution $(\rho^k, u^k, q^k)$ to the approximate system \eqref{e:approximate} converges to a solution to the original system \eqref{s-Prandtl}
	as $k$ tends to infinity and $\theta$ tends to zero.
	To prove this, we define
	$\hr^{k+1}\overset{def}{=}\rho^{k+1}-\rho^{k}$,
	$\hu^{k+1}\overset{def}{=}u^{k+1}-u^{k}$ and $\hq^{k+1}\overset{def}{=}q^{k+1}-q^{k}$.
	Then, the quantity $(\hr^{k+1}, \hu^{k+1},\hq^{k+1})$ satisfies the following system:
	\beq\label{e:linear}
	\left\{
	\bal
	&\p_x \hr^{k+1}+q^k \p_y \hr^{k+1}+\hq^k \p_y \rho^{k}=0,\\
	&\p_x \{\rho^{k+1}(u^{k}+\theta)^2 \p_y{\hq^{k+1}}\} +
	\p_x \{\hr^{k+1}(u^{k}+\theta)^2 \p_y{q^{k}}\}\\
	&\quad +\p_x \{\rho^{k} \hu^k (u^{k}+u^{k-1}+2\theta) \p_y{q^{k}}\}
	-\p_y^3 (u^k \hq^{k+1})-\p_y^3 (\hu^k q^{k})=0,\\
	&\p_x \hu^{k+1}=-(\p_y u^{k-1} \hq^{k+1} +  u^{k-1} \p_y \hq^{k+1} + \p_y \hu^{k} q^{k+1} + \hu^{k-1} \p_y q^{k+1}),\\
	&(\hq^{k+1},\hu^{k+1})|_{y=0}=(0,0),\\
	&\lim_{y \rightarrow+\infty}(\hr^{k+1},\hu^{k+1},\p_y \hq^{k+1})
	=(0, 0, 0),\\
	&(\hr^{k+1}, \hu^{k+1}, \hq^{k+1})|_{x=0}=(0, 0, 0),
	\dal\right.
	\deq
	and
	\beq\label{e:linear-py2}
	\p_y^2{\hu^{k+1}}=\hr^{k+1}(u^{k}+\theta)^2 \p_y{q^{k+1}} + \rho^{k} \hu^k (u^{k}+u^{k-1}+2\theta) \p_y{q^{k+1}} + \rho^{k}(u^{k-1}+\theta)^2 \p_y{\hq^{k+1}}.
	\deq
	\textbf{Step 1:} Under the uniform estimate \eqref{uniform-s} and the lower bound \eqref{prop-re-urhok},
	multiplying the equation $\eqref{e:linear}_2$ by $\p_y
	\hq^{k+1} \xw^{2\ssg} $ and integrating over $\mathbb{R^+}$, we have
	\beqq
	\bal
	&\f12 \f{d}{dx} \j \rho^{k+1} (u^k+\theta)^2 |\p_y \hq^{k+1}|^2 \xw^{2\ssg}dy
	+ \underbrace{\f12  \j \p_x\{\rho^{k+1} (u^k+\theta)^2 \}|\p_y \hq^{k+1}|^2 \xw^{2\ssg} dy}_{K_{5,1}}\\
	&+\underbrace{ \j \p_x\{\hr^{k+1} (u^k+\theta)^2 \p_y q^k\}\p_y \hq^{k+1} \xw^{2\ssg} dy}_{K_{5,2}}
	+\underbrace{\j \p_x \{\rho^{k} \hu^k (u^{k}+u^{k-1}+2\theta) \p_y{q^{k}}\} \p_y  \hq^{k+1} \xw^{2\ssg} dy }_{K_{5,3}}\\
	&-\underbrace{\j \p_y^3 (u^k \hq^{k+1}) \p_y \hq^{k+1} \xw^{2\ssg} dy}_{K_{5,4}}
	-\underbrace{\j \p_y^3 (\hu^k q^{k}) \p_y \hq^{k+1} \xw^{2\ssg} dy}_{K_{5,5}}=0.
	\dal
	\deqq
	
	\textbf{Deal with the term $K_{5,1}$.} It is easy to check that
	\beqq
	\bal
	|K_{5,1}| & \le \f12  |\j \p_x \rho^{k+1} (u^k+\theta)^2 |\p_y \hq^{k+1}|^2 \xw^{2\ssg} dy |+ |\j \rho^{k+1} (u^k+\theta)^2 \f{\p_x (u^k+\theta)}{u^k+\theta} |\p_y \hq^{k+1}|^2 \xw^{2\ssg} dy |\\
	& \le C \|(u^k+\theta)  \p_y \hq^{k+1} \xw^{\ssg}\|_{L^{2}_y}^2 (\|\p_x \rho^{k+1} \|_{L^{\infty}_y} + \|\rho^{k+1} \|_{L^{\infty}_y} \|\f{\p_x (u^k+\theta)}{u^k+\theta}\|_{L^{\infty}_y}).
	\dal
	\deqq
	
	\textbf{Deal with the term $K_{5,2}$.} Integrating by part and using the equation $\eqref{e:linear}_1$, we have
	\beqq
	\bal
	K_{5,2}=&\;\j \p_x \hr^{k+1} (u^k+\theta)^2 \p_y q^k\p_y \hq^{k+1} \xw^{2\ssg}  dy +\j  \hr^{k+1} \p_x\{(u^k+\theta)^2 \p_y q^k\} \p_y \hq^{k+1} \xw^{2\ssg}  dy \\
	=&\; \j \hr^{k+1} q^k (u^k+\theta)^2 \p_y q^k \hq^{k+1} \xw^{2\ssg}
	+ \hr^{k+1} \p_y \{ q^k (u^k+\theta)^2 \p_y q^k \xw^{2\ssg} \} \p_y^2 \hq^{k+1} dy\\
	&\; + \j \hq^{k} \p_y \rho^k (u^k+\theta)^2 \p_y q^k\p_y \hq^{k+1} \xw^{2\ssg}  dy  +\j  \hr^{k+1} \p_x\{(u^k+\theta)^2 \p_y q^k\} \p_y \hq^{k+1} \xw^{2\ssg}  dy.
	\dal
	\deqq
	Due to the estimate of $u^{k}$ in \eqref{prop-re-urhok}, we can obtain
	\beqq
	\| \f{q^k}{u^k}\|_{L_y^{\infty}} \le   \| \f{q^k}{u^k} \chi \|_{L_y^{\infty}} +\| \f{q^k}{u^k}(1-\chi)\|_{L_y^{\infty}}\\
	\le  C_{\lambda_0} \| \f{q^k}{y} \chi \|_{L_y^{\infty}} + C_{\xi_0} \| q^k\|_{L_y^{\infty}} \\
	\le C_{\lambda_0, \xi_0} (\| \p_y {q^k} \|_{L_y^{\infty}} + \| q^k\|_{L_y^{\infty}}), \\
	\deqq
	which yields
	\beqq
	\bal
	|K_{5,2}|
	\le &\; C \| \hr^{k+1} \|_{L^2_y}
	((1+\|u^k\|_{L_y^{\infty}}^3)
	\| \f{q^k}{u^k}\|_{L_y^{\infty}}
	\| \sqrt{u^k} \p_y^2 \hq^{k+1} \xw^{\ssg}\|_{L^2_y}
	+ \| (u^k+\theta) \p_y \hq^{k+1} \xw^{\ssg}\|_{L^2_y} )\\
	&\; + (1+\|u^k\|_{L_y^{\infty}})
	\| (u^k+\theta) \p_y \hq^{k+1} \xw^{\ssg}\|_{L^2_y} \|\hq^k\|_{L_y^{\infty}} \|\p_y \rho^k\|_{L_y^{\infty}}
	\|\p_y q^k \xw^{\ssg}\|_{L_y^{\infty}}\\
	&\; +\| \hr^{k+1} \|_{L_y^2}
	\| (u^k+\theta) \p_y \hq^{k+1} \xw^{\ssg}\|_{L^2_y}
	( \|\p_x u^k\|_{L_y^{\infty}}
	\|\p_y q^k \xw^{\ssg}\|_{L_y^{\infty}}
	+ \| u^k +\theta\|_{L_y^{\infty}}
	\|\p_{xy} q^k \xw^{\ssg}\|_{L_y^{\infty}}
	)\\
	\le &\; \mu \| \sqrt{u^k} \p_y^2 \hq^{k+1} \xw^{\ssg}\|_{L^2_y}^2
	+ C_{\mu,\lambda_0, \xi_0,\kappa_3} \| ((u^k+\theta) \p_y \hq^{k+1} \xw^{\ssg}, \p_y \hq^{k}, \hr^{k+1})\|_{L^2_y}^2.
	\dal
	\deqq
	
	\textbf{Deal with the term $K_{5,3}$.} Using the equation $\eqref{e:linear}_3$ and integrating by part, we have
	\beqq
	\bal
	K_{5,3}=&\;-\j (\p_y u^{k-1} \hq^k + u^{k-1} \p_y \hq^k +\p_y \hu^{k-1} q^k + \hu^{k-1} \p_y q^k )(\rho^{k}(u^{k}+u^{k-1}+2\theta) \p_y{q^{k}})
	\p_y \hq^{k+1} \xw^{2 \ssg} dy\\
	&\; +\j \p_x \{\rho^{k} (u^{k}+u^{k-1}+2\theta) \p_y{q^{k}} \} \hu^k \p_y  \hq^{k+1} \xw^{2\ssg} dy\\
	=&\;\j \hu^{k-1} \p_y  \{q^k \rho^k
	(u^{k}+u^{k-1}+2\theta) \p_y{q^{k}}
	\xw^{2 \ssg} \} \p_y \hq^{k+1} dy \\
	&\; +\j \hu^{k-1} q^k \rho^k
	(u^{k}+u^{k-1}+2\theta) \p_y{q^{k}}
	\xw^{2 \ssg}  \p_y^2 \hq^{k+1} dy \\
	&\; -\j (\p_y u^{k-1} \hq^k + u^{k-1} \p_y \hq^k +  \hu^{k-1} \p_y q^k )(\rho^{k}(u^{k}+u^{k-1}+2\theta) \p_y{q^{k}})
	\p_y \hq^{k+1} \xw^{2 \ssg} dy\\
	&+\j \p_x \{\rho^{k} (u^{k}+u^{k-1}+2\theta) \p_y{q^{k}} \} \hu^k \p_y  \hq^{k+1} \xw^{2\ssg} dy.
	\dal
	\deqq
	Similar to the estimate of $K_{5,2}$, we have
	\beqq
	\bal
	|K_{5,3}| \le &\;C_{\rho_0,u_0,\xi_0, \lambda_0}  \|\hu^{k-1}\|_{L_y^2}
	(
	\|\p_y \hq^{k+1} \xw^{\ssg}\|_{L_y^2}
	+ \| \sqrt{u^k} \p_y^2 \hq^{k+1} \xw^{\ssg}\|_{L^2_y}
	)\\
	&\; +C_{\rho_0,u_0} (\|\p_y \hq^{k+1} \xw^{\ssg}\|_{L_y^2}(
	\| \hq^{k}\|_{L^{\infty}_y} +\| \p_y \hq^{k}\|_{L^{2}_y} + \| \hu^{k-1}\|_{L^{2}_y}
	)
	+ \| \hu^{k}\|_{L^{2}_y} \|\p_y \hq^{k+1} \xw^{\ssg}\|_{L_y^2})\\
	\le &\;\mu \| \sqrt{u^k} \p_y^2 \hq^{k+1} \xw^{\ssg}\|_{L^2_y}^2
	+ C_{\mu,\lambda_0, \xi_0, \kappa_3} (\| (\p_y \hq^{k+1} \xw^{\ssg},  \hr^{k+1})\|_{L^2_y}^2+ \| (\hu^{k}, \p_y \hq^{k} \xw)\|_{L^2_y}^2 +\|\hu^{k-1}\|_{L^2_y}^2).
	\dal
	\deqq
	
	\textbf{Deal with the term $K_{5,4}$.}
	Using the equation \eqref{e:linear-py2} and integrating by part, we have
	\beqq
	\bal
	-K_{5,4}=&\;-\j \p_y\{\p_y^2 u^{k} {\hq^{k+1}} + 2\p_y u^{k} \p_y \hq^{k+1} +  u^{k} \p_y^2 \hq^{k+1}\} \p_y \hq^{k+1} \xw^{2\ssg} dy\\
	=&\; -\j (\p_y^3 u^{k} \hq^{k+1} + 3\p_y^2 u^{k} \p_y \hq^{k+1} + 2 \p_y u^{k} \p_y^2 \hq^{k+1} ) \p_y \hq^{k+1} \xw^{2\ssg} dy\\
	&\;+ \j u^k |\p_y \hq^{k+1}|^2 \xw^{2 \ssg} dy
	+ \j u^k \p_y \hq^{k+1} \p_y^2 \hq^{k+1}\xw^{2 \ssg} dy\\
	=&\; \j u^k |\p_y \hq^{k+1}|^2 \xw^{2 \ssg} dy
	+ (\p_y u^k |\p_y \hq^{k+1}|^2)_{y=0}
	-\j (\p_y^3 u^{k} \hq^{k+1} + 3\p_y^2 u^{k} \p_y \hq^{k+1} ) \p_y \hq^{k+1} \xw^{2\ssg} dy\\
	&\; + 2 \ssg \j \p_y u^k |\p_y \hq^{k+1}|^2 \xw^{2 \ssg-1}
	+\j u^k \p_y \hq^{k+1} \p_y^2 \hq^{k+1}\xw^{2 \ssg} dy.
	\dal
	\deqq
	Therefore, it holds
	\beqq
	\bal
	- K_{5,4} \geq &\;  \j u^k |\p_y \hq^{k+1}|^2 \xw^{2 \ssg} dy
	+ (\p_y u^k |\p_y \hq^{k+1}|^2)_{y=0}
	+C_{\ssg} (\| \p_y^3 u^k \xw^{\ssg} \|_{L_y^2} \| \hq^{k+1}\|_{L^{\infty}_y} \| \p_y \hq^{k+1} \xw^{\ssg}  \|_{L^2_y}  \\
	&\; +(\| \p_y^2 u^k\|_{L_y^{\infty}}  +\| \p_y u^k\|_{L_y^{\infty}} )\| \p_y \hq^{k+1} \xw^{\ssg}  \|_{L^2_y}^2 +\| u^k\|_{L_y^{\infty}}
	\| \sqrt{u^k} \p_y^2 \hq^{k+1}\|_{L_y^{2}}
	\| \p_y \hq^{k+1} \xw^{\ssg}  \|_{L^2_y}
	)\\
	\geq &\;  \j u^k |\p_y \hq^{k+1}|^2 \xw^{2 \ssg} dy
	+ (\p_y u^k |\p_y \hq^{k+1}|^2)_{y=0}  - C_{\ssg,\lambda_0, \xi_0, \kappa_3} \| \p_y \hq^{k+1} \xw^{\ssg}  \|_{L^2_y}^2.
	\dal
	\deqq
	
	\textbf{Deal with the term $K_{5,5}$.} Integrating by part, we have
	\beqq
	\bal
	K_{5,5}=&\;\j \p_y\{\p_y^2 \hu^{k} {q^{k}} + 2\p_y \hu^{k} \p_y q^{k} +  \hu^{k} \p_y^2 q^{k}\} \p_y \hq^{k+1} \xw^{2\ssg} dy\\
	=&\; -\j \p_y^2 \hu^k ( q^k \p_y^2 \hq^{k+1} \xw^{2\ssg} + 2 \ssg q^k \p_y \hq^{k+1} \xw^{2\ssg-1} +2 \p_y q^k \p_y \hq^{k+1} \xw^{2\ssg}) dy\\
	&\; +3 \j \p_y \hu^k  \p_y^2 q^k \p_y \hq^{k+1} \xw^{2\ssg} dy + \j \hu^k  \p_y^3 q^k \p_y \hq^{k+1} \xw^{2\ssg} dy.
	\dal
	\deqq
	Then we can obtain that
	\beqq
	\bal
	|K_{5,5}|
	\le &\; C_{\lambda_0, \xi_0, \kappa_3,\ssg} \| \p_y^2 \hu^{k} \xw^{\ssg}\|_{L^2_y}
	(\| \sqrt{u^k} \p_y^2 \hq^{k+1} \xw^{\ssg}\|_{L_y^{2}} +
	\| \p_y \hq^{k+1} \xw^{\ssg}  \|_{L^2_y})\\
	&+C_{\lambda_0, \xi_0, \kappa_3 }\| (\hu^{k}, \p_y \hu^{k})\|_{L^2_y}
	\| \p_y \hq^{k+1} \xw^{\ssg}  \|_{L^2_y} \\
	\le &\;\mu \| \sqrt{u^k} \p_y^2 \hq^{k+1} \xw^{\ssg}\|_{L^2_y}^2
	+ C_{\mu, \lambda_0, \xi_0, \kappa_3,\ssg} (\| \p_y \hq^{k+1} \xw^{\ssg}\|_{L^2_y}^2+ \| (\hr^{k}, \hu^{k}, \p_y \hu^{k}, \p_y \hq^{k} \xw^{\ssg})\|_{L^2_y}^2 +\|\hu^{k-1}\|_{L^2_y}^2),
	\dal
	\deqq
	where we have used the following estimate according to the equation \eqref{e:linear-py2}
	\beqq
	\bal
	\| \p_y^2 \hu^{k} \xw^{\ssg}\|_{L^2_y}
	=&\;  \| (\hr^{k}(u^{k-1}+\theta)^2 \p_y{q^{k}} + \rho^{k-1} \hu^{k-1} (u^{k-1}+u^{k-2}+2\theta) \p_y{q^{k}} + \rho^{k-1}(u^{k-2}+\theta)^2 \p_y{\hq^{k}}) \xw^{\ssg}\|_{L^2_y} \\
	\le &\; C_{\lambda_0, \xi_0, \kappa_3} (\| \hr^{k} \|_{L^2_y} +  \| \hu^{k-1} \|_{L^2_y}  + \| \p_y{\hq^{k}} \xw^{\ssg} \|_{L^2_y}).
	\dal
	\deqq
	Combining the estimate from $K_{5,1}$ to $K_{5,5}$ and choosing $\mu$ small enough, we have
	\beq\label{ie:hq-k+1}
	\bal
	& \f12 \f{d}{dx} \j \rho^{k+1} (u^k+\theta)^2 |\p_y \hq^{k+1}|^2 \xw^{2\ssg}dy + \j u^k |\p_y \hq^{k+1}|^2 \xw^{2 \ssg} dy
	+ (\p_y u^k |\p_y \hq^{k+1}|^2)_{y=0}  \\
	\le &\; C_{\lambda_0, \xi_0, \kappa_3,\ssg} (\| \p_y \hq^{k+1} \xw^{\ssg}\|_{L^2_y}^2+ \| (\hr^{k}, \hu^{k}, \p_y \hu^{k}, \p_y \hq^{k} \xw^{\ssg})\|_{L^2_y}^2 +\|\hu^{k-1}\|_{L^2_y}^2).
	\dal
	\deq
	\textbf{Step 2:}
	Multiplying the equation $\eqref{e:linear}_1$ by $
	\hr^{k+1} $ and integrating over $\mathbb{R^+}$, we have
	\beqq
	\bal
	\f12 \f{d}{dx} \j |\hr^{k+1}|^2 dy + \underbrace{\j q^k \p_y (\f{|\hr^{k+1}|^2}{2}) dy }_{K_{6,1}}+ \underbrace{\j \hq^k \p_y \rho_k \cdot \hr_{k+1} dy}_{K_{6,2}}=0.
	\dal
	\deqq
	Integrating by part, we have
	\beqq
	\bal
	|K_{6,1}| & =|\f12 \j \p_y q^k |\hr^{k+1}|^2 dy| \le C_{\lambda_0, \xi_0, \kappa_3}  \|\hr^{k+1} \|_{L^2_y}^2.
	\dal
	\deqq
	It is easy to obtain
	\beqq
	\bal
	|K_{6,2}| \le \| \hq^k\|_{L^{\infty}_y}   \|\hr^{k+1} \|_{L^2_y} \|\p_y \rho^{k} \|_{L^2_y}
	\le C_{\lambda_0, \xi_0, \kappa_3} \| \p_y \hq^k \xw\|_{L^2_{y}}   \|\hr^{k+1} \|_{L^2_y}.
	\dal
	\deqq
	Therefore we have
	\beq \label{ie:hr-k+1}
	\bal
	\f12 \f{d}{dx} \j |\hr^{k+1}|^2 dy \le C_{\lambda_0, \xi_0, \kappa_3} \| (\p_y \hq^k \xw, \hr^{k+1}) \|_{L^2_y}^2.
	\dal
	\deq
	
	Multiplying the equation $\eqref{e:linear}_3$ by $
	\hu^{k+1} $ and integrating over $\mathbb{R^+}$, we have
	\beq\label{ie:hu-k+1}
	\bal
	\f12 \f{d}{dx} \j |\hu^{k+1}|^2 dy  =&\; - \j  (\p_y u^{k-1} \hq^{k+1} +  u^{k-1} \p_y \hq^{k+1} + \p_y \hu^{k} q^{k+1} + \hu^{k-1} \p_y q^{k+1}) \cdot \hu^{k+1} dy\\
	\le &\; C_{\lambda_0, \xi_0, \kappa_3} (\| \hq^{k+1}\|_{L^{\infty}_y} +
	\| \p_y \hq^{k+1}\|_{L^2_{y}} + \| \hu^{k} \|_{L^2_y}+ \|\p_y \hu^{k} \|_{L^2_y}
	) \|\hu^{k+1} \|_{L^2_y}.
	\dal
	\deq
	Applying the $\p_y$ operator to $\eqref{e:linear}_3$,
	multiplying the equation by $
	\p_y \hu^{k+1} $ and integrating over $\mathbb{R^+}$, we have
	\beq\label{ie:pyhu-k+1}
	\bal
	& \f12 \f{d}{dx} \j |\p_y \hu^{k+1}|^2 dy \\
	= &\; - \p_y (\p_y u^{k-1} \hq^{k+1} +  u^{k-1} \p_y \hq^{k+1} + \p_y \hu^{k} q^{k+1} + \hu^{k-1} \p_y q^{k+1}) \p_y \hu^{k+1} dy\\
	= &\; (\p_y u^{k-1} \hq^{k+1} +  u^{k-1} \p_y \hq^{k+1} + \p_y \hu^{k} q^{k+1} + \hu^{k-1} \p_y q^{k+1}) \p_y^2 \hu^{k+1} dy\\
	\le &\; C_{\lambda_0, \xi_0, \kappa_3} (\| \hq^{k+1}\|_{L^{\infty}_y} +
	\| \p_y \hq^{k+1}\|_{L^2_{y}} + \| \hu^{k} \|_{L^2_y}+ \|\p_y \hu^{k} \|_{L^2_y}
	)  (\| \hr^{k+1} \|_{L^2_y} +  \| \hu^{k} \|_{L^2_y}  + \| \p_y{\hq^{k+1}} \xw^{\ssg} \|_{L^2_y})\\
	\le &\;  C_{\lambda_0, \xi_0, \kappa_3} (\| \p_y \hq^{k+1} \xw, \hr^{k+1}\|_{L^2_y}^2+ \| (\hu^{k}, \p_y \hu^{k})\|_{L^2_y}^2).
	\dal
	\deq
	Finally, we have
	\beqq
	\p_x \hu^k= -(\p_y u^{k-2} \hq^{k} +  u^{k-2} \p_y \hq^{k} + \p_y \hu^{k-1} q^{k} + \hu^{k-2} \p_y q^{k}).
	\deqq
	Multiplying the equation by $
	\hu^{k} $ and integrating over $\mathbb{R^+}$, we have
	\beq\label{ie:hu-k}
	\bal
	\f12 \f{d}{dx} \j |\hu^{k}|^2 dy & \;= -\j (\p_y u^{k-2} \hq^{k} +  u^{k-2} \p_y \hq^{k} + \p_y \hu^{k-1} q^{k} + \hu^{k-2} \p_y q^{k}) \hu^{k} dy  \\
	&\; = -\j (\p_y u^{k-2} \hq^{k} +  u^{k-2} \p_y \hq^{k}
	+ \hu^{k-2} \p_y q^{k}) \hu^{k} dy\\
	&\quad + \j  \hu^{k-1} \p_y \{ \p_y q^{k} \hu^k +  q^{k} \p_y \hu^k \} dy\\
	&\; \le  C_{\lambda_0, \xi_0, \kappa_3} (\| (\hu^{k}, \p_y \hu^{k}, \p_y \hq^{k} \xw)\|_{L^2_y}^2 +\| \hu^{k-1}\|_{L^2_y}^2).
	\dal
	\deq
	Combining the estimates \eqref{ie:hr-k+1}-\eqref{ie:hu-k}, we have
	\beqq
	\bal
	&\;\f{d}{dx} \j |\hr^{k+1}|^2+|\hu^{k+1}|^2+|\p_y\hu^{k+1}|^2+|\hu^{k}|^2 dy   \\
	\le &\; C_{\lambda_0, \xi_0, \kappa_3} (\| (\hr^{k+1},  \hu^{k+1}, \p_y \hq^{k+1} \xw)\|_{L^2_y}^2+ \| (\hu^{k}, \p_y \hu^{k}, \p_y \hq^{k} \xw)\|_{L^2_y}^2 +\| \hu^{k-1}\|_{L^2_y}^2),
	\dal
	\deqq
	which, together with \eqref{ie:hq-k+1}, yields that
	\beqq
	\bal
	& \f{d}{dx} \j \rho^{k+1} (u^k+\theta)^2 |\p_y \hq^{k+1}|^2 \xw^{2\ssg}+ |\hr^{k+1}|^2+|\hu^{k+1}|^2+|\p_y\hu^{k+1}|^2+|\hu^{k}|^2 dy \\
	&\;+ \j u^k |\p_y \hq^{k+1}|^2 \xw^{2 \ssg} dy
	+ C (\p_y u^k |\p_y \hq^{k+1}|^2)_{y=0}  \\
	\le &\; C_{\lambda_0, \xi_0, \kappa_3,\ssg}(\| (\hr^{k+1},  \hu^{k+1},  \p_y \hq^{k+1} \xw^{\ssg})\|_{L^2_y}^2+ \| (\hr^{k}, \hu^{k}, \p_y \hu^{k}, \p_y \hq^{k} \xw^{\ssg})\|_{L^2_y}^2 +\|\hu^{k-1}\|_{L^2_y}^2).
	\dal
	\deqq
	\textbf{Step 3:}
	Define
	\beqq
	\phi^{k+1}(x)\overset{def}{=} \j (\rho^{k+1} (u^k+\theta)^2 |\p_y \hq^{k+1}|^2 \xw^{2\ssg}+ |\hr^{k+1}|^2+|\hu^{k+1}|^2+|\p_y\hu^{k+1}|^2+|\hu^{k}|^2)dy,
	\deqq
	then we have
	\beqq
	\f{d}{dx} \phi^{k+1}(x) \le C_{\lambda_0,\xi_0,\kappa_3,\ssg}(\phi^{k+1}(x) + \phi^{k}(x)).
	\deqq
	Choosing $x^*>0$ such that $e^{C_{\lambda_0,\xi_0,\kappa_3,\ssg}
		x} \le \f32$ for all $0\le x \le x^*$, we can deduce that for all $K \geq 1$,
	\beqq
	\sum_{k=1}^K \sup_{0\le x \le x^*} \phi^{k+1}(x) < +\infty.
	\deqq
	Recall that $\theta >0$ and $\rho^{k+1} \geq \kappa_3 > 0$  and the uniform estimate of $\x^{k+1}$, we conclude that there exists $\rho_{\theta}, u_{\theta}, \p_y q_{\theta} \in L^{\infty}(0,L_a; L^2_y)$ such that
	\beqq
	\left\{
	\bal
	\rho^{k+1} &\to \rho_{\theta}~ &&\text{in}~ L^{\infty}(0,L_a; L^2_y);\\
	u^{k+1} &\to u_{\theta} ~ &&\text{in}~ L^{\infty}(0,L_a; L^2_y);\\
	\p_y q^{k+1} &\to \p_y q_{\theta} ~&&\text{in} ~L^{\infty}(0,L_a; L^2_y).
	\dal\right.
	\deqq
	It can be checked that $(\rho_{\theta}, u_{\theta},  q_{\theta})$ solves the following equations \eqref{e:theta} for the fixed constant $\theta >0$.
	\beq \label{e:theta}
	\left\{
	\bal
	&\p_x \rho_{\theta} + q_{\theta} \p_y \rho_{\theta} =0,\\
	& \rho_{\theta} (u_{\theta}+\theta)^2 \p_y q_{\theta} +\p_y^2 u_{\theta}  = r_0,\\
	&\p_x u_{\theta} + \p_y\{u_{\theta} q_{\theta}\} =0,\\
	&(\rho_{\theta},u_{\theta})|_{y=0}=(0,0),\\
	&\lim_{y \rightarrow+\infty}(\rho_{\theta},u_{\theta},\p_y q_{\theta})=(\varrho_{\infty}, u_{\infty},0),\\
	&(\rho_{\theta},u_{\theta},q_{\theta})|_{x=0}=(\rho_0,u_0,q_0).\\
	\dal\right.
	\deq
	Define the energy norm and dissipation norm for the solution ($\rho_{\theta}, q_{\theta}, u_{\theta}$) as following,
	\beqq
	\bal
	\x_{\theta}(x) \overset{def}{=} & \sum_{|\alpha|\le m}\|\p^\alpha \br_{\theta} \|_{L_y^2}^2
	+\sum_{|\alpha|\le m}\|\sqrt{\rho_{\theta}} (u_{\theta}+\theta) \p^{\al}\p_y q_{\theta} \xw^{\ssg}\|_{L_y^2}^2
	+\|\bu_{\theta} \|_{L_y^2}^2+\|\p_y \bu_{\theta} \|_{L_y^2}^2;\\
	\y_{\theta}(x) \overset{def}{=}
	&\sum_{|\alpha|\le m}\|\sqrt{u_{\theta}} \p^\alpha \p_y^2 q_{\theta} \xw^{\ssg}\|_{L_y^2}^2.
	\dal
	\deqq
	Then, by virtue of semi-continuity of norm, we deduce from the uniform bound \eqref{uniform-s} and the lower bound \eqref{prop-re-urhok} that $(\rho_{\theta}, u_{\theta},  q_{\theta})$ satisfies the estimates
	\beq \label{estimate-xtheta}
	\sup_{0 \le x \le L_a}  \x_{\theta}(x) + \int_{0}^{L_a} \y_{\theta} (x) dx \le 2 C_{\rho_0, u_0},
	\deq
	and
	\beq \label{estimate-rhotheta}
	u_{\theta}(x,y)>0, \quad \rho_{\theta}(x,y) \geq \kappa_3, \quad \p_y u_{\theta}(x,0) \ge \f{ u_0'(0)}{2}, \forall (x,y) \in [0, L_a] \times (0,+\infty).
	\deq
Similar to the analysis of a priori estimate of the steady Prandtl equations \eqref{s-Prandtl} in Section \ref{sec:priori-steady}, we can establish a priori estimate and derive the uniform estimate independent of $\theta$ for the equation \eqref{e:theta}. Here we omit the proof for the sake of simplicity.
Taking $\theta \to 0$, we conclude that there exists solution $(\rho, u, q)$
satisfies the following system
	\beqq
	\left\{
	\bal
	&\p_x \rho + q \p_y \rho =0,\\
	& \rho u^2 \p_y q + \p_y^2 u =0,\\
	&\p_x u + \p_y\{u q\} =0,\\
	&(q,u)|_{y=0}=(0,0),\\
	&\lim_{y \rightarrow+\infty}(\rho,u,\p_y q)=(\varrho_{\infty}, u_{\infty},0),\\
	&(\rho,u,q)|_{x=0}=(\rho_0, u_0,q_0).\\
	\dal\right.
	\deqq
	Let us define $v\overset{def}{=}uq$. Therefore $(\rho,u,v)$ satisfies the following system:
	\beqq
	\left\{
	\bal
	& \p_x \rho + \f{v}{u} \p_y \rho =0,\\
	& \rho u^2 \p_y (\f{v}{u})  + \p_y^2 u =0,\\
	&\p_x u + \p_y v =0,\\
	&(v,u)|_{y=0}=(0,0),\\
	&\lim_{y \rightarrow+\infty}(\rho,u)=(\varrho_{\infty}, u_{\infty}),\\
	&(\rho,u)|_{x=0}=(\rho_0,u_0),
	\dal\right.
	\deqq
	which is equivalent to the system \eqref{s-Prandtl}.
	At the same time, due to the uniform estimate \eqref{estimate-xtheta}, the lower bound \eqref{estimate-rhotheta} and semi-continuity of norm, we can obtain
	\beqq
	\sup_{0 \le x \le L_a}  \x(x) + \int_{0}^{L_a} \y (x) dx \le 2 C_{\rho_0, u_0},
	\deqq	
	and
	\beqq
	u(x,y)>0, \quad \rho(x,y) \ge \kappa_3, \quad \p_y u(x,0) \ge \f{ u_0'(0)}{2}, \quad \forall (x,y) \in [0, L_a] \times (0,+\infty),
	\deqq
	which yields the estimate \eqref{estimate-s-rho} and
	\eqref{estimate-uX} in Theorem \ref{main-result-steady}.
	Therefore, we complete the proof of local-in-time existence  in Theorem \ref{main-result-steady}.
	
	\subsection{Uniqueness of solution of the original system}\label{sec:unique}
	In this subsection, we will show the uniqueness of solution to the original system \eqref{s-Prandtl}.
	Let $(u_1, v_1, \rho_1)$ and $(u_2, v_2, \rho_2)$ be two solutions in the 
existence life-span $[0, L_a]$, constructed in the previous subsection, with the same initial data $(u_{0},  \rho_{0})$.
	Let us denote $(\tilde{u}, \tilde{v}, \tilde{\rho}) \overset{def}{=} (u_1, v_1, \rho_1) - (u_2, v_2, \rho_2)$,
	$\tilde{q} \overset{def}{=} \frac{\tilde{v}}{u_2}$. They satisfy
	\beq \label{equ-tilde}
	\left\{
	\begin{aligned}
		&\tilde{u} \p_x \rho_1 +u_2 \p_x \tilde{\rho} + \tilde{v} \p_y \rho_1 + v_2 \p_y \tilde{\rho}=0;\\
		& \tilde{\rho} (u_2 \p_x u_2 + v_2 \p_y u_2 )+\rho_1 (\tilde{u} \p_x u_1 + v_1 \p_y \tilde{u} -u_2^2 \p_y \tilde{q}) -\p_y^2 \tilde{u}=0,
	\end{aligned}\right.
	\deq	
	with the zero boundary condition and initial data
	\beqq
	(\tu, \tv)|_{y=0}=(0,0), \quad \lim_{y \rightarrow+\infty}(\tr,\tu)=(0,0),
	\quad (\tr, \tu)|_{x=0}=(0,0).
	\deqq
	Taking $\p_x$-operator to the above equation $\eqref{equ-tilde}_2$ and
	using the divergence-free condition, we have
	\beq\label{equation-d}
	\p_x \{\rho_1 u_2^2 \p_y \tq\}-\p_y^3 \tv
	=\p_x\{\tilde{\rho} (u_2 \p_x u_2 + v_2 \p_y u_2 )+\rho_1 (\tilde{u} \p_x u_1 + v_1 \p_y \tilde{u})\}.
	\deq
	
	Next, we will establish the following estimate, which is useful to prove the uniqueness of the system \eqref{s-Prandtl}.
	\begin{prop}\label{lemma-uniq}
		Let $(u_1, v_1,\rho_1)$ and $(u_2, v_2,\rho_2)$ be two solutions of the equation \eqref{s-Prandtl}
		with the same initial data $(u_0, \rho_0)$
		satisfying the estimate	\eqref{estimate-s-rho}  and \eqref{estimate-uX}.
		Then, the quantity $(\tu, \tv, \tr, \tq)$ satisfies
		\beq\label{stability}
		\bal
		&\; \|( \sqrt{\rho_1}u_2\p_y \tq \xw^{\ssg}, \tr)(x,y) \|_{L^2_y}^2
		+\int_0^x \|(\sqrt{u_2} \p_y^2 \tq \xw^{\ssg})(\tau,y)\|_{L^2_y}^2 d\tau \\
		\le &\; C_{\lambda_0,\xi_0,\kappa_3,\ssg,L_a} \int_0^x  \|(\sqrt{\rho_1}u_2\p_y \tq \xw^{\ssg},  \tr )(\tau,y) \|_{L^2_y}^2 d\tau,
		\dal
		\deq
		for any $x\in (0, L_a]$.
	\end{prop}
	\begin{proof}[\textbf{Proof of Uniqueness.}]
		Applying the Gronwall's lemma to the estimate \eqref{stability},
		we obtain
		\beqq
		(\sqrt{\rho_1}u_2\p_y \tq \xw^{\ssg},\tr)(x, y)\equiv 0
		.
		\deqq
		 Then we have $\rho_1(x,y)\equiv\rho_2(x,y)$.
		Due to the fact $\rho_1>0$ and $u_2>0$ in the fluid domain, then we have
		\beqq
		v_2-v_1=u_2 w,
		\deqq
		for some function $w \overset{def}{=}w(x)$.
		Using the assumption $u_2>0$  in the fluid domain
		and boundary condition $\tv|_{y=0}=0$,
		we know that $w(x,y) \equiv 0$ and hence $v_1(x,y) \equiv v_2(x,y)$.
		Using the divergence-free condition, it holds
		\beqq
		u_2(x, y)-u_1(x, y)
		=\int_0^x \partial_\tau\{u_2(\tau, y)-u_1(\tau, y)\}d\tau
		=-\int_0^x \partial_y \{v_2(\tau, y)-v_1(\tau, y)\}d\tau
		=0.
		\deqq
		Then, we have $u_2(x, y)\equiv u_1(x, y)$.
		This proves the uniqueness of solution in Theorem \ref{main-result-steady}.
	\end{proof}
	In the rest of this subsection, we will give the proof
	of Proposition \ref{lemma-uniq} as follows.
	\begin{proof}[\textbf{Proof of Proposition \ref{lemma-uniq}}]
		Multiplying the equation \eqref{equation-d}
		by $\p_y \tq \xw^{2 \ssg}$, integrating over $[0, x]\times[0,+\infty)$, $x\in (0, L_a]$
		and integrating by part, we have
		\beqq
		\bal
		&\f12\j \!(\rho_1 u_2^2 |\p_y \tq|^2 \xw^{2\ssg}) (x, y)dy
		+\!\int_0^x \!\!\j \! \!u_2 |\p_y^2 \tq|^2 \xw^{2\ssg} dy d\tau
		+\!\int_0^x \!\! (\p_y u_2 |\p_y \tq|^2)_{y=0} d\tau\\
		=&-\f12 \int_0^x \j \p_x (\rho_1 u_2^2 ) |\p_y \tq|^2 \xw^{2\ssg} dy d\tau
		-2\ssg \int_0^x \j  u_2\p_y^2 \tq  \cdot \p_y \tq  \xw^{2\ssg-1} dy d\tau\\
		&+2\int_0^x \j  \p_y^2 u_2 |\p_y \tq|^2 \xw^{2\ssg} dy d\tau
		-2 \ssg \int_0^x \j  \p_y u_2 |\p_y \tq|^2 \xw^{2\ssg-1} dy d\tau \\
		&+\int_0^x \j \p_y^3 u_2 \tq \cdot \p_y \tq  \xw^{2\ssg} dy d\tau \int_0^x \j \p_x\{\tilde{\rho} (u_2 \p_x u_2 + v_2 \p_y u_2 )\} \p_y \tq \xw^{2\ssg} dy  d\tau\\
		&+ \int_0^x \j \p_x\{\rho_1 (\tilde{u} \p_x u_1 + v_1 \p_y \tilde{u})\} \p_y \tq \xw^{2\ssg}dy  d\tau\\
		\overset{def}{=}&\sum_{i=1}^{7}II_{i}.
		\dal
		\deqq
		It's easy to obtain that
		\beqq
		\bal
		|II_1| &\le C(\| \f{\p_x u_2}{u_2} \|_{L^{\infty}} \| \sqrt{\rho_1} u_2 \p_y \tq  \xw^{\ssg} \|_{L^{2}}^2
		+  \| \p_x \rho_1 \|_{L^{\infty}} \| u_2 \p_y \tq  \xw^{\ssg} \|_{L^{2}}^2) ;\\
		|II_2| &\le C_{\ssg} \| u_2 \p_y^2 \tq  \xw^{\ssg} \|_{L^{2}} \| \p_y \tq  \xw^{\ssg-1} \|_{L^{2}};\\
		|II_3| &\le C \| \p_y^2 u_2 \|_{L^{\infty}} \| \p_y \tq  \xw^{\ssg} \|_{L^{2}}^2;\\
		|II_4| &\le C_{\ssg} \| \p_y u_2 \|_{L^{\infty}} \| \p_y \tq  \xw^{\ssg} \|_{L^{2}}^2;\\
		|II_5| &\le C \| \p_y^3 u_2 \xw\|_{L_x^{\infty} L_y^2}  \| \tq  \xw^{\ssg-1}\|_{L_x^2 L_y^{\infty} } \| \p_y \tq  \xw^{\ssg} \|_{L^{2}} \le C \| \p_y^3 u_2 \xw\|_{L_x^{\infty} L_y^2}  \| \p_y \tq  \xw^{\ssg} \|_{L^{2}}^2.\\	
		\dal
		\deqq
		Then we deal with the term $II_{6}$.
		\beqq
		\bal
		II_6 =& \int_0^x \j \p_x \{\tr u_2^2 \p_y(\f{v_2}{u_2})\}  \p_y \tq  \xw^{2\ssg} dy  d\tau\\
		=& \int_0^x \j \left(u_2^2 \p_x \tr  \p_y(\f{v_2}{u_2}) +  2 \tr \p_x u_2 \p_y(\f{v_2}{u_2}) u_2  +  \tr u_2^2 \p_{xy}(\f{v_2}{u_2})  \right)
		\p_y \tq  \xw^{2\ssg} dy d\tau\\
		\overset{def}{=}&II_{6,1}+II_{6,2}+II_{6,3}.
		\dal
		\deqq
		Using the equation $\eqref{equ-tilde}_1$ and integrating by part, we have
		\beqq
		\bal
		II_{6,1}=&\int_0^x \j (\tilde{u} \p_x \rho_1  + \tilde{v} \p_y \rho_1 + v_2 \p_y \tilde{\rho})
		\p_y(\f{v_2}{u_2}) u_2 \p_y \tq  \xw^{2\ssg} dy d\tau\\
		=&\int_0^x \j (\tilde{u} \p_x \rho_1  + \tilde{v} \p_y \rho_1 )
		\p_y(\f{v_2}{u_2}) u_2 \p_y \tq   \xw^{2\ssg} dy d\tau
		- \int_0^x \j    \tilde{\rho} \p_y \left\{v_2 \p_y(\f{v_2}{u_2}) u_2 \p_y \tq  \xw^{2\ssg}\right\} dy d\tau \\
		\overset{def}{=}&II_{6,1,1}+II_{6,1,2}+II_{6,1,3}.
		\dal
		\deqq
		We can check that
		\beqq
		\bal
		|II_{6,1,1}|
		&\le \| \p_y(\f{v_2}{u_2})  \xw^{\ssg} \|_{L^{\infty}} \| \p_x \rho_1 \|_{L^{2}} \| \tu \xw^{\ssg-1} \|_{L^{\infty}}  \| u_2 \p_y \tq \xw \|_{L^{2}}\\
		&\le C_{L_a} \| \p_y(\f{v_2}{u_2}) \xw^{\ssg} \|_{L^{\infty}} \| \p_x \rho_1 \|_{L^{2}} \| \p_y^2 \tv  \xw^{\ssg}  \|_{L^{2}}  \| u_2 \p_y \tq \xw \|_{L^{2}};\\
		|II_{6,1,2}| 	&\le \| \p_y(\f{v_2}{u_2})  \xw^{\ssg} \|_{L^{\infty}} \| \p_y \rho_1 \|_{L^{\infty}} \| \tv  \|_{L^{2}}  \| u_2 \p_y \tq \xw^{\ssg} \|_{L^{2}}\\
		&\le C \| \p_y(\f{v_2}{u_2}) \xw^{\ssg} \|_{L^{\infty}} \| \p_y \rho_1 \|_{L^{\infty}} \| \p_y^2 \tv \xw^{2} \|_{L^{2}}  \| u_2 \p_y \tq  \xw^{\ssg}\|_{L^{2}}.
		\dal
		\deqq
		Then we deal with the term $II_{6,1,3}$.
		\beqq
		\bal
		|II_{6,1,3}| \le &\;\| \p_y(\f{v_2}{u_2}) \xw^{\ssg} \|_{L^{\infty}} (1+\|  u_2\|_{L^{\infty}}) \| v_2\|_{L^{\infty}} \| \tr \|_{L^{2}}  \| \sqrt{u_2} \p_y^2 \tq \xw^{\ssg} \|_{L^{2}}\\
		&\;+\| \tr \|_{L^{2}}  \| \p_y \tq \xw^{\ssg} \|_{L^{2}}
		\| \p_y(\f{v_2}{u_2})\xw^{\ssg} \|_{L^{\infty}} \| \p_y u_2\|_{L^{\infty}} \| v_2\|_{L^{\infty}}\\
		&\;+\| \tr \|_{L^{2}}  \| \p_y \tq \xw^{\ssg}\|_{L^{2}}
		(\| \p_y(\f{v_2}{u_2}) \xw^{\ssg} \|_{L^{\infty}} \|  u_2\|_{L^{\infty}} \| \p_y v_2\|_{L^{\infty}}
		+\| \p_y^2(\f{v_2}{u_2})\xw^{\ssg} \|_{L^{\infty}} \|  u_2\|_{L^{\infty}} \| v_2\|_{L^{\infty}}).
		\dal
		\deqq
		Thus we finish the estimate of the term $II_{6,1}$.
		Next we deal with the term $II_{6,2}$ and $II_{6,3}$.
		\beqq
		\bal
		|II_{6,2}| &\le C \| \p_{y}(\f{v_2}{u_2})\xw^{\ssg} \|_{L^{\infty}} \|  \p_x u_2\|_{L^{\infty}}
		\| \tr \|_{L^{2}}  \| u_2 \p_y \tq \xw^{\ssg}\|_{L^{2}};\\
		|II_{6,3}| &\le \| \p_{xy}(\f{v_2}{u_2})\xw^{\ssg} \|_{L^{\infty}} \|  u_2\|_{L^{\infty}}
		\| \tr \|_{L^{2}}  \|u_2 \p_y \tq \xw^{\ssg}\|_{L^{2}}.
		\dal
		\deqq
		Combining the estimate of $II_{6}$, we have
		\beqq
		\bal
		|II_{6}|	
		\le &\; C_{L_a} \| \p_y^2 \tv  \xw^{\ssg}  \|_{L^{2}}  \| u_2 \p_y \tq \xw \|_{L^{2}}
		(
		\| \p_y(\f{v_2}{u_2}) \xw^{\ssg} \|_{L^{\infty}} \| \p_x \rho_1 \|_{L^{2}}+
		\| \p_y(\f{v_2}{u_2}) \xw^{\ssg} \|_{L^{\infty}} \| \p_y \rho_1 \|_{L^{\infty}}
		)\\
		&\; + C\| \tr \|_{L^{2}}  \| \sqrt{u_2} \p_y^2 \tq \xw^{\ssg} \|_{L^{2}} \| \p_y(\f{v_2}{u_2}) \xw^{\ssg} \|_{L^{\infty}} (1+\|  u_2\|_{L^{\infty}}) \| v_2\|_{L^{\infty}} \\
		&\;+C \| \tr \|_{L^{2}}  \| \p_y \tq \xw^{\ssg} \|_{L^{2}}
		\| \p_y(\f{v_2}{u_2})\xw^{\ssg} \|_{L^{\infty}} \| (\p_y u_2, \p_y v_2, u_2, v_2)\|_{L^{\infty}}^2\\
		&\; + C \| \tr \|_{L^{2}}  \|u_2 \p_y \tq \xw^{\ssg}\|_{L^{2}}
		(
		\p_{y}(\f{v_2}{u_2})\xw^{\ssg} \|_{L^{\infty}} \|  \p_x u_2\|_{L^{\infty}}+\| \p_{xy}(\f{v_2}{u_2})\xw^{\ssg} \|_{L^{\infty}} \|  u_2\|_{L^{\infty}}
		).
		\dal
		\deqq
		Similar to the estimate of $II_{6,1}$,
		\beqq
		\bal
		|II_{7}| \le &\; \| \p_y \tq \xw^{\ssg}\|_{L^{2}}
		(\| \tu \xw^{\ssg-1} \|_{L^{\infty}} \|  \p_x u_1 \xw\|_{L^{2}} \| \p_x \rho_1\|_{L^{\infty}} +\| \p_y \tu \xw^{\ssg} \|_{L_x^{\infty} L_y^2} \|  v_1\|_{L_x^{2} L_y^{\infty} } \| \p_x \rho_1\|_{L^{\infty}} )\\
		&\;+\| \p_y \tq \xw^{\ssg} \|_{L^{2}} \| \rho_1\|_{L^{\infty}}
		(\|\p_x \tu\xw^{\ssg-1} \|_{L_x^2L_y^{\infty}} \|  \p_x u_1 \xw\|_{L_x^{\infty}L_y^2}
		+\| \tu \xw^{\ssg-1} \|_{L^{2}} \|  \p_x u_1 \xw \|_{L^{2}}\\
		&\;+\|  \p_y \tu \xw^{\ssg}\|_{L_x^{\infty} L_y^2} \|  \p_x v_1\|_{L_x^2 L_y^{\infty} }
		+\|  \p_x \p_y \tu \xw^{\ssg} \|_{L^2} \|  v_1\|_{L^{\infty} } )\\
		\le &\; C_{L_a} \| \p_y \tq \xw^{\ssg}\|_{L^{2}}
		\| \p_y^2 \tv \xw^{\ssg}\|_{L^{2}}(\|  \p_x u_1 \xw\|_{L^{2}} \| \p_x \rho_1\|_{L^{\infty}}+\|  v_1\|_{L_x^{2} L_y^{\infty} } \| \p_x \rho_1\|_{L^{\infty}}\\
		&\;+
		\| \rho_1\|_{L^{\infty}}\|  \p_x u_1 \xw\|_{L^{2}}
		+\| \rho_1\|_{L^{\infty}} \|  \p_x v_1\|_{L_x^2 L_y^{\infty}}
		+\| \rho_1\|_{L^{\infty}} \|   v_1\|_{ L^{\infty}} ).
		\dal	
		\deqq
		Using the equation $\eqref{equ-tilde}_1$,
		we have
		\beqq
		\f{\tilde{u}}{u_2} \p_x \rho_1 + \p_x \tilde{\rho} + \f{\tilde{v}}{u_2} \p_y \rho_1 + \f{v_2}{u_2}  \p_y \tilde{\rho}=0.
		\deqq
		Multiplying the above equation by $\tr$, integrating over $[0, x]\times[0,+\infty)$ and integrating by part, we have
		\beqq
		\bal
		\f12\j \! \tr^2 (x, y)dy&=
		-\f12 \!\int_0^x \!\!\j \p_y(\f{v_2}{u_2}) |\tr|^2 dy d\tau
		-\!\int_0^x \!\!\j \! \!  \f{\tu}{u_2} \p_x \rho_1 \tr  dy d\tau
		-\!\int_0^x \!\!\j \! \!  \f{\tv}{u_2} \p_y \rho_1 \tr  dy d\tau\\
		&\overset{def}{=}II_{8}+II_{9}+II_{10}.
		\dal
		\deqq
		It is easy to check that
		\beqq
		\bal
		|II_{8}| &\le C \|  \p_y(\f{v_2}{u_2})\|_{L^{\infty}}
		\| \tr \|_{L^2}^2;\\
		|II_{10}| &\le  \|  \p_y \rho_1 \|_{L_x^{\infty} L_y^2}
		\| \tq  \|_{L^2_x L^{\infty}_y}
		\| \tr \|_{L^2} \le C \|  \p_y \rho_1 \|_{L_x^{\infty} L_y^2}
		\| \p_y \tq \xw \|_{L^2}
		\| \tr \|_{L^2}.\\
		\dal
		\deqq
		Let us deal with the term $II_{9}$.
		\beqq
		\bal
		|II_{9}| &\le  \|  \f{\tu}{u_2} \p_x \rho_1  \|_{L^2}
		\| \tr \|_{L^2}\\
		& \le \|  \f{\tu}{u_2} \p_x \rho_1 \chi_{\sde} \|_{L^2}
		+\|  \f{\tu}{u_2} \p_x \rho_1 (1-\chi_{\sde}) \|_{L^2}\| \tr \|_{L^2}\\
		& \le C_{\lambda_0,\xi_0} (\|  \f{\tu}{y} \|_{L_x^{\infty} L_y^2}
		\|\p_x \rho_1  \|_{L_x^{2} L_y^{\infty}}
		+ \|  {\tu}  \|_{L_x^{\infty} L_y^{\infty}}
		\|\p_x \rho_1  \|_{L_x^{2} L_y^{2}})\| \tr \|_{L^2}\\
		&\le   C_{\lambda_0,\xi_0,L_a}(\|\p_x \rho_1  \|_{L_x^{2} L_y^{\infty}}+
		\|\p_x \rho_1  \|_{L^2})
		\|  \p_y^2 \tv \xw\|_{L^2}
		\| \tr \|_{L^2}.\\
		\dal
		\deqq
		Recall the definition $\tq\overset{def}{=}\f{\tv}{u_2}$,
		then it holds true that
		\beqq
		\bal
		&\;\|\p_y^2 \tv \xw^{\ssg}\|_{L^2}\\
		\le
		&\; \| \p_y u_2 \|_{L^{\infty}} \|\p_y \tq \xw^{\ssg}\|_{L^2}
		+ \| \p_y^2 u_2 \xw\|_{L_x^{\infty} L_y^2} \| \tq \xw^{\ssg-1} \|_{L^2_x L^{\infty}_y}
		+(1+ \|  u_2\|_{L^{\infty}} )
		\|\sqrt{u_2} \p_y^2 \tq \xw^{\ssg}\|_{L^2}\\
		\le &\;  (\| \p_y u_2\|_{L^{\infty}}
		+ \| \p_y^2 u_2 \xw\|_{L_x^{\infty} L_y^2})\|\p_y \tq \xw^{\ssg}\|_{L^2}
		+(1+ \|  u_2\|_{L^{\infty}} )
		\|\sqrt{u_2} \p_y^2 \tq \xw^{\ssg}\|_{L^2},
		\dal
		\deqq
		and similar to the estimate in Lemma \eqref{Lemma:help}, we can obtain for the constant $\xi_0$,
		\beqq
		\|\p_y \tq \xw^{\ssg} \|_{L^2}
		\le C_{\xi_0,\kappa_3}\|\sqrt{\rho_1}u_2 \p_y \tq \xw^{\ssg}\|_{L^2}
		+\sde^{\f12} C_{\lambda_0} \|\sqrt{u_2} \p_y^2 \tq \xw^{\ssg} \|_{L^2}.
		\deqq
		By virtue of the uniform estimate \eqref{estimate-uX} in Theorem \ref{main-result-steady} for the two solutions $(u_1, v_1,\rho_1)$ and $(u_2, v_2,\rho_2)$, combining the estimates of $II_{1}$  to $II_{10}$, using the smallness of $\sde$, we have
		\beqq
		\bal
		&\;\f12\j \! \tr^2 (x, y)dy+ \f12\j \!(\rho_1 u_2^2 |\p_y \tq |^2 \xw^{2 \ssg}) (x, y)dy
		+\!\int_0^x \!\!\j \! \!u_2 |\p_y^2 \tq |^2 \xw^{2\ssg}  dy d\tau
		+\!\int_0^x \!\! (\p_y u_2 |\p_y \tq|^2)_{y=0} d\tau\\
		\le &\; C_{\lambda_0,\xi_0,\kappa_3,L_a}\!\int_0^x \!\!\j (\tr^2, \rho_1 u_2^2 |\p_y \tq| \xw^{2\ssg}, |\p_y \tq|^2 \xw^{2\ssg} )dy d\tau  + \f14 \!\int_0^x \!\!\j u_2 |\p_y^2 \tq|^2 \xw^{2\ssg} dy d\tau\\
		\le &\;  C_{\lambda_0,\xi_0,\kappa_3,L_a} \!\int_0^x \!\!\j (\tr^2, \rho_1 u_2^2 |\p_y \tq|^2 \xw^{2\ssg}) dy d\tau +  \f12 \!\int_0^x \!\!\j u_2 |\p_y^2 \tq|^2 \xw^{2 \ssg}dy d\tau.
		\dal
		\deqq
		Therefore we have finished the proof of Proposition \ref{lemma-uniq}.
	\end{proof}

	\section*{Data availability}
	No data was used for the research described in the article.
	
	\section*{Acknowledgments}
	This research was partially supported by
	National Key Research and Development Program of China(2021YFA1002100, 2020YFA0712500),
	Guangdong Basic and Applied Basic Research Foundation(2022A1515011798, 2021B1515310003),
	National Natural Science Foundation of China(12126609),
	Guangzhou Science and technology project(2024A04J6410).
	
	\begin{appendices}
		
		\section{Some useful Sobolev inequalities and maximum principles}\label{appendix-a}
		First, we will state two elementary inequalities without proof (i.e., Lemma \ref{appendix-hardy} and Lemma  \ref{appendix-Sobolev} below), refer to Lemmas
 B.1 and B.2 in  \cite{Masmoudi}.
		Let us first state the Hardy type inequalities and the Sobolev type inequality.
		\begin{lemm}[Hardy Type Inequalities]
			\label{appendix-hardy}
			Let $f:\mathbb{T}\times\mathbb{R}^+ \rightarrow \mathbb{R}$. \\
			{\rm(i)}
			If $\lambda>-\frac{1}{2}$ and $\underset{y\rightarrow+\infty}{\lim}f(x, y)=0$, then
			\beq\label{Hardy1}
			\|f\xw^{\lambda}\|_{L^2(\mathbb{T}\times\mathbb{R}^+)}
			\le \frac{2}{2\lambda+1}\|\p_y f\xw^{\lambda+1}\|_{L^2(\mathbb{T}\times\mathbb{R}^+)};
			\deq
			{\rm(ii)}If $\lambda<-\frac{1}{2}$, then
			\beq\label{Hardy2}
			\|f\xw^{\lambda}\|_{L^2(\mathbb{T}\times\mathbb{R}^+)}
			\le \sqrt{-\frac{1}{2\lambda+1}}\|f|_{y=0}\|_{L^2(\mathbb{T})}
			-\frac{2}{2\lambda+1}\|\p_y f\xw^{\lambda+1}\|_{L^2(\mathbb{T}\times\mathbb{R}^+)}.
			\deq
		\end{lemm}
		\begin{lemm}[Sobolev Type Inequality]
			\label{appendix-Sobolev}
			Let $f: \mathbb{T}\times \mathbb{R}^+\rightarrow \mathbb{R}$.
			Then there exists a universal constant $C>0$ such that
			\beq\label{inf}
			\|f\|_{L^\infty(\mathbb{T}\times \mathbb{R}^+)}
			\le C(\|f\|_{L^2(\mathbb{T}\times \mathbb{R}^+)}
			+\|\p_x f\|_{L^2(\mathbb{T}\times \mathbb{R}^+)}
			+\|\p_y^2 f\|_{L^2(\mathbb{T}\times \mathbb{R}^+)}).
			\deq
		\end{lemm}
		Then we will introduce some inequalities that will be used frequently in Section \ref{prior-estimate}.
		\begin{lemm}
			For proper functions $f, g$, the following inequalities holds\\
			{\rm(i)} $\underset{y\rightarrow+\infty}{\lim}fg(x, y)=0$, we have
			\beq\label{trace}
			\int_{\mathbb{T}}fg|_{y=0}dx
			\le \|\p_y f\|_{L^2}\|g\|_{L^2}+\|f\|_{L^2}\|\p_y g\|_{L^2}.
			\deq
			{\rm(ii)}For $\lambda \in \mathbb{R}$ and an integer $m\ge 3$,
			any $\alpha=(\alpha_1, \alpha_2),
			\tilde{\alpha}=(\tilde{\alpha}_1, \tilde{\alpha}_2)$
			with $|\alpha|+|\tilde{\alpha}|\le m$, it holds
			\beq\label{morse}
			\|\p^{\alpha}f \p^{\tilde{\alpha}}g\xw^{\lambda+\alpha_2+\tilde{\alpha}_2}\|_{L^2}
			\le C\|f\|_{H^m_{\lambda_1}}\|g\|_{H^m_{\lambda_2}},
			\deq
			with $\lambda_1+\lambda_2=\lambda$.\\
			\rm(iii) Assume $g\ge \kappa$ and $g$ satisfies:
			\beqq
			\underset{y \rightarrow +\infty}{\lim}\p^{e_i}g=0(i=1,2), \underset{y \rightarrow +\infty}{\lim}g=g_{\infty},
			\deqq
			where $\kappa$ and $g_{\infty}$ are the constants.
			Denote $\bar{g}\overset{def}{=} g-g_{\infty}$,
			we have for all $|\alpha|\ge 1$ that
			\beq\label{f02}
			\|\p^\al (\frac{1}{g})\xw^{\lambda+\al_2}\|_{L^2}
			\le C_\kappa (1+\| \bar{g}\|_{H^{|\al|}_{\lambda}}^{|\al|}),
			\deq
			and furthermore, if $\underset{y \rightarrow +\infty}{\lim}f=0$,
			we have for all $|\alpha|\ge 1$
			\beq\label{f03}
			\|\p^\al (\frac{f}{g})\xw^{\lambda+\al_2}\|_{L^2}
			\le C_\kappa (1+\|\bar{g} \|_{H^{\max\{|\al|, 2\}}_{{0}}}^{|\al|})
			\|f\|_{H^{|\al|}_{\lambda}}.
			\deq
		\end{lemm}
		\begin{proof}
			The proof of \eqref{trace} and \eqref{morse} can be found in the Appendix of \cite{Liu-Xie-Yang2019}.					
			As for the estimate \eqref{f02}, we will give the proof  by induction.
			First of all, if $|\alpha|=1$, then we have
			\beqq
			0=\p^\alpha\left\{\frac{1}{g}g\right\}
			=\p^\alpha(\frac{1}{g})g+\frac{1}{g}\p^\alpha g,
			\deqq
			which yields directly
			\beqq
			\p^\alpha(\frac{1}{g})=-\frac{1}{g^2}\p^\alpha g,
			\deqq
			Thus,  we have
			\beq\label{a102}
			\|\p^\alpha(\frac{1}{g})\xw^{\lambda+\al_2}\|_{L^2}
			\le C_\kappa\|\p^\alpha g\xw^{\lambda+\al_2}\|_{L^2}
			\deq
			Thus, the estimate \eqref{a102} implies that \eqref{f02}
			holds  for the case $|\alpha|=1$.
			Now, let us assume that \eqref{f02} holds  for $|\alpha|\ge 1$.
			Then, for $i,j=1,2$, we have for $|\tilde{\alpha}|=|\alpha|+1$
			\beqq
			0=\p^{\tilde{\alpha}}\left\{\frac{1}{g}g\right\}
			=\frac{1}{g}\p^{\tilde{\alpha}}g
			+\p^{\tilde{\alpha}}(\frac{1}{g})g
			+\sum_{0<\tilde{\beta}<\tilde{\alpha}}
			C_{\tilde{\alpha}}^{\tilde{\beta}}
			\p^{\tilde{\beta}-e_i}\p^{e_i}(\frac{1}{g})
			\p^{\tilde{\alpha}-\tilde{\beta}-e_j}\p^{e_j}g,
			\deqq
			which yields directly
			\beqq
			\p^{\tilde{\alpha}}(\frac{1}{g})
			=-\frac{1}{g^2}\p^{\tilde{\alpha}}g
			-\frac{1}{g}\sum_{0<\tilde{\beta}<\tilde{\alpha}}
			C_{\tilde{\alpha}}^{\tilde{\beta}}
			\p^{\tilde{\beta}-e_i}\p^{e_i}(\frac{1}{g})
			\p^{\tilde{\alpha}-\tilde{\beta}-e_j}\p^{e_j}g.
			\deqq
			Thus, using the Morse-type inequality \eqref{morse}, we have
			\beq\label{a103}
			\begin{aligned}
				&\;\|\p^{\tilde{\alpha}}(\frac{1}{g})\xw^{\lambda+\tilde{\alpha}_2}\|_{L^2}\\
				\le &\;C_\kappa\|\p^{\tilde{\alpha}}g\xw^{\lambda+\tilde{\alpha}_2}\|_{L^2}
				+C_\kappa\|\p^{e_i}(\frac{1}{g})\|_{H^{|\alpha|-1}_{i-1}}
				\|\p^{e_j}g\|_{H^{|\alpha|-1}_{\lambda+j-1}}\\
				\le &\;C_\kappa\|\p^{\tilde{\alpha}}g\xw^{\lambda+\tilde{\alpha}_2}\|_{L^2}
				+C_\kappa (1+\|\bar{g} \|_{H^{|\al|}_{0}}^{|\al|})
				\|\p^{e_j}g\|_{H^{|\alpha|-1}_{\lambda+j-1}}\\
				\le &\;C_\kappa (1+\|\bar {g}\|_{H^{|\tilde{\al}|}_{\lambda}}^{|\tilde{\al}|}).
			\end{aligned}
			\deq
			Thus, the estimate \eqref{a103} implies that \eqref{f02}
			holds for the case $|\tilde{\alpha}|=|\alpha|+1$.
			Therefore, the induction implies that the estimate \eqref{f02}
			for all $|\alpha|\ge 1$.
			
			Finally, let us give the proof for the estimate \eqref{f03}.
			If $|\al|=1$, it is easy to check that
			\beqq
			|\p^\al(\frac{f}{g})|\le C_\dl(|\p^\al f|+|f||\p^\al g|).
			\deqq
			Thus, we apply the Sobolev inequality to obtain for all $|\al|=1$
			\beq\label{a104}
			\begin{aligned}
				&\;\|\p^\al(\frac{f}{g})\xw^{\lambda+\alpha_2}\|_{L^2}\\
				\le &\;C_\kappa(\|f\xw^{\lambda}\|_{L^2}+\|f\xw^{\lambda}\|_{L^2}^{\frac{1}{2}}
				\|\p_x f\xw^{\lambda}\|_{L^2}^{\frac{1}{2}})\\
				&\;\times \|\p^\al g\xw^{\alpha_2}\|_{L^2}^{\frac{1}{2}}
				\|\p_y \p^\al g \xw^{\alpha_2}\|_{L^2}^{\frac{1}{2}}
				+C_\kappa \|\p^\al f \xw^{\lambda+\alpha_2}\|_{L^2}\\
				\le &\;C_\kappa(1+\|\bar{g} \|_{H^2_{0}})\| f\|_{H^1_\lambda}.
			\end{aligned}
			\deq
			If $|\al|=2$, $|\beta|=1$, it is easy to check that
			\beqq
			|\p^\al(\frac{f}{g})|
			\le C_{\kappa}\left\{ |\p^\al f|+|\p^\beta  f||\p^{\alpha-\beta}  g|
			+|f|(|\p^\al g|+|\p^{\beta} g||\p^{\alpha-\beta} g|)\right\},
			\deqq
			which yields directly
			\beq\label{a105}
			\begin{aligned}
				&\;\|\p^\al(\frac{f}{g})\xw^{\lambda+\alpha_2}\|_{L^2}\\
				\le &\;C_\kappa\|f\xw^{\lambda}\|_{L^\infty}
				(\|\p^\beta g \xw^{\beta_2}\|_{L^2}+\|\p^\beta g \xw^{\beta_2}\|_{L^2}^{\frac12}
				\|\p_x \p^\beta g \xw^{\beta_2}\|_{L^2}^{\frac12})\\
				&\;\times\|\p^{\alpha-\beta} g \xw^{\alpha_2-\beta_2}\|_{L^2}^{\frac12}
				\|\p_y\p^{\alpha-\beta} g \xw^{\alpha_2-\beta_2}\|_{L^2}^{\frac12}\\
				&\;+C_\kappa\|f\xw^{\lambda}\|_{L^\infty}\|\xw^{\alpha_2}\p^\al g\|_{L^2}
				+C_\kappa \|\p^\al f\xw^{\lambda+\alpha_2}\|_{L^2}\\
				&\;+C_\kappa(\|\p^\beta  f \xw^{\lambda+\beta_2}\|_{L^2}
				+\|\p^\beta  f\xw^{\lambda+\beta_2}\|_{L^2}^{\frac{1}{2}}
				\|\p_x \p^\beta  f \xw^{\lambda+\beta_2}\|_{L^2}^{\frac{1}{2}})\\
				&\;\times \|\p^{\alpha-\beta}  g \xw^{\alpha_2-\beta_2}\|_{L^2}^{\frac{1}{2}}
				\|\p_y \p^{\alpha-\beta}  g \xw^{\alpha_2-\beta_2}\|_{L^2}^{\frac{1}{2}}\\
				\le &\;C_\kappa (1+\|\bar {g}\|_{H^2_{0}}^2)\|f\|_{H^2_\lambda}.
			\end{aligned}
			\deq
			For all $|\alpha| \ge 3$, it is easy to check that for $i= 1,2$,
			\beqq
			\begin{aligned}
				\p^\al (\frac{f}{g})
				=\frac{1}{g}\p^\al f
				+\sum_{0<\beta \le \alpha}C_\alpha^\beta
				\p^{\beta-e_i}\p^{e_i}(\frac{1}{g})\p^{\alpha-\beta}f.
			\end{aligned}
			\deqq
			Thus, we have
			\beq\label{a106}
			\begin{aligned}
				&\;\|\p^\al (\frac{f}{g})\xw^{\lambda+\alpha_2}\|_{L^2}\\
				\le &\;C_\kappa\|\p^\al f\xw^{\lambda+\alpha_2}\|_{L^2}
				+C_{|\alpha|}\|\p^{e_i}(\frac{1}{g})\|_{H^{|\alpha|-1}_{i-1}}
				\|f\|_{H^{|\alpha|-1}_{\lambda}}\\
				\le &\;C_\kappa\|\p^\al f\xw^{\lambda+\alpha_2}\|_{L^2}
				+C_{|\alpha|, \kappa}(1+\|\bar {g}\|_{H^{|\al|}_{{0}}}^{|\al|})
				\|f\|_{H^{|\alpha|-1}_{\lambda}}\\
				\le &\;C_\kappa (1+\|\bar {g}\|_{H^{|\al|}_{{0}}}^{|\al|})
				\|f\|_{H^{|\al|}_{\lambda}}.
			\end{aligned}
			\deq
			Thus, the combination of estimate \eqref{a104},
			\eqref{a105} and \eqref{a106} yields \eqref{f03}.
			Therefore, we complete the proof of this lemma.
			
		\end{proof}
		Finally we will state two classical maximum principles that are useful in Section \ref{prior-estimate}, for parabolic equations, which can be found in \cite{Masmoudi}.
		\begin{lemm}[Maximum Principle for Parabolic Equations]\label{max}
			Let $\var \ge 0$.
			Under the condition \eqref{Lowerb} that $\vr^\var+\vrf$ has an upper bound, if $H\in C([0, T]; C^2(\mathbb{T}\times \mathbb{R}^+))
			\cap C^1([0, T]; C^0(\mathbb{T}\times \mathbb{R}^+))$
			is a bounded function that satisfies the differential inequality
			\beqq
			\{\p_t+b_1\p_x +b_2 \p_y -\var \p_x^2-\frac{1}{\vr^\var+\vrf}\p_y^2\}H\le fH,
			~\rm{in}~ [0, T]\times \mathbb{T} \times \mathbb{R}^+,
			\deqq
			where the coefficients $b_1, b_2$ and $f$ are continuous and satisfy
			\beq\label{a101}
			\|b_2 \xw^{-1}\|_{L^\infty([0, T]\times \mathbb{T} \times \mathbb{R}^+)}
			<+\infty ~{\rm and}~\|f\|_{L^\infty([0, T]\times \mathbb{T} \times \mathbb{R}^+)}
			\le \lambda,
			\deq
			then for any $t\in [0, T]$,
			\beq\label{max-est}
			\underset{\mathbb{T} \times \mathbb{R}^+}{\sup}H(t)
			\le \max\{e^{\lambda t}\|H(0)\|_{L^\infty(\mathbb{T} \times \mathbb{R}^+)},
			\underset{\tau \in [0, t]}{\max}\{e^{\lambda(t-\tau)}
			\|H(\tau)|_{y=0}\|_{L^\infty(\mathbb{T})}\}\}.
			\deq
		\end{lemm}

		\begin{lemm}[Minimum Principle for Parabolic Equations]\label{min}
			Let $\var \ge 0$.
			Under the condition \eqref{Lowerb} that $\vr^\var+\vrf$ has an upper bound,
			if $H\in C([0, T]; C^2(\mathbb{T}\times \mathbb{R}^+))
			\cap C^1([0, T]; C^0(\mathbb{T}\times \mathbb{R}^+))$
			is a bounded function with
			\beqq
			\kappa(t)\overset{def}{=}
			\min\{\underset{\mathbb{T}\times\mathbb{R}^+}{H(0)},\quad
			\underset{[0, t]\times \mathbb{T}}{\min}H|_{y=0}\}\ge 0,
			\deqq
			and satisfies:
			\beqq
			\{\p_t+b_1\p_x +b_2 \p_y -\var \p_x^2-\frac{1}{\vr^\var+\vrf}\p_y^2\}H= fH,
			\deqq
			where the coefficients $b_1, b_2$ and $f$ are continuous and
			satisfy \eqref{a101}, then for any $t \in[0, T]$,
			\beq\label{min-est}
			\underset{\mathbb{T}\times \mathbb{R}^+}{\min H(t)}
			\ge (1-\lambda t e^{\lambda t})\kappa(t).
			\deq
			
		\end{lemm}
		
		\section{Boundary reduction of vorticity}\label{appendix-b}
		
		In this section, we will give the proof for the boundary reduction
		representation \eqref{bl-redu}.
		
		\textbf{Step 1:
		Proof of the boundary reduction representation $\eqref{bl-redu}_1$.}
		Indeed, from the vorticity equation $\eqref{Prandtl-02}_1$, we have
		\beqq
		\frac{1}{\vr^\var+\vrf}\p_y^3 w
		=\p_t \p_y w^\var +\p_y(u^\var \p_x w^\var+ v^\var \p_y w^\var)
		-\var \p_x^2 \p_y w^\var-\p_y^2(\frac{1}{\vr^\var+\vrf}) \p_y w^\var
		-2\p_y(\frac{1}{\vr^\var+\vrf})\p_y^2 w^\var,
		\deqq
		which yields on the boundary $y=0$ that
		\beq\label{a201}
		\frac{1}{\vr^\var+\vrf}\p_y^3 w^\var
		=w^\var \p_x w^\var+\frac{2\p_y \vr^\var}{(\vr^\var+\vrf)^2}\p_y^2 w^\var.
		\deq
		Thus, \eqref{a201} implies the boundary reduction relation $\eqref{bl-redu}_1$.\\
		
		\textbf{Step 2:
		Proof of the boundary reduction representation $\eqref{bl-redu}_2$.}
		Indeed, the vorticity equation $\eqref{Prandtl-02}_1$ yields directly
		\beqq
		\frac{1}{\vr^\var+\vrf}\p_y^5 w^\var
		=\p_t \p_y^3 w^\var+\p_y^3(u^\var \p_x w^\var +v^\var \p_y w^\var)-\var \p_x^2 \p_y^3 w^\var
		-\p_y^4 (\frac{1}{\vr^\var+\vrf})\p_y w^\var
		-\sum_{1\le k \le 3}C_4^k \p_y^k(\frac{1}{\vr^\var+\vrf})\p_y^{5-k}w^\var,
		\deqq
		which, on the boundary $y=0$, yields
		\beqq
		\frac{1}{\vr^\var+\vrf}\p_y^5 w^\var
		=\p_t \p_y^3 w^\var+\p_y^3(u^\var \p_x w^\var +v^\var \p_y w^\var)-\var \p_x^2 \p_y^3 w^\var
		-\sum_{1\le k \le 3}C_4^k \p_y^k(\frac{1}{\vr^\var+\vrf})\p_y^{5-k}w^\var.
		\deqq
		Indeed, on the boundary $y=0$, it holds
		\beq\label{a203}
		\p_y^3(u^\var \p_x w^\var+v^\var \p_y w^\var)=3w^\var \p_x \p_y^2 w^\var-2\p_x w^\var \p_y^2 w^\var,
		\deq
		The application of \eqref{a201}, on the boundary $y=0$, yields directly
		\beqq
		\var \p_x^2 \p_y^3 w^\var
		=\var \p_x^2\{(\vr^\var+\vrf) w^\var \p_x w^\var\}
		+2\var \p_x^2(\frac{\p_y \vr^\var \p_y^2 w^\var }{\vr^\var+\vrf}).
		\deqq
		Due to the boundary condition $u^\var|_{y=0}=v^\var|_{y=0}=0$, we have
		\beq\label{a204}
		\begin{aligned}
			&\p_t \vr^\var|_{y=0}=-(u \p_x \vr^\var+v^\var \p_y \vr^\var-\var \p_x^2 \vr^\var)_{y=0}=\var \p_x^2 \vr^\var|_{y=0},\\
			&\p_t \p_y \vr^\var|_{y=0}=-\p_y\{u^\var \p_x \vr^\var+v^\var \p_y \vr^\var-\var \p_x^2 \vr^\var\}|_{y=0}=(-w^\var\p_x \vr^\var + \var \p_x^2 \p_y \vr^\var)_{y=0}.
		\end{aligned}
		\deq
		and hence, on the boundary $y=0$, we have
		\beq\label{a205}
		\begin{aligned}
			\frac{1}{\vr^\var+\vrf}\p_t \p_y^3 w^\var
			=&\p_t(w^\var \p_x w^\var)-2\p_t\{\p_y(\frac{1}{\vr^\var+\vrf})\p_y^2 w^\var\}-\p_t(\frac{1}{\vr^\var+\vrf})\p_y^3 w^\var\\
			=&\frac{\p_x(\p_y^2 w^\var  w^\var)}{\vr^\var+\vrf} -\frac{{\p_x \vr^\var w^\var}\p_y^2 w^\var}{(\vr^\var+\vrf)^2}
			+ 2\p_y^2 w^\var \left(\frac{-w^\var\p_x \vr^\var +\var \p_x^2 \p_y \vr^\var}{(\vr^\var+\vrf)^2}-\frac{2 \var  \p_y \vr^\var \p_x^2 \vr^\var  }{(\vr^\var+\vrf)^3}\right)\\
			&+2\frac{\p_y \vr^\var}{(\vr^\var+\vrf)^2}\left(\var \p_x^2\p_y^2 w^\var +\p_y^3(\frac{\p_y w^\var}{\vr^\var+\vrf})\right)-\frac{\var \p_x^2 \vr^\var}{(\vr^\var+\vrf)^2} \p_y^3 w^\var.
		\end{aligned}
		\deq
		Thus, on the boundary $y=0$, it holds
		\beq\label{a206}
		\begin{aligned}
			\p_y^5 w^\var
			=&(\vr^\var+\vrf)\p_x(\p_y^2 w^\var  w^\var)
			-3{\p_x \vr^\var w^\var}\p_y^2 w^\var
			+2\p_y^2 w^\var\left(\var \p_x^2 \p_y \vr^\var- \frac{2 \var \p_y \vr^\var \p_x^2 \vr^\var  }{\vr^\var+\vrf}\right)\\
			&+{2\p_y \vr^\var}\p_y^3(\frac{\p_y w^\var}{\vr^\var+\vrf})
			+2 \var \p_y \vr^\var  \p_x^2\p_y^2 w^\var -\var \p_x^2 \vr^\var \p_y^3 w^\var-\sum_{1\le k \le 3}C_4^k (\vr^\var+\vrf)\p_y^k(\frac{1}{\vr^\var+\vrf})\p_y^{5-k}w^\var\\
			&
			+(\vr^\var+\vrf)\left(3w^\var \p_x \p_y^2 w^\var-2\p_x w^\var \p_y^2 w^\var
			-\var \p_x^2\{(\vr^\var+\vrf) w^\var \p_x w^\var\}
			-2\var \p_x^2(\frac{\p_y \vr^\var \p_y^2 w^\var }{\vr^\var+\vrf})\right)\\
			\overset{def}{=}&N_1.
		\end{aligned}
		\deq
		Thus, \eqref{a206} implies the boundary reduction relation $\eqref{bl-redu}_2$.\\

		\textbf{Step 3:
		Proof of the boundary reduction representation $\eqref{bl-redu}_3$.}
		Using the vorticity equation $\eqref{Prandtl-02}_1$, we have
		\beq\label{a207}
		\bal
		\frac{1}{\vr^\var+\vrf}\p_y^7 w^\var
		=&\;\p_t \p_y^5 w^\var+\p_y^5\{u^\var \p_x w^\var +v^\var \p_y w^\var\}-\var \p_x^2 \p_y^5 w^\var\\
		&\;-\p_y^6 (\frac{1}{\vr^\var+\vrf})\p_y w^\var
		-\sum_{1\le k \le 5}C_6^k \p_y^k(\frac{1}{\vr^\var+\vrf})\p_y^{7-k}w^\var.
		\dal
		\deq
		Then, we will give the estimate for $\p_t \p_y^5 w^\var|_{y=0}$ in the sequence.
		Using the vorticity equation $\eqref{Prandtl-02}_1$, we have
		\beqq
		\begin{aligned}
			\frac{1}{\vr^\var+\vrf}\p_t \p_y^5 w^\var
			=&-\p_t(\frac{1}{\vr^\var+\vrf})\p_y^5 w^\var
			+\p_t^2 \p_y^3 w^\var+\p_t \p_y^3\{u^\var \p_x w^\var +v^\var \p_y w^\var\}
			-\var \p_t \p_x^2 \p_y^3 w^\var\\
			&-\p_t \left\{\p_y^4 (\frac{1}{\vr^\var+\vrf})\p_y w^\var\right\}
			-\p_t \left\{\sum_{1\le k \le 3}C_4^k \p_y^k(\frac{1}{\vr^\var+\vrf})\p_y^{5-k}w^\var\right\},
		\end{aligned}
		\deqq
		which, together with \eqref{a204}, yields directly on the $y=0$ that
		\beq\label{a208}
		\begin{aligned}
			\frac{1}{\vr^\var+\vrf}\p_t \p_y^5 w^\var
			=& \frac{\var \p_x^2 \vr^\var}{(\vr^\var+\vrf)^2} \p_y^5 w^\var+
			\p_t^2 \p_y^3 w^\var+\p_t \p_y^3 \{u^\var \p_x w^\var +v^\var \p_y w^\var\}
			-\var \p_t \p_x^2 \p_y^3 w^\var\\
			&-\p_t \left\{\sum_{1\le k \le 3}C_4^k \p_y^k(\frac{1}{\vr^\var+\vrf})\p_y^{5-k}w^\var\right\}.
		\end{aligned}
		\deq
	First, we give the boundary reduction for $\p_t^2 \p_y^3 w^\var|_{y=0}$.
		Indeed, the vorticity equation $\eqref{Prandtl-02}_1$ yields
		\beqq
		\begin{aligned}
			\frac{1}{\vr^\var+\vrf}\p_t^2 \p_y^3 w^\var
			=&\p_t^2\left\{\p_t \p_y w^\var+\p_y(u^\var \p_x w^\var+v^\var \p_y w^\var)-\var \p_x^2 \p_y w^\var\right\}
			-\sum_{1\le k \le 2}C_2^k \p_t^k(\frac{1}{\vr^\var+\vrf})\p_t^{2-k}\p_y^3 w^\var\\
			&-\p_t^2 \left\{\p_y^2(\frac{1}{\vr^\var+\vrf})\p_y w^\var
			+2\p_y(\frac{1}{\vr^\var+\vrf})\p_y^2 w^\var\right\}.
		\end{aligned}
		\deqq
		Thus, on the boundary $y=0$, we have
		\beq\label{a209}
		\begin{aligned}
			\frac{1}{\vr^\var+\vrf}\p_t^2 \p_y^3 w^\var
			=&\p_t^2(w^\var \p_x w^\var)
			-2\p_t^2 \left\{\p_y(\frac{1}{\vr^\var+\vrf})\p_y^2 w^\var\right\}
			-\sum_{1\le k \le 2}C_2^k \p_t^k(\frac{1}{\vr^\var+\vrf})\p_t^{2-k}\p_y^3 w^\var
			.
		\end{aligned}
		\deq
		It is easy to check that
		\beq\label{a210}
		\p_t^2\{w^\var\p_x w^\var\}
		=\p_t^2 w^\var \p_x w^\var+2\p_t w^\var \p_t \p_x w^\var +w^\var\p_x \p_t^2 w^\var.
		\deq
		Obviously, on the boundary $y=0$, the vorticity equation
		$\eqref{Prandtl-02}_1$ yields
		\beq\label{a211}
		\p_t w^\var=\frac{1}{\vr^\var+\vrf}\p_y^2 w^\var +\var \p_x^2 w^\var.
		\deq
		Similarly, it is easy to check that
		\beqq
		\p_t^2 w^\var=\p_t \left\{\p_y(\frac{1}{\vr^\var+\vrf})\p_y w^\var\right\}
		+\p_t(\frac{1}{\vr^\var+\vrf})\p_y^2 w^\var
		+\frac{1}{\vr^\var+\vrf}\p_t \p_y^2 w^\var
		-\p_t\{u^\var \p_x w^\var+v^\var \p_y w^\var\}+\var\p_t \p_x^2 w^\var,
		\deqq
		which, on the boundary $y=0$, yields directly
		\beq\label{a212}
		\begin{aligned}
			\p_t^2 w^\var =& -\frac{\p_t \vr^\var}{(\vr^\var+\vrf)^2} \p_y^2 w^\var
			+\frac{1}{\vr^\var+\vrf}\p_t \p_y^2 w^\var +\var\p_t \p_x^2 w^\var\\
			=&-\frac{\var \p_x^2 \vr^\var}{(\vr^\var+\vrf)^2} \p_y^2 w^\var
			+ \frac{\var \p_x^2 \p_y^2 w^\var}{\vr^\var+\vrf}
			+\frac{1}{\vr^\var+\vrf}\p_y^3(\frac{1}{\vr^\var+\vrf}\p_y w^\var)+\var \p_x^2 \p_y\left\{\frac{1}{\vr^\var+\vrf}\p_y w^\var\right\} +\var^2 \p_x^4 w^\var,
		\end{aligned}
		\deq
		where we have used the condition
		\beq\label{a214}
		\begin{aligned}
			\p_t \p_y^2 w^\var|_{y=0}=(\p_y^3(\frac{1}{\vr^\var+\vrf}\p_y w^\var)+\var \p_x^2 \p_y^2 w^\var)_{y=0},\\
			\p_t \p_x^2 w^\var|_{y=0}= (\p_x^2 \p_y(\frac{1}{\vr^\var+\vrf}\p_y w^\var) + \var \p_x^4 w^\var)_{y=0}.
		\end{aligned}
		\deq
		Using the relations \eqref{a211} and \eqref{a212}, on the boundary $y=0$,
		the relation \eqref{a210} yields
		\beq\label{a213}
		\begin{aligned}
			&\;\p_t^2(w^\var\p_x w^\var)\\
			=&\;\p_x w^\var \left(-\frac{\var \p_x^2 \vr^\var}{(\vr^\var+\vrf)^2} \p_
			y^2 w^\var
			+ \frac{\var \p_x^2 \p_y^2 w^\var}{\vr^\var+\vrf}
			+\frac{1}{\vr^\var+\vrf}\p_y^3(\frac{1}{\vr^\var+\vrf}\p_y w^\var)+\var \p_x^2 \p_y(\frac{1}{\vr^\var+\vrf}\p_y w^\var) +\var^2 \p_x^4 w^\var\right)\\
			&\;+w^\var\p_x\left\{-\frac{\var \p_x^2 \vr^\var \p_
				y^2 w^\var}{(\vr^\var+\vrf)^2}
			+ \frac{\var \p_x^2 \p_y^2 w^\var}{\vr^\var+\vrf}
			+\frac{1}{\vr^\var+\vrf}\p_y^3(\frac{1}{\vr^\var+\vrf}\p_y w^\var)+\var \p_x^2 \p_y(\frac{1}{\vr^\var+\vrf}\p_y w^\var) +\var^2 \p_x^4 w^\var
			\right\}\\
			&\;+2(\frac{1}{\vr^\var+\vrf}\p_y^2 w^\var +\var \p_x^2 w^\var )\cdot \p_x\left\{\frac{1}{\vr^\var+\vrf}\p_y^2 w^\var+\var \p_x^2 w^\var\right\}.
		\end{aligned}
		\deq
		Finally, it is easy to check that
		\beq\label{a215}
		\begin{aligned}
			\p_t^2\left\{\p_y(\frac{1}{\vr^\var+\vrf})\p_y^2 w^\var\right\}
			=\p_t^2 \p_y(\frac{1}{\vr^\var+\vrf})\p_y^2 w^\var
			+2\p_t\p_y(\frac{1}{\vr^\var+\vrf})\p_t \p_y^2 w^\var
			+\p_y(\frac{1}{\vr^\var+\vrf}) \p_t^2  \p_y^2 w^\var.
		\end{aligned}
		\deq
		Using the density equation $\eqref{Prandtl-01}_1$, then we have
		\beqq
		\p_t\p_y (\frac{1}{\vr^\var+\vrf})
		=\p_y\left\{\frac{u^\var \p_x \vr^\var+v^\var \p_y \vr^\var-\var \p_x^2 \vr^\var}{(\vr^\var+\vrf)^2}\right\},
		\deqq
		and hence, on the boundary $y=0$, it holds
		\beq\label{a216}
		\p_t \p_y(\frac{1}{\vr^\var+\vrf})=\frac{2\var  \p_y \vr^\var \p_x^2 \vr^\var}{(\vr^\var+\vrf)^3}
		+\frac{w^\var\p_x \vr^\var-\var \p_x^2\p_y \vr^\var}{(\vr^\var+\vrf)^2},
		\deq
		and
		\beq\label{a217}
		\begin{aligned}
			&\;\p_t^2 \p_y (\frac{1}{\vr^\var+\vrf})\\
			=&\;\p_t \left\{\frac{2\var  \p_y \vr^\var \p_x^2 \vr^\var}{(\vr^\var+\vrf)^3}
			+\frac{w^\var\p_x \vr^\var-\var \p_x^2\p_y \vr^\var}{(\vr^\var+\vrf)^2}\right\}\\
			=&\;-\var \p_x^2 \vr^\var \left(\frac{6\var  \p_y \vr^\var \p_x^2 \vr^\var}{(\vr^\var+\vrf)^4}
			+2\frac{w^\var \p_x \vr^\var-\var \p_x^2\p_y \vr^\var}{(\vr^\var+\vrf)^3}\right)
			+\frac{2\var\p_x^2 \vr^\var (-w^\var \p_x \vr^\var + \var \p_x^2 \p_y \vr^\var) +2\var^2  \p_y \vr^\var \p_x^4 \vr^\var}{(\vr^\var+\vrf)^3}\\
			&\;+\frac{1}{(\vr^\var+\vrf)^2}\left(\var w^\var \p_x^3 \vr^\var + \p_x \vr^\var (\var \p_x^2 w^\var + \p_y(\frac{\p_y w^\var}{\vr^\var+\vrf}))+\var \p_x^2 \p_y (u^\var\p_x \vr^\var+v^\var \p_y \vr^\var -\var \p_x^2 \vr^\var)\right).
			\\
		\end{aligned}
		\deq
		Using the vorticity equation $\eqref{Prandtl-02}_1$,
		on the boundary $y=0$, it holds
		\beq\label{a218}
		\begin{aligned}
			&\;\p_t^2 \p_y^2 w^\var
			\\=&\;\p_t \left\{ \p_y^3(\frac{1}{\vr^\var+\vrf}\p_y w^\var) +\var \p_x^2\p_y^2 w^\var
			\right\}\\
			=&\;\p_y^3\left\{\frac{(u^\var \p_x \vr^\var+v^\var \p_y \vr^\var- \var \p_x^2 \vr^\var)\p_y w^\var}{(\vr^\var+\vrf)^2}
			+\frac{1}{\vr^\var+\vrf}\p_y^2(\frac{\p_y w^\var}{\vr^\var+\vrf})
			-\frac{\p_y(u^\var \p_x w^\var+v^\var \p_y w^\var)}{\vr^\var+\vrf}+\frac{\var \p_y \p_x^2 w^\var}{\vr^\var+\vrf}
			\right\}\\
			&- \p_x^2 \p_y^2\left\{u^\var \p_x w^\var+v^\var \p_y w^\var- \var \p_x^2 w^\var -\p_y(\frac{\p_y w^\var}{\vr^\var+\vrf})\right\}.
		\end{aligned}
		\deq
		Substituting the equalities \eqref{a214}, \eqref{a216}
		\eqref{a217} and \eqref{a218} into \eqref{a215},
		we have on the boundary $y=0$
		\beq\label{a219}
		\begin{aligned}
			&\;\p_t^2\left\{\p_y(\frac{1}{\vr^\var+\vrf})\p_y^2 w^\var\right\}\\
			=&\;\p_y(\frac{1}{\vr^\var+\vrf})\p_y^3\left\{\frac{(u^\var \p_x \vr^\var+v^\var \p_y \vr^\var- \var \p_x^2 \vr^\var)\p_y w^\var}{(\vr^\var+\vrf)^2}
			+\frac{1}{\vr^\var+\vrf}\p_y^2(\frac{\p_y w^\var}{\vr^\var+\vrf})
			-\frac{\p_y(u^\var \p_x w^\var+v^\var \p_y w^\var)}{\vr^\var+\vrf}
			\right\}\\
			&\;+\var\p_y(\frac{1}{\vr^\var+\vrf}) \p_y^3( \frac{ \p_y \p_x^2 w^\var}{\vr^\var+\vrf}) -\p_y(\frac{1}{\vr^\var+\vrf}) \p_x^2 \p_y^2\left\{u^\var \p_x w^\var+v^\var \p_y w^\var- \var \p_x^2 w^\var -\p_y(\frac{\p_y w^\var}{\vr^\var+\vrf})\right\}\\
			&\;-\var \p_x^2 \vr^\var \p_y^2 w^\var \left(\frac{6\var  \p_y \vr^\var \p_x^2 \vr^\var}{(\vr^\var+\vrf)^4}
			+2\frac{w^\var\p_x \vr^\var-\var \p_x^2\p_y \vr^\var}{(\vr^\var+\vrf)^3}\right)
			+\p_y^2 w^\var \frac{2\var\p_x^2 \vr^\var (-w^\var \p_x \vr^\var + \var \p_x^2 \p_y \vr^\var) +2\var^2  \p_y \vr^\var \p_x^4 \vr^\var}{(\vr^\var+\vrf)^3}\\
			&\;+\frac{\p_y^2 w^\var}{(\vr^\var+\vrf)^2}\left(\var w^\var \p_x^3 \vr^\var + \p_x \vr^\var (\var \p_x^2 w^\var + \p_y(\frac{\p_y w^\var}{\vr^\var+\vrf}))+\var \p_x^2 \p_y \{u^\var \p_x \vr^\var+v^\var \p_y \vr^\var -\var \p_x^2 \vr^\var\}\right)\\
			&\;+\left(\frac{2\var  \p_y \vr^\var \p_x^2 \vr^\var}{(\vr^\var+\vrf)^3}
			+\frac{w^\var\p_x \vr^\var-\var \p_x^2\p_y \vr^\var}{(\vr^\var+\vrf)^2}\right) \left(\p_y^3(\frac{1}{\vr^\var+\vrf}\p_y w^\var)+\var \p_x^2 \p_y^2 w^\var \right) .
		\end{aligned}
		\deq
		Substituting \eqref{a205}, \eqref{a213} and \eqref{a219} into \eqref{a209},
		on the boundary $y=0$, we have
		\beq\label{a220}
		\begin{aligned}
			&\frac{1}{\vr^\var+\vrf}\p_t^2 \p_y^3 w^\var\\
			=&-2\p_y(\frac{1}{\vr^\var+\vrf})\p_y^3\left\{\frac{(u^\var \p_x \vr^\var+v^\var \p_y \vr^\var- \var \p_x^2 \vr^\var)\p_y w^\var}{(\vr^\var+\vrf)^2}
			+\frac{1}{\vr^\var+\vrf}\p_y^2(\frac{\p_y w^\var}{\vr^\var+\vrf})
			-\frac{\p_y(u^\var \p_x w^\var+v^\var \p_y w^\var)}{\vr^\var+\vrf}
			\right\}\\
			&-2\var\p_y(\frac{1}{\vr^\var+\vrf})  \p_y^3 (\frac{ \p_y \p_x^2 w^\var}{\vr^\var+\vrf})
			+2 \p_y(\frac{1}{\vr^\var+\vrf}) \p_x^2 \p_y^2\left\{u^\var_1\p_x w^\var+u^\var_2 \p_y w^\var- \var \p_x^2 w^\var -\p_y(\frac{\p_y w^\var}{\vr^\var+\vrf})\right\}\\
			& +2 \var \p_x^2 \vr^\var \p_y^2 w^\var \left(\frac{6\var  \p_y \vr^\var \p_x^2 \vr^\var}{(\vr^\var+\vrf)^4}
			+2\frac{w^\var\p_x \vr^\var-\var \p_x^2\p_y \vr^\var}{(\vr^\var+\vrf)^3}\right)
			-2\p_y^2 w^\var \frac{2\var\p_x^2 \vr^\var (-w^\var \p_x \vr^\var + \var \p_x^2 \p_y \vr^\var) +2\var^2  \p_y \vr^\var \p_x^4 \vr^\var}{(\vr^\var+\vrf)^3}\\
			&-2\frac{\p_y^2 w^\var}{(\vr^\var+\vrf)^2}\left(\var w^\var \p_x^3 \vr^\var + \p_x \vr^\var (\var \p_x^2 w^\var + \p_y(\frac{\p_y w^\var}{\vr^\var+\vrf}))+\var \p_x^2 \p_y \{u^\var \p_x \vr^\var+v^\var \p_y \vr^\var -\var \p_x^2 \vr^\var\}\right)\\
			&-2\left(\frac{2\var  \p_y \vr^\var \p_x^2 \vr^\var}{(\vr^\var+\vrf)^3}
			+\frac{w^\var\p_x \vr^\var-\var \p_x^2\p_y \vr^\var}{(\vr^\var+\vrf)^2}\right) \left(\p_y^3(\frac{1}{\vr^\var+\vrf}\p_y w^\var)+\var \p_x^2 \p_y^2 w^\var \right)
			\\&+2(\frac{1}{\vr^\var+\vrf}\p_y^2 w^\var +\var \p_x^2 w^\var )\cdot \p_x\left\{\frac{1}{\vr^\var+\vrf}\p_y^2 w^\var+\var \p_x^2 w^\var\right\}\\
\end{aligned}
\deq
		\beqq
		\begin{aligned}
			&+\p_x w^\var \left(-\frac{\var \p_x^2 \vr^\var}{(\vr^\var+\vrf)^2} \p_
			y^2 w^\var
			+ \frac{\var \p_x^2 \p_y^2 w^\var}{\vr^\var+\vrf}
			+\frac{1}{\vr^\var+\vrf}\p_y^3(\frac{1}{\vr^\var+\vrf}\p_y w^\var)+\var \p_x^2 \p_y(\frac{1}{\vr^\var+\vrf}\p_y w^\var) +\var^2 \p_x^4 w^\var\right)\\
			&+w^\var\p_x\left\{-\frac{\var \p_x^2 \vr^\var}{(\vr^\var+\vrf)^2} \p_
			y^2 w^\var
			+ \frac{\var \p_x^2 \p_y^2 w^\var}{\vr^\var+\vrf}
			+\frac{1}{\vr^\var+\vrf}\p_y^3(\frac{1}{\vr^\var+\vrf}\p_y w^\var)+\var \p_x^2 \p_y(\frac{1}{\vr^\var+\vrf}\p_y w^\var) +\var^2 \p_x^4 w^\var
			\right\}\\
			&-2 \var^2 \p_y^3 w^\var\left(\frac{2(\p_x^2 \vr^\var)^2}{(\vr^\var+\vrf)^3} - \frac{\p_x^4 \vr^\var}{(\vr^\var+\vrf)^2}\right)
			+\frac{ \var \p_x^2 \vr^\var}{\vr^\var+\vrf}\Big(
			\frac{\p_x(\p_y^2 w^\var  w^\var)}{\vr^\var+\vrf} -\frac{{\p_x \vr^\var w^\var}\p_y^2 w^\var}{(\vr^\var+\vrf)^2}-\frac{\var \p_x^2 \vr^\var}{(\vr^\var+\vrf)^2} \p_y^3 w^\var\Big)\\
			&+ \frac{  2 \var \p_y^2 w^\var  \p_x^2 \vr^\var}{\vr^\var+\vrf}
			\left(\frac{-w^\var\p_x \vr^\var +\var \p_x^2 \p_y \vr^\var}{(\vr^\var+\vrf)^2}-\frac{2 \var  \p_y \vr^\var \p_x^2 \vr^\var  }{(\vr^\var+\vrf)^3}\right)
			+2\frac{\var \p_y \vr^\var \p_x^2 \vr^\var }{(\vr^\var+\vrf)^3}\left(\var \p_x^2\p_y^2 w^\var +\p_y^3(\frac{\p_y w^\var}{\vr^\var+\vrf})\right) .
		\end{aligned}
		\deqq
		Using the equation \eqref{a203}, we have on the boundary $y=0$ that
		\beq\label{a221}
		\begin{aligned}
			&\;\p_t \p_y^3\{u^\var \p_x w^\var+v^\var \p_y w^\var\}\\
			=&\;\p_t\{3w^\var \p_x \p_y^2 w^\var-2\p_x w^\var \p_y^2 w^\var\}\\
			=&\;3\p_x \p_y^2 w^\var \left(\p_y(\frac{\p_y w^\var}{\vr^\var+\vrf})+\var \p_x^2 w^\var\right)
			+3 w^\var \p_x \p_y^2\left\{\p_y(\frac{\p_y w^\var}{\vr^\var+\vrf})+\var \p_x^2 w^\var\right\}\\
			&\;-2\p_y^2 w^\var \p_x\left\{\p_y(\frac{\p_y w^\var}{\vr^\var+\vrf})+\var \p_x^2 w^\var\right\}
			-2\p_x w^\var \p_y^2\left\{\p_y(\frac{\p_y w^\var}{\vr^\var+\vrf})+\var \p_x^2 w^\var\right\}.
		\end{aligned}
		\deq
		Using the equation \eqref{a205}, we have on the boundary $y=0$ that
		\beq\label{a222}
		\begin{aligned}
			&\;\var\p_t \p_x^2 \p_y^3 w^\var\\
			=&\;\var\p_x^2\Bigg\{ \p_x(\p_y^2 w^\var  w^\var) -\frac{{\p_x \vr^\var w^\var}\p_y^2 w^\var}{\vr^\var+\vrf}
			+ 2\p_y^2 w^\var(\frac{-w^\var\p_x \vr^\var +\var \p_x^2 \p_y \vr^\var}{\vr^\var+\vrf}-\frac{2 \var \p_y \vr^\var \p_x^2 \vr^\var  }{(\vr^\var+\vrf)^2})\\
			&\;+2\frac{\p_y \vr^\var}{\vr^\var+\vrf}(\var \p_x^2\p_y^2 w^\var +\p_y^3(\frac{\p_y w^\var}{\vr^\var+\vrf}))-\frac{\var \p_x^2 \vr^\var}{\vr^\var+\vrf} \p_y^3 w^\var \Bigg\}.
		\end{aligned}
		\deq
		Using the density equation $\eqref{Prandtl-01}_1$
		and vorticity equation $\eqref{Prandtl-02}_1$, it is easy to check that
		\beq\label{a223}
		\begin{aligned}
			&\;\p_t \left\{\sum_{1\le k \le 3}C_4^k \p_y^k(\frac{1}{\vr^\var+\vrf})\p_y^{5-k}w^\var\right\}\\
			=&\;\sum_{1\le k \le 3}C_4^k \p_y^k(\frac{1}{\vr^\var+\vrf})
			\p_y^{5-k}\left\{\var \p_x^2 w^\var+\p_y(\frac{\p_y w^\var}{\vr^\var+\vrf})-(u^\var \p_x w^\var+v^\var \p_y w^\var)\right\}\\
			&\;+\sum_{1\le k \le 3}C_4^k \p_y^k
			\left\{\frac{u^\var \p_x \vr^\var+v^\var \p_y \vr^\var-\var \p_x ^2 \vr^\var}{(\vr^\var+\vrf)^2}\right\}\p_y^{5-k}w^\var.
		\end{aligned}
		\deq
		Substituting equations \eqref{a220}, \eqref{a221}, \eqref{a222},
		\eqref{a223} into \eqref{a208}, we have on the boundary $y=0$ that
		\beq\label{a225}
		\begin{aligned}
			&\frac{1}{\vr^{\var}+\vrf}\p_t \p_y^5 w^{\var}\\
			=&\frac{2\p_y \vr^\var}{\vr^\var+\vrf}\p_y^3\left\{\frac{(u^\var \p_x \vr^\var+v^\var \p_y \vr^\var- \var \p_x^2 \vr^\var)\p_y w^\var}{(\vr^\var+\vrf)^2}
			+\frac{\p_y^2(\frac{\p_y w^\var}{\vr^\var+\vrf})}{\vr^\var+\vrf}
			-\frac{\p_y(u^\var \p_x w^\var+v^\var \p_y w^\var)}{\vr^\var+\vrf}+\frac{\var \p_y \p_x^2 w^\var}{\vr^\var+\vrf}
			\right\}\\
			&-\frac{2\p_y \vr^\var}{\vr^\var+\vrf} \p_x^2 \p_y^2\left\{u^\var \p_x w^\var+v^\var \p_y w^\var- \var \p_x^2 w^\var -\p_y(\frac{\p_y w^\var}{\vr^\var+\vrf})\right\}\\
			& +2 \var \p_x^2 \vr^\var \p_y^2 w^\var \left(\frac{6\var  \p_y \vr^\var \p_x^2 \vr^\var}{(\vr^\var+\vrf)^3}
			+2\frac{w^\var\p_x \vr^\var-\var \p_x^2\p_y \vr^\var}{(\vr^\var+\vrf)^2}\right)
			-2\p_y^2 w^\var \frac{2\var\p_x^2 \vr^\var (-w^\var \p_x \vr^\var + \var \p_x^2 \p_y \vr^\var) +2\var^2  \p_y \vr^\var \p_x^4 \vr^\var}{(\vr^\var+\vrf)^2}\\
			&-2\frac{\p_y^2 w^\var}{\vr^\var+\vrf}\left(\var w^\var \p_x^3 \vr^\var + \p_x \vr^\var (\var \p_x^2 w^\var + \p_y(\frac{\p_y w^\var}{\vr^\var+\vrf}))+\var \p_x^2 \p_y (u^\var\p_x \vr^\var+v^\var \p_y \vr^\var -\var \p_x^2 \vr^\var)\right)\\
\end{aligned}
\deq
		\beqq
		\begin{aligned}
			&-2\left(\frac{2\var  \p_y \vr^\var \p_x^2 \vr^\var}{(\vr^\var+\vrf)^2}
			+\frac{w^\var\p_x \vr^\var-\var \p_x^2\p_y \vr^\var}{\vr^\var+\vrf}\right) \left(\p_y^3(\frac{1}{\vr^\var+\vrf}\p_y w^\var)+\var \p_x^2 \p_y^2 w^\var \right)
			\\&+2(\p_y^2 w^\var -\var (\vr^\var+\vrf)\p_x^2 w^\var )\cdot \p_x\left\{\frac{1}{\vr^\var+\vrf}\p_y^2 w^\var-\var \p_x^2 w^\var\right\}\\
			&+\p_x w^\var \left(-\frac{\var \p_x^2 \vr^\var}{\vr^\var+\vrf} \p_
			2^2 w^\var
			- \var \p_x^2 \p_y^2 w^\var
			+\p_y^3(\frac{1}{\vr^\var+\vrf}\p_y w^\var)-(\vr^\var+\vrf)(\var \p_x^2 \p_y(\frac{1}{\vr^\var+\vrf}\p_y w^\var) +\var^2 \p_x^4 w^\var)\right)\\
			&+w^\var(\vr^\var+\vrf) \p_x\left\{-\frac{\var \p_x^2 \vr^\var \p_
				2^2 w^\var}{(\vr^\var+\vrf)^2}
			- \frac{\var \p_x^2 \p_y^2 w^\var}{\vr^\var+\vrf}
			+\frac{1}{\vr^\var+\vrf}\p_y^3(\frac{1}{\vr^\var+\vrf}\p_y w^\var)-\var \p_x^2 \p_y(\frac{1}{\vr^\var+\vrf}\p_y w^\var) -\var^2 \p_x^4 w^\var
			\right\}\\
			&-2 \var^2 \p_y^3 w^\var\left(\frac{2(\p_x^2 \vr^\var)^2}{(\vr^\var+\vrf)^2} - \frac{\p_x^4 \vr^\var}{\vr^\var+\vrf}\right)
			+ \var \p_x^2 \vr^\var\Big(
			\frac{\p_x(\p_y^2 w^\var  w^\var)}{\vr^\var+\vrf} -\frac{{\p_x \vr^\var w^\var}\p_y^2 w^\var}{(\vr^\var+\vrf)^2}-\frac{\var \p_x^2 \vr^\var}{(\vr^\var+\vrf)^2} \p_y^3 w^\var\Big)\\
			&+   2 \var \p_y^2 w^\var  \p_x^2 \vr^\var
			\left(\frac{-w^\var\p_x \vr^\var +\var \p_x^2 \p_y \vr^\var}{(\vr^\var+\vrf)^2}-\frac{2 \var  \p_y \vr^\var \p_x^2 \vr^\var  }{(\vr^\var+\vrf)^3}\right)
			+2\frac{\var \p_y \vr^\var \p_x^2 \vr^\var }{(\vr^\var+\vrf)^2}\left(\var \p_x^2\p_y^2 w^\var +\p_y^3(\frac{\p_y w^\var}{\vr^\var+\vrf})\right)\\
			&+3\p_x \p_y^2 w^\var \left(\p_y(\frac{\p_y w^\var}{\vr^\var+\vrf})+\var \p_x^2 w^\var\right)
			+3 w^\var \p_x \p_y^2\left\{\p_y(\frac{\p_y w^\var}{\vr^\var+\vrf})+\var \p_x^2 w^\var\right\}\\
			&-2\p_y^2 w^\var \p_x\left\{\p_y(\frac{\p_y w^\var}{\vr^\var+\vrf})+\var \p_x^2 w^\var\right\}
			-2\p_x w^\var \p_y^2\left\{\p_y(\frac{\p_y w^\var}{\vr^\var+\vrf})+\var \p_x^2 w^\var\right\}\\
			&-\var\p_x^2\left\{ \p_x(\p_y^2 w^\var  w^\var) -\frac{{\p_x \vr^\var w^\var}\p_y^2 w^\var}{\vr^\var+\vrf}
			+ 2\p_y^2 w^\var(\frac{-w^\var\p_x \vr^\var +\var \p_x^2 \p_y \vr^\var}{\vr^\var+\vrf}-\frac{2 \var \p_x^2 \p_y \vr^\var  }{(\vr^\var+\vrf)^2})\right\}\\
			&-\var\p_x^2\left\{2\frac{\p_y \vr^\var}{\vr^\var+\vrf}(\var \p_x^2\p_y^2 w^\var +\p_y^3(\frac{\p_y w^\var}{\vr^\var+\vrf}))-\frac{\var \p_x^2 \vr^\var}{\vr^\var+\vrf} \p_y^3 w^\var \right\}
			\\
			&-\sum_{1\le k \le 3}C_4^k \p_y^k(\frac{1}{\vr^\var+\vrf})
			\p_y^{5-k}\left\{\var \p_x^2 w^\var+\p_y(\frac{\p_y w^\var}{\vr^\var+\vrf})-(u^\var \p_x w^\var+v^\var \p_y w^\var)\right\}\\
			&+\frac{\var \p_x^2 \vr^\var}{(\vr^\var+\vrf)^2} \p_y^5 w^\var
			-\sum_{1\le k \le 3}C_4^k \p_y^k
			\left\{\frac{u^\var \p_x \vr^\var+v^\var \p_y \vr^\var-\var \p_x ^2 \vr^\var}{(\vr^\var+\vrf)^2}\right\}\p_y^{5-k}w^\var\\
			\overset{def}{=}&N_2.
		\end{aligned}
		\deqq
		Substituting  \eqref{a206} and \eqref{a225} into \eqref{a207},
		on the boundary $y=0$, we have
		\beq\label{a226}
		\begin{aligned}
			\frac{1}{\vr^\var+\vrf}\p_y^7 w^\var
			=&(\vr^\var+\vrf)N_2-\var \p_x^2 N_1
			+\p_y^5(u^\var \p_x w^\var +v^\var \p_y w^\var)
			-\sum_{1\le k \le 5}C_6^k \p_y^k(\frac{1}{\vr^\var+\vrf})\p_y^{7-k}w^\var.
		\end{aligned}
		\deq
		Thus, \eqref{a226} implies the boundary reduction relation $\eqref{bl-redu}_3$.
		Therefore, we complete the proof of the boundary reduction relation \eqref{bl-redu}.

		\section{Good unknown equations}\label{good-unkonw}
		In this section, we will give the proof for the equations
		\eqref{206} and \eqref{2011}.
		
		\begin{proof}[{Proof of \eqref{206}}]
			Indeed, subtracting \eqref{205} from \eqref{201}, we obtain
			\beq\label{a301}
			\begin{aligned}
				&\;\p_t w_g^\var +u^\var \p_x w_g^\var +v^\var \p_y w_g^\var-\var \p_x^2 w_g^\var
				-\p_x^6\p_y(\frac{1}{\vr^\var+\vrf}\p_y w^\var)
				+\p_x^6(\frac{1}{\vr^\var+\vrf}\p_y w^\var)g_w^\var\\
				=&\;Q_1-Q_2\cdot g_w^\var-\p_x^6 u^\var(\p_t g_w^\var+u^\var \p_x g_w^\var+v^\var \p_y g_w^\var-\var \p_x^2 g_w^\var)
				+2\var\p_x^7 u^\var \p_x g_w^\var.
			\end{aligned}
			\deq
			It is easy to check that
			\beqq
			\begin{aligned}
				\frac{1}{\vr^\var+\vrf}\p_x^6 \p_y w^\var g_w^\var
				=&\p_y\left\{\frac{1}{\vr^\var+\vrf}\p_y\p_x^6 u^\var g_w^\var\right\}
				-\p_x^6 w^\var \p_y\left\{\frac{1}{\vr^\var+\vrf}g_w^\var\right\}\\
				=&\p_y\left\{\frac{1}{\vr^\var+\vrf}\p_y(\p_x^6 u^\var g_w^\var)\right\}
				-\p_y\left\{\frac{1}{\vr^\var+\vrf}\p_x^6 u^\var \p_y g_w^\var\right\}
				-\p_x^6 w^\var \p_y\left\{\frac{1}{\vr^\var+\vrf}g_w^\var\right\}.
			\end{aligned}
			\deqq
			Thus, we have
			\beq\label{a302}
			\begin{aligned}
				&-\p_x^6\p_y\{\frac{1}{\vr^\var+\vrf}\p_y w^\var\}
				+\p_x^6\{\frac{1}{\vr^\var+\vrf}\p_y w^\var\}g_w^\var\\
				=&-\p_y\left\{\frac{1}{\vr^\var+\vrf}\p_y \p_x^6 w^\var\right\}
				-\p_y \left\{\sum_{1\le k\le 6}C_6^k \p_x^k(\frac{1}{\vr^\var+\vrf})\p_x^{6-k}\p_y w^\var\right\}
				+\p_y\left\{\frac{1}{\vr^\var+\vrf}\p_y(\p_x^6 u^\var g_w^\var)\right\}\\
				&-\p_y\left\{\frac{1}{\vr^\var+\vrf}\p_x^6 u^\var \p_y g_w^\var\right\}
				-\p_x^6 w^\var \p_y\left\{\frac{1}{\vr^\var+\vrf}g_w^\var\right\}
				+\sum_{1\le k\le 6}C_6^k \p_x^k(\frac{1}{\vr^\var+\vrf})\p_x^{6-k}\p_y w^\var g_w^\var\\
				=&-\p_y\left\{\frac{1}{\vr^\var+\vrf}\p_y\{\p_x^6 w^\var-\p_x^6 u^\var g_w^\var\}\right\}+Q_3^*,
			\end{aligned}
			\deq
			where $Q_3^*$ is defined as follows
			\beqq
			\begin{aligned}
				Q_3^*=&-\p_y\left\{\sum_{1\le k \le 6}C_6^k \p_x^6(\frac{1}{\vr^\var+\vrf})\p_x^{6-k}\p_y w^\var\right\}
				-\p_y\left\{\frac{1}{\vr^\var+\vrf}\p_x^6 u^\var \p_y g_w^\var\right\}\\
				&-\p_x^6 w^\var \p_y\left\{\frac{1}{\vr^\var+\vrf}g_w^\var\right\}
				+\sum_{1\le k \le 6}C_6^k \p_x^k(\frac{1}{\vr^\var+\vrf})\p_x^{6-k}\p_y w^\var g_w^\var.
			\end{aligned}
			\deqq
			Using the vorticity equation, it is easy to check that
			\beq\label{a303}
			\begin{aligned}
				&\p_t g_w^\var+u^\var \p_x g_w^\var+v^\var \p_y g_w^\var-\var \p_x^2 g_w^\var\\
				=&\;\frac{1}{w^\var}\p_t \p_y w^\var-\frac{\p_y w^\var}{(w^\var)^2}\p_t w^\var
				+u^\var \p_x(\frac{\p_y w^\var}{w^\var})+v^\var \p_y (\frac{\p_y w^\var}{w^\var})-\var \p_x^2 (\frac{\p_y w^\var}{w^\var})\\
				=&\;  Q_4.
			\end{aligned}
			\deq
			Then, substituting the quantities \eqref{a302} and \eqref{a303}
			into \eqref{a301}, we have
			\beqq
			\p_t w_g^\var+u^\var \p_x w_g^\var +v^\var \p_y w_g^\var-\var \p_x^2 w_g^\var
			-\p_y\left\{\frac{1}{\vr^\var+\vrf}\p_y w_g^\var\right\}
			=Q_1-Q_2\cdot g_w^\var-Q_3^*-Q_4\cdot \p_x^6 u^\var
			+2\var\p_x^7 u^\var \p_x g_w^\var.
			\deqq
			Therefore, we complete the proof of equation \eqref{206}.
		\end{proof}
		
		\begin{proof}[Proof of \eqref{2011}]
			Indeed, subtracting \eqref{2010} from \eqref{208}, we obtain
			\beq\label{a304}
			\begin{aligned}
				&\p_t \vr_g^\var+u^\var \p_x \vr_g^\var+v^\var \p_y \vr_g^\var-\var\p_x^2 \vr_g^\var\\
				=&\;Q_5-Q_6\cdot g_\vr^\var
				-\p_x^6 u^\var(\p_t g_\vr^\var+u^\var \p_x g_\vr^\var+v^\var \p_y g_\vr^\var-\var \p_x^2 g_\vr^\var)
				+2\var\p_x^7 u^\var \p_x g_\vr^\var.
			\end{aligned}
			\deq
			where $Q_6$ is defined as follows
			\beqq
			Q_6\overset{def}{=}Q_2+\p_x^6 \left\{ \frac{1}{\vr^\var+\vrf}\p_y w^\var\right\}.
			\deqq
			Using the density and vorticity equation, we have
			\beq\label{a305}
			\p_t g_\vr^\var+u^\var \p_x g_\vr^\var+v^\var \p_y g_\vr^\var-\var \p_x^2 g_\vr^\var
			=\frac{1}{w}\p_t \p_y \vr^\var-\frac{\p_y \vr^\var}{(w^\var)^2}\p_t w^\var
			+u^\var \p_x (\frac{\p_y \vr^\var}{w^\var})+v^\var \p_y (\frac{\p_y \vr^\var}{w^\var})-\var \p_x^2 (\frac{\p_y \vr^\var}{w^\var})
			\overset{def}{=}Q_7,
			\deq
			where $Q_7$ is defined as follows
			\beqq
			\begin{aligned}
				Q_7\overset{def}{=}&-\frac{\p_y(u^\var \p_x \vr^\var+v^\var \p_y \vr^\var-\var \p_x^2 \vr^\var)}{w^\var}
				-\frac{\p_y \vr^\var}{(w^\var)^2}
				\left(-u^\var \p_x w^\var-v^\var \p_y w^\var+\p_y(\frac{1}{\vr^\var+\vrf}\p_y w^\var)+\var \p_x^2 w^\var\right)\\
				&+u^\var \left(\frac{\p_{xy}\vr^\var}{w^\var}-\frac{\p_y \vr^\var \p_x w^\var}{(w^\var)^2}\right)
				+v^\var\left(\frac{\p_y^2 \vr^\var}{w^\var}-\frac{\p_y \vr^\var \p_y w^\var}{(w^\var)^2}\right)
				-\var \p_x^2 (\frac{\p_y \vr^\var}{w^\var}).
			\end{aligned}
			\deqq
			Therefore, substituting the quantity \eqref{a305} into \eqref{a304},
			we complete the proof of \eqref{2011}.
		\end{proof}

		\section{Compatibility of initial data}\label{appendix-com}
		
		In this section, we will give the control of initial data involved vertical
		velocity that is not given initial data.
		Indeed, we will introduce the compatibility of initial data at the
		corner $(0, 0)$. This can explain the reason why initial data required
		in Theorem \ref{main-result-steady}.
		First of all, from the equation $\eqref{s-Prandtl}_2$, we have
		\beq\label{c001}
		{\rho} u \p_{y}(\f{v}{u})|_{x=0}
		=-\frac{\p_y^2 u}{u}|_{x=0}
		=-\frac{\p_y^2 u_0}{u_0}.
		\deq
		In order to make sure that $\frac{\p_y^2 u_0}{u_0}\in L^2(\mathbb{R}^+)$,
		we should at least require
		\beq\label{c01}
		\p_y^2 u_0|_{y=0}=0.
		\deq
		On the other hand, by evaluating the equation $\eqref{s-Prandtl}_2$
		at the boundary $y=0$, it is easy to check that
		\beqq
		\p_y^2 u |_{y=0}=0.
		\deqq
		Thus, we have $\underset{x \rightarrow 0^+}{\lim}\p_y^2 u |_{y=0}=0$.
		Therefore, the initial data condition \eqref{c01} essentially
		implies the initial compatibility condition at the corner $(0,0)$.
		
		Secondly, from the equation $\eqref{s-Prandtl}$, we have
		\beqq
		\rho u \p_{xy}(\frac{v}{u})
		=-\p_x \rho u \p_y (\frac{v}{u})
		-2 \rho \p_x u \p_y (\frac{v}{u})
		+\frac{\p_y^3 v}{u}
		=\frac{v}{u}\p_y \rho u \p_y (\frac{v}{u})
		+2 \rho \p_y v \p_y (\frac{v}{u})
		+\frac{\p_y^3 v}{u},
		\deqq
		which yields directly
		\beqq
		\rho u \p_{xy}(\frac{v}{u})|_{x=0}
		=\left(\frac{v}{u}\p_y \rho u \p_y (\frac{v}{u})
		+2 \rho \p_y v \p_y (\frac{v}{u})\right)_{x=0}
		+\frac{\p_y^3 v}{u}|_{x=0}.
		\deqq
		From the equation \eqref{s-Prandtl}, we have
		\beq\label{v-ex}
		v=-u \int_0^y \frac{\p_y^2 u}{\rho u^2}d\tau,
		\deq
		and hence, we define
		\beqq
		v_0(y)\overset{def}{=}-u_0 \int_0^y \frac{\p_y^2 u_0}{\rho_0 u_0^2}d\tau.
		\deqq
		Thus, we have
		\beqq
		\rho u \p_{xy}(\frac{v}{u})|_{x=0}
		=\left\{\frac{v}{u}\p_y \rho u \p_y (\frac{v}{u})
		+2 \rho \p_y v \p_y (\frac{v}{u})\right\}|_{x=0}
		+\frac{\p_y^3 v_0}{u_0}.
		\deqq
		In order to make sure that
		$\sqrt{\rho} u \p_{xy}(\frac{v}{u})|_{x=0}\in L^2(\mathbb{R}^+)$
		and avoid the singularity at $y=0$,
		we should at least require
		\beq\label{c02}
		\p_y^3 v_0(y)|_{y=0}=0.
		\deq
		On the other hand,  from the equation $\eqref{s-Prandtl}_2$
		and boundary condition $\eqref{s-Prandtl}_4$, one may check that
		\beqq
		0=\p_x (\rho u \p_x u+\rho v \p_y u-\p_y^2 u)|_{y=0}
		=\p_y^3 v|_{y=0},
		\deqq
		which implies
		$\underset{x \rightarrow 0^+}{\lim}\p_y^3 v|_{y=0}=0$.
		Therefore, the initial data condition \eqref{c02} essentially
		implies the initial compatibility condition at the corner $(0,0)$.
		In our paper, the initial compatibility  data \eqref{c02}
		is called the first order generic compatibility condition
		at the corner $(0, 0)$.
		Similarly, taking $\p_x^m(m \ge 2)$ operator
        to the equation $\eqref{s-Prandtl}_2$, we have
		\beqq
		\rho u \p_x^m \p_{y} (\frac{v}{u})=
		-\frac{1}{u}\sum_{1\le k \le m}C^m_k \p_x^k (\rho u^2)
		\p_x^{m-k}\p_{y}(\frac{v}{u})+\frac{\p_x^{m-1} \p_y^3 v}{u}.
		\deqq
		In order to make sure that
		$\sqrt{\rho} u \p_x^m \p_{y}(\frac{v}{u})|_{x=0}\in L^2(\mathbb{R}^+)$
		and avoid the singularity at $y=0$,
		we should at least require
		\beq\label{c0m}
		\left(\p_x^{m-1}\p_y^3 v(x,y)|_{x=0}\right)_{y=0}=0.
		\deq
		On the other hand, we have
		\beqq
		0=\p_x^m(\rho u \p_x u+\rho v \p_y u-\p_y^2 u)|_{y=0}
		=\p_x^{m-1}\p_y^3 v|_{y=0},
		\deqq
		which implies
		$\underset{x \rightarrow 0^+}{\lim}\p_x^{m-1} \p_y^3 v|_{y=0}=0$.
		Therefore, the initial data condition \eqref{c02} essentially
		implies the initial compatibility condition at the corner $(0,0)$.
		In our paper, the initial compatibility  data \eqref{c0m}
		is called the m-th order generic compatibility condition
		at the corner $(0, 0)$.
		At the same time, we will explain the initial data $u_0$ satisfying
		this compatibility condition.
		From the equation \eqref{v-ex}, we have
		\beqq
		\begin{aligned}
			\p_y^3 v
			=&-\p_y^3 u \int_0^y \frac{\p_{y'}^2 u}{\rho u^2}d y'
			-\frac{2 (\p_{y}^2 u)^2}{\rho u^2}
			-\p_y u \left( \p_y(\frac{1}{\rho})\frac{\p_y^2 u}{u^2}
			+\frac{1}{\rho}(\frac{\p_y^3 u}{u^2}-\frac{2\p_y^2 u \p_y u}{u^3})\right)\\
			&-\p_y^2(\frac{1}{\rho})\frac{\p_y^2 u}{u}
			-2\p_y(\frac{1}{\rho})\left(\frac{\p_y^2 u}{u}-\frac{\p_y^2 u \p_y u}{u^2}\right)
			-\frac{1}{\rho}\left(\frac{\p_y^4 u}{u}-\frac{2 \p_y^3 u \p_y u}{u^2}
			-\frac{(\p_y^2 u)^2}{u^2}+\frac{2\p_y^2 u (\p_y u)^2}{u^3}\right).
		\end{aligned}
		\deqq
		In order to make sure that $\p_y^3 v_0(y)|_{y=0}=0$,
		we require the initial data of $u_0$ to satisfy
		\beqq
		\p_y^3 u_0|_{y=0}=\p_y^4 u_0|_{y=0}=\p_y^5 u_0|_{y=0}=0.
		\deqq
		Now, let us consider the initial data
		$\|\sqrt{\rho}u \p_{y}^2(\frac{v}{u}) \xw^{\ssg}|_{x=0}\|_{L^2}$.
		Indeed, from the equation $\eqref{s-Prandtl}_2$, we have
		\beq\label{c021}
		\rho u \p_y^2(\frac{v}{u})|_{x=0}=
		-\left(\p_y \rho u \p_y(\frac{v}{u})
		+2\rho \p_y u \p_y(\frac{v}{u})\right)_{x=0}-\frac{\p_y^3 u_0}{u_0},
		\deq
		and
		\beq\label{c022}
		0=\p_y(\rho u \p_x u+\rho v \p_y u-\p_y^2 u)|_{y=0}
		=-\p_y^3 u|_{y=0}.
		\deq
		In order to avoid the singularity of $y$ in \eqref{c021}, we require the initial data
		of $u_0$ satisfying
		\beq\label{c03}
		\p_y^3 u_0(y)|_{y=0}=0.
		\deq
		The equation \eqref{c022} implies directly
		$\underset{x \rightarrow 0^+}{\lim}\p_y^3 u|_{y=0}=0$.
		Therefore, the initial data condition \eqref{c03} essentially
		implies the initial compatibility condition at the corner $(0,0)$.
		Finally, let us control the initial data $\p_x \rho$.
		Indeed, from the equation $\eqref{s-Prandtl}_1$, it is easy to check that
		\beqq
		\|\p_x \rho|_{x=0}\|_{L^2}
		\le \|(\frac{v}{u})|_{x=0}\p_y \rho_0 \|_{L^2}
		\le \|\p_y(\frac{v}{u})\xw|_{x=0}\|_{L^2}\|\p_y \rho_0 \|_{L^2}
		\deqq
		The conditions \eqref{c001} and  \eqref{c03} implies directly
		$\p_x \rho|_{x=0} \in L^2(\mathbb{R}^+)$.
		Therefore, the initial data $\mathcal{X}(0)$
can be controlled of  in terms of $\rho_0$ and $u_0$.

		\section{Derivative of inhomogeneous Prandtl equations}\label{appendix-derivation}
		
		In this section, we will give a derivation of the Prandtl equations \eqref{Prandtl}.
		We consider the two-dimensional inhomogeneous incompressible Navier-Stokes equations with a small viscosity coefficient $\mu$ in a period domain $\{(x, y)|(x, y) \in \mathbb{T} \times \mathbb{R}^{+} \}:$
		\beq\label{0-NS}
		\left\{\begin{aligned}
			&\p_t \rho^{\mu}+ {\rm div}(\rho^{\mu} \mathbf{u^{\mu}})=0,\\
			&\rho^{\mu} \p_t \mathbf{u^{\mu}}+ \rho^{\mu} (\mathbf{u^{\mu}} \cdot \nabla)\mathbf{u^{\mu}}+\nabla p^{\mu}= \mu \triangle \mathbf{u^{\mu}},\\
			&{\rm div} \mathbf{u^{\mu}}=0,
		\end{aligned}\right.
		\deq
		where $t>0$,~$\mathbf{u^{\mu}}=(u^{\mu},v^{\mu})$ denotes the velocity, and $p^{\mu}$ represents the pressure of fluid. To complete the system \eqref{0-NS}, the initial data is given by
		$$
		\mathbf{u^{\mu}}(0, x, y)= \mathbf{u}^{\mu}_0(x, y).
		$$
		The boundary conditions are given by
		\beq\label{boundary}
		u^{\mu}(t, x, 0) = v^{\mu}(t, x, 0).
		\deq
		
		As the viscosity coefficient $\mu$ tends to zero, we obtain the following system formally.
		\beqq
		\left\{\begin{aligned}
			&\p_t \rho^0+ {\rm div}(\rho^0 \mathbf{u^{0}})=0,\\
			&\rho^{0} \p_t \mathbf{u^{0}}+ \rho^{0} (\mathbf{u^{0}} \cdot \nabla)\mathbf{u^{0}}+\nabla p^{0}=0,\\
			&{\rm div} \mathbf{u^{0}}=0.
		\end{aligned}\right.
		\deqq
		To find out the terms in \eqref{0-NS} whose contributions are essential for the boundary layer, we use the following scale transformation:
		$$ t=t,x=x,\tilde{y}=\mu^{-\f12}y,$$
		and set
		\beqq
		\begin{aligned}
			\rho(t,x,\tilde{y})&=\rho^{\mu}(t,x,y),\quad
			p(t,x,\tilde{y})=p^{\mu}(t,x,y),\\
			u(t,x,\tilde{y})&=u^{\mu}(t,x,y),\quad
			v(t,x,\tilde{y})=\mu^{-\f12} v^{\mu}(t,x,y).\\
		\end{aligned}
		\deqq
		Taking the leading order, the system \eqref{0-NS} is reduced to
		\beq\label{1-NS}
		\left\{\begin{aligned}
			&\rho_t + u \p_{x} \rho + v \p_{\tilde{y}} \rho =0,\\
			&\rho \p_t u + \rho u \p_{x} u + \rho v \p_{\tilde{y}} u  +  \p_{x} p = \p_{\tilde{y}}^2 u,\\
			& \p_{\tilde{y}} p=0,\\
			&\p_{x} u+\p_{\tilde{y}} v =0.
		\end{aligned} \right.
		\deq
		The third equation in the system \eqref{1-NS} implies that the leading order of boundary layers for the total pressure $p(t, x,\tilde{y})$ is invariant across the boundary layer, and should be matched to the outflow pressure $p(t, x)$ on top of boundary layer, that is, the trace of pressure of ideal MHD flow. Hence, we obtain
		$$p(t, x,\tilde{y})\equiv p(t, x).$$
		Furthermore, the density $\rho(t, x,\tilde{y})$, tangential component $u(t, x,\tilde{y})$ of velocity flied, should match the outflow density $\rho_e (t, x)$ and tangential velocity $u_e(t, x)$, on the top of boundary layer, that is
		$$\rho(t, x,\tilde{y}) \to \rho_e(t, x), u(t, x,\tilde{y}) \to u_e(t, x),  \text{as} ~ \tilde{y} \to +\infty.$$
		Then, we have the following matching conditions:
		\beqq
		\left\{\begin{aligned}
			&\p_t \re +\ue \p_x \re=0,\\
			&\re \p_t \ue+\re \ue \p_x \ue+\p_x p=0.\\
		\end{aligned}\right.
		\deqq
		Moreover, by virtue of the boundary condition \eqref{boundary}, one attains the following boundary condition
		\beqq
		u(t, x, 0) = v(t, x, 0)= 0.
		\deqq
		Therefore, one can obtain the following Prandtl system
		\beqq
		\left\{\begin{aligned}
			&\p_t \rho+u \p_x \rho+ v \p_y \rho=0,\\
			&\rho \p_t u+\rho u \p_x u+\rho v \p_y u+\p_x p
			-\p_y^2 u=0,\\
			&\p_x u+\p_y v=0,\\
			&(u, v)|_{y=0}=0,
			\underset{y\rightarrow +\infty}{\lim}(\rho, u)(t,x,y)
			=(\re(t,x),\ue(t,x)),\\
			&(\rho, u)|_{t=0}=(\rho_0, u_{0})(x, y),
		\end{aligned}\right.
		\deqq
		here we have replaced $\tilde{y}$ by $y$ for simplicity of notations.
		Thus, we obtain the Prandtl system \eqref{Prandtl}.
		
	\end{appendices}
	\phantomsection
	
\end{document}